\documentclass[12pt]{article}
\usepackage{amsthm,amsfonts,amssymb,epsfig,graphics,amsmath,amsbsy,color,enumerate,tikz}

\theoremstyle{plain}
\newtheorem{theorem}{Theorem}[section]
\newtheorem{lemma}[theorem]{Lemma}
\newtheorem{proposition}[theorem]{Proposition}
\newtheorem{remark}[theorem]{Remark}
\newtheorem{definition}[theorem]{Definition}

\newtheorem{claim}[theorem]{Claim}

\newcommand{\mas}{\operatorname{Mas}}

\newcommand{\dom}{\operatorname{dom}}
\newcommand{\ran}{\operatorname{ran}}

\newcommand{\loc}{\operatorname{loc}}

\newcommand{\AC}{\operatorname{AC}}
\newcommand{\Span}{\operatorname{Span}}
\newcommand{\rank}{\operatorname{rank}}
\newcommand{\nullity}{\operatorname{nullity}}
\newcommand{\ess}{\operatorname{ess}}
\newcommand{\pt}{\operatorname{pt}}
\newcommand{\im}{\operatorname{Im}}

\numberwithin{equation}{section}
\numberwithin{figure}{section}

\title{Renormalized Oscillation Theory for Singular Linear 
Hamiltonian Pencils}

\author{Peter Howard and Alim Sukhtayev}

\begin{document}

\maketitle

\begin{abstract} 
For many applications, critical information about system dynamics 
is encoded in associated eigenvalue problems that can be posed as 
linear Hamiltonian systems with suitable boundary conditions. Motivated
by examples from hydrodynamics, quantum mechanics, and magnetohydrodynamics (MHD),
we develop a general framework for analyzing a broad class of 
linear Hamiltonian systems  with at least one singular boundary condition 
and possible nonlinear dependence on the spectral parameter. We show that 
renormalized oscillation results can be obtained in a 
natural way through consideration of the Maslov index 
associated with appropriately chosen paths of Lagrangian 
subspaces of $\mathbb{C}^{2n}$. This extends previous work
by the authors for regular linear Hamiltonian systems
that depend nonlinearly on the spectral parameter 
and singular linear Hamiltonian systems that depend linearly 
on the spectral parameter. We conclude the study by using our 
framework to study the spectrum in the setting of each of our motivating 
examples. 
\end{abstract}

\section{Introduction} \label{new-introduction}

In this study, we employ renormalized oscillation theory 
(via the Maslov index) to analyze the spectrum associated 
with a general class of differential operators 
that depend nonlinearly on their spectral parameter 
$\lambda$ and have one or two singular endpoints. 
Such operators arise naturally in a wide
range of applications, including three that will 
be emphasized in this work, from hydrodynamics, 
quantum mechanics, and magnetohydrodynamics (MHD). For 
the first of these, it's shown in \cite{SYZ2020}
that when the Saint--Venant model for inclined shallow 
water flow is linearized about a hydraulic shock 
profile, the resulting eigenvalue problem takes the 
form 
\begin{equation} \label{saint-venant-evp}
    A v' = (E - \lambda I - A_x) v,
\end{equation}
where 
\begin{equation}
    A (x) = \begin{pmatrix}
    - c & 1 \\
    \frac{H (x)}{F^2} - \frac{Q(x)^2}{H(x)^2} & \frac{2 Q(x)}{H(x)} - c
    \end{pmatrix},
    \quad 
    E (x) = \begin{pmatrix}
        0 & 0 \\
        \frac{2 Q(x)^2}{H(x)^3} + 1 & - \frac{2 Q(x)}{H(x)^2}
    \end{pmatrix}.
\end{equation}
Here, $(H(x),Q(x))$ comprises the hydraulic shock profile, 
$c$ denotes the profile wave speed, and $F > 0$ is the {\it Froude}
number (a nondimensional constant; see Section \ref{hydraulic-shocks-section} 
below for details on 
the Saint--Venant system). In the case of a 
discontinuous profile, it's shown in \cite{YZ2020} that 
(\ref{saint-venant-evp}) can be understood on the half-line
$(- \infty, 0)$ with right boundary condition 
\begin{equation} \label{saint-venant-bc}
    \Big(\lambda Q(0) - H(0) + \frac{|Q(0)|Q(0)}{H(0)^2}, \,\, - \lambda H(0) \Big) A(0) v(0) = 0.
\end{equation}
We will observe in Section \ref{hydraulic-shocks-section} below that with suitable transformations, the 
system \eqref{saint-venant-evp}--\eqref{saint-venant-bc} can be expressed 
as 
\begin{equation} \label{main-form-intro}
    Jy' = \mathbb{B} (x; \lambda) y, \quad \alpha (\lambda) y(0) = 0,
    \quad y(x;\lambda) \in \mathbb{C}^2,
\end{equation}
where 
\begin{equation} \label{J-defined2D}
    J = \begin{pmatrix} 0 & -1 \\ 1 & 0 \end{pmatrix},
\end{equation}
$\mathbb{B} (x; \lambda)$ is self-adjoint for all $x \in (-\infty, 0)$,
$\lambda > 0$ (see \eqref{isl-B}), and $\alpha (\lambda)$ is a row vector depending 
on $\lambda$. Equations with form \eqref{main-form-intro} will be addressed
in Theorem \ref{regular-singular-theorem}. 
 
For the application to quantum mechanics, it's known that 
the potential function in Schr\"odinger's wave equation 
can depend on energy encoded in the equation's eigenvalues, 
giving rise to ODE of the form 
\begin{equation} \label{schrodinger-intro}
    - \psi'' + V (x; \lambda) \psi = \lambda \psi,
\end{equation}
for which a typical form of $V(x; \lambda)$ (taken from 
\cite{JJ1972}) is 
\begin{equation*}
    V^{\pm} (x; \lambda)
    = U(x) \pm 2 \sqrt{\lambda} Q(x),
\end{equation*}
where $U(x)$ and $Q(x)$ are appropriately specified functions 
(see Section \ref{quadratic-schrodinger-example} for particular
examples). Expressed as a first-order system with components
$y_1 = \phi$ and $y_2 = \phi'$, (\ref{schrodinger-intro}) 
becomes 
\begin{equation} \label{nonlinear-schrodinger-form}
    J y' = \mathbb{B} (x; \lambda) y,
    \quad 
    \mathbb{B} (x; \lambda)
    = \begin{pmatrix}
        \lambda - V(x; \lambda) & 0 \\
        0 & 1
    \end{pmatrix}, \quad x \in \mathbb{R}
\end{equation}
which is the same form as \eqref{main-form-intro}, except
with domain $\mathbb{R}$ rather than $(-\infty, 0)$. Equations
of form \eqref{nonlinear-schrodinger-form} will be addressed in 
Theorem \ref{singular-theorem}. 

In Section \ref{mhd-section}, we take from \cite{GP2004}
a model of incompressible ideal MHD flow, and show that when 
this system is linearized about certain stationary solutions 
the resulting eigenvalue problem can be expressed as 
\begin{equation} \label{mhd-intro} 
    - (P(x;\lambda) \phi')' + V(x; \lambda) \phi = 0,
\end{equation}
where the coefficient function $V (x; \lambda)$ is nonlinear
in the spectral parameter $\lambda$ (see equations (\ref{mhd-P})
and (\ref{mhd-V}) below for the specifications of $P (x; \lambda)$
and $V(x; \lambda)$ respectively). Similarly as with 
(\ref{schrodinger-intro}), (\ref{mhd-intro}) can be expressed 
as a first-order system with form \eqref{main-form-intro}, this 
time with natural domain $(0, \infty)$.

In order to develop a framework in which to study the applications
described above, along with numerous others, we consider systems of the general form 
\begin{equation} \label{hammy-intro}
    J y' = \mathbb{B} (x; \lambda) y, 
    \quad x \in (a, b), \quad y(x; \lambda) \in \mathbb{C}^n,
\end{equation}
where $J$ is as in \eqref{J-defined2D} with each 1 replaced by 
an $n \times n$ identity matrix, $\mathbb{B} (x; \lambda) \in \mathbb{C}^{n \times n}$ 
is self-adjoint for a.~e.~$x \in (a, b)$, $\lambda \in I \subset \mathbb{R}$, and  
\begin{equation*}
    - \infty \le a < b \le + \infty,
\end{equation*}
either with two singular endpoints or with one singular endpoint 
and one regular endpoint with boundary condition expressed as 
$\alpha (\lambda) y(a) = 0$ or $\alpha (\lambda) y(b) = 0$. (The 
case of two regular endpoints has been analyzed previously 
in \cite{HS2022}). 

For each of our motivating applications, it is important to 
understand the eigenvalues of \eqref{hammy-intro}, by which we 
mean the values $\lambda \in I$ for which \eqref{hammy-intro} 
has a solution in an appropriate function space, determined by 
the nature of the endpoints and the boundary conditions. However, 
the possibility of singular endpoints along with the possibly 
nonlinear dependence on $\lambda$ complicate the specification of
such spaces, and indeed such specification is a key point of the 
current analysis. 

Precisely, our goals for the current study are as follows: 

\vspace{.1in}
\noindent
1. Identify general conditions on the matrix function $\mathbb{B} (x; \lambda)$
under which the spectrum of \eqref{hammy-intro} can be analyzed by 
our geometric approach based on the Maslov index;

\vspace{.1in}
\noindent
2. For singular endpoints, determine suitable boundary conditions to 
impose on \eqref{hammy-intro}, and relate these to physically relevant 
behavior; 

\vspace{.1in}
\noindent
3. With $\mathbb{B} (x; \lambda)$ chosen, along with suitable boundary 
conditions, compute the number of eigenvalues $\mathcal{N} ([\lambda_1, \lambda_2))$
that the specified eigenvalue problem has on the interval $[\lambda_1, \lambda_2)$. 

\vspace{.1in}

Our approach to the problem will be through oscillation theory, and in particular 
the {\it renormalized oscillation} approach introduced in \cite{GST1996}. Developed
initially in the context of single Sturm-Liouville equations, this method has 
been extended in \cite{Teschl1996, Teschl1998} to the cases of Jacobi and 
Dirac operators, in \cite{GZ2017, HS2022} to \eqref{hammy-intro} when 
$\mathbb{B} (x; \lambda)$ depends linearly on $\lambda$, in \cite{HS2} to 
\eqref{hammy-intro} when $\mathbb{B} (x; \lambda)$ depends nonlinearly on $\lambda$
(but restricted to regular endpoints), and in \cite{Howard2022} when 
$\mathbb{B} (x; \lambda)$ is not self-adjoint (again restricted to the case of 
regular endpoints). The current analysis extends the approach of \cite{HS2,HS2022} 
to the setting in which $\mathbb{B} (x; \lambda)$ both depends nonlinearly on $\lambda$ and 
has one or two singular endpoints.  

{\it Plan of the paper}. In Section \ref{old-introduction}, we develop the 
notation and background needed to state our main results, and we state these 
results as Theorems \ref{regular-singular-theorem} and \ref{singular-theorem}. 
In Section \ref{operator-section}, we prove Lemma \ref{self-adjoint-operator-lemma}, 
establishing the existence and nature of the family of self-adjoint 
pencils $\mathcal{T} (\lambda)$ and $\mathcal{T}^{\alpha} (\lambda)$ 
that will be the principal objects of our study, and in Section 
\ref{extension-section}, we construct the frames $\mathbf{X}_a (x; \lambda)$ 
and $\mathbf{X}_b (x; \lambda)$, introduced respectively in (\ref{frame-a}) 
and (\ref{frame-b}). In Section \ref{maslov-section}, we 
provide additional background on the Maslov index, along with 
some results that will be needed for the subsequent analysis, 
and in Section \ref{theorems-section}, we prove Theorems 
\ref{regular-singular-theorem} and \ref{singular-theorem}. Finally, 
in Section \ref{applications-section}
we conclude with the three specific illustrative applications discussed
in this introduction.

{\it Notational conventions}. Throughout the analysis, we 
will use the notation $\|\cdot\|_{\mathbb{B}_{\lambda}}$ and 
$\langle \cdot, \cdot \rangle_{\mathbb{B}_{\lambda}}$ respectively for our 
weighted norm and inner product. In the case that 
(\ref{hammy}) is regular at $x = a$, we will denote 
the associated map of Lagrangian subspaces by $\ell_{\alpha}$,
and we will denote by $\mathbf{X}_{\alpha}$ a specific
corresponding map of frames. Likewise, if (\ref{hammy})
is singular at $x = a$, we will use $\ell_a$ and $\mathbf{X}_a$,
and for $x = b$ (always assumed singular), we will use 
$\ell_b$ and $\mathbf{X}_b$. In order to accommodate limits
associated with our bilinear form, we will adopt the notation 
\begin{equation*}
    (Jy,z)_a := \lim_{x \to a^+} (J y(x), z(x)); 
    \quad (Jy,z)_b := \lim_{x \to b^-} (J y(x), z(x)),
\end{equation*}
along with 
\begin{equation*}
 (Jy, z)_a^b := (Jy,z)_b - (Jy,z)_a.   
\end{equation*}
Here and throughout, we use 
$(\cdot, \cdot)$ to denote the usual inner product 
for $\mathbb{C}^{2n}$.

\section{Statements of the Main Results}\label{old-introduction}

For values $\lambda$ in some interval $I \subset \mathbb{R}$, 
we consider linear Hamiltonian systems  
\begin{equation} \label{hammy}
J y' = \mathbb{B} (x; \lambda) y; \quad y (x; \lambda) \in \mathbb{C}^{2n},
\quad n \in \{1, 2, \dots \},
\end{equation}
where $J$ denotes the standard symplectic matrix 
\begin{equation*}
J 
=
\begin{pmatrix}
0_n & - I_n \\
I_n & 0_n
\end{pmatrix}.
\end{equation*}
We specify (\ref{hammy}) on intervals $(a, b)$, with 
$-\infty \le a < b \le +\infty$, and we assume throughout that 
for each $\lambda \in I$, $\mathbb{B} (\cdot; \lambda)$ is a 
measurable function with $\mathbb{B} (x; \lambda)$ self-adjoint
for a.e. $x \in (a, b)$, and that there exists $b_0 \in L^1_{\loc} ((a,b), \mathbb{R})$
so that for each $\lambda \in I$, $|\mathbb{B} (x; \lambda)| \le b_0 (x)$ for  
a.e. $x \in (a, b)$. In addition, we assume that for each 
$\lambda \in I$, $\mathbb{B}_{\lambda} (x; \lambda)$ exists
for a.e. $x \in (a, b)$, with $\mathbb{B}_{\lambda} (x; \lambda)$
self-adjoint and non-negative, and that there exists $b_1 \in L^1_{\loc} ((a,b), \mathbb{R})$
so that for each $\lambda \in I$, $|\mathbb{B}_{\lambda} (x; \lambda)| \le b_1 (x)$ for  
a.e. $x \in (a, b)$. For convenient reference, we refer to these 
basic assumptions as Assumptions {\bf (A)}. In addition, we make the following 
Atkinson-type positivity assumption.

\medskip
{\bf (B)} For each $\lambda \in I$, if $y(\cdot; \lambda) \in \AC_{\loc} ((a, b), \mathbb{C}^{2n})$ 
is any non-trivial solution of (\ref{hammy}), then 
\begin{equation} \label{positive-definite}
\int_c^d (\mathbb{B}_{\lambda} (x; \lambda)  y(x; \lambda), y(x; \lambda)) dx > 0,
\end{equation}
for all $[c, d] \subset (a, b)$, $c < d$. (Here and throughout, $\AC_{\loc} (\cdot)$ denotes local absolute
continuity, and $(\cdot, \cdot)$ denotes the usual inner product on $\mathbb{C}^{2n}$.)

\medskip
Our goal is to associate (\ref{hammy}) with one or more
self-adjoint operator pencils $\mathcal{L} (\lambda)$ (see Lemma 
\ref{self-adjoint-operator-lemma} below), and to use renormalized 
oscillation theory to count the number of eigenvalues 
$\mathcal{N} ([\lambda_1, \lambda_2))$, $\lambda_1, \lambda_2 \in I$,
$\lambda_1 < \lambda_2$, that each 
such operator has on a given interval $[\lambda_1, \lambda_2)$ 
for which the closure $[\lambda_1, \lambda_2]$
has empty intersection with the essential spectrum of 
the pencil.
We will formulate our results for two cases: (1) when $x = a$ 
is a regular boundary point for (\ref{hammy}); and 
(2) when $x = a$ is a singular boundary point for 
(\ref{hammy}). (We take (\ref{hammy}) to 
be singular at $x = b$ in both cases; the case in which 
(\ref{hammy}) is regular at both endpoints has 
been analyzed in \cite{HS2}.) The case in which 
(\ref{hammy}) is regular at $x = a$ corresponds 
with the following additional assumption. 

\medskip
{\bf (A)$^\prime$} The value $a$ is finite, and for any 
$c \in (a, b)$, we have 
\begin{equation*}
B (\cdot; \lambda), B_{\lambda} (\cdot; \lambda) 
\in L^1 ((a, c), \mathbb{C}^{2n \times 2n}).    
\end{equation*}

For the case in which $\mathbb{B} (x; \lambda)$ has the standard form 
\begin{equation} \label{linear-in-lambda}
    \mathbb{B} (x; \lambda) = B_0 (x) + \lambda B_1 (x),
\end{equation}
results along these lines have been obtained in \cite{GZ2017} for 
the limit-point case and in \cite{HS2022} for all limit-point, 
limit-circle, and limit-intermediate cases. These and related results
will be discussed at the end of this introduction.

Our starting point will be to specify an appropriate Hilbert 
space to work in, and for this we begin by fixing any 
$\lambda \in I$ and denoting by $\tilde{L}^2_{\mathbb{B}_{\lambda}} ((a, b), \mathbb{C}^{2n})$
the set of all Lebesgue measureable functions $f$ defined on $(a, b)$
so that 
\begin{equation*}
\|f\|_{\mathbb{B}_{\lambda}} := \Big(\int_a^b (\mathbb{B}_{\lambda} (x; \lambda) f(x), f(x)) dx\Big)^{1/2} 
< \infty.
\end{equation*}
Correspondingly, we denote by $\mathcal{Z}_{\mathbb{B}_{\lambda}}$ the subset of 
$\tilde{L}^2_{\mathbb{B}_{\lambda}} ((a, b), \mathbb{C}^{2n})$ comprising 
elements $f \in \tilde{L}^2_{\mathbb{B}_{\lambda}} ((a, b), \mathbb{C}^{2n})$ so that 
$\|f\|_{\mathbb{B}_{\lambda}} = 0$. Our Hilbert space will be the quotient
space, 
\begin{equation*}
L^2_{\mathbb{B}_{\lambda}} ((a, b), \mathbb{C}^{2n})
:= \tilde{L}^2_{\mathbb{B}_{\lambda}} ((a, b), \mathbb{C}^{2n})/\mathcal{Z}_{\mathbb{B}_{\lambda}}.
\end{equation*}
I.e., two functions $f, g \in L^2_{\mathbb{B}_{\lambda}} ((a, b), \mathbb{C}^{2n})$ 
are equivalent if and only if $\|f - g\|_{\mathbb{B}_{\lambda}} = 0$. With this 
convention, it follows 
that $\| \cdot \|_{\mathbb{B}_{\lambda}}$ is a norm on 
$L^2_{\mathbb{B}_{\lambda}} ((a, b), \mathbb{C}^{2n})$, and we equip 
$L^2_{\mathbb{B}_{\lambda}} ((a, b), \mathbb{C}^{2n})$ with the inner product
\begin{equation*}
\langle f, g \rangle_{\mathbb{B}_{\lambda}} := \int_a^b (\mathbb{B}_{\lambda} (x; \lambda) f(x), g(x)) dx.
\end{equation*}
We emphasize that Assumptions {\bf (A)} and {\bf (B)} are sufficient for defining
$L^2_{\mathbb{B}_{\lambda}} ((a, b), \mathbb{C}^{2n})$ in this way, 
and in particular that $\mathbb{B}_{\lambda} (x; \lambda)$ need not be an 
invertible matrix. We next begin our specification of the operator 
pencils we'll work with by defining what we will mean by the maximal domain 
$\mathcal{D}_M (\lambda)$ and the maximal operator $\mathcal{T}_M (\lambda)$. 

\begin{definition} \label{maximal-operator}
(i) For each fixed $\lambda \in I$, we denote by $\mathcal{D}_M (\lambda)$ the collection of 
all 
\begin{equation} \label{y-specified}
    y \in \AC_{\loc} ((a, b), \mathbb{C}^{2n}) \cap L^2_{\mathbb{B}_{\lambda}} ((a, b), \mathbb{C}^{2n})
\end{equation} 
for which there exists some $f \in L^2_{\mathbb{B}_{\lambda}} ((a, b), \mathbb{C}^{2n})$ so that 
\begin{equation} \label{domain-relation1}
Jy' - \mathbb{B} (x; \lambda) y
= \mathbb{B}_{\lambda} (x; \lambda) f,
\end{equation}
for a.e. $x \in (a, b)$. We note that the function $f$ in this specification 
is uniquely determined 
in $L^2_{\mathbb{B}_{\lambda}} ((a, b), \mathbb{C}^{2n})$. (If $f$ and $g$ are two 
functions associated with the same $y \in \mathbf{\mathcal{D}}_M (\lambda)$, then 
$\mathbb{B}_{\lambda} (x;\lambda) (f - g) = 0$ for a.e. $x \in (a, b)$,
so that $f = g$ in $L^2_{\mathbb{B}_{\lambda}} ((a, b), \mathbb{C}^{2n})$.)
We will refer to $\mathcal{D}_M (\lambda)$ 
as the maximal domain associated with (\ref{hammy}).

(ii) For each fixed $\lambda \in I$, we define the maximal operator 
$\mathcal{T}_{M} (\lambda): L^2_{B_{\lambda}} ((a, b), \mathbb{C}^{2n}) 
\to L^2_{\mathbb{B}_{\lambda}} ((a, b), \mathbb{C}^{2n})$ to be the operator 
with domain $\mathcal{D}_M (\lambda)$ taking 
a given $y \in \mathcal{D}_M (\lambda)$ to the unique 
$f \in L^2_{\mathbb{B}_{\lambda}} ((a, b), \mathbb{C}^{2n})$ guaranteed by 
the definition of $\mathcal{D}_M (\lambda)$. We note 
particularly that $y (\cdot; \lambda) \in \mathcal{D}_M (\lambda)$
solves (\ref{hammy}) if and only if 
$\mathcal{T}_M (\lambda) y = 0$ a.e. in $(a, b)$. 
\end{definition}

\begin{remark} \label{maximal-domain-remark} 
In the definition of $\mathcal{D}_M (\lambda)$, relation 
(\ref{domain-relation1}) could be replaced by 
\begin{equation} \label{domain-relation2}
Jy' - (\mathbb{B} (x; \lambda) - \lambda \mathbb{B}_{\lambda} (x; \lambda)) y
= \mathbb{B}_{\lambda} (x; \lambda) f.
\end{equation}
This follows because for a given $y$ as specified in (\ref{y-specified}),
if there exists $f \in L^2_{\mathbb{B}_{\lambda}} ((a, b), \mathbb{C}^{2n})$
so that (\ref{domain-relation1}) holds then it must be the case that 
(\ref{domain-relation2}) holds with $f$ replaced by $\tilde{f} = f + \lambda y$,
and the other direction is similar. In \cite{HS2022} the authors analyze
the case (\ref{linear-in-lambda}) and (following \cite{Krall2002}) 
use (\ref{domain-relation2}) to work with $J y' - B_0 (x) y = B_1 (x) f$.

We emphasize that in Definition \ref{maximal-operator} and throughout,
the designation a.e. always means almost everywhere in the usual 
sense, and in particular when we write $\mathcal{T}_M (\lambda) y = 0$ 
a.e. in $(a, b)$, we mean that $y$ is not only in the kernel 
of $\mathcal{T}_M (\lambda)$, but additionally that the element 
$f \in L_{\mathbb{B}_{\lambda}}^2 ((a, b), \mathbb{C}^n)$ for 
which $\mathcal{T}_M (\lambda) y = f$ is 
equivalent in $L_{\mathbb{B}_{\lambda}}^2 ((a, b), \mathbb{C}^n)$
to a function that vanishes a.e. in $(a, b)$. 
\end{remark}

With Definition \ref{maximal-operator} in place, 
we make the following Assumptions {\bf (C)}. 

\medskip
{\bf (C)} For any values $\lambda, \tilde{\lambda} \in I$, the 
spaces $L^2_{\mathbb{B}_{\lambda}} ((a, b), \mathbb{C}^{2n})$ and 
$L^2_{\mathbb{B}_{\tilde{\lambda}}} ((a, b), \mathbb{C}^{2n})$ 
are equivalent. In addition, the maximal domains $\mathcal{D}_{M} (\lambda)$ 
and $\mathcal{D}_{M} (\tilde{\lambda})$ are identical as sets 
(henceforth, we will denote this set $\mathcal{D}_{M}$).
\medskip

\begin{remark} \label{assumptions-remark}
We will check in Section \ref{applications-section} that Assumptions 
{\bf (A)}, {\bf (B)}, and {\bf (C)}, along with Assumptions 
{\bf (D)}, {\bf (E)}, and {\bf (F)} below hold in a wide range of 
important cases. 
\end{remark}

We will often find it convenient to fix $\lambda \in I$ and consider 
the linear eigenvalue problem 
\begin{equation} \label{linear-ev-problem}
    \mathcal{T} (\lambda) y = \mu y,
\end{equation}
for which our primary interest will be whether $\mu = 0$ is an eigenvalue. 
Nonetheless, we observe that for any $\mu \in \mathbb{C}$ (\ref{linear-ev-problem})
corresponds with the Hamiltonian system 
\begin{equation} \label{linear-hammy}
    Jy' = \mathbb{B} (x; \lambda) y + \mu \mathbb{B}_{\lambda} (x; \lambda) y, 
\end{equation}
which has precisely the structure of the Hamiltonian systems considered 
in \cite{HS2022} (i.e., is a Hamiltonian system linear in the 
spectral parameter $\mu$). 

The following terminology will be convenient for the discussion.

\begin{definition} \label{left-right-definition}
For each fixed pair $(\mu,\lambda) \in \mathbb{C} \times I$, we will say 
that a solution $y (\cdot;\mu, \lambda) \in \AC_{\loc} ((a,b),\mathbb{C}^{2n})$
of (\ref{linear-hammy}) {\it lies left} in $(a, b)$ if for any $c \in (a,b)$, 
the restriction of $y(\cdot; \mu, \lambda)$ to $(a,c)$ is in 
$L^2_{\mathbb{B}_{\lambda}} ((a, c), \mathbb{C}^{2n})$. Likewise, 
we will say that a solution $y (\cdot;\mu, \lambda) \in \AC_{\loc} ((a,b),\mathbb{C}^{2n})$
of (\ref{linear-hammy}) {\it lies right} in $(a, b)$ if for any $c \in (a,b)$, 
the restriction of $y(\cdot; \mu, \lambda)$ to $(c,b)$ is in 
$L^2_{\mathbb{B}_{\lambda}} ((c,b), \mathbb{C}^{2n})$. We will denote by 
$m_a (\mu, \lambda)$ the dimension of the space of solutions to 
(\ref{linear-hammy}) that lie left in $(a, b)$, and 
we will denote by $m_b (\mu, \lambda)$ the dimension of the 
space of solutions to (\ref{linear-hammy}) that lie right 
in $(a, b)$.  
\end{definition}

We can take advantage of the observation that (\ref{linear-hammy})
has the form of the systems analyzed in \cite{HS2022} to draw the 
following conclusions. If Assumptions {\bf (A)} and 
{\bf (B)} hold and $\lambda \in I$ is fixed, then for 
any $\mu \in \mathbb{C} \backslash \mathbb{R}$, 
(\ref{linear-hammy}) admits at least $n$ linearly independent 
solutions that lie
left in $(a, b)$ and at least $n$ linearly independent solutions 
that lie right in 
$(a, b)$. According to Theorem V.2.2 in \cite{Krall2002}, 
$m_a (\mu, \lambda)$ and $m_b (\mu, \lambda)$ are both 
constant for all $\mu$ with $\textrm{Im}\,\mu > 0$, and the same statement is 
true for $\textrm{Im}\,\mu < 0$. In the event that 
$\mathbb{B} (x; \lambda)$ has real-valued entries for 
a.e. $x \in (a, b)$, it is furthermore the case that $m_a (\mu, \lambda)$
and $m_b (\mu, \lambda)$ are both constant for all 
$\mu \in \mathbb{C} \backslash \mathbb{R}$. For the 
current analysis we will further assume a type of 
uniformity for this behavior as $\lambda$ varies in 
$I$. 

\medskip
{\bf (D)} The values $m_a (\mu, \lambda)$ and $m_b (\mu, \lambda)$ are both 
constant for all $(\mu, \lambda) \in (\mathbb{C} \backslash \mathbb{R}) \times I$. 
We denote these common values $m_a$ and $m_b$.
\medskip

In the event that Assumption {\bf (A)$^\prime$} also holds, it's 
clear that $m_a (\mu, \lambda) = 2n$ for all $(\mu, \lambda) \in \mathbb{C} \times I$.
In the terminology of our next definition, this means that under
Assumption {\bf (A)$^\prime$}, (\ref{linear-hammy}) 
is in the limit circle case at $x = a$. In this case, Assumption 
{\bf (D)} holds immediately for $x = a$, with $m_a = 2n$. We note here
that whenever we state that a result holds under 
Assumptions {\bf (A)} through {\bf (D)}, we will mean that 
{\bf (A)$^\prime$} doesn't necessarily hold. If {\bf (A)$^\prime$}
is needed, it will always be explicitly included in the list of 
assumptions. 

\begin{remark} \label{change-of-J-remark}
    In the event that we have an equation of form (\ref{hammy}) 
    with an alternative non-singular skew-symmetric matrix, say 
    $\mathcal{J}$, we can always make a change of variables 
    to obtain our preferred form. Precisely, suppose the form we have 
    is $\mathcal{J} z' = \mathcal{B} (x; \lambda) z$, with 
    $\mathcal{B} (x; \lambda)$ satisfying our assumptions 
    {\bf (A)} through {\bf (D)}. As 
    noted in \cite{HJK2018}, there exists an invertible matrix 
    $M$ so that $J = M^T \mathcal{J} M$. If we set $z = M y$
    then $\mathcal{J} M y' = \mathcal{B} (x; \lambda) My$, so 
    that $M^T \mathcal{J} M y' = M^T \mathcal{B} (x; \lambda) My$,
    giving (\ref{hammy}) with $\mathbb{B} (x; \lambda) = M^T \mathcal{B} (x; \lambda) M$. 
    It's straightforward to see that $\mathbb{B} (x; \lambda)$ 
    also satisfies Assumptions {\bf (A)} through {\bf (D)}.
\end{remark}
 
\begin{definition} If $m_a = n$, we say that (\ref{linear-hammy}) 
is in the limit point case at $x = a$, and if $m_a = 2n$, we say 
that (\ref{linear-hammy}) is in the limit circle case at $x = a$.
If $m_a \in (n, 2n)$, we say that (\ref{linear-hammy}) 
is in the limit-$m_a$ case at $x = a$. Analogous specifications
are made at $x = b$.
\end{definition}

Under Assumptions {\bf (A)} through {\bf (D)}, and
for some fixed pair $(\mu_0, \lambda_0) \in (\mathbb{C} \backslash \mathbb{R}) \times I$
we will show that by taking an appropriate
selection of $n$ solutions to (\ref{linear-hammy}) that lie left in $(a, b)$, 
$\{u^a_j (x; \mu_0, \lambda_0)\}_{j=1}^n$, and an appropriate 
selection of $n$ solutions to (\ref{linear-hammy}) that lie right in $(a, b)$,
$\{u^b_j (x; \mu_0, \lambda_0)\}_{j=1}^n$, we can specify, for each 
$\lambda \in I$, the domain
of a self-adjoint restriction of $\mathcal{T}_M (\lambda)$, which we will 
denote $\mathcal{T} (\lambda)$. For the purposes of this introduction, we
will sum this development up in the following lemma, for which 
we denote by $U^a (x; \mu_0, \lambda_0)$ the matrix comprising the 
vector functions $\{u^a_j (x; \mu_0, \lambda_0)\}_{j=1}^n$ as its columns, 
and by $U^b (x; \mu_0, \lambda_0)$ the matrix comprising the 
vector functions $\{u^b_j (x; \mu_0, \lambda_0)\}_{j=1}^n$ as its columns.
(The selection process is taken from Section 2.1 of \cite{HS2022}, with 
an overview given in Section \ref{operator-section} below.) The proof of this lemma is 
the main content of Section \ref{operator-section}. 

\begin{lemma} \label{self-adjoint-operator-lemma}
(i) Let Assumptions {\bf (A)} through {\bf (D)} hold, and let 
$(\mu_0, \lambda_0) \in (\mathbb{C} \backslash \mathbb{R}) \times I$ be fixed. 
Then there exists a selection of $n$ solutions $\{u^a_j (x; \mu_0, \lambda_0)\}_{j=1}^n$
to (\ref{linear-hammy}) (with $(\mu, \lambda) = (\mu_0, \lambda_0)$) 
that lie left in $(a, b)$, along with 
a selection of $n$ solutions $\{u^b_j (x; \mu_0, \lambda_0)\}_{j=1}^n$
to (\ref{linear-hammy}) (with $(\mu, \lambda) = (\mu_0, \lambda_0)$) 
that lie right in $(a, b)$ so 
that for each $\lambda \in I$ the restriction of $\mathcal{T}_M (\lambda)$ to the 
($\lambda$-independent) domain 
\begin{equation*}
    \mathcal{D} := \{y \in \mathcal{D}_M: \lim_{x \to a^+} U^a (x; \mu_0, \lambda_0)^* J y(x) = 0,
    \quad \lim_{x \to b^-} U^b (x; \mu_0, \lambda_0)^* J y(x) = 0\}
\end{equation*}
is a self-adjoint operator. We will denote this operator $\mathcal{T} (\lambda)$.

\medskip
(ii) In addition to Assumptions {\bf (A)} through {\bf (D)}, suppose that Assumption
{\bf (A)$^\prime$} holds, and let 
$(\mu_0, \lambda_0) \in (\mathbb{C} \backslash \mathbb{R}) \times I$ be fixed. Let 
$\alpha \in C^1(I, \mathbb{C}^{n \times 2n})$ be a continuously differentiable 
matrix-valued function so that for each $\lambda \in I$, $\rank \alpha (\lambda) = n$ 
and $\alpha (\lambda) J \alpha (\lambda)^* = 0$. Then there
exists a selection of solutions $\{u^b_j (x; \mu_0, \lambda_0)\}_{j=1}^n$
to (\ref{linear-hammy}) (with $(\mu, \lambda) = (\mu_0, \lambda_0)$) that lie right in $(a, b)$ so 
that for each $\lambda \in I$ the restriction of $\mathcal{T}_M (\lambda)$ to the 
($\lambda$-dependent) domain 
\begin{equation*}
    \mathcal{D}^{\alpha} (\lambda) := \{y \in \mathcal{D}_M: \alpha (\lambda) y(a) = 0,
    \quad \lim_{x \to b^-} U^b (x; \mu_0, \lambda_0)^* J y(x) = 0\}
\end{equation*}
is a self-adjoint operator. We will denote this operator $\mathcal{T}^{\alpha} (\lambda)$.
\end{lemma}

\begin{remark} \label{lemma1-remark}
We emphasize that while the operator pencils $\mathcal{T} (\lambda)$ and $\mathcal{T}^{\alpha} (\lambda)$
generally depend on $\lambda$, the domain $\mathcal{D}$ depends only on the 
fixed value $\lambda_0 \in I$, and the domain $\mathcal{D}^{\alpha} (\lambda)$ depends on 
$\lambda$ only through the boundary matrix $\alpha (\lambda)$, and otherwise only depends on the 
fixed value $\lambda_0 \in I$. 
\end{remark}

In addition to Assumptions {\bf (A)} through {\bf (D)}, we will require two additional 
assumptions, the first of which is somewhat technical. To motivate this, 
we note that for the case of (\ref{linear-in-lambda}), we have the convenient 
relation 
\begin{equation*}
    \mathbb{B} (x; \lambda_2) - \mathbb{B} (x; \lambda_1)
    = (\lambda_2 - \lambda_1) B_1 (x),
\end{equation*}
and Assumption {\bf (E)} (just below) can be viewed as a generalization of this relation. As 
expected, in the case of (\ref{linear-in-lambda}), this assumption holds trivially.
To set some notation, for any $\lambda_* \in I$ and any $r > 0$ (typically to be 
taken small), we set 
\begin{equation} \label{interval-notation}
I_{\lambda_*, r} := (\lambda_* - r, \lambda_* + r) \cap I.    
\end{equation}

\medskip
{\bf (E)} For each $\lambda_* \in I$, 
there exists a positive constant $r > 0$ and a 
map $\mathcal{E} (\cdot; \cdot, \lambda_*): (a, b) \times I_{\lambda_*, r}
\to \mathbb{C}^{2n \times 2n}$ so that for each 
$\lambda \in I_{\lambda_*, r}$ we have 
\begin{equation} \label{E-B-diff}
    \mathbb{B} (x; \lambda) - \mathbb{B} (x; \lambda_*)
    = \mathbb{B}_{\lambda} (x; \lambda_*) \mathcal{E} (x; \lambda, \lambda_*)
\end{equation}
for a.e. $x \in (a, b)$. In addition, the following hold: (i) for each $\lambda \in I_{\lambda_*, r}$,
$\mathcal{E} (\cdot; \lambda; \lambda_*) \in \mathcal{B} (L^2_{\mathbb{B}_{\lambda}} ((a, b), \mathbb{C}^{2n}))$
(i.e., when viewed as a multiplication operator, the matrix function $\mathcal{E} (\cdot; \lambda; \lambda_*)$
is a bounded linear operator taking $L^2_{\mathbb{B}_{\lambda}} ((a, b), \mathbb{C}^{2n})$
to itself); (ii) $\|\mathcal{E} (\cdot; \lambda, \lambda_*)\| = \mathbf{o} (1)$, $\lambda \to \lambda_*$;
(iii) for a.e. $x \in (a, b)$, the matrix function $\mathcal{E} (x; \lambda, \lambda_*)$
is continuously differentiable in $\lambda$ on $I_{\lambda_*, r}$, 
and the map $\lambda \mapsto \mathcal{E} (\cdot; \lambda; \lambda_*)$ is continuously 
differentiable as a map from $I_{\lambda_*, r}$ to $\mathcal{B} (L^2_{\mathbb{B}_{\lambda}} ((a, b), \mathbb{C}^{2n}))$;
and (iv) given any 
$f, g \in L^2_{\mathbb{B}_{\lambda}} ((a, b), \mathbb{C}^{2n})$, there exists 
$h \in L^1 ((a, b), \mathbb{R})$, depending on 
$f$ and $g$, so that for all $\lambda \in I_{\lambda_*, r}$
\begin{equation*}
|f(x)^* \mathbb{B}_{\lambda} (x; \lambda) g(x) | \le h (x), 
\end{equation*}
for a.e. $x \in (a, b)$. By (\ref{E-B-diff}), this is equivalent to 
\begin{equation*}
|f(x)^* \mathbb{B}_{\lambda} (x; \lambda_*) \mathcal{E}_{\lambda} (x; \lambda, \lambda_*) g(x) | \le h (x).
\end{equation*}

In order to set some notation and terminology for this discussion, we make the following standard
definitions.

\begin{definition} \label{spectrum-definitions}
For each fixed $\lambda \in I$, we define the resolvent set $\rho (\mathcal{T} (\lambda))$,
the spectrum $\sigma (\mathcal{T} (\lambda))$, the point spectrum $\sigma_p (\mathcal{T} (\lambda))$,
and the essential spectrum $\sigma_{\ess} (\mathcal{T} (\lambda))$ in the usual way 
(see, e.g., Definition 1.4 in \cite{HS2022}). In all cases, we categorize $\lambda$
with respect to the operator pencil $\mathcal{T} (\cdot)$ 
according to the categorization of $\mu = 0$ with respect to the 
operator $\mathcal{T} (\lambda)$ (i.e., with respect to the eigenvalue problem
(\ref{linear-ev-problem})). 
Analogous specifications are taken to hold for $\mathcal{T}^{\alpha} (\lambda)$. 
\end{definition}

Our primary tool for this analysis will be the 
Maslov index, and as a starting point for a discussion of this object, we 
define what we will mean by a Lagrangian subspace of $\mathbb{C}^{2n}$. 

\begin{definition} \label{lagrangian_subspace}
We say $\ell \subset \mathbb{C}^{2n}$ is a Lagrangian subspace of $\mathbb{C}^{2n}$
if $\ell$ has dimension $n$ and
\begin{equation} 
(J u, v) = 0, 
\end{equation} 
for all $u, v \in \ell$. In addition, we denote by 
$\Lambda (n)$ the collection of all Lagrangian subspaces of $\mathbb{C}^{2n}$, 
and we will refer to this as the {\it Lagrangian Grassmannian}. 
\end{definition}

Any Lagrangian subspace of $\mathbb{C}^{2n}$ can be
spanned by a choice of $n$ linearly independent vectors in 
$\mathbb{C}^{2n}$. We will generally find it convenient to collect
these $n$ vectors as the columns of a $2n \times n$ matrix $\mathbf{X}$, 
which we will refer to as a {\it frame} for $\ell$. Moreover, we will 
often coordinatize our frames as $\mathbf{X} = \genfrac{(}{)}{0pt}{1}{X}{Y}$, 
where $X$ and $Y$ are 
$n \times n$ matrices. Following \cite{F} (p. 274), we specify 
a metric on $\Lambda (n)$ in terms of appropriate orthogonal projections. 
Precisely, let $\mathcal{P}_i$ 
denote the orthogonal projection matrix onto $\ell_i \in \Lambda (n)$
for $i = 1,2$. I.e., if $\mathbf{X}_i$ denotes a frame for $\ell_i$,
then $\mathcal{P}_i = \mathbf{X}_i (\mathbf{X}_i^* \mathbf{X}_i)^{-1} \mathbf{X}_i^*$.
We take our metric $d$ on $\Lambda (n)$ to be defined 
by 
\begin{equation*}
d (\ell_1, \ell_2) := \|\mathcal{P}_1 - \mathcal{P}_2 \|,
\end{equation*} 
where $\| \cdot \|$ can denote any matrix norm. We will say 
that a path of Lagrangian subspaces 
$\ell: \mathcal{I} \to \Lambda (n)$ is continuous provided it is 
continuous under the metric $d$. 

Suppose $\ell_1 (\cdot), \ell_2 (\cdot)$ denote continuous paths of Lagrangian 
subspaces $\ell_i: \mathcal{I} \to \Lambda (n)$, $i = 1,2$, for some parameter interval 
$\mathcal{I}$ (not necessarily closed and bounded). 
The Maslov index associated with these paths, which we will 
denote $\mas (\ell_1, \ell_2; \mathcal{I})$, is a count of the number of times
the subspaces $\ell_1 (t)$ and $\ell_2 (t)$ intersect as 
$t$ traverses $\mathcal{I}$, counted
with both multiplicity and direction. (In this setting, if we let 
$t_*$ denote the point of intersection (often referred to as a 
{\it crossing point}), then multiplicity corresponds with the dimension 
of the intersection $\ell_1 (t_*) \cap \ell_2 (t_*)$; a precise definition of what we 
mean in this context by {\it direction} will be
given in Section \ref{maslov-section}.) 

In order to relate our results to previous work on renormalized
oscillation theory, we observe that in some cases the Maslov
index can be expressed as a sum of nullities for certain 
evolving matrix Wronskians. To understand this, we first 
specify the following terminology: for two paths of Lagrangian subspaces 
$\ell_1, \ell_2: [a, b] \to \Lambda (n)$, we say 
that the evolution of the pair $\ell_1, \ell_2$
is {\it monotonic} provided all intersections (including full multiplicities) 
occur in the same direction. If the intersections all correspond with the 
positive direction, then we can compute 
\begin{equation*}
    \mas (\ell_1, \ell_2; [a,b])
    = \sum_{t \in (a,b]} \dim (\ell_1 (t) \cap \ell_2 (t)).
\end{equation*}
(Here, we exclude $a$ from the interval we sum over, because
there will be no contribution to the Maslov index
from an initial positive crossing; see Section \ref{maslov-section}
for details on our conventions with the endpoints.)
Suppose $\mathbf{X}_1 (t) = \genfrac{(}{)}{0pt}{1}{X_1 (t)}{Y_1 (t)}$ and 
$\mathbf{X}_2 (t) = \genfrac{(}{)}{0pt}{1}{X_2 (t)}{Y_2 (t)}$ respectively 
denote frames for Lagrangian subspaces of $\mathbb{C}^{2n}$,
$\ell_1 (t)$ and $\ell_2 (t)$. Then we can express
this last relation as 
\begin{equation*}
\mas (\ell_1, \ell_2; [a, b]) 
= \sum_{t \in (a,b]} \dim \ker (\mathbf{X}_1 (t)^* J \mathbf{X}_2 (t)).
\end{equation*}
(See Lemma 2.2 of \cite{HS2}.)

In preparation for formulating our results in the case that (\ref{hammy}) 
is regular at $x = a$, we introduce the $2n \times n$ matrix solution
$\mathbf{X}_{\alpha} (x; \lambda)$ to the initial value problem
\begin{equation} \label{frame-alpha}
    J \mathbf{X}_{\alpha}' = \mathbb{B} (x; \lambda) \mathbf{X}_{\alpha},
    \quad \mathbf{X}_{\alpha} (a; \lambda) = J \alpha (\lambda)^*.
\end{equation}
Under our assumptions {\bf (A)} and {\bf (A)$^\prime$}, 
we can conclude that for each $\lambda \in I$, 
$\mathbf{X}_{\alpha} (\cdot; \lambda) \in AC_{\loc} ([a,b), \mathbb{C}^{2n \times n})$. 
In addition, $\mathbf{X}_{\alpha} \in C([a,b) \times I, \mathbb{C}^{2n \times n})$,
and $\mathbf{X}_{\alpha} (x; \cdot)$ is differentiable in 
$\lambda$. (See, for example, \cite{Weidmann1987}.)
As shown in \cite{HJK2018}, for each pair $(x, \lambda) \in [a,b) \times I$,
$\mathbf{X}_{\alpha} (x; \lambda)$ is the frame for a Lagrangian subspace
of $\mathbb{C}^{2n}$, which we will denote $\ell_{\alpha} (x; \lambda)$. 
(In \cite{HJK2018}, the authors make slightly 
stronger assumptions on $\mathbb{B} (x; \lambda)$, but their proof
carries over immediately into our setting.)

For the frame associated with the right endpoint, 
we let $[\lambda_1, \lambda_2] \subset I$, $\lambda_1 < \lambda_2$,
be such that for all $\lambda \in [\lambda_1, \lambda_2]$,
$0 \notin \sigma_{\ess} (\mathcal{T}^{\alpha} (\lambda))$ 
(equivalently, in our notation, 
$\lambda \notin \sigma_{\ess} (\mathcal{T}^{\alpha} (\cdot))$).
In Section \ref{operator-section}, we will show that 
for each $\lambda \in [\lambda_1, \lambda_2]$, there exists a
$2n \times n$ matrix solution $\mathbf{X}_b (x; \lambda)$
to the ODE 
\begin{equation} \label{frame-b}
J \mathbf{X}_b' = \mathbb{B} (x; \lambda) \mathbf{X}_b,
\quad \lim_{x \to b^-} U^b (x; \mu_0, \lambda_0)^* J \mathbf{X}_b (x; \lambda) = 0,  
\end{equation}
where the matrix $U^b (x; \mu_0, \lambda_0)$ is described in 
Lemma \ref{self-adjoint-operator-lemma} (and the paragraph
leading into that lemma).
In addition, we will check that for each pair $(x, \lambda) \in [a,b) \times [\lambda_1, \lambda_2]$,
$\mathbf{X}_{b} (x; \lambda)$ is the frame for a Lagrangian subspace
of $\mathbb{C}^{2n}$, which we will denote $\ell_b (x; \lambda)$, 
and we will also check that 
$\ell_b \in C([a, b) \times [\lambda_1, \lambda_2], \Lambda (n))$.

In order to conclude monotonicity, we require one final assumption.

\medskip
{\bf (F)} For specified values $\lambda_1, \lambda_2 \in I$, 
$\lambda_1 < \lambda_2$, the matrix $(\mathbb{B} (x; \lambda_2) - \mathbb{B} (x; \lambda_1))$ 
is non-negative for a.e. $x \in (a,b)$, and moreover 
there is no interval $[c, d] \subset (a, b)$, $c < d$, so that 
\begin{equation*}
\dim (\ell_a (x; \lambda_1) \cap \ell_b (x; \lambda_2)) \ne 0
\end{equation*}
for all $x \in [c, d]$. In the case that Assumption {\bf (A)$^\prime$}
holds, the subscript $a$ is replaced by $\alpha$ in this statement. 
\medskip

\begin{remark} \label{assumption-F-remark}
In \cite{HS2} the authors verify that the moreover part of 
{\bf (F)} is implied by the following form of Atkinson 
positivity: for any $[c, d] \subset (a, b)$, $c < d$, and any non-trivial 
solution $y(\cdot; \lambda_1) \in \AC_{\loc} ((a, b), \mathbb{C}^{2n})$ 
of (\ref{hammy}), we must have 
\begin{equation} \label{difference-definite}
\int_c^d (\mathbb{B} (x; \lambda_2) - \mathbb{B} (x; \lambda_1)) y(x; \lambda_1), y(x; \lambda_1)) dx > 0. 
\end{equation}
In the case $\mathbb{B} (x; \lambda) = B_0 (x) + \lambda B_1 (x)$, 
non-negativity of $(\mathbb{B} (x; \lambda_2) - \mathbb{B} (x; \lambda_1))$ corresponds
with non-negativity of the matrix $B_1 (x)$ (since $\lambda_1 < \lambda_2$), 
and the integral conditions
(\ref{positive-definite}) and (\ref{difference-definite})
are both equivalent to Atkinson positivity (see, e.g., Section 4 in \cite{LS2012}
and Section IV.4 in \cite{Krall2002}). 
\end{remark}

In Section \ref{theorems-section}, we will establish the following 
theorem. 

\begin{theorem} \label{regular-singular-theorem}
Let Assumptions {\bf (A)} through {\bf (E)} hold, along with Assumption {\bf (A)$^\prime$}, 
and assume that for some pair $\lambda_1, \lambda_2 \in I$,
$\lambda_1 < \lambda_2$, we have 
$0 \notin \sigma_{\ess} (\mathcal{T}^{\alpha} (\lambda))$ for all 
$\lambda \in [\lambda_1, \lambda_2]$. In addition, let Assumption {\bf (F)} hold
for the values $\lambda_1$ and $\lambda_2$, and for the boundary matrix $\alpha (\lambda)$
specified in Lemma \ref{self-adjoint-operator-lemma}(ii), assume that 
$\alpha (\lambda) J \partial_{\lambda} \alpha^* (\lambda)$ (which is 
necessarily self-adjoint) is 
non-negative at each $\lambda \in [\lambda_1, \lambda_2]$. 
If $\ell_{\alpha} (\cdot; \lambda_1)$ and $\ell_b (\cdot; \lambda_2)$
denote the paths of Lagrangian subspaces of $\mathbb{C}^{2n}$ constructed just above, 
and $\mathcal{N}^{\alpha} ([\lambda_1, \lambda_2))$
denotes a count of the number of eigenvalues that $\mathcal{T}^{\alpha} (\cdot)$
has on the interval $[\lambda_1, \lambda_2)$, then 
\begin{equation} \label{regular-singular-theorem-inequality}
  \mathcal{N}^{\alpha} ([\lambda_1, \lambda_2))
  \ge \mas (\ell_{\alpha} (\cdot; \lambda_1), \ell_b (\cdot; \lambda_2); [a, b))
  - \mas (\ell_{\alpha} (a; \cdot), \ell_b (a; \lambda_2); [\lambda_1, \lambda_2]),
\end{equation}
where 
\begin{equation*}
    \mas (\ell_{\alpha} (\cdot; \lambda_1), \ell_b (\cdot; \lambda_2); [a, b))
    := \lim_{c \to b^-} \mas (\ell_{\alpha} (\cdot; \lambda_1), \ell_b (\cdot; \lambda_2); [a, c]),
\end{equation*}
and part of the assertion is that this limit exists.
If additionally there exists a value $c_b \in (a, b)$ so that for 
all $c \in (c_b, b)$
\begin{equation} \label{condition1}
    \ell_{\alpha} (c; \lambda_1) \cap \ell_b (c; \lambda)
    = \{0\}, \quad \forall \,\, \lambda \in [\lambda_1, \lambda_2), 
\end{equation}
then we have equality in (\ref{regular-singular-theorem-inequality}).
\end{theorem}

\begin{remark} \label{regular-singular-remark}
In the case that $\alpha (\lambda)$ is constant in $\lambda$ (i.e., the 
boundary condition at $x = a$ is independent of $\lambda$), we necessarily
have $\mas (\ell_{\alpha} (a; \cdot), \ell_b (a; \lambda_2), [\lambda_1, \lambda_2]) = 0$,
because the maps $\ell_{\alpha} (a; \lambda)$ and $\ell_b (a; \lambda_2)$
are both independent of $\lambda$. As shown in \cite{HS2022}, in the case of (\ref{linear-in-lambda}), 
(\ref{condition1}) can be replaced by the simpler requirement
$0 \notin \sigma_p (\mathcal{T}^{\alpha} (\lambda_1)) \cup \sigma_p (\mathcal{T}^{\alpha} (\lambda_2))$. 
More generally, (\ref{condition1}) implies 
$0 \notin \sigma_p (\mathcal{T}^{\alpha} (\lambda_1))$, and in this 
case $\ell_{\alpha} (c; \lambda_1)$ is the space of solutions of 
(\ref{hammy}) that do not satisfy the specified boundary conditions
at $x = b$, while $\ell_b (c; \lambda)$ is the space of solutions of 
(\ref{hammy}) that do satisfy such conditions (though for varying 
values of $\lambda$). For many important cases, we can use a converse
argument to verify that (\ref{condition1}) holds. 
\end{remark}

In the case that {\bf (A)$^\prime$} doesn't hold, so that (\ref{hammy})
is singular at $x = a$, we let $[\lambda_1, \lambda_2] \subset I$, 
$\lambda_1 < \lambda_2$,
be such that $0 \notin \sigma_{\ess} (\mathcal{T} (\lambda))$ for all 
$\lambda \in [\lambda_1, \lambda_2]$. 
We will show in Section \ref{operator-section} 
that for each $\lambda \in [\lambda_1, \lambda_2]$ there exists a
$2n \times n$ matrix solution $\mathbf{X}_a (x; \lambda)$
to the ODE 
\begin{equation} \label{frame-a}
J \mathbf{X}_a' = \mathbb{B} (x; \lambda) \mathbf{X}_a,
\quad \lim_{x \to a^+} U^a (x; \mu_0, \lambda_0)^* J \mathbf{X}_a (x; \lambda) = 0,  
\end{equation}
where the matrix $U^a (x; \mu_0, \lambda_0)$ is described in 
Lemma \ref{self-adjoint-operator-lemma} (and the paragraph
leading into that lemma). In addition, we will check that for each 
pair $(x, \lambda) \in (a,b) \times [\lambda_1, \lambda_2]$,
$\mathbf{X}_a (x; \lambda)$ is the frame for a Lagrangian subspace
of $\mathbb{C}^{2n}$, which we will denote $\ell_a (x; \lambda)$, 
and that 
$\ell_a \in C((a, b) \times [\lambda_1, \lambda_2], \Lambda (n))$.

In Section \ref{theorems-section}, we will establish the following 
theorem. 

\begin{theorem} \label{singular-theorem}
Let Assumptions {\bf (A)} through {\bf (E)} hold, and
assume that for some pair $\lambda_1, \lambda_2 \in I$,
$\lambda_1 < \lambda_2$, we have $0 \notin \sigma_{\ess} (\mathcal{T} (\lambda))$ for all 
$\lambda \in [\lambda_1, \lambda_2]$. In addition, let Assumption {\bf (F)} hold
for the values $\lambda_1$ and $\lambda_2$. 
If $\ell_a (\cdot; \lambda_1)$ and $\ell_b (\cdot; \lambda_2)$
denote the paths of Lagrangian subspaces of $\mathbb{C}^{2n}$ constructed just above, 
and $\mathcal{N} ([\lambda_1, \lambda_2))$
denotes a count of the number of eigenvalues that $\mathcal{T} (\cdot)$
has on the interval $[\lambda_1, \lambda_2)$, then 
\begin{equation} \label{singular-theorem-inequality}
  \mathcal{N} ([\lambda_1, \lambda_2))
  \ge \mas (\ell_a (\cdot; \lambda_1), \ell_b (\cdot; \lambda_2); (a, b)),
\end{equation}
where the Maslov index $\mas (\ell_{\alpha} (\cdot; \lambda_1), \ell_b (\cdot; \lambda_2); (a, b))$
is computed by taking a limit of the values 
$\mas (\ell_{\alpha} (\cdot; \lambda_1), \ell_b (\cdot; \lambda_2); [c_1, c_2])$
as $c_1 \to a^+$ and $c_2 \to b^-$, and part of the assertion is that 
this double limit exists. 
If additionally there exists a value $c_a \in (a, b)$ so that for 
all $c \in (a, c_a)$
\begin{equation} \label{condition2}
    \ell_a (c; \lambda) \cap \ell_b (c; \lambda_2)
    = \{0\}, \quad \forall \,\, \lambda \in [\lambda_1, \lambda_2), 
\end{equation}
and also a value $c_b \in (a, b)$ so that for 
all $c \in (c_b, b)$
\begin{equation} \label{condition3}
    \ell_a (c; \lambda_1) \cap \ell_b (c; \lambda)
    = \{0\}, \quad \forall \,\, \lambda \in [\lambda_1, \lambda_2), 
\end{equation}
then we have equality in (\ref{singular-theorem-inequality}).
\end{theorem}

\begin{remark} \label{singular-remark}
Similarly as with Theorem \ref{regular-singular-theorem}, in the case 
of (\ref{linear-in-lambda}), conditions (\ref{condition2}) and (\ref{condition3})
can be replaced by the simpler requirement
$0 \notin \sigma_p (\mathcal{T} (\lambda_1)) \cup \sigma_p (\mathcal{T} (\lambda_2))$. 
Conditions (\ref{condition2}) and (\ref{condition3}) can be interpreted similarly 
as in Remark \ref{regular-singular-remark}. 
\end{remark}

In the current setting, the necessary monotonicity follows 
from Claims 4.1 and 4.2 of \cite{HS2} (with $(0, 1)$ replaced
by $(a, b)$). With this observation, we obtain the following 
theorem.

\begin{theorem} \label{nullity-theorem}
Under the assumptions of Theorem \ref{regular-singular-theorem} (without 
condition (\ref{condition1})), we can write 
\begin{equation*}
    \mas (\ell_{\alpha} (\cdot; \lambda_1), \ell_b (\cdot; \lambda_2); [a, b))
    = \sum_{x \in (a, b)} \dim \ker \mathbf{X}_{\alpha} (x; \lambda_1)^* J \mathbf{X}_b (x; \lambda_2),
\end{equation*}
and under the assumptions of Theorem \ref{singular-theorem} (without 
conditions (\ref{condition2}) and (\ref{condition3})), 
we can write 
\begin{equation*}
    \mas (\ell_a (\cdot; \lambda_1), \ell_b (\cdot; \lambda_2); (a, b))
    = \sum_{x \in (a, b)} \dim \ker \mathbf{X}_a (x; \lambda_1)^* J \mathbf{X}_b (x; \lambda_2).
\end{equation*}
\end{theorem}

\begin{remark} \label{nullity-theorem-remark}
For the first assertion in Theorem \ref{nullity-theorem}, the left endpoint
$a$ is not included on the right-hand side because the intersections 
counted by the Maslov index in this case are monotonically associated
with a direction (counterclockwise) that does not increment the Maslov 
index at departures (as discussed in Section \ref{maslov-section}).
\end{remark}

\section{The Self-Adjoint Operator Pencils $\mathcal{T} (\cdot)$ and $\mathcal{T}^{\alpha} (\cdot)$} 
\label{operator-section}

In this section, we construct the self-adjoint operator pencils 
$\mathcal{T} (\cdot)$ and $\mathcal{T}^{\alpha} (\cdot)$ 
described in Lemma \ref{self-adjoint-operator-lemma}. We begin 
by formulating a version of Green's identity appropriate for this 
setting.

\begin{lemma}[Green's Identity] \label{green-lemma}
Let Assumptions {\bf (A)} hold, and for any fixed $\lambda \in I$ 
let $\mathcal{T}_M (\lambda)$
be the maximal operator specified in Definition \ref{maximal-operator}.
Then for any $y, z \in \mathcal{D}_M (\lambda)$, 
\begin{equation} \label{green1}
    \langle \mathcal{T}_M (\lambda) y, z \rangle_{\mathbb{B}_{\lambda}}
    - \langle y, \mathcal{T}_M (\lambda) z \rangle_{\mathbb{B}_{\lambda}}
    = (Jy, z)_a^b,
\end{equation}
where 
\begin{equation*}
     (Jy,z)_a^b =  (Jy,z)_b -  (Jy,z)_a,
\end{equation*}
with 
\begin{equation*}
    \begin{aligned}
    (Jy,z)_a &:= \lim_{x \to a^+} (J y(x),z(x)), \\
    (Jy,z)_b &:= \lim_{x \to b^-} (J y(x),z(x))
    \end{aligned}
\end{equation*} 
(for which the limits are well-defined).
In particular, if $y$ and $z$ satisfy 
$\mathcal{T}_M (\lambda) y = \mu y$ and $\mathcal{T}_M (\lambda) z = \mu z$
for some $\mu \in \mathbb{C}$, then 
\begin{equation} \label{green2}
    2 i (\mathrm{Im }\mu) \langle y, z \rangle_{\mathbb{B}_{\lambda}} = (Jy,z)_a^b.
\end{equation}
\end{lemma} 

\begin{proof} 
To begin, we fix any $\lambda \in I$, and for 
any $y, z \in \mathcal{D}_M (\lambda)$, we let $f, g \in L^2_{\mathbb{B}_{\lambda}} ((a, b), \mathbb{C}^{2n})$
respectively denote the uniquely defined functions so that $\mathcal{T}_M (\lambda) y = f$ and 
$\mathcal{T}_M (\lambda) z = g$. By definition of $\mathcal{D}_M$, this means that we have
the relations
\begin{equation*}
    \begin{aligned}
    J y' - \mathbb{B} (x; \lambda)y &= \mathbb{B}_{\lambda} (x; \lambda) f \\
    J z' - \mathbb{B} (x; \lambda) z &= \mathbb{B}_{\lambda} (x; \lambda) g,
    \end{aligned}
\end{equation*}
for a.e. $x \in (a, b)$. We compute the $\mathbb{C}^{2n}$ inner
product
\begin{equation*}
    (\mathbb{B}_{\lambda} (x; \lambda) \mathcal{T}_M (\lambda) y, z) 
    = (\mathbb{B}_{\lambda} (x; \lambda) f, z) 
    = (J y' - \mathbb{B} (x; \lambda) y, z) 
    = (Jy', z) - (y, \mathbb{B} (x; \lambda) z),
\end{equation*}
where in obtaining the final equality we have used our assumption that 
$\mathbb{B} (x; \lambda)$ is self-adjoint for a.e. $x \in (a, b)$. Likewise, 
\begin{equation*}
\begin{aligned}
    (\mathbb{B}_{\lambda} (x; \lambda) y, \mathcal{T}_M (\lambda) z) 
    &= (\mathbb{B}_{\lambda} (x; \lambda) y, g) 
    = (y, \mathbb{B}_{\lambda} (x; \lambda) g) \\
    &= (y, J z' - \mathbb{B} (x; \lambda) z)
    = (y, Jz') - (y, \mathbb{B} (x; \lambda) z).
\end{aligned}
\end{equation*}
Subtracting the latter of these relations from the former, 
we see that 
\begin{equation*}
    \frac{d}{d x} (Jy,z)
    =  (\mathbb{B}_{\lambda} (x; \lambda) \mathcal{T}_M (\lambda) y, z)
    -  (\mathbb{B}_{\lambda} (x; \lambda) y, \mathcal{T}_M (\lambda) z). 
\end{equation*}
For any $c, d \in (a, b)$, $c < d$, we can integrate this 
last relation to see that 
\begin{equation*}
\begin{aligned}
    &(J y(d), z(d)) - (Jy (c),z (c)) \\
    &= \int_c^d (\mathbb{B}_{\lambda} (x; \lambda) \mathcal{T}_M (\lambda) y (x), z (x)) dx
    - \int_c^d (\mathbb{B}_{\lambda} (x; \lambda) y (x), \mathcal{T}_M (\lambda) z (x)) dx.
\end{aligned}
\end{equation*}
If we allow $d$ to remain fixed, then since $y, z \in L^2_{\mathbb{B}_{\lambda}} ((a, b), \mathbb{C}^{2n})$
we see that the limit 
\begin{equation*}
    (Jy, z)_a := \lim_{c \to a^+} (Jy (c),z (c))
\end{equation*}
is well-defined. In particular, we can write 
\begin{equation*}
    (J y(d), z(d)) - (Jy,z)_a
    = \int_a^d (\mathbb{B}_{\lambda} (x; \lambda) \mathcal{T}_M (\lambda) y (x), z (x)) dx
    - \int_a^d (\mathbb{B}_{\lambda} (x; \lambda) y (x), \mathcal{T}_M (\lambda) z (x)) dx.
\end{equation*}
If we now take $d \to b^-$, we obtain precisely 
(\ref{green1}). Relation (\ref{green2})
is an immediately consequence of (\ref{green1}).
\end{proof}

\begin{remark}
Throughout the proof of Lemma \ref{green-lemma}, $\lambda$ remains fixed, 
so there is no requirement that either the weighted space 
$L^2_{\mathbb{B}_{\lambda}} ((a, b), \mathbb{C}^{2n})$ or the 
maximal domain $D_M (\lambda)$ be independent of $\lambda$. 
\end{remark}

We turn next to the identification of appropriate domains 
$\mathcal{D}$ and $\mathcal{D}^{\alpha}$ on which the 
respective restrictions of $\mathcal{T}_M (\lambda)$ 
are self-adjoint. 
This development is adapted from Section 2 of \cite{HS2022}, 
which in turn follows Chapter 6 in \cite{Pearson1988}.
We begin by making some preliminary definitions. 
We set 
\begin{equation*}
\mathcal{D}_c := \{y \in \mathcal{D}_M: y \textrm{ has compact support in } (a,b)\},
\end{equation*}
and we denote by $\mathcal{T}_c (\lambda)$ the restriction of $\mathcal{T}_M (\lambda)$
to $\mathcal{D}_c$. We can show, as in Theorem 3.9 of \cite{Weidmann1987}
that for each $\lambda \in I$,
$\mathcal{T}_c (\lambda)^* = \mathcal{T}_M (\lambda)$, and from Theorem 3.7 of 
that same reference (adapted to the current setting) that 
$\mathcal{D}_c$ is dense in $L^2_{\mathbb{B}_{\lambda}} ((a, b), \mathbb{C}^{2n})$. 

At this point, we fix some $\lambda_0 \in I$, and in 
addition we fix some $\mu_0 \in \mathbb{C} \backslash \mathbb{R}$, and we emphasize
that these values will remain fixed throughout the analysis. Under our 
assumptions {\bf (A)} through {\bf (C)}, the linear Hamiltonian 
system 
\begin{equation} \label{fixed-hammy}
    J y' = (\mathbb{B} (x; \lambda_0) + \mu_0 \mathbb{B}_{\lambda} (x; \lambda_0))    
\end{equation}
satisfies all the assumptions of the corresponding systems analyzed 
in \cite{HS2022}. This allows us to adapt four useful lemmas from that 
reference, stated here as Lemmas \ref{subspace-dimensions-lemma}
through \ref{niessen-lemma2}. First, we summarize some notation and terminology from 
\cite{HS2022} associated with the {\it Niessen spaces} that will have a 
critical role in our development. We begin by fixing some 
$c \in (a, b)$, and for 
$(\mu, \lambda) \in (\mathbb{C} \backslash \mathbb{R}) \times I$ 
letting $\Phi (x; \mu, \lambda)$ denote the fundamental matrix specified 
by
\begin{equation} \label{phi-specified}
J \Phi' = (\mathbb{B} (x; \lambda) + \mu \mathbb{B}_{\lambda} (x; \lambda)) \Phi; 
\quad \Phi (c; \mu, \lambda) = I_{2n}.
\end{equation}
We define
\begin{equation} \label{mathcal-A-defined}
\mathcal{A} (x; \mu, \lambda) := \frac{1}{2 \rm{Im }\mu} 
\Phi (x; \mu, \lambda)^* (J/i) \Phi (x; \mu, \lambda),
\end{equation}
on $(a,b) \times (\mathbb{C} \backslash \mathbb{R}) \times I$. 
It's clear from this definition that with $\lambda \in I$
fixed, for each $\mu \in \mathbb{C} \backslash \mathbb{R}$, we have
$\mathcal{A} (\cdot; \mu, \lambda) \in \AC_{\loc} ((a, b), \mathbb{C}^{2n \times 2n})$,  
with $\mathcal{A} (x; \mu, \lambda)$ self-adjoint 
for all $(x, \mu) \in (a,b) \times \mathbb{C} \backslash \mathbb{R}$.
It follows that the eigenvalues $\{\nu_j (x; \mu, \lambda)\}_{j=1}^{2n}$ 
of $\mathcal{A} (x; \mu, \lambda)$ can be 
ordered so that $\nu_j (x; \mu, \lambda) \le \nu_{j+1} (x; \mu, \lambda)$
for all $j \in \{1, 2, \dots, 2n-1\}$. In addition, it follows from 
Assumption {\bf (B)} that each $\nu_j (x; \lambda)$ is non-decreasing 
as $x$ increases. 

The following lemma is proven as Lemma 2.1 in \cite{HS2022}.

\begin{lemma} \label{subspace-dimensions-lemma}
Let Assumptions {\bf (A)} and {\bf (B)} hold,
and let $(\mu, \lambda) \in (\mathbb{C} \backslash \mathbb{R}) \times I$ 
be fixed. Then the dimension $m_a (\mu, \lambda)$ of the subspace of solutions to 
(\ref{linear-hammy}) that lie left in $(a, b)$ is precisely 
the number of eigenvalues $\nu_j (x; \mu, \lambda) \in \sigma (\mathcal{A} (x; \mu, \lambda))$
that approach a finite limit as $x \to a^+$. Likewise, 
the dimension $m_b (\mu, \lambda)$ of the subspace of solutions to 
(\ref{linear-hammy}) that lie right in $(a, b)$ is precisely 
the number of eigenvalues $\nu_j (x; \mu, \lambda) \in \sigma (\mathcal{A} (x; \mu, \lambda))$
that approach a finite limit as $x \to b^-$.

In addition, for each eigenvalue-eigenvector pair 
$(\nu_j (x; \mu, \lambda), v_j (x; \mu, \lambda))$ (whether or
not the limits described above exist), there exists a sequence 
$\{x_k\}_{k=1}^{\infty}$, with $x_k \to a^+$ so that
$v_j^a (\mu, \lambda) := \lim_{k \to \infty} v_j (x_k; \mu, \lambda)$ is 
well defined, and also a sequence 
$\{\tilde{x}_k\}_{k=1}^{\infty}$, with $\tilde{x}_k \to b^-$ so that
$v_j^b (\mu, \lambda) := \lim_{k \to \infty} v_j (\tilde{x}_k; \mu, \lambda)$ is well defined.
The collection 
\begin{equation*}
\{\Phi (x; \mu, \lambda) v_j^a (\mu, \lambda) \}_{j=2n-m_a (\mu, \lambda)+1}^{2n}    
\end{equation*}
comprises a basis for the space of solutions to (\ref{linear-hammy}) that lie
left in $(a, b)$, and the collection 
\begin{equation*}
\{\Phi (x; \mu, \lambda) v_j^a (\mu, \lambda) \}_{j = 1}^{2n-m_a (\mu, \lambda)+1}    
\end{equation*}
comprises a basis for the space of solutions to (\ref{linear-hammy}) that do not lie
left in $(a, b)$. Likewise, the collection 
\begin{equation*}
\{\Phi (x; \mu, \lambda) v_j^b (\mu, \lambda) \}_{j=1}^{m_b (\mu, \lambda)}    
\end{equation*}
comprises a basis for the space of solutions to (\ref{linear-hammy}) that lie
right in $(a, b)$, and the collection 
\begin{equation*}
\{\Phi (x; \mu, \lambda) v_j^b (\mu, \lambda) \}_{j = m_b (\mu, \lambda)+1}^{2n}    
\end{equation*}
comprises a basis for the space of solutions to (\ref{linear-hammy}) that do not lie
right in $(a, b)$. 
\end{lemma}

To set some notation, for each $j \in \{2n - m_a (\mu, \lambda) + 1, \dots, 2n\}$,
we set 
\begin{equation*}
    \nu_j^a (\mu, \lambda)
    := \lim_{x \to a^+} \nu_j (x; \mu, \lambda),
\end{equation*}
and likewise for each $j \in \{1, 2, \cdots, m_b (\mu, \lambda)\}$,
we set 
\begin{equation*}
    \nu_j^b (\mu, \lambda)
    := \lim_{x \to b^-} \nu_j (x; \mu, \lambda).
\end{equation*}
Using the elements described in Lemma \ref{subspace-dimensions-lemma}, we can specify
Niessen elements associated with (\ref{hammy}). These specifications are adapted 
from \cite{Niessen70, Niessen71, Niessen72} (as developed in Chapter VI of \cite{Krall2002}).
Since the development is similar at the two endpoints $x = a$ and $x = b$, we will 
work through full details only for $x = b$ and summarize results for $x = a$. 

For $j = 1, 2, \dots, n$, we set 
\begin{equation} \label{niessen-components}
\begin{aligned}
y_j^b (x; \mu, \lambda) &:= \Phi (x; \mu, \lambda) v_j^b (\mu, \lambda) \\
z_j^b (x; \mu, \lambda) &:= \Phi (x; \mu, \lambda) v_{n+j}^b (\mu, \lambda). 
\end{aligned}
\end{equation} 
It's clear from our construction that 
$y_j^b (\cdot; \mu, \lambda)$ lies right in $(a, b)$
for each $j \in \{1, 2, \dots, n\}$, while 
$z_j^b (\cdot; \mu, \lambda)$ lies right in $(a, b)$
if and only if $n+j \in \{1, 2, \dots, m_b (\mu, \lambda)\}$ 
is finite. In what follows, we will find it convenient
to introduce the value $r_b (\mu, \lambda) := m_b (\mu, \lambda) - n$. 
For each $j \in \{1, 2, \dots, n\}$,
we define the two-dimensional space
\begin{equation} \label{niessen-spaces}
N_j^b (\mu, \lambda) := \Span \{y_j^b (\cdot; \mu, \lambda), z_j^b (\cdot; \mu, \lambda)\},
\end{equation}
and following \cite{Krall2002} we refer to the collection 
$\{N_j^b (\mu, \lambda)\}_{j=1}^n$ as the {\it Niessen subspaces} at $b$.
According to our labeling convention, the Niessen subspaces
$\{N_j^b (\mu, \lambda)\}_{j=1}^{r_b (\mu, \lambda)}$ all satisfy 
$\dim N_j^b (\mu, \lambda) \cap L^2_{\mathbb{B}_{\lambda}} ((c, b), \mathbb{C}^{2n}) = 2$,
while the remaining Niessen subspaces $\{N_j^b (\mu, \lambda)\}_{r_b (\mu, \lambda) + 1}^n$ 
satisfy $\dim N_j^b (\mu, \lambda) \cap L^2_{\mathbb{B}_{\lambda}} ((c, b), \mathbb{C}^{2n}) = 1$. 
(Here, $c$ continues to be the value $c \in (a, b)$ fixed just prior to (\ref{phi-specified}).)

In the development so far, it hasn't necessarily been the case that 
$m_a (\mu, \lambda)$ and $m_b (\mu, \lambda)$ are independent of $\mu$ and 
$\lambda$, but at this point we add our Assumption {\bf (D)} and henceforth
denote $m_a (\mu, \lambda)$, $m_b (\mu, \lambda)$, $r_a (\mu, \lambda)$, 
and $r_b (\mu, \lambda)$ respectively $m_a$, $m_b$, $r_a$, and $r_b$.
In this setting, we choose $n$ 
solutions of (\ref{linear-hammy}) that lie right in 
$(a, b)$, taking precisely 
one from each Niessen subspace $N_j^b (\mu, \lambda)$ in the following way. 
First, for each $j \in \{1, 2, \dots, r_b\}$,
we let $\beta_j (\mu, \lambda)$ be any complex
number on the circle 
\begin{equation*}
    |\beta_j^b (\mu, \lambda)| = \sqrt{-\nu_j^b (\mu, \lambda)/\nu_{n+j}^b (\mu, \lambda)},
\end{equation*}
where as discussed in \cite{HS2022} these ratios cannot be 0, and we set 
\begin{equation*}
    u^b_j (x; \mu, \lambda) := y^b_j (x; \mu, \lambda) + \beta_j^b (\mu, \lambda) z^b_j (x; \mu, \lambda). 
\end{equation*}
Next, for each $j \in \{r_b + 1, r_b + 2, \dots, n\}$,
we set 
\begin{equation*}
    u^b_j (x; \mu, \lambda) = y^b_j (x; \mu, \lambda). 
\end{equation*}
Correspondingly, we will denote by $\{r_j^b (\mu, \lambda)\}_{j=1}^n$ 
the vectors specified so that $u^b_j (x; \mu, \lambda) = \Phi (x; \mu, \lambda) r_j^b (\mu, \lambda)$
for each $j \in \{1, 2, \dots, n\}$. Precisely, this means that 
\begin{equation*}
    \begin{aligned}
    r^b_j (\mu, \lambda) &= v_j^b (\mu, \lambda) + \beta^b_j (\mu, \lambda) v_{n+j}^b (\mu, \lambda),
    \quad j \in \{1, 2, \dots, r_b\}, \\
    r^b_j (\mu, \lambda) &= v_j^b (\mu, \lambda), 
    \quad \quad \quad \quad \quad \quad  j \in \{r_b + 1, r_b + 2, \dots, n\}. 
    \end{aligned}
\end{equation*}
We can now collect the vectors $\{r_j^b (\mu, \lambda)\}_{j=1}^n$ 
into a frame 
\begin{equation} \label{b-frame}
\mathbf{R}^b (\mu, \lambda) =
\begin{pmatrix}
r_1^b (\mu, \lambda) & r_2^b (\mu, \lambda) & \dots & r_n^b (\mu, \lambda)
\end{pmatrix}.
\end{equation}

In addition to the above specifications, for the Niessen subspaces 
$\{N_j^b (\mu, \lambda)\}_{j=1}^{r_b}$, it will be useful to introduce
notation for elements linearly independent to the 
$\{u_j^b (x; \mu, \lambda)\}_{j = 1}^{r_b}$. For each 
$j \in \{1, 2, \dots, r_b\}$, we take any 
complex number $\gamma_j (\mu, \lambda)$ so that 
$|\gamma_j (\mu, \lambda)| = |\beta_j (\mu, \lambda)|$ 
but  $\gamma_j (\mu, \lambda) \ne \beta_j (\mu, \lambda)$,
and we define the {\it Niessen complement} to $u_j^b (x; \mu, \lambda)$
to be 
\begin{equation} \label{niessen-complements}
v_j^b (x; \mu, \lambda) 
= y^b_j (x; \mu, \lambda) + \gamma_j^b (\mu, \lambda) z^b_j (x; \mu, \lambda). 
\end{equation}

The following three lemmas are adapted respectively from Lemma 2.3,
Claim 2.1, and Claim 2.2 of \cite{HS2022}.

\begin{lemma} \label{krall-niessen-lemma} Let 
Assumptions {\bf (A)} through {\bf (D)} hold, and 
take the collection of Niessen elements $\{u_j^b (x; \mu, \lambda)\}_{j=1}^{n}$
and the collection of Niessen complements $\{v_j^b (x; \mu, \lambda)\}_{j=1}^{r_b}$
to be specified as above. Then the following hold: 

\medskip
(i) For each $j, k \in \{1, 2, \dots, n\}$, 
\begin{equation*}
    (J u_j^b (\cdot; \mu, \lambda), u_k^b (\cdot; \mu, \lambda))_b = 0.
\end{equation*}

\medskip
(ii) For each $j \in \{1, 2, \dots, n\}$, $k \in \{1, 2, \dots, r_b\}$, 
\begin{equation*}
    (J u_j^b (\cdot; \mu, \lambda), v_k^b (\cdot; \mu, \lambda))_b 
    = \begin{cases}
    0 & j \ne k \\
    \kappa_j^b = 2 i \mathrm{Im }\mu (\nu_j^b (\mu, \lambda) 
    + \gamma^b_j (\mu, \lambda) \beta^b_j (\mu, \lambda) \nu_{n+j}^b (\mu, \lambda)) \ne 0 & j = k. 
    \end{cases}
\end{equation*}
\end{lemma}

\begin{lemma} \label{niessen-lemma1}
Let Assumptions {\bf (A)} through {\bf (D)} hold, and suppose
the Niessen elements for (\ref{linear-hammy}) are chosen to be
\begin{equation*}
    \begin{aligned}
    u_j^b (x; \mu, \lambda) &= \Phi (x; \mu, \lambda) (v_j^b (\mu, \lambda) + \beta_j^b (\mu, \lambda) v^b_{n+j} (\mu, \lambda)),
    \quad j \in \{1, 2, \dots, r_b\} \\
    v_j^b (x; \mu, \lambda) &= \Phi (x; \mu, \lambda) (v_j^b (\mu, \lambda) + \gamma_j^b (\mu, \lambda) v^b_{n+j} (\mu, \lambda)),
    \quad j \in \{1, 2, \dots, r_b\} \\
    u_j^b (x; \mu, \lambda) &= \Phi (x; \mu, \lambda) v_j^b (\mu, \lambda), 
    \quad \quad \quad \quad \quad \quad j \in \{r_b + 1, r_b + 2, \dots, n\},
    \end{aligned}
\end{equation*}
with $\beta_j^b (\mu, \lambda)$ and $\gamma_j^b (\mu, \lambda)$ specified just above (in particular, 
as well-defined non-zero values).
Then the Niessen elements for (\ref{linear-hammy}) with $\mu$ 
replaced by $\bar{\mu}$ (and $\lambda$ unchanged) can be chosen to be
\begin{equation*}
    \begin{aligned}
    u_j^b (x; \bar{\mu}, \lambda) &= \Phi (x; \bar{\mu}, \lambda) (v_j^b (\bar{\mu}, \lambda) 
    + \beta_j^b (\bar{\mu}, \lambda) v^b_{n+j} (\bar{\mu}, \lambda)), \quad j \in \{1, 2, \dots, r_b\} \\
    v_j^b (x; \bar{\mu}, \lambda) &= \Phi (x; \bar{\mu}, \lambda) (v_j^b (\bar{\mu}, \lambda) 
    + \gamma_j^b (\bar{\mu}, \lambda) v^b_{n+j} (\bar{\mu}, \lambda)), \quad j \in \{1, 2, \dots, r_b\} \\
    u_j^b (x; \bar{\mu}, \lambda) &= \Phi (x; \bar{\mu}, \lambda) v_j^b (\bar{\mu}, \lambda), 
    \quad \quad \quad \quad \quad \quad j \in \{r_b + 1, r_b + 2, \dots, n\},
    \end{aligned}
\end{equation*}
with $\beta_j^b (\bar{\mu}, \lambda) =  - \overline{\beta_j^b (\mu, \lambda)}$
and $\gamma_j^b (\bar{\mu}, \lambda) = - \overline{\gamma_j^b (\mu, \lambda)}$ 
for all $j \in \{1, 2, \dots r_b\}$.
\end{lemma}

\begin{lemma} \label{niessen-lemma2}
Let the Assumptions and notation of Lemma \ref{niessen-lemma1} hold,
and let $\mathbf{R}^b (\mu, \lambda)$ denote the matrix defined in (\ref{b-frame}).
If $\mathbf{R}^b (\bar{\mu}, \lambda)$ denotes the matrix defined in 
(\ref{b-frame}) with $\mu$ replaced by $\bar{\mu}$
and the Niessen elements described in Lemma \ref{niessen-lemma1},
then 
\begin{equation*}
    \mathbf{R}^b (\bar{\mu}, \lambda)^* J \mathbf{R}^b (\mu, \lambda) = 0.
\end{equation*}
\end{lemma}

As noted in \cite{HS2022}, with appropriate labeling, statements analogous to 
Lemmas \ref{krall-niessen-lemma},
\ref{niessen-lemma1} and \ref{niessen-lemma2}
can be established with $b$ replaced by $a$. 

Proceeding now with fixed values $\lambda_0 \in I$, 
and $\mu_0 \in \mathbb{C} \backslash \mathbb{R}$, we let 
$\{u_j^b (x; \mu_0, \lambda_0)\}_{j=1}^n$ denote a selection 
of Niessen elements as described in Lemma 
\ref{niessen-lemma1}, and we denote by $U^b (x; \mu_0, \lambda_0)$
the $2n \times n$ matrix comprising the vectors 
$\{u_j^b (x; \mu_0, \lambda_0)\}_{j=1}^n$ as its columns.
Likewise we let $\{u_j^a (x; \mu_0, \lambda_0)\}_{j=1}^n$
denote a selection of Niessen elements that can 
similarly be specified in association with $x = a$,
and we denote by $U^a (x; \mu_0, \lambda_0)$
the $2n \times n$ matrix comprising the vectors 
$\{u_j^a (x; \mu_0, \lambda_0)\}_{j=1}^n$ as its columns. 
In \cite{HS2022}, the authors verify that we can construct functions 
$\{\tilde{u}_j^a (x; \mu_0, \lambda_0)\}_{j=1}^n$
and $\{\tilde{u}_j^b (x; \mu_0, \lambda_0)\}_{j=1}^n$
so that for each $j \in \{1, 2, \dots, n\}$
we have 
$\tilde{u}_j^a (\cdot; \mu_0, \lambda_0), \tilde{u}_j^b (\cdot; \mu_0, \lambda_0)
\in \mathcal{D}_M$, and moreover 
\begin{equation} \label{niessen-modified}
    \tilde{u}_j^a (x; \mu_0, \lambda_0)
    = \begin{cases}
    u_j^a (x; \mu_0, \lambda_0) & \mathrm{near }\,\, x = a \\
    0 & \mathrm{near }\,\, x = b
    \end{cases}; \quad \quad
    \tilde{u}_j^b (x; \mu_0, \lambda_0)
    = \begin{cases}
    0 & \mathrm{near }\,\, x = a \\
    u_j^b (x; \mu_0, \lambda_0) & \mathrm{near }\,\, x = b.
    \end{cases}
\end{equation}
(See Lemma 2.5 in \cite{HS2022}, which is adapted directly 
from Lemma 3.1 in \cite{SunShi2010}.)

We now specify the domain 
\begin{equation} \label{D1-def}
    \mathcal{D}_{\mu_0, \lambda_0} := \mathcal{D}_c 
    + \mathrm{Span}\,\Big{\{} \{\tilde{u}_j^a (\cdot; \mu_0, \lambda_0)\}_{j=1}^n, 
    \{\tilde{u}_j^b (\cdot; \mu_0, \lambda_0)\}_{j=1}^n \Big{\}},
\end{equation}
and for each $\lambda \in I$ we denote by $\mathcal{T}_{\mu_0, \lambda_0} (\lambda)$ the restriction of $\mathcal{T}_M (\lambda)$ 
to $\mathcal{D}_{\mu_0, \lambda_0}$. We know from Theorem 2.1 of \cite{HS2022} that 
$\mathcal{T}_{\mu_0, \lambda_0} (\lambda_0)$ is a self-adjoint operator. In addition, 
we establish in the next theorem that for each $\lambda \in I$,
$\mathcal{T}_{\mu_0, \lambda_0} (\lambda)$ is essentially self-adjoint. We note that the proof of Theorem \ref{self-adjointness-theorem} also establishes Item (i) in Lemma \ref{self-adjoint-operator-lemma}.

\begin{theorem} \label{self-adjointness-theorem}
Let Assumptions {\bf (A)} through {\bf (D)} hold. 
Then for each $\lambda \in I$, 
the operator $\mathcal{T}_{\mu_0, \lambda_0} (\lambda)$ is essentially self-adjoint, and 
so in particular, $\mathcal{T} (\lambda) := \overline{\mathcal{T}_{\mu_0, \lambda_0} (\lambda)} 
= \mathcal{T}_{\mu_0, \lambda_0} (\lambda)^*$ 
is self-adjoint. (Here and below, overbar denotes closure.) The domain $\mathcal{D}$ of $\mathcal{T} (\lambda)$
is 
\begin{equation}
    \mathcal{D} = \{y \in \mathcal{D}_M: \lim_{x \to a^+} U^a (x; \mu_0, \lambda_0)^* J y(x) = 0,
    \quad \lim_{x \to b^-} U^b (x; \mu_0, \lambda_0)^* J y(x) = 0\}.
\end{equation}
\end{theorem}

\begin{proof}[Proof of Theorem \ref{self-adjointness-theorem}]
We begin by fixing some $\lambda \in I$ and checking that the 
operator $\mathcal{T}_{\mu_0, \lambda_0} (\lambda)$ is symmetric. 
Using (\ref{green1}), we immediately see that for any $y, z \in \mathcal{D}_c$ we have 
\begin{equation*}
    \langle \mathcal{T}_{\mu_0, \lambda_0} (\lambda) y, z \rangle_{\mathbb{B}_{\lambda}} 
    - \langle  y, \mathcal{T}_{\mu_0, \lambda_0} (\lambda) z \rangle_{\mathbb{B}_{\lambda}}
    = (Jy,z)_a^b = 0,
\end{equation*}
and we can similarly use (\ref{green1}) along with the identities 
\begin{equation*}
    (J y, \tilde{u}_j^a)_a^b = 0,
    \quad (Jy, \tilde{u}_j^b)_a^b = 0,
    \quad (J \tilde{u}_j^a, \tilde{u}_k^b)_a^b = 0,
\end{equation*}
for all $j, k \in \{1, 2, \dots, n\}$ (following from support 
of the elements in all cases). It remains to show that 
\begin{equation}
    (J \tilde{u}_j^a, \tilde{u}_k^a)_a^b = 0,
    \quad (J \tilde{u}_j^b, \tilde{u}_k^b)_a^b = 0,
\end{equation}
but these identities are immediate from Lemma \ref{krall-niessen-lemma} (along
with the analogous statement associated with $x = a$),  
so symmetry is established. 

Next, we will show that $\mathcal{T}_{\mu_0, \lambda_0} (\lambda)$ is essentially self-adjoint. 
According to Theorem 5.21 in \cite{Weidmann1980}, it suffices to 
show that for some (and hence for all) $\mu \in \mathbb{C}\backslash \mathbb{R}$,
\begin{equation} \label{essential-sa-condition}
    \overline{\ran(\mathcal{T}_{\mu_0, \lambda_0} (\lambda) - \mu I)} = L^2_{\mathbb{B}_{\lambda}} ((a, b), \mathbb{C}^{2n}),
    \quad {\rm and} \quad \overline{\ran(\mathcal{T}_{\mu_0, \lambda_0} (\lambda) - \bar{\mu} I)} = L^2_{\mathbb{B}_{\lambda}} ((a, b), \mathbb{C}^{2n}).
\end{equation}
Since we can proceed with any $\mu \in \mathbb{C} \backslash \mathbb{R}$, we can take
$\mu_0$ from (\ref{D1-def}) as our choice. 

We will show that 
\begin{equation} \label{ess-sa-cond}
    \ran(\mathcal{T}_{\mu_0, \lambda_0} (\lambda) - \mu_0 I)^{\perp} = \{0\},
    \quad {\rm and} \quad \ran(\mathcal{T}_{\mu_0, \lambda_0} (\lambda) - \bar{\mu}_0 I)^{\perp} = \{0\}, 
\end{equation}
from which (\ref{essential-sa-condition}) is clear, since 
\begin{equation}
    L^2_{\mathbb{B}_{\lambda}} ((a, b), \mathbb{C}^{2n}) = 
    \ran(\mathcal{T}_{\mu_0, \lambda_0} (\lambda) - \mu_0 I)^{\perp} \oplus  
    \overline{\ran(\mathcal{T}_{\mu_0, \lambda_0} (\lambda) - \mu_0 I)},
\end{equation}
and likewise with $\mu_0$ replaced by $\bar{\mu}_0$.

Starting with the second relation in (\ref{ess-sa-cond}),
we suppose that for some $u \in L^2_{\mathbb{B}_{\lambda}} ((a, b), \mathbb{C}^{2n})$,
\begin{equation*}
\langle (\mathcal{T}_{\mu_0, \lambda_0} (\lambda) - \bar{\mu}_0 I) \psi, u \rangle_{\mathbb{B}_{\lambda}} = 0,
\quad \forall \, \psi \in \mathcal{D}_{\mu_0, \lambda_0},
\end{equation*}
and our goal is to show that 
this implies that $u = 0$ in $L^2_{\mathbb{B}_{\lambda}} ((a, b), \mathbb{C}^{2n})$. 
First, if we restrict to $\psi \in \mathcal{D}_c$, then we have
\begin{equation}
  \langle (\mathcal{T}_c (\lambda) - \bar{\mu}_0 I) \psi, u \rangle_{\mathbb{B}_{\lambda}} = 0, 
  \quad \forall \, \psi \in \mathcal{D}_c.   
\end{equation}
This relation implies that $u \in \dom ((\mathcal{T}_c (\lambda) - \bar{\mu}_0 I)^*) \,\, (= \mathcal{D}_M)$,
so we're justified in writing 
\begin{equation}
  \langle \psi, (\mathcal{T}_M (\lambda) - \mu_0 I) u \rangle_{\mathbb{B}_{\lambda}} = 0, 
  \quad \forall \, \psi \in \mathcal{D}_c.   
\end{equation}
Since $\mathcal{D}_c$ is dense in $L^2_{\mathbb{B}_{\lambda}} ((a, b), \mathbb{C}^{2n})$,
we can conclude that $u$ must satisfy $(\mathcal{T}_M (\lambda) - \mu_0 I) u = 0$. 

Next, we also have the relation 
\begin{equation} \label{with-psi}
  \langle (\mathcal{T}_{\mu_0, \lambda_0} (\lambda) - \bar{\mu}_0 I) \psi, u \rangle_{\mathbb{B}_{\lambda}} = 0, 
  \quad \forall \, \psi \in \mathrm{Span}\,\Big{\{} \{\tilde{u}_j^a\}_{j=1}^n, \{\tilde{u}_j^b\}_{j=1}^n \Big{\}}.   
\end{equation}
For each $j \in \{1, 2, \dots, n\}$, $\tilde{u}^{a, b}_j (\cdot; \mu_0, \lambda_0) \in \mathcal{D}_M$, and 
we've already established that $u \in \mathcal{D}_M$, so we can apply 
Green's identity (\ref{green1}) to see that with $\psi$ as in (\ref{with-psi})
\begin{equation}
    0 = \langle (\mathcal{T}_{\mu_0, \lambda_0} (\lambda) - \bar{\mu}_0 I) \psi, u \rangle_{\mathbb{B}_{\lambda}}
    = \langle \psi, (\mathcal{T}_M (\lambda) - \mu_0 I) u \rangle_{\mathbb{B}_{\lambda}}
    + (J \psi, u)_a^b.
\end{equation}
Since $(\mathcal{T}_M (\lambda) - \mu_0 I) u = 0$, we see that 
$(J \psi, u)_a^b = 0$. In addition, since 
$\tilde{u}_j^b$ is zero near $x=a$, we have
(taking $\psi = \tilde{u}_j^b$)
$(J \tilde{u}_j^b, u)_a = 0$, and consequently we can 
conclude $(J \tilde{u}_j^b, u)_b = 0$.
That is, 
\begin{equation*}
    \lim_{x \to b^-} u(x)^* J \tilde{u}_j^b (x; \mu_0, \lambda_0) = 0.
\end{equation*}
If we take the adjoint of this relation, and recall that 
$\tilde{u}_j^b$ is identical to $u_j^b$ for $x$ near
$b$, then we can express this limit in our preferred form 
\begin{equation*}
    \lim_{x \to b^-} u_j^b (x; \mu_0, \lambda_0)^* J u(x) = 0.
\end{equation*}
This last relation is true for all $j \in \{1, 2, \dots, n\}$,
and a similar relation holds near $x = a$. We can summarize these
observations with the following limits
\begin{equation} \label{u-limits}
    \begin{aligned}
        B^a (\mu_0, \lambda_0) u 
        &:= \lim_{x \to a^+} U^a (x; \mu_0, \lambda_0)^* J u(x) = 0, \\ 
        B^b (\mu_0, \lambda_0) u 
        &:= \lim_{x \to b^-} U^b (x; \mu_0, \lambda_0)^* J u(x) = 0,
    \end{aligned}
\end{equation}
where we have also specified some convenient boundary operator notation. 

At this point, we introduce a useful way of representing elements of $\mathcal{D}_M$. 
Since $\mathcal{D}_M$ does not depend on $\lambda$, we can assert that if $y \in \mathcal{D}_M$,
then for any $\lambda \in I$ there exists 
$f (\cdot; \lambda) \in L^2_{\mathbb{B}_{\lambda}} ((a, b), \mathbb{C}^{2n})$
so that $\mathcal{T}_M (\lambda) y = f (\cdot; \lambda)$. In particular, this is true for 
$\lambda_0$ as above, allowing us to write 
\begin{equation} \label{inhomogeneous-equation}
    \mathcal{T}_M (\lambda_0) y - \mu_0 y = f (\cdot; \lambda_0) - \mu_0 y. 
\end{equation}
Recalling the definition of $\mathcal{T}_M (\lambda_0)$, 
(\ref{inhomogeneous-equation}) can be viewed as an inhomogeneous 
system of ODE for which solutions can be expressed in the usual 
way as the sum of a solution to the inhomogeneous system and 
an appropriate solution to the associated homogeneous system. For 
the particular solution, we can take advantage of the fact that 
$\mathcal{T} (\lambda_0)$ is already known to be a self-adjoint 
operator with domain $\mathcal{D}$, so that with 
$\mu_0 \in \mathbb{C} \backslash \mathbb{R}$,
we can solve (\ref{inhomogeneous-equation}) with 
\begin{equation*}
    y_p = (\mathcal{T} (\lambda_0) - \mu_0)^{-1} (f (\cdot; \lambda_0) - \mu_0 y).
\end{equation*}
We have $y_p \in \mathcal{D}$, so  
\begin{equation} \label{yp-limits}
    \begin{aligned}
        B^a (\mu_0, \lambda_0) y_p(\cdot) &= 0, \\ 
        B^b (\mu_0, \lambda_0) y_p(\cdot) &= 0.
    \end{aligned}
\end{equation}
Turning to solutions of the homogeneous equation associated with 
(\ref{inhomogeneous-equation}), the Niessen elements 
$\{u_j^b (x; \mu_0, \lambda_0)\}_{j=1}^n$ and $\{v_j^b (x; \mu_0, \lambda_0)\}_{j=1}^{r_b}$
comprise (by construction) a basis for the $m_b$-dimensional space of solutions 
to the homogeneous equation associated with 
(\ref{inhomogeneous-equation}) that lie right in $(a, b)$. It follows
that $y$ can be expressed as 
\begin{equation*}
    y(x) = y_p (x) + \sum_{j=1}^n c_j^b (\mu_0, \lambda_0) u_j^b (x; \mu_0, \lambda_0)
    + \sum_{j=1}^{r_b} d_j^b (\mu_0, \lambda_0) v_j^b (x; \mu_0, \lambda_0),
\end{equation*}
for some constants $\{c_j^b (\mu_0, \lambda_0)\}_{j=1}^n$ and 
$\{d_j^b (\mu_0, \lambda_0)\}_{j=1}^{r_b}$. Using (\ref{yp-limits})
along with Lemma \ref{krall-niessen-lemma}, we can write 
\begin{equation} \label{u-limits-compare}
     B^b (\mu_0, \lambda_0) y (\cdot)
     = \sum_{j=1}^{r_b} d_j^b (\mu_0, \lambda_0) 
     B^b (\mu_0, \lambda_0) v_j^b (\cdot; \mu_0, \lambda_0).
\end{equation}

Returning now to (\ref{u-limits}), since $u \in \mathcal{D}_M$, we know 
from (\ref{u-limits-compare}) that for some constants 
$\{d_j^b (\mu_0, \lambda_0)\}_{j=1}^{r_b}$
we have 
\begin{equation*}
     B^b (\mu_0, \lambda_0) u (\cdot)
     = \sum_{j=1}^{r_b} d_j^b (\mu_0, \lambda_0) 
     B^b (\mu_0, \lambda_0) v_j^b (\cdot; \mu_0, \lambda_0).
\end{equation*}
According to Lemma \ref{krall-niessen-lemma}, 
\begin{equation*}
    (B^b (\mu_0, \lambda_0) v_j^b (\cdot; \mu_0, \lambda_0))_i
    = \begin{cases}
    0 & i \ne j \\
    \kappa_j^b (\mu_0, \lambda_0) \ne 0 & i = j 
    \end{cases}.
\end{equation*}
In this way, we see that 
\begin{equation*}
    B^b (\mu_0, \lambda_0) u (\cdot)
    = \begin{pmatrix}
    d_1^b (\mu_0, \lambda_0) \kappa_1^b (\mu_0, \lambda_0) 
    & \dots 
    & d_{r_b}^b (\mu_0, \lambda_0) \kappa_{r_b}^b (\mu_0, \lambda_0) & 0 & \dots & 0
    \end{pmatrix}^T.
\end{equation*}
Since the constants $\{\kappa_j^b (\mu_0, \lambda_0)\}_{j=1}^{r_b}$ are 
all non-zero, we can only have the required relation $B^b (\mu_0, \lambda_0) u (\cdot) = 0$
if $d_j^b (\mu_0, \lambda_0) = 0$ for all 
$j \in \{1, 2, \dots, r_b\}$, and in this case 
\begin{equation} \label{u-at-b}
    u(x) = u_p (x) + \sum_{j=1}^n c_j^b (\mu_0, \lambda_0) u_j^b (x; \mu_0, \lambda_0),
\end{equation}
for some $u_p \in \mathcal{D}$ and some constants $\{c_j^b (\mu_0, \lambda_0)\}_{j=1}^n$. 
Likewise, there exist expansion
constants $\{c_j^a (\mu_0, \lambda_0)\}_{j=1}^n$ so that 
\begin{equation} \label{u-at-a}
    u(x) = u_p (x) + \sum_{j=1}^n c_j^a (\mu_0, \lambda_0) u_j^a (x; \mu_0, \lambda_0).
\end{equation}
For convenient reference, these last observations will be summarized in Lemma \ref{representation-lemma}
at the end of this section. 

Since $\mathcal{T}_M (\lambda) u = \mu_0 u$, we can use (\ref{green2}) from Lemma
\ref{green-lemma} to see that 
\begin{equation} \label{green2a}
    2i ({\rm Im}\,\mu_0) \|u\|_{\mathbb{B}_{\lambda}}^2
    = (Ju, u)_a^b = (Ju, u)_b - (Ju, u)_a.
\end{equation}
We can now check that relation (\ref{u-at-b}) allows us to conclude that 
$(Ju, u)_b = 0$. First, to see that $(Ju_p, u_p)_b = 0$, we use 
Lemma 2.5 in \cite{HS2022} to construct $\tilde{u}_p \in \mathcal{D}$
so that 
\begin{equation*}
    \tilde{u}_p (x) = 
    \begin{cases}
    0 & x \, {\rm near} \, a \\
    u_p (x) & x \, {\rm near} \, b.
    \end{cases}
\end{equation*}
Then, since $\mathcal{T} (\lambda_0)$ is self-adjoint, we can use 
(\ref{green1}) from Lemma \ref{green-lemma} to write 
\begin{equation*}
    0 = \langle \mathcal{T} (\lambda_0) \tilde{u}_p, \tilde{u}_p \rangle
      -  \langle \tilde{u}_p, \mathcal{T} (\lambda_0) \tilde{u}_p \rangle
      = (J\tilde{u}_p, \tilde{u}_p)_a^b = (J\tilde{u}_p, \tilde{u}_p)_b,
\end{equation*}
where the final equality follows from the support of $\tilde{u}_p$. 
Next, according to Lemma \ref{krall-niessen-lemma}
$(J u_k^b (\cdot; \mu_0, \lambda_0), u_j^b (\cdot; \mu_0, \lambda_0))_b = 0$
for all $j, k \in \{1, 2, \dots, n\}$. Last, for terms of the form 
$(J u_p, u_j^b (\cdot; \mu_0, \lambda_0))_b$, we can use 
Lemma 2.5 in \cite{HS2022} to truncate $u_j^b (\cdot; \mu_0, \lambda_0)$ 
and proceed as with $(Ju_p, u_p)_b$. Similarly, using (\ref{u-at-a})
we can show that $(Ju, u)_a = 0$, and consequently from (\ref{green2a})
we conclude that $\|u\|_{\mathbb{B}_{\lambda}} = 0$, and so 
$u = 0$ in $L^2_{\mathbb{B}_{\lambda}} ((a, b), \mathbb{C}^{2n})$, which 
is what we wanted to show (i.e., we have verified the second condition
in (\ref{ess-sa-cond})). 

We turn now to verifying the first condition in (\ref{ess-sa-cond}).
For this, we suppose that for some $u \in L^2_{\mathbb{B}_{\lambda}} ((a, b), \mathbb{C}^{2n})$
we have 
\begin{equation} \label{condition-for-first}
    \langle (\mathcal{T}_{\mu_0, \lambda_0} (\lambda) - \mu_0 I) \psi, u \rangle_{\mathbb{B}_{\lambda}} = 0,
    \quad \forall \, \psi \in \mathcal{D}_{\mu_0, \lambda_0},
\end{equation}
and our goal is to show that this implies $u = 0$ in $L^2_{\mathbb{B}_{\lambda}} ((a, b), \mathbb{C}^{2n})$.
Precisely as with the second condition in (\ref{ess-sa-cond}), we can check 
that (\ref{condition-for-first}) implies that $u \in \mathcal{D}_M$, 
$\mathcal{T}_{\mu_0, \lambda_0} (\lambda) u = \bar{\mu}_0 u$, and also that 
relations (\ref{u-limits}) hold. Since $u \in \mathcal{D}_M$ and relations (\ref{u-limits}) hold, 
we can conclude that there exist constants $\{\tilde{c}_j^b (\mu_0, \lambda_0)\}_{j=1}^n$ so 
that 
\begin{equation} \label{u-at-b-bar}
    u(x) = \tilde{u}_p (x) + \sum_{j=1}^n \tilde{c}_j^b (\mu_0, \lambda_0) u_j^b (x; \mu_0, \lambda_0),
\end{equation}
for some $\tilde{u}_p \in \mathcal{D}$. As in 
the previous case, we can conclude that $(Ju, u)_b = 0$, and by a similar 
argument that $(Ju, u)_a = 0$. Since $\mathcal{T}_{M} (\lambda) u = \bar{\mu}_0 u$, 
we can use (\ref{green2}) from Lemma \ref{green-lemma} to see that 
\begin{equation*}
    2i ({\rm Im}\, \bar{\mu}_0) \|u\|_{\mathbb{B}_{\lambda}}^2 
    = (Ju, u)_a^b = 0, 
\end{equation*}
confirming that $u = 0$ in $L^2_{\mathbb{B}_{\lambda}} ((a, b), \mathbb{C}^{2n})$. 

In summary to this point, we have shown that for each $\lambda \in I$,
if we restrict the maximal operator $\mathcal{T}_M (\lambda)$ to the domain 
$\mathcal{D}_{\mu_0, \lambda_0}$ specified in (\ref{D1-def}), then we obtain 
an essentially self-adjoint operator $\mathcal{T}_{\mu_0, \lambda_0} (\lambda)$. 
Since $\mathcal{T}_{\mu_0, \lambda_0} (\lambda)$ is essentially self-adjoint, 
we can set $\mathcal{T} (\lambda) := \mathcal{T}_{\mu_0, \lambda_0} (\lambda)^*$,
and conclude that $\mathcal{T} (\lambda)$ is self-adjoint (see, e.g. Theorem 5.20
in \cite{Weidmann1980}). 

The final item of Theorem \ref{self-adjointness-theorem} to be clear about is 
the domain of $\mathcal{T} (\lambda)$. For this, we first observe that 
\begin{equation*}
    \mathcal{T}_c (\lambda) \subset \mathcal{T}_{\mu_0, \lambda_0} (\lambda)
    \implies \mathcal{T}_{\mu_0, \lambda_0} (\lambda)^* \subset \mathcal{T}_c (\lambda)^*.
\end{equation*}
Using the relations $\mathcal{T} (\lambda) = \mathcal{T}_{\mu_0, \lambda_0} (\lambda)^*$
and $\mathcal{T}_M (\lambda) = \mathcal{T}_c (\lambda)^*$, we see that 
$\mathcal{T} (\lambda) \subset \mathcal{T}_M (\lambda)$. This leaves only 
the question of what, if any, additional restrictions elements in the 
domain of $\mathcal{T} (\lambda)$ must satisfy. To understand this, we recall 
that, by definition, the domain of $\mathcal{T} (\lambda)$ is
\begin{equation*}
\begin{aligned}
    \mathcal{D}
    &= \{ u\in \mathcal{D}_M: \textrm{there exists } v \in L^2_{\mathbb{B}_{\lambda}} ((a, b), \mathbb{C}^{2n}) \\
    &\textrm{ so that } \langle \mathcal{T}_{\mu_0, \lambda_0} (\lambda) \psi, u \rangle_{\mathbb{B}_{\lambda}}
    = \langle \psi, v \rangle_{\mathbb{B}_{\lambda}} \, \forall \, \psi \in \mathcal{D}_{\mu_0, \lambda_0} \}.
\end{aligned}
\end{equation*}

Let $u \in \mathcal{D}_M$. For all $\psi \in \mathcal{D}_c$, we can immediately write
\begin{equation*}
    \langle \mathcal{T}_{\mu_0, \lambda_0} (\lambda) \psi, u \rangle_{\mathbb{B}_{\lambda}}
    = \langle \mathcal{T}_c (\lambda) \psi, u \rangle_{\mathbb{B}_{\lambda}}
    = \langle \psi, \mathcal{T}_M (\lambda) u \rangle_{\mathbb{B}_{\lambda}}
    = \langle \psi,  v \rangle_{\mathbb{B}_{\lambda}},
    \quad v = \mathcal{T}_M (\lambda) u.
\end{equation*}
In particular, this places no additional restrictions on $u$. On the other hand, for 
any $j \in \{1, 2, \dots, n\}$, we have from (\ref{green1}) in Lemma \ref{green-lemma}
\begin{equation*}
    \langle \mathcal{T}_{\mu_0, \lambda_0} (\lambda) \tilde{u}_j^b, u \rangle_{\mathbb{B}_{\lambda}}
    - \langle \tilde{u}_j^b,  \mathcal{T}_M (\lambda) u \rangle_{\mathbb{B}_{\lambda}}
    = (J \tilde{u}_j^b, u)_b, 
\end{equation*}
where we've recalled that $\tilde{u}_j^b$ is identically zero for $x$ near $a$. We see 
that in order 
to have $u \in \mathcal{D}$, we must have $(J \tilde{u}_j^b, u)_b = 0$, and since this is
true for all $j \in \{1, 2, \dots, n\}$, we can conclude that 
\begin{equation*}
    \lim_{x \to b^-} U^b (x; \mu_0, \lambda_0)^* J u(x) = 0. 
\end{equation*}
A similar argument holds with each $\tilde{u}_j^b$ replaced by 
$\tilde{u}_j^a$, leading to the full characterization of $\mathcal{D}$
given in the statement of Theorem \ref{self-adjointness-theorem}.
\end{proof}

By essentially identical considerations, we can establish 
a similar theorem for $\mathcal{T}^{\alpha} (\cdot)$. In this 
case, we take $\alpha (\lambda)$ as specified in the statement 
of Lemma \ref{self-adjoint-operator-lemma}(ii), and we take 
$U^{\alpha} (x; \lambda)$ to solve
\begin{equation*}
    J (U^{\alpha})' (x; \lambda) = \mathbb{B} (x; \lambda) U^{\alpha} (x; \lambda),
    \quad U^{\alpha} (a; \lambda) = J \alpha (\lambda)^*.
\end{equation*}
If we denote the columns of $U^{\alpha} (x; \lambda)$ by 
$\{u_j^{\alpha} (x; \lambda)\}_{j=1}^n$, and their respective 
truncations $\{\tilde{u}_j^{\alpha} (x; \lambda)\}_{j=1}^n$,
then for a fixed pair $(\mu_0, \lambda_0) \in (\mathbb{C} \backslash \mathbb{R}) \times I$ 
we can define the domain  
\begin{equation} \label{D1-def-alpha}
    \mathcal{D}_{\mu_0, \lambda_0}^{\alpha} := \mathcal{D}_c 
    + \mathrm{Span}\,\Big{\{} \{\tilde{u}_j^{\alpha} (\cdot; \mu_0, \lambda_0)\}_{j=1}^n, \{\tilde{u}_j^b (\cdot; \mu_0, \lambda_0)\}_{j=1}^n \Big{\}}.
\end{equation}
We denote by $\mathcal{T}_{\mu_0, \lambda_0}^{\alpha} (\lambda)$ 
the restriction of $\mathcal{T}_M (\lambda)$ to $\mathcal{D}_{\mu_0, \lambda_0}^{\alpha}$. We note that the proof of Theorem 
\ref{L-alpha-theorem} (omitted due to its similarity to the proof of Theorem \ref{self-adjointness-theorem})
also establishes Item (ii) in Lemma \ref{self-adjoint-operator-lemma}.

\begin{theorem} \label{L-alpha-theorem}
Let Assumptions {\bf (A)} through {\bf (D)} hold, along with 
{\bf (A)$^\prime$}. Then for each $\lambda \in I$ 
the operator $\mathcal{T}_{\mu_0, \lambda_0}^{\alpha} (\lambda)$ is essentially self-adjoint, and 
so in particular, $\mathcal{T}^{\alpha} (\lambda) := \overline{\mathcal{T}_{\mu_0, \lambda_0}^{\alpha} (\lambda)} 
= (\mathcal{T}_{\mu_0, \lambda_0}^{\alpha} (\lambda))^*$ 
is self-adjoint. The domain $\mathcal{D}^{\alpha} (\lambda)$ of $\mathcal{T}^{\alpha} (\lambda)$
is 
\begin{equation}
    \mathcal{D}^{\alpha} (\lambda) = \{y \in \mathcal{D}_M: \alpha (\lambda) y(a) = 0,
    \quad \lim_{x \to b^-} U^b (x; \mu_0, \lambda_0)^* J y(x) = 0\}.
\end{equation}
\end{theorem}

During the proof of Theorem \ref{self-adjointness-theorem}, we established a 
useful representation for elements $y \in \mathcal{D}_M$ in terms of Niessen 
elements, and for future reference we summarize this representation as
a lemma. 

\begin{lemma} \label{representation-lemma}
Let Assumptions {\bf (A)} through {\bf (D)} hold, and fix
any $\lambda_0 \in I$ and 
any $\mu_0 \in \mathbb{C} \backslash \mathbb{R}$. In addition, 
let $\{u_j^b (x; \mu_0, \lambda_0)\}_{j=1}^n$ 
and $\{v_j^b (x; \mu_0, \lambda_0)\}_{j=1}^{r_b}$ denote 
the Niessen elements constructed in Lemma \ref{niessen-lemma1} 
for (\ref{linear-hammy}) with $\lambda = \lambda_0$ and $\mu = \mu_0$.
Then given any $y \in \mathcal{D}_M$, there exist values $\{c_j^b (\mu_0, \lambda_0)\}_{j=1}^n$ and 
$\{d_j^b (\mu_0, \lambda_0)\}_{j=1}^{r_b}$ so that 
$y$ can be expressed as 
\begin{equation*}
    y(x) = y_p (x) + \sum_{j=1}^n c_j^b (\mu_0, \lambda_0) u_j^b (x; \mu_0, \lambda_0)
    + \sum_{j=1}^{r_b} d_j^b (\mu_0, \lambda_0) v_j^b (x; \mu_0, \lambda_0),
\end{equation*}
where
\begin{equation}
    y_p = (\mathcal{T} (\lambda_0) - \mu_0)^{-1} (f (\cdot; \lambda_0) - \mu_0 y),
\end{equation}
Moreover, 
\begin{equation*} 
     B^b (\mu_0, \lambda_0) y (\cdot)
     = \sum_{j=1}^{r_b} d_j^b (\mu_0, \lambda_0) 
     B^b (\mu_0, \lambda_0) v_j^b (\cdot; \mu_0, \lambda_0).
\end{equation*}
A similar statement holds with $b$ replaced by $a$.
\end{lemma}

\section{Continuation to $\mathbb{R}$}
\label{extension-section}

In the preceding considerations, we fixed some 
$\lambda_0 \in I$, along with 
some $\mu_0 \in \mathbb{C} \backslash \mathbb{R}$ 
and used these values to specify the self-adjoint
operators $\mathcal{T} (\lambda)$ and $\mathcal{T}^{\alpha} (\lambda)$
for each $\lambda \in I$.
With these operators in hand, we fix some interval $[\lambda_1, \lambda_2] \subset I$, 
$\lambda_1 < \lambda_2$, for which we have the exclusion
$0 \notin \sigma_{\ess} (\mathcal{T} (\lambda))$ for 
all $\lambda \in [\lambda_1, \lambda_2]$. Our next goal 
is to fix any $\lambda \in [\lambda_1, \lambda_2]$
and construct a collection $\{u^a_j (x; \lambda)\}_{j=1}^n$
of linearly independent solutions to
\begin{equation} \label{extension-hammy}
    Ju' = \mathbb{B} (x; \lambda) u
\end{equation}
that lie left in $(a, b)$, along with 
a collection $\{u^b_j (x; \lambda)\}_{j=1}^n$
of linearly independent solutions to (\ref{extension-hammy})
that lie right in $(a, b)$. One difficulty we 
encounter is that the matrix $\mathcal{A} (x; \mu, \lambda)$
specified in (\ref{mathcal-A-defined})
is not defined for $\mu = 0$,
and so we cannot directly extend Niessen's 
development to this setting. Instead of 
extending Niessen's development directly, 
we will take advantage of our assumption that 
for all $\lambda \in [\lambda_1, \lambda_2]$,
$0 \notin \sigma_{\ess} (\mathcal{T} (\lambda))$,
along with a theorem from \cite{Weidmann1987} about self-adjoint
operators. 

As a starting point, we fix some $c \in (a, b)$ and consider 
(\ref{extension-hammy}) on $(c, b)$ with 
boundary conditions 
\begin{equation} \label{boundary-c}
    \gamma y(c) = 0,
\end{equation}
and 
\begin{equation} \label{boundary-b}
    \lim_{x \to b^-} U^b (x; \mu_0, \lambda_0)^* J y(x) = 0,
\end{equation}
where $U^b (x; \mu_0, \lambda_0)$ is as in Theorem \ref{self-adjointness-theorem}
and the boundary matrix $\gamma \in \mathbb{C}^{n \times 2n}$
satisfies 
\begin{equation} \label{boundary-matrix}
  \rank \gamma = n, 
    \, \textrm{ and } \, \gamma J \gamma^* = 0,
\end{equation}
and will be specified more precisely below as needed. 
Similarly as in Section \ref{operator-section}, we can associate 
this boundary value problem 
with a self-adjoint operator $\mathcal{T}_{c, b}^{\gamma} (\lambda)$,
with domain 
\begin{equation*}
    \mathcal{D}_{c, b}^{\gamma} := \{y \in \mathcal{D}_{c, b, M}: 
    \gamma y (c) = 0, 
    \quad \lim_{x \to b^-} U^b (x; \mu_0, \lambda_0) J y(x) = 0 \}.
\end{equation*}
Here, $\mathcal{D}_{c, b, M}$ denotes the domain of the 
maximal operator associated with (\ref{extension-hammy}) 
on $(c, b)$.

We start with a lemma. 

\begin{lemma} \label{lemma2-6prime} 
Let Assumptions {\bf (A)} through {\bf (D)} hold.  
For any fixed $\lambda \in [\lambda_1, \lambda_2]$,
suppose $u^b (x; \lambda)$ and $v^b (x; \lambda)$ denote any two solutions 
of (\ref{extension-hammy}) (if such solutions exist) that lie 
right in $(c, b)$ and satisfy (\ref{boundary-b}). 
Then $(J u^b (\cdot; \lambda), v^b (\cdot; \lambda))_b = 0$.
\end{lemma}

\begin{proof} First, using Lemma 2.5 from \cite{HS2022}, we can construct 
functions $\tilde{u}^b (\cdot; \lambda), \tilde{v}^b (\cdot; \lambda)
\in \mathcal{D}_{c, b, M}$ so that 
\begin{equation*}
    \tilde{u}^b (x; \lambda)
    = \begin{cases}
    0 & \mathrm{near }\,\, x = c \\
    u^b (x; \lambda) & \mathrm{near }\,\, x = b,
    \end{cases}\quad \quad
    \tilde{v}^b (x; \lambda)
    = \begin{cases}
    0 & \mathrm{near }\,\, x = c \\
    v^b (x; \lambda) & \mathrm{near }\,\, x = b.
    \end{cases}
\end{equation*}
Since $\tilde{u}^b (x; \lambda)$ and $\tilde{v}^b (x; \lambda)$  
lie right in $(c, b)$ and satisfy (\ref{boundary-b}), 
it's clear that $\tilde{u}^b (x; \lambda), \tilde{v}^b (x; \lambda)$
are contained in $\mathcal{D}_{c, b}^{\gamma}$. 
Using self-adjointness 
of $\mathcal{T}_{c, b}^{\gamma} (\lambda)$, along with 
Green's identity, we can write 
\begin{equation*}
    \begin{aligned}
    0 &= \langle \mathcal{T}_{c, b}^{\gamma} (\lambda) \tilde{u}^b (\cdot; \lambda), \tilde{v}^b (\cdot; \lambda) \rangle_{\mathbb{B}_{\lambda}}
    - \langle \tilde{u}^b (\cdot; \lambda), \mathcal{T}_{c, b}^{\gamma} (\lambda) \tilde{v}^b (\cdot; \lambda) \rangle_{\mathbb{B}_{\lambda}} \\
    &= (J \tilde{u}^b (\cdot; \lambda), \tilde{v}^b (\cdot; \lambda))_c^b 
    = (J \tilde{u}^b (\cdot; \lambda), \tilde{v}^b (\cdot; \lambda))_b. 
    \end{aligned}
\end{equation*}
Since  $\tilde{u}^b (x; \lambda), \tilde{v}^b (x; \lambda)$ are identical 
to $u^b (x; \lambda), v^b(x; \lambda)$ for $x$ near $b$, this gives the claim.
\end{proof}

\begin{lemma} \label{lemma2-7prime}
Let Assumptions {\bf (A)} through {\bf (D)} hold.
Then for any fixed $\lambda \in [\lambda_1, \lambda_2]$, 
the space of solutions 
of (\ref{extension-hammy}) (if such solutions exist)
that lie right in $(c, b)$ and satisfy (\ref{boundary-b})
has dimension at most $n$. In the event that the dimension 
of this space is $n$, we let $\{u^b_j (x; \lambda)\}_{j=1}^n$
denote a choice of basis. Then for each $x \in (c, b)$ the 
vectors $\{u^b_j (x; \lambda)\}_{j=1}^n$ comprise the 
basis for a Lagrangian subspace of $\mathbb{C}^{2n}$.
\end{lemma}

\begin{proof}
Let $d$ denote the dimension of the space of solutions 
of (\ref{extension-hammy}) that lie right in $(c, b)$ and 
satisfy (\ref{boundary-b}), and suppose $d \ge n$. Let
$\{u_j^b (x; \lambda)\}_{j=1}^d$ denote a basis for 
this space, and notice that for any 
$j, k \in \{1, 2, \dots, d\}$ (and with $^{\prime}$ denoting
differentiation with respect to $x$), 
\begin{equation*}
    \begin{aligned}
    (u_j^b &(x; \lambda)^* J u^b_k (x; \lambda))'
    = u_j^{b \, \prime} (x; \lambda)^* J u^b_k (x; \lambda)
    + u^b_j (x; \lambda)^* J u_k^{b \, \prime} (x; \lambda) \\
    &= - (J u_j^{b \, \prime} (x; \lambda))^* u^b_k (x; \lambda)
    + u_j (x; \lambda)^* J u_k^{b \, \prime} (x; \lambda) \\
    &= - (\mathbb{B} (x; \lambda) u_j^b (x; \lambda))^* u^b_k (x; \lambda)
    +  u_j^b (x; \lambda)^* \mathbb{B} (x; \lambda) u_k^b (x; \lambda) \\
    &= - u_j^b (x; \lambda)^* \mathbb{B} (x; \lambda) u_k^b (x; \lambda)
    + u_j^b (x; \lambda)^* \mathbb{B} (x; \lambda) u_k^b (x; \lambda)
    = 0.
    \end{aligned}
\end{equation*}
We see that $u_j^b (x; \lambda)^* J u^b_k (x; \lambda)$ is constant for 
all $x \in (c, b)$. In addition, according to Lemma 
\ref{lemma2-6prime}, we have 
\begin{equation*}
    \lim_{x \to b^-} u_j^b (x; \lambda)^* J u^b_k (x; \lambda) = 0.
\end{equation*}
We conclude that $u_j^b (x; \lambda)^* J u^b_k (x; \lambda) = 0$ for 
all $x \in (c, b)$.

We see immediately that the first $n$ elements $\{u^b_j (x; \lambda)\}_{j=1}^n$
(or any other $n$ elements taken from $\{u_j^b (x; \lambda)\}_{j=1}^d$) 
form the basis for a Lagrangian subspace of $\mathbb{C}^{2n}$ for all
$x \in (c, b)$. If $d > n$, we get a contradiction to the maximality 
of Lagrangian subspaces, and so we can conclude that $d = n$ (recalling
that this is under the assumption that $d \ge n$). This, of course, 
leaves open the possibility that the dimension 
of the space of solutions 
of (\ref{extension-hammy}) that lie right in $(c, b)$ and 
satisfy (\ref{boundary-b}) is less than $n$.
\end{proof}

\begin{lemma} \label{lemma2-8prime}
Let Assumptions {\bf (A)} through {\bf (D)} hold. 
Then for any fixed $\lambda \in [\lambda_1, \lambda_2]$, 
there exists a matrix $\gamma \in \mathbb{C}^{n \times 2n}$ 
satisfying (\ref{boundary-matrix}) so that $0$
is not an eigenvalue of $\mathcal{T}_{c, b}^{\gamma} (\lambda)$.
\end{lemma}

\begin{proof}
First, we recall that $0$ is an eigenvalue of 
$\mathcal{T}_{c, b}^{\gamma} (\lambda)$ if and only if 
there exists a solution 
\begin{equation*}
    y (\cdot; \lambda) 
    \in \AC_{\loc} ([c, b), \mathbb{C}^{2n}) 
    \cap L^2_{\mathbb{B}_{\lambda}} ((c, b), \mathbb{C}^{2n}) 
\end{equation*}
to (\ref{extension-hammy}) so that (\ref{boundary-c})
and (\ref{boundary-b}) are both satisfied. 
Also, according to Lemma \ref{lemma2-7prime}, 
the space of solutions of (\ref{extension-hammy})
that lie right in $(c, b)$ and satisfy (\ref{boundary-b})
has dimension at most $n$. We begin by assuming
that this space of solutions has dimension $n$, and we denote
a basis for the space by $\{u^b_j (x; \lambda)\}_{j=1}^n$.

We let $\Phi (x; \lambda)$ denote a fundamental 
matrix for (\ref{extension-hammy}), initialized 
by $\Phi (c; \lambda) = I_{2n}$. If $U^b (x; \lambda)$
denotes the matrix comprising $\{u^b_j (x; \lambda)\}_{j=1}^n$
as its columns, then there exists a $2n \times n$ 
matrix $\mathbf{R}^b (\lambda) = \genfrac{(}{)}{0pt}{1}{R^b (\lambda)}{S^b (\lambda)}$
so that 
\begin{equation*}
    U^b (x; \lambda) = \Phi (x; \lambda) \mathbf{R}^b (\lambda),
\end{equation*}
for all $x \in [c, b)$ (i.e., $\mathbf{R}^b (\lambda) = U^b (c; \lambda)$). 
We know from Lemma \ref{lemma2-7prime} that $U^b (c; \lambda)$ is a frame 
for a Lagrangian subspace of $\mathbb{C}^{2n}$, and it follows 
immediately that the same is true for $\mathbf{R}^b (\lambda)$.
By taking a derivative in $x$, we can readily 
check that $\Phi (x; \lambda)^* J \Phi (x; \lambda)$ is constant 
in $x$, and evaluation at $x=c$ yields the useful identity  
\begin{equation} \label{fundamental-matrix-identity}
    \Phi (x; \lambda)^* J \Phi (x; \lambda) = J.
\end{equation}
Using this relation, we can compute 
\begin{equation*}
    U^b (x; \lambda)^* J U^b (x; \lambda) = 
    \mathbf{R}^b (\lambda)^* \Phi (x; \lambda)^* J \Phi (x; \lambda) \mathbf{R}^b (\lambda)
     = \mathbf{R}^b (\lambda)^* J \mathbf{R}^b (\lambda).
\end{equation*}

The value $\mu = 0$ will be an eigenvalue of 
$\mathcal{T}_{c, b}^{\gamma} (\lambda)$
if and only if there exists a vector 
$v \in \mathbb{C}^n$ so that 
$y(x; \lambda) = \Phi (x; \lambda) \mathbf{R}^b (\lambda) v$
satisfies 
\begin{equation*}
    \gamma y(c; \lambda) = 0,
\end{equation*}
which we can express (since $\Phi (c; \lambda) = I_{2n}$)
as $\gamma \mathbf{R}^b (\lambda) v = 0$. This relation will 
hold for a vector $v \ne 0$ if and only if the Lagrangian 
spaces with frames $J \gamma^*$ and $\mathbf{R}^b (\lambda)$
intersect. We choose $\gamma = \mathbf{R}^b (\lambda)^*$, noting
that in this case 
\begin{equation*}
    \gamma J \gamma^*
    = \mathbf{R}^b (\lambda)^* J \mathbf{R}^b (\lambda) = 0
\end{equation*}
(i.e., this is a valid choice for $\gamma$, satisfying (\ref{boundary-matrix})) 
but $\gamma \mathbf{R}^b (\lambda) =  \mathbf{R}^b (\lambda)^* \mathbf{R}^b (\lambda)$
is certainly non-singular, so $0$ is not an 
eigenvalue of $\mathcal{T}_{c, b}^{\gamma} (\lambda)$.

In the event that the space of solutions of (\ref{extension-hammy})
that lie right in $(c, b)$ and satisfy (\ref{boundary-b})
has dimension less than $n$, the matrix $\mathbf{R}^b (\lambda)$
(as constructed just above) will have fewer than $n$ columns, but we can add columns 
(which don't correspond with solutions of (\ref{extension-hammy})
that lie right in $(c, b)$ and satisfy (\ref{boundary-b}))
to create the basis for a Lagrangian subspace of $\mathbb{C}^{2n}$.
We can then proceed precisely as before, and we conclude
that the Lagrangian subspace with frame $J \gamma^*$ does
not intersect the Lagrangian subspace with frame 
$\mathbf{R}^b (\lambda)$, certainly including the elements that
correspond with solutions of (\ref{extension-hammy})
that lie right in $(c, b)$ and satisfy (\ref{boundary-b}).
\end{proof}

\begin{lemma} \label{lemma2-9prime}
Let Assumptions {\bf (A)} through {\bf (D)} hold,  
and suppose that for each  
$\lambda \in [\lambda_1, \lambda_2]$, we have the exclusion
$0 \notin \sigma_{\ess} (\mathcal{T} (\lambda))$. 
Then for each $\lambda \in [\lambda_1, \lambda_2]$, 
the space of solutions of (\ref{extension-hammy})
that lie right in $(c, b)$ and satisfy (\ref{boundary-b}) has 
dimension $n$. If we let $\{u^b_j (x; \lambda)\}_{j=1}^n$ denote
a basis for this space, then for each $x \in (c, b)$,
the vectors $\{u^b_j (x; \lambda)\}_{j=1}^n$ comprise a basis
for a Lagrangian subspace of $\mathbb{C}^{2n}$.
\end{lemma}

\begin{proof}
We fix any $\lambda \in [\lambda_1, \lambda_2]$, and observe 
from Lemma \ref{lemma2-8prime} that we can select
$\gamma \in \mathbb{C}^{n \times 2n}$ satisfying  
(\ref{boundary-matrix}) so that $0$ is not 
an eigenvalue of the self-adjoint operator 
$\mathcal{T}_{c, b}^{\gamma} (\lambda)$.
In addition, we know from Theorem 11.5 in \cite{Weidmann1987},
appropriately adapted to our setting, 
that $\sigma_{\ess} (\mathcal{T}_{c, b}^{\gamma} (\lambda)) 
\subset \sigma_{\ess} (\mathcal{T} (\lambda))$,
so we can conclude (using our assumption  
$0 \notin \sigma_{\ess} (\mathcal{T} (\lambda))$)
that, in fact, $0 \in \rho (\mathcal{T}_{c, b}^{\gamma} (\lambda))$.
This last inclusion allows us to apply Theorem 7.1 in 
\cite{Weidmann1987}, which asserts (among other things) that 
the space of solutions of (\ref{extension-hammy}) that lie 
right in $(c, b)$ and satisfy (\ref{boundary-b}) has
the same dimension for each  
$\mu \in \rho (\mathcal{T}_{c, b}^{\gamma} (\lambda))$. We know by 
construction that for $\mu_0$ and $\lambda = \lambda_0$ as in the 
specification of $\mathcal{D}$ this dimension is 
precisely $n$, and so we can conclude that it must be
$n$ for $\mu = 0$ as well (still with $\lambda = \lambda_0$). 
We can now conclude from 
Lemma \ref{lemma2-7prime} that this space must be 
a Lagrangian subspace of $\mathbb{C}^{2n}$ for 
each $x \in (c, b)$. This gives the claim for the 
specific choice $\lambda = \lambda_0$. 

For $\lambda \in [\lambda_1, \lambda_2] \backslash \lambda_0$,
we have from Assumption {\bf (C)} that there exist $n + r_b$
solutions $\{u_j^b (x; \mu_0, \lambda)\}_{j=1}^{n + r_b}$ to 
$Ju' - \mathbb{B} (x; \lambda) u = \mu_0 \mathbb{B}_{\lambda} (x; \lambda) u$ 
that lie right in $(c, b)$. Such functions are solutions to the 
eigenvalue problem $\mathcal{T}_{c, b, M} (\lambda) u = \mu_0 u$
(noting that the equation is regular at $x = c$, so 
$u \in \mathcal{D}_{c, b, M} = \dom (\mathcal{T}_{c, b, M} (\lambda))$).
According to Lemma \ref{representation-lemma},  
for each $j \in \{1, 2, \dots, n+r_b\}$, we can write 
\begin{equation} \label{u-representation}
    u_j^b (x; \mu_0; \lambda) = u_{p, j} (x) 
    + \sum_{k=1}^n c_k^j (\mu_0, \lambda, \lambda_0) u_k^b (x; \mu_0, \lambda_0)
    + \sum_{k=1}^{r_b} d_k^j (\mu_0, \lambda, \lambda_0) v_k^b (x; \mu_0, \lambda_0),
\end{equation}
for some constants $\{c_k^j (\mu_0, \lambda, \lambda_0)\}_{k = 1}^{n}$
and $\{d_k^j (\mu_0, \lambda, \lambda_0)\}_{k = 1}^{r_b}$, and 
where $u_{p, j} \in \mathcal{D}_{c, b}^{\gamma}$. (Here, we recall that under
Assumption {\bf (C)} $r_b$ does not depend on $\lambda$.) 
We would like to show that by taking appropriate 
linear combinations of the functions $\{u_j^b (x; \mu_0, \lambda)\}_{j=1}^{n + r_b}$,
we can construct $n$ linearly independent solutions of 
$Ju' - \mathbb{B} (x; \lambda) u = \mu_0 \mathbb{B}_{\lambda} (x; \lambda) u$ 
that lie right in $(c, b)$ and satisfy (\ref{boundary-b}). 

First, suppose we're in the limit point case so that $r_b = 0$. Then 
\begin{equation*}
    u_j^b (x; \mu_0; \lambda) = u_{p, j} (x) 
    + \sum_{k=1}^n c_k^j (\mu_0, \lambda, \lambda_0) u_k^b (x; \mu_0, \lambda_0),
\end{equation*}
and since the elements $u_{p, j}$ and $\{u_k^b (x; \mu_0, \lambda_0)\}_{k=1}^n$
all lie right in $(c, b)$ and satisfy (\ref{boundary-b}), we see immediately
that the elements $\{u_j^b (x; \mu_0; \lambda)\}_{j=1}^n$ must also have these
properties. 

For the general case with $r_b \in \{1, 2, \dots, n\}$, we will show that 
by taking appropriate linear combinations of the elements 
$\{u_j^b (x; \mu_0, \lambda)\}_{j=1}^{n + r_b}$ we can construct at 
least $n$ linearly independent solutions of 
$Ju' - \mathbb{B} (x; \lambda) u = \mu_0 \mathbb{B}_{\lambda} (x; \lambda) u$ 
that lie right in $(c, b)$ and satisfy (\ref{boundary-b}). Using 
(\ref{u-representation}), we see that linear combinations of the 
elements $\{u_j^b (x; \mu_0, \lambda)\}_{j=1}^{n + r_b}$ can be 
expressed as 
\begin{equation*}
    \begin{aligned}
    \sum_{j=1}^{n+r_b} e_j u_j^b (x; \mu_0, \lambda)
    &= \sum_{j=1}^{n+r_b} e_j u_{p, j} (x) 
    + \sum_{j=1}^{n+r_b} e_j \sum_{k=1}^n c_k^j (\mu_0, \lambda, \lambda_0) u_k^b (x; \mu_0, \lambda_0) \\
    &+ \sum_{j=1}^{n+r_b} e_j \sum_{k=1}^{r_b} d_k^j (\mu_0, \lambda, \lambda_0) v_k^b (x; \mu_0, \lambda_0),
    \end{aligned}
\end{equation*}
for some constants $\{e_j\}_{j=1}^{n+r_b}$. We want to find all collections
of such sets for which the sum involving $\{v_k^b (x; \mu_0, \lambda_0)\}_{k=1}^{r_b}$
is eliminated. Recalling again that 
the elements $u_{p, j}$ and $\{u_k^b (x; \mu_0, \lambda_0)\}_{k=1}^n$
all lie right in $(c, b)$ and satisfy (\ref{boundary-b}), our goal is 
to show that the coefficients $\{e_j\}_{j=1}^{n+r_b}$ can be chosen 
so that the Niessen elements $\{v_k^b (x; \mu_0, \lambda_0)\}_{k=1}^{r_b}$
are eliminated entirely from this sum for at least $n$ choices of 
the coefficients. In order to effect this elimination, we need to 
choose the coefficients $\{e_j\}_{j=1}^{n+r_b}$ so that for 
each $k \in \{1, 2, \dots, r_b\}$ we have 
\begin{equation*}
\sum_{j=1}^{n+r_b} e_j d_k^j (\mu_0, \lambda, \lambda_0) = 0.
\end{equation*}
I.e., we have $r_b$ equations for the $n+r_b$ unknowns 
$\{e_j\}_{j=1}^{n+r_b}$. For purposes of notation, it will be convenient to 
let $D$ denote the $r_b \times (n+r_b)$ matrix 
$D = (d_k^j)_{k, j = 1}^{r_b, n+r_b}$ and likewise to let 
$\mathbf{e}$ denote the column vector of length $n+r_b$ 
with entries $\{e_j\}_{j=1}^{n+r_b}$. Then we can express the 
system we need to solve as $D \mathbf{e} = 0$. Here, since $D$
only has $r_b$ rows, $\rank D \le r_b$, so that $\nullity D \ge n$,
from which we conclude that we can find at least $n$ suitable 
choices of the coefficients. 

At this point, we've shown that there exist at least 
$n$ linearly independent functions 
$\{u_j^b (x; \lambda)\}_{j=1}^{n}$ 
that solve 
\begin{equation*}
    Ju' - \mathbb{B} (x; \lambda) u = \mu_0 \mathbb{B} (x; \lambda) u
\end{equation*}
and additionally 
lie right in $(c, b)$ and satisfy (\ref{boundary-b}).
(The collection $\{u_j^b (x; \lambda)\}_{j=1}^{n}$ comprises 
linear combinations of the elements 
$\{u_j^b (x; \mu_0, \lambda)\}_{j=1}^{n + r_b}$
just above, and our convention of suppressing dependence on 
$\mu_0$ in the former is intended merely to draw a distinction
between the two collections without introducing additional
cumbersome notation.) As described at the outset of the proof, 
since $0 \in \rho(\mathcal{T}^{\gamma}_{c, b} (\lambda))$,
we can conclude that there must be at least $n$ linearly 
independent solutions $\{u_j^b (x; \lambda)\}_{j=1}^n$ 
of (\ref{extension-hammy}) that 
lie right in $(c, b)$ and satisfy (\ref{boundary-b}).
Last, according to Lemma \ref{lemma2-7prime}, for 
each $x \in (c, b)$ these elements comprise a basis for 
a Lagrangian subspace of $\mathbb{C}^{2n}$.
\end{proof}

We next develop a Green's function for the inhomogeneous 
system
\begin{equation} \label{inhomogeneous-problem}
    Jy' - \mathbb{B} (x; \lambda) y = \mathbb{B}_{\lambda} (x; \lambda) f,
\end{equation}
The construction follows a standard argument, and is 
included in the appendix. 

\begin{lemma} \label{green-function-lemma} 
Let Assumptions {\bf (A)} through {\bf (D)} hold, and  
fix any $\lambda \in [\lambda_1, \lambda_2]$ for which 
$0 \notin \sigma_{\ess} (\mathcal{T} (\lambda))$. Using 
Lemma \ref{lemma2-8prime}, take $\gamma$ satisfying (\ref{boundary-matrix}) 
so that $0 \notin \sigma_p (\mathcal{T}^{\gamma}_{c, b} (\lambda))$
(and so consequently $0 \in \rho (\mathcal{T}^{\gamma}_{c, b} (\lambda))$). 
Then for any $f \in L^2_{\mathbb{B}_{\lambda}} ((c, b), \mathbb{C}^{2n})$
the inhomogeneous problem (\ref{inhomogeneous-problem})
can be solved for $y \in \mathcal{D}^{\gamma}_{c, b}$ via the 
Green's function formulation
\begin{equation*}
    y (x; \lambda) 
    = \int_c^b G^{\gamma}_{c, b} (x, \xi; \lambda) \mathbb{B}_{\lambda} (\xi; \lambda) f(\xi) d\xi,
\end{equation*}
where 
\begin{equation*}
    G^{\gamma}_{c, b} (x, \xi; \lambda)
    = \begin{cases}
    - \Phi (x; \lambda) 
    \begin{pmatrix}
    0 & \mathbf{R}^b (\lambda)
    \end{pmatrix} 
    \mathbb{M} (\lambda)
    \begin{pmatrix}
    J \gamma^* & 0
    \end{pmatrix}^*
    \Phi (\xi; \lambda)^* & c < \xi < x < b \\
    \Phi (x; \lambda) 
    \begin{pmatrix}
    J \gamma^* & 0
    \end{pmatrix} 
    \mathbb{M} (\lambda)
    \begin{pmatrix}
    0 & \mathbf{R}^b (\lambda)
    \end{pmatrix}^*
    \Phi (\xi; \lambda)^* & c < x < \xi < b,
    \end{cases}
\end{equation*}
and $\Phi (x; \lambda)$ denotes a 
fundamental matrix for (\ref{extension-hammy}) satisfying
\begin{equation*}
    J \Phi' = \mathbb{B} (x; \lambda) \Phi, \quad \Phi (c; \lambda) = I_{2n}.
\end{equation*}
Here, $\mathbf{R}^b (\lambda)$ is the frame for a Lagrangian subspace of 
$\mathbb{C}^{2n}$, and 
\begin{equation*}
    \mathbb{M} (\lambda) = \mathbb{E} (\lambda)^{-1} J (\mathbb{E}(\lambda)^*)^{-1},
    \quad \mathbb{E} (\lambda) = (J \gamma^* \,\, \mathbf{R}^b (\lambda)).
\end{equation*}
\end{lemma}

We next use our Green's function formulation from Lemma 
\ref{green-function-lemma} to obtain useful pointwise estimates on 
elements $(\mathcal{T}_{c, b}^{\gamma} (\lambda_*)^{-1} \mathcal{E} (\cdot; \lambda, \lambda_*) f) (x)$.
For this discussion, it will be convenient to set 
\begin{equation} \label{mathcalR}
    \mathcal{R}_{\mathcal{E}} (\lambda)
    := \mathcal{T}^{\gamma}_{c, b} (\lambda_*)^{-1}
    \mathcal{E} (\cdot; \lambda, \lambda_*),
\end{equation}
and we note the inequalities 
\begin{equation} \label{mathcalR-inequalities}
\| \mathcal{R}_{\mathcal{E}} (\lambda)^k\|
\le \| \mathcal{T}^{\gamma}_{c, b} (\lambda_*)^{-1} \|^k
\| \mathcal{E} (\cdot; \lambda, \lambda_*) \|^k,
\end{equation}
where in all cases $\| \cdot \|$ denotes operator norm, with 
the operator viewed as a map from $L^2_{\mathbb{B}_{\lambda}} ((c, b), \mathbb{C}^{2n})$
to itself.

\begin{lemma} \label{useful-estimates-lemma}
Let Assumptions {\bf (A)} through {\bf (D)} hold, and 
fix any $\lambda_* \in [\lambda_1, \lambda_2]$ for which 
$0 \notin \sigma_{\ess} (\mathcal{T} (\lambda_*))$. Using 
Lemma \ref{lemma2-8prime}, take $\gamma$, depending on 
$\lambda_*$, satisfying (\ref{boundary-matrix})
so that $0 \notin \sigma_p (\mathcal{T}^{\gamma}_{c, b} (\lambda_*))$
(and so consequently $0 \in \rho (\mathcal{T}^{\gamma}_{c, b} (\lambda_*))$).
Finally, for some interval $\mathcal{I} \subset [\lambda_1, \lambda_2]$,
let $\mathcal{E} (x; \lambda, \lambda_*)$ denote any 
measurable map from $[c, b) \times \mathcal{I}$
to $\mathbb{C}^{2n \times 2n}$ (in particular, not necessarily the map specified in 
Assumption {\bf (E)}) so that for all $\lambda \in \mathcal{I}$ 
$\mathcal{E} (\cdot; \lambda, \lambda_*)$ is bounded as a map 
from $L^2_{\mathbb{B}_{\lambda}} ((c, b), \mathbb{C}^{2n})$ to 
itself. 
Then for any fixed $x \in [c, b)$ there exists a value $C (x; \lambda_*)$
so that for any $f \in L^2_{\mathbb{B}_{\lambda}} ((c, b), \mathbb{C}^{2n})$
we have 
\begin{equation*}
  |(\mathcal{R}_{\mathcal{E}} (\lambda) f) (x)|
  \le C (x; \lambda_*) \|\mathcal{E} (\cdot; \lambda, \lambda_*)\| \|f\|_{\mathbb{B}_{\lambda}},
\end{equation*}
for all $\lambda \in \mathcal{I}$. Moreover, for any $b' \in (c, b)$ there exists
a value $C_{b'} (\lambda_*)$ so that $C(x; \lambda_*) \le C_{b'} (\lambda_*)$
for all $x \in [c, b']$. 
\end{lemma}

\begin{proof} By construction of $G^{\gamma}_{c, b} (x, \xi; \lambda_*)$,
along with the definition of $\mathcal{R}_{\mathcal{E}} (\lambda)$,
for any $f \in L^2_{\mathbb{B}_{\lambda}} ((c, b), \mathbb{C}^{2n})$, 
\begin{equation*}
(\mathcal{R}_{\mathcal{E}} (\lambda) f) (x)
= \int_c^b G^{\gamma}_{c, b} (x, \xi; \lambda_*) \mathbb{B}_{\lambda} (\xi; \lambda_*) 
\mathcal{E} (\xi; \lambda, \lambda_*) f (\xi) d\xi.
\end{equation*}
Using Lemma \ref{green-function-lemma}, we obtain the inequality 
\begin{equation} \label{IandII}
\begin{aligned}
&|(\mathcal{R}_{\mathcal{E}} (\lambda) f) (x)| \\
&\le
\Big| \Phi (x; \lambda_*) 
    \begin{pmatrix}
    0 & \mathbf{R}^b (\lambda_*)
    \end{pmatrix} 
    \mathbb{M} (\lambda_*)
    \int_c^x
    \begin{pmatrix}
    J \gamma^* & 0
    \end{pmatrix}^*
    \Phi (\xi; \lambda_*)^* \mathbb{B}_{\lambda} (\xi; \lambda_*) 
\mathcal{E} (\xi; \lambda, \lambda_*) f (\xi) d\xi \Big| \\
&\quad + \Big|  \Phi (x; \lambda_*) 
    \begin{pmatrix}
    J \gamma^* & 0
    \end{pmatrix} 
    \mathbb{M} (\lambda_*)
\int_{x}^{b}  
    \begin{pmatrix}
    0 & \mathbf{R}^b (\lambda_*)
    \end{pmatrix}^*
    \Phi (\xi; \lambda_*)^* \mathbb{B}_{\lambda} (\xi; \lambda_*) 
\mathcal{E} (\xi; \lambda, \lambda_*) f (\xi) d\xi \Big| \\
&=: I_1 + I_2.
\end{aligned}
\end{equation}

Beginning with $I_1$, since $\Phi (x; \lambda_*)$ is absolutely continuous
on $[c, b']$ for any $c < b' < b$, the value 
\begin{equation*}
    C_1 (x; \lambda_*) 
    := |\Phi (x; \lambda_*) 
    \begin{pmatrix}
    0 & \mathbf{R}^b (\lambda_*)
    \end{pmatrix} 
    \mathbb{M} (\lambda_*)| 
\end{equation*}
is fixed and finite. The integral in $I_1$ can be associated with a 
collection of $L^2_{\mathbb{B}_{\lambda}} ((c, b), \mathbb{C}^{2n})$ 
inner products (i.e., one for each row of the matrix 
$(J \gamma^* \,\,\, 0)^* \Phi (\xi; \lambda_*)^*$), allowing 
us to write 
\begin{equation*}
    I_1 \le C_1 (x; \lambda_*) C_2 (x; \lambda_*) 
    \| \mathcal{E} (\cdot; \lambda, \lambda_*) f\|_{\mathbb{B}_{\lambda}} 
    \le C_1 (x; \lambda_*) C_2 (x; \lambda_*) 
    \| \mathcal{E} (\cdot; \lambda, \lambda_*)\| \|f\|_{\mathbb{B}_{\lambda}} 
\end{equation*}
where $C_2 (x; \lambda_*)$ is the Euclidean length of the vector in 
$\mathbb{R}^n$ whose $i^{\rm th}$ component is the 
$L^2_{\mathbb{B}_{\lambda}} ((c, b), \mathbb{C}^{2n})$ norm 
of the $i^{\rm th}$ column of $\Phi (x; \lambda_*) (J \gamma^* \,\,\, 0)$. 
We emphasize that $\Phi (\cdot; \lambda_*) (J \gamma^* \,\,\, 0)$ is not 
generally in $L^2_{\mathbb{B}_{\lambda}} ((c, b), \mathbb{C}^{2n \times 2n})$,
but since $x \in [c, b)$ is fixed, $C_2 (x; \lambda_*)$ is finite by 
the local absolute continuity of $\Phi (x; \lambda_*)$. Also, we see 
that $C_2 (x; \lambda_*)$ is uniformly bounded for all 
$x \in [c, b']$.  

Likewise, for $I_2$ in (\ref{IandII}), we can write  
\begin{equation*}
    I_2 \le \tilde{C}_1 (x; \lambda_*) \tilde{C}_2 (x; \lambda_*) 
    \| \mathcal{E} (\cdot; \lambda, \lambda_*)\| \|f\|_{\mathbb{B}_{\lambda}}, 
\end{equation*}
where in this case 
\begin{equation*}
    \tilde{C}_1 (x; \lambda_*)
    = |\Phi (x; \lambda_*)  \begin{pmatrix} J \gamma^* & 0 \end{pmatrix} \mathbb{M} (\lambda_*)|,
\end{equation*}
and $\tilde{C}_2 (x; \lambda_*)$ is the Euclidean length of the vector in 
$\mathbb{R}^n$ whose $i^{\rm th}$ component is the 
$L^2_{\mathbb{B}_{\lambda}} ((c, b), \mathbb{C}^{2n})$ norm 
of the $i^{\rm th}$ column of $\Phi (x; \lambda_*) (0 \,\,\, \mathbf{R}^b (\lambda_*))$. 
In this case, it's important that the columns of 
$\Phi (x; \lambda_*) (0 \,\,\, \mathbf{R}^b (\lambda_*))$ lie in 
 $L^2_{\mathbb{B}_{\lambda}} ((c, b), \mathbb{C}^{2n \times 2n})$.
Combining 
these estimates on $I_1$ and $I_2$, we arrive that the claimed estimate 
\begin{equation*}
    |(\mathcal{R}_{\mathcal{E}} (\lambda) f) (x)|
    \le C(x; \lambda_*) \| \mathcal{E} (\cdot; \lambda, \lambda_*)\| \|f\|_{\mathbb{B}_{\lambda}},
\end{equation*}
where 
\begin{equation*}
C(x; \lambda_*) = C_1 (x; \lambda_*) C_2 (x; \lambda_*) + \tilde{C}_1 (x; \lambda_*) \tilde{C}_2 (x; \lambda_*),  
\end{equation*}
with $C (x; \lambda_*)$ uniformly bounded for all $x \in [c, b']$.
\end{proof}

For the final two lemmas of this section we will add Assumption {\bf (E)}
to our list of hypotheses. 

\begin{lemma} \label{continuation-lemma}
Let Assumptions {\bf (A)} through {\bf (E)} hold, and  
suppose that for some fixed $\lambda_* \in [\lambda_1, \lambda_2]$
there is an open interval $I_*$ containing $\lambda_*$ so that 
for each $\lambda \in I_* \cap [\lambda_1, \lambda_2]$,
we have $0 \notin \sigma_{\ess} (\mathcal{T} (\lambda))$.
Let $\{u^b_j (x; \lambda_*)\}_{j=1}^n$ denote a basis for the 
$n$-dimensional space of solutions of (\ref{extension-hammy})
(with $\lambda = \lambda_*$)
that lie right in $(c, b)$ and satisfy (\ref{boundary-b})
(guaranteed to exist by Lemma \ref{lemma2-9prime}). Then 
there exists a constant $r > 0$, depending on both $\lambda_*$ 
and $\mathcal{T}_{c, b}^{\gamma} (\lambda_*)$ 
(including the choice of $\gamma$) so that the elements 
$\{u^b_j (x; \lambda_*)\}_{j=1}^n$ can be
extended in $\lambda$ to the interval $I_{\lambda_*, r}$
(as specified in Assumption {\bf (E)}).
The extensions $\{u^b_j (x; \lambda)\}_{j=1}^n$
comprise a basis for the space of solutions 
of (\ref{linear-hammy}) that lie right in $(a, b)$ and 
satisfy (\ref{boundary-b}), and moreover they 
are continuously differentiable on $I_{\lambda_*, r}$, and for every 
$\lambda \in I_{\lambda_*, r}$ satisfy the 
relations
\begin{equation} \label{diff1}
    J (\partial_{\lambda} u_j^b)' (x; \lambda)
    = \mathbb{B}_{\lambda} (x; \lambda) u_j^b (x; \lambda) 
    + \mathbb{B} (x; \lambda) \partial_{\lambda} u_j^b (x; \lambda),
\end{equation}
for a.e. $x \in (a, b)$, and 
\begin{equation} \label{diff2}
    \lim_{x \to b^-} u_j^b (x; \lambda_*)^* J \partial_{\lambda} u_k^b (x; \lambda_*)
    = 0,
    \quad \forall \,\, j, k \in \{1, 2, \dots, n\}.
\end{equation}
\end{lemma}

\begin{proof}
We begin by observing that the starting element $u_j^b (x; \lambda_*)$
solves 
\begin{equation*}
    J (u_j^{b})' = \mathbb{B} (x; \lambda_*) u_j^b.
\end{equation*}
Our goal is to construct $u_j^b (x; \lambda)$ so that 
\begin{equation} \label{lambda-equation}
    J (u_j^{b})' = \mathbb{B} (x; \lambda) u_j^b,
\end{equation}
for $\lambda$ near $\lambda_*$. For this, we express the 
latter equation as 
\begin{equation*}
     J (u_j^{b})' - \mathbb{B} (x; \lambda_*) u_j^b
     = (\mathbb{B} (x; \lambda) - \mathbb{B} (x; \lambda_*)) u_j^b.
\end{equation*}
Our strategy will be to look for solutions of this equation of 
the form 
\begin{equation} \label{ujb-equation}
    u_j^b (x; \lambda)  = u_j^b (x; \lambda_*) 
    + F_j^b (x; \lambda, \lambda_*),
\end{equation}
where $F_j^b (\cdot; \lambda, \lambda_*) \in \mathcal{D}^{\gamma}_{c, b}$
satisfies 
\begin{equation*}
     J (F_j^{b})' - \mathbb{B} (x; \lambda_*) F_j^b
     = (\mathbb{B} (x; \lambda) - \mathbb{B} (x; \lambda_*)) u_j^b (x; \lambda),
\end{equation*}
or equivalently 
\begin{equation} \label{resolvent-star}
    \mathcal{T}_{c, d}^{\gamma} (\lambda_*) F_j^b (\cdot; \lambda, \lambda_*)
    = \mathcal{E} (x; \lambda, \lambda_*) u_j^b (\cdot; \lambda),
\end{equation}
where $\mathcal{E} (x; \lambda, \lambda_*)$ is described in Assumption {\bf (E)}. 

If a solution of (\ref{resolvent-star}) exists with $F_j^b (x; \lambda, \lambda_*)$ 
contained in 
$\mathcal{D}_{c, b}^{\gamma}$, then we have the relation 
\begin{equation*}
    F_j^b (\cdot; \lambda, \lambda_*)
    = \mathcal{R}_{\mathcal{E}} (\lambda) u_j^b (\cdot; \lambda),
\end{equation*}
where since $0 \in \rho (\mathcal{T}_{c, b}^{\gamma} (\lambda_*))$,
we have that $\mathcal{T}_{c, b}^{\gamma} (\lambda_*)^{-1}$ is a bounded
linear operator mapping $L^2_{\mathbb{B}_{\lambda}} ((c, b), \mathbb{C}^{2n})$
into $\mathcal{D}_{c, b}^{\gamma}$, so that in particular 
$ F_j^b (\cdot; \lambda, \lambda_*)$ satisfies
(\ref{boundary-b}). In this way, we arrive at the integral 
equation 
\begin{equation*}
    u_j^b (\cdot; \lambda)
    = u_j^b (\cdot; \lambda_*) 
    +\mathcal{R}_{\mathcal{E}} (\lambda) u_j^b (\cdot; \lambda).
\end{equation*}
Rearranging, we can express this equation as 
\begin{equation} \label{neumann-type}
   (I - \mathcal{R}_{\mathcal{E}} (\lambda)) u_j^b (\cdot; \lambda)
   = u_j^b (\cdot; \lambda_*).
\end{equation}
According to Assumption {\bf (E)} we can choose $r > 0$ sufficiently small so that 
\begin{equation} \label{the-inequality}
    \| \mathcal{T}^{\gamma}_{c, b} (\lambda_*)^{-1} \|
    \|\mathcal{E} (\cdot; \lambda, \lambda_*) \| < 1,
\end{equation}
for all $\lambda \in I_{\lambda_*, r}$.
Accordingly (using (\ref{mathcalR-inequalities})), $\|\mathcal{R}_{\mathcal{E}} (\lambda)\| < 1$ 
for all $\lambda \in I_{\lambda_*, r}$,
and by the standard theory of Neumann series (for example, the 
discussion of Example 4.9 on p. 32 of \cite{Kato}), 
we can solve (\ref{neumann-type}) with 
\begin{equation} \label{neumann-series}
    u_j^b (\cdot; \lambda)
   = (I - \mathcal{R}_{\mathcal{E}} (\lambda))^{-1} u_j^b (\cdot; \lambda_*)
   \in L^2_{\mathbb{B}_{\lambda}} ((a, b), \mathbb{C}^{2n}),
\end{equation}
for all $\lambda \in I_{\lambda_*, r}$.

We've already noted that $F_j^b (x; \lambda_*, \lambda)$ is 
contained in $\mathcal{D}_{c, b}^{\gamma}$, and so in particular
lies right in $(c,b)$ and satisfies (\ref{boundary-b}). In addition,
$u_j^b (\cdot; \lambda_*)$ lies right in $(c,b)$ and satisfies 
(\ref{boundary-b}), so we can conclude that $u_j^b (x; \lambda)$
is a solution of (\ref{lambda-equation}) that 
lies right in $(c,b)$ and satisfies (\ref{boundary-b}). 
Proceeding similarly for
each $j \in \{1, 2, \dots, n\}$, we obtain a collection 
of extensions $\{u^b_j (x; \lambda)\}_{j=1}^n$.

In addition, by virtue of (\ref{neumann-type}) and (\ref{neumann-series}),
we see that $\{u^b_j (x; \lambda)\}_{j=1}^n$ inherits 
linear independence from the set $\{u^b_j (x; \lambda_*)\}_{j=1}^n$.
We conclude from Lemma \ref{lemma2-7prime} that 
the set $\{u^b_j (x; \lambda)\}_{j=1}^n$ comprises a 
basis for the space of solutions 
of (\ref{extension-hammy}) that lie right in $(c, b)$ and 
satisfy (\ref{boundary-b}), and additionally
that for each $x \in (c, b)$ the vectors
$\{u^b_j (x; \lambda)\}_{j=1}^n$ comprise the basis
of a Lagrangian subspace of $\mathbb{C}^{2n}$. 

Turning now to the dependence of our extensions 
$\{u_j^b (\cdot; \lambda)\}_{j=1}^n$ on $\lambda$, 
we observe that by standard Neumann expansion, we can 
express (\ref{neumann-series}) as 
\begin{equation} \label{the-real-neumann-series}
u_j^b (\cdot; \lambda) = \sum_{k=0}^{\infty} 
\mathcal{R}_{\mathcal{E}} (\lambda)^k u_j^b (\cdot; \lambda_*),
\end{equation}
and conclude that for each $\lambda \in I_{\lambda_*, r}$ this 
series converges absolutely in 
$L^2_{\mathbb{B}_{\lambda}} ((c, b), \mathbb{C}^{2n})$. In
order to conclude something about the dependence of 
$u_j^b (\cdot; \lambda)$ on $\lambda$, we will 
proceed by showing that the series of term-by-term $\lambda$-derivatives
of (\ref{the-real-neumann-series}) converges absolutely for all 
$\lambda \in I_{\lambda_*, r}$. As a starting point, we need to verify 
that the individual summands are in fact differentiable in $\lambda$,
and we will do this iteratively, verifying that if some 
\begin{equation*}
    f (\cdot; \lambda) 
    \in \AC_{\loc} ([c, b), \mathbb{C}^{2n}) \cap L^2_{\mathbb{B}_{\lambda}} ((c, b), \mathbb{C}^{2n}), 
\end{equation*}
is differentiable as a map $\lambda \mapsto L^2_{\mathbb{B}_{\lambda}} ((c, b), \mathbb{C}^{2n})$,
then so is 
\begin{equation} \label{cap-F-defined}
    F(\cdot; \lambda) = \mathcal{R}_{\mathcal{E}} (\lambda) f(\cdot; \lambda).
\end{equation}
To this end, we write
\begin{equation*}
    \begin{aligned}
    \frac{1}{h} (F (\cdot; \lambda+h) &- F(\cdot; \lambda))
    = \frac{1}{h} (\mathcal{R}_{\mathcal{E}} (\lambda+h) f(\cdot; \lambda+h) 
    - \mathcal{R}_{\mathcal{E}} (\lambda) f(\cdot; \lambda)) \\
    &= \frac{1}{h} \Big(\mathcal{R}_{\mathcal{E}} (\lambda+h) - \mathcal{R}_{\mathcal{E}} (\lambda) \Big) f(\cdot; \lambda+h) 
    - \frac{1}{h} \mathcal{R}_{\mathcal{E}} (\lambda) \Big( f(\cdot; \lambda) - f(\cdot; \lambda+h) \Big),
    \end{aligned}
\end{equation*}
from which we see that 
$F (\cdot; \lambda)$ inherits the differentiability in 
$\lambda$ of $f(\cdot; \lambda)$ (by virtue of the assumed differentiability
of $\mathcal{E} (\cdot; \lambda, \lambda_*)$ in $\lambda$). In addition,
we can use the inequality 
\begin{equation*}
\begin{aligned}
    \| \frac{1}{h} (F (\cdot; \lambda+h) &- F(\cdot; \lambda))  \|_{\mathbb{B}_{\lambda}} 
    \le \| \mathcal{T}_{c, b}^{\gamma} (\lambda_*)^{-1} \|
    \| \frac{1}{h} (\mathcal{E} (\cdot; \lambda+h, \lambda_*) - \mathcal{E} (\cdot; \lambda, \lambda_*)) \| 
    \| f(\cdot; \lambda+h)  \|_{\mathbb{B}_{\lambda}} \\
    &+ \| \mathcal{T}_{c, b}^{\gamma} (\lambda_*)^{-1} \|
    \| \mathcal{E} (\cdot; \lambda, \lambda_*) \| 
    \| \frac{1}{h} (f(\cdot; \lambda+h) - f(\cdot; \lambda) )  \|_{\mathbb{B}_{\lambda}}
\end{aligned}
\end{equation*}
to obtain the relation
\begin{equation} \label{iteration-relation-derivative}
    \| F_{\lambda} (\cdot; \lambda) \|_{\mathbb{B}_{\lambda}}
    \le \| \mathcal{T}_{c, b}^{\gamma} (\lambda_*)^{-1} \|
    \| \mathcal{E}_{\lambda} (\cdot; \lambda, \lambda_*) \| 
    \| f(\cdot; \lambda)  \|_{\mathbb{B}_{\lambda}}
    +
    \| \mathcal{T}_{c, b}^{\gamma} (\lambda_*)^{-1} \|
    \| \mathcal{E} (\cdot; \lambda, \lambda_*) \|
    \| f_{\lambda} (\cdot; \lambda)  \|_{\mathbb{B}_{\lambda}}.
\end{equation}

To see that 
\begin{equation} \label{the-real-neumann-series-derivatives}
\sum_{k=0}^{\infty} 
\frac{d}{d \lambda} \mathcal{R}_{\mathcal{E}} (\lambda)^k u_j^b (\cdot; \lambda_*),
\end{equation}
converges absolutely, we begin by 
using (\ref{iteration-relation-derivative})
to observe that (for $k=1$)
\begin{equation} \label{k1estimate1}
    \| \frac{d}{d \lambda} \mathcal{R}_{\mathcal{E}} (\lambda) u_j^b (\cdot; \lambda_*) \|_{\mathbb{B}_{\lambda}}
    \le \| \mathcal{T}_{c, b}^{\gamma} (\lambda_*)^{-1} \|
    \| \mathcal{E}_{\lambda} (\cdot; \lambda, \lambda_*) \| 
    \| u^b_j (\cdot; \lambda_*)  \|_{\mathbb{B}_{\lambda}},
\end{equation}
where the contribution associated with $f_{\lambda} (\cdot; \lambda)$ 
doesn't appear because $u^b_j (\cdot; \lambda_*)$ doesn't depend on 
$\lambda$. Next, using (\ref{iteration-relation-derivative}) again, 
this time with $f (\cdot; \lambda) = \mathcal{R}_{\mathcal{E}} (\lambda) u_j^b (\cdot; \lambda_*)$
we combine (\ref{mathcalR-inequalities}) (with $k=1$) and (\ref{k1estimate1})
to obtain the estimate 
\begin{equation*}
    \| \frac{d}{d \lambda} \mathcal{R}_{\mathcal{E}} (\lambda)^2 u_j^b (\cdot; \lambda_*) \|_{\mathbb{B}_{\lambda}}
    \le 2 \| \mathcal{T}_{c, b}^{\gamma} (\lambda_*)^{-1} \|^2
    \| \mathcal{E} (\cdot; \lambda, \lambda_*) \| 
    \| \mathcal{E}_{\lambda} (\cdot; \lambda, \lambda_*) \| 
    \| u^b_j (\cdot; \lambda_*)  \|_{\mathbb{B}_{\lambda}}.
\end{equation*}
Continuing this way, we obtain the inequality 
\begin{equation*}
    \| \frac{d}{d \lambda} \mathcal{R}_{\mathcal{E}} (\lambda)^k u_j^b (\cdot; \lambda_*) \|_{\mathbb{B}_{\lambda}}
    \le k \| \mathcal{T}_{c, b}^{\gamma} (\lambda_*)^{-1} \|^k
    \| \mathcal{E} (\cdot; \lambda, \lambda_*) \|^{k-1} 
    \| \mathcal{E}_{\lambda} (\cdot; \lambda, \lambda_*) \| 
    \| u^b_j (\cdot; \lambda_*)  \|_{\mathbb{B}_{\lambda}},
\end{equation*}
from which the absolute convergence of (\ref{the-real-neumann-series-derivatives}),
uniform for $\lambda$ in compact subsets of $I_{\lambda_*, r}$,
is immediate from the ratio test. We can conclude that 
$\lambda \mapsto u_j^b (\cdot; \lambda)$ 
is differentiable as a map taking $\lambda \in I_{\lambda_*, r}$
to $L^2_{\mathbb{B}_{\lambda}} ((c, b), \mathbb{C}^{2n})$.

Next, we check that for each fixed $x \in [c, b)$, the map 
$\lambda \mapsto u_j^b (x; \lambda)$ taking $I_{\lambda_*, r}$
to $\mathbb{C}^{2n}$ is differentiable in 
$\lambda$. For this, we 
need to understand the convergence of the series
\begin{equation} \label{pointwise}
    \sum_{k=0}^{\infty} 
\Big(\mathcal{R}_{\mathcal{E}} (\lambda)^k u_j^b (\cdot; \lambda_*)\Big) (x),
\end{equation}
in $\mathbb{C}^{2n}$, and also the associated series of derivatives 
\begin{equation} \label{pointwise-derivatives}
    \sum_{k=0}^{\infty} 
\frac{\partial}{\partial \lambda} 
\Big(\mathcal{R}_{\mathcal{E}} (\lambda)^k u_j^b (\cdot; \lambda_*)\Big) (x).
\end{equation}
First, we check that (\ref{pointwise}) is absolutely 
convergent for all $\lambda \in I_{\lambda_*, r}$. To this end, we 
write (for $k \ge 1$)
\begin{equation} \label{iteration-inequality-pointwise}
\begin{aligned}
\Big|\Big(\mathcal{R}_{\mathcal{E}} (\lambda)^k u_j^b (\cdot; \lambda_*)\Big) (x)\Big|
&= \Big| \Big(\mathcal{R}_{\mathcal{E}} (\lambda) \mathcal{R}_{\mathcal{E}} (\lambda)^{(k-1)} u_j^b (\cdot; \lambda_*)\Big) (x)\Big| \\
&\le C(x; \lambda_*) \|\mathcal{E} (\cdot; \lambda, \lambda_*)\| 
\| \mathcal{R}_{\mathcal{E}} (\lambda)^{(k-1)} u_j^b (\cdot; \lambda_*) \|_{\mathbb{B}_{\lambda}} \\
&\le C(x; \lambda_*) \|\mathcal{E} (\cdot; \lambda, \lambda_*)\| 
\| \mathcal{R}_{\mathcal{E}} (\lambda)\|^{(k-1)} \|u_j^b (\cdot; \lambda_*) \|_{\mathbb{B}_{\lambda}},
\end{aligned}
\end{equation}
where in obtaining the first inequality we have used Lemma \ref{useful-estimates-lemma}. Recalling 
that $\|\mathcal{R}_{\mathcal{E}} (\lambda)\| < 1$ for $\lambda \in I_{\lambda_*, r}$,
we see that 
\begin{equation} \label{undifferentiated-estimates}
\begin{aligned}
     \sum_{k=0}^{\infty} &\Big|\Big(\mathcal{R}_{\mathcal{E}} (\lambda)^k u_j^b (\cdot; \lambda_*)\Big) (x)\Big|
     \le |u_j^b (x; \lambda_*)| 
     + \sum_{k=1}^{\infty} \Big|\Big(\mathcal{R}_{\mathcal{E}} (\lambda)^k u_j^b (\cdot; \lambda_*)\Big) (x)\Big| \\ 
     &\le |u_j^b (x; \lambda_*)| 
     + C(x; \lambda_*) \|\mathcal{E} (\cdot; \lambda, \lambda_*)\| \|u_j^b (\cdot; \lambda_*) \|_{\mathbb{B}_{\lambda}} 
     \sum_{k=1}^{\infty} \| \mathcal{R}_{\mathcal{E}} (\lambda)\|^{(k-1)} \\
     & = |u_j^b (x; \lambda_*)| + C(x; \lambda_*) \|\mathcal{E} (\cdot; \lambda, \lambda_*)\| \|u_j^b (\cdot; \lambda_*) \|_{\mathbb{B}_{\lambda}} 
     \frac{1}{1-\| \mathcal{R}_{\mathcal{E}} (\lambda)\|},
\end{aligned}     
\end{equation}
verifying that (\ref{pointwise}) is absolutely convergent for all $\lambda \in I_{\lambda_*, r}$. 

Turning to (\ref{pointwise-derivatives}), we first need to verify that each summand in 
(\ref{pointwise}) is indeed differentiable in $\lambda$. For this, it suffices to show that 
if 
\begin{equation} \label{inclusions}
    f (\cdot; \lambda), f_\lambda (\cdot; \lambda) 
    \in \AC_{\loc} ([c, b), \mathbb{C}^{2n}) \cap L^2_{\mathbb{B}_{\lambda}} ((c, b), \mathbb{C}^{2n}), 
\end{equation}
then $F (\cdot; \lambda)$ from (\ref{cap-F-defined}) and $F_{\lambda} (\cdot; \lambda)$ 
are also contained in this intersection. The required differentiability then follows upon 
iteration of 
\begin{equation*}
    (\mathcal{R}_{\mathcal{E}} (\lambda)^k u_j^b (\cdot; \lambda_*)) (x)
    = \int_c^b G_{c, b}^{\gamma} (x; \lambda, \lambda_*) \mathbb{B}_{\lambda} (\xi; \lambda_*)
    \mathcal{E} (\xi; \lambda, \lambda_*) (\mathcal{R}_{\mathcal{E}} (\lambda)^{k-1} u_j^b (\cdot; \lambda_*)) (\xi) d\xi,
\end{equation*}
starting with $f(x; \lambda) = u^b_j (x; \lambda_*)$. To see that 
$F (\cdot; \lambda) \in \AC_{\loc} ([c, b), \mathbb{C}^{2n}) \cap L^2_{\mathbb{B}_{\lambda}} ((c, b), \mathbb{C}^{2n})$,
we recall that since $0 \in \rho (\mathcal{T}^{\gamma}_{c, b} (\lambda_*))$,
$\mathcal{R}_{\mathcal{E}} (\lambda)$ maps $L^2_{\mathbb{B}_{\lambda}} ((c, b), \mathbb{C}^{2n})$
into $\mathcal{D}^{\gamma}_{c, b}$, and moreover that all elements of $\mathcal{D}^{\gamma}_{c, b}$
lie in this intersection. Turning to $F_{\lambda} (x; \lambda)$, the 
inclusions (\ref{inclusions}) allow us to use the Lebesgue Dominated
Convergence Theorem to differentiate through the integral to write 
\begin{equation} \label{pointwise-integral}
    F_{\lambda} (x; \lambda) 
    = \int_c^b G^{\gamma}_{c, b} (x, \xi; \lambda_*) \mathbb{B}_{\lambda} (\xi; \lambda_*)
    \Big(\mathcal{E}_{\lambda} (\xi; \lambda, \lambda_*) f(\xi; \lambda) 
    + \mathcal{E} (\xi; \lambda, \lambda_*) f_{\lambda} (\xi; \lambda) \Big) d\xi.
\end{equation} 
From Lemma \ref{green-function-lemma}, we see that we only need to verify
the inclusions 
\begin{equation*}
\mathcal{E}_{\lambda} (\cdot; \lambda, \lambda_*) f(\cdot; \lambda),
\mathcal{E} (\cdot; \lambda, \lambda_*) f_{\lambda} (\cdot; \lambda)
\in L^2_{\mathbb{B}_{\lambda}} ((c, b), \mathbb{C}^{2n}),
\end{equation*}
each of which is immediate from our Assumption {\bf (E)}, which
asserts that $\mathcal{E} (\cdot; \lambda, \lambda_*)$ 
and  $\mathcal{E}_{\lambda} (\cdot; \lambda, \lambda_*)$ are both 
bounded linear operators mapping $L^2_{\mathbb{B}_{\lambda}} ((c, b), \mathbb{C}^{2n})$ 
to itself. We conclude that each summand in (\ref{pointwise}) is differentiable in 
$\lambda$ for all $\lambda \in I_{\lambda_*, r}$. 

In order to complete this part of the proof, 
we need to verify that (\ref{pointwise-derivatives}) 
converges absolutely for all $\lambda \in I_{\lambda_*, r}$. 
To this end, 
we observe from (\ref{pointwise-integral}) and Lemma \ref{useful-estimates-lemma}
the estimate 
\begin{equation} \label{iteration-inequality-derivative}
|F_{\lambda} (x; \lambda)| \le C(x; \lambda_*) 
 \| \mathcal{E}_{\lambda} (\cdot; \lambda, \lambda_*) \| 
  \| f (\cdot; \lambda)  \|_{\mathbb{B}_{\lambda}}
 +
 C(x; \lambda_*) \| \mathcal{E} (\cdot; \lambda, \lambda_*) \|
  \| f_{\lambda} (\cdot; \lambda)  \|_{\mathbb{B}_{\lambda}},
\end{equation}
where $C(x; \lambda_*)$ is as in Lemma \ref{useful-estimates-lemma}.
For $k=1$, this allows us to write 
\begin{equation} \label{k1estimate}
    | \frac{\partial}{\partial \lambda} (\mathcal{R}_{\mathcal{E}} (\lambda) u_j^b (\cdot; \lambda_*))(x) |
    \le C (x; \lambda_*)
    \| \mathcal{E}_{\lambda} (\cdot; \lambda, \lambda_*) \| 
    \| u^b_j (\cdot; \lambda_*)  \|_{\mathbb{B}_{\lambda}},
\end{equation}
where the contribution associated with $f_{\lambda} (\cdot; \lambda)$ 
doesn't appear because $u^b_j (\cdot; \lambda_*)$ doesn't depend on 
$\lambda$. Next, using (\ref{iteration-inequality-derivative}) again, 
this time with $f (x; \lambda) = (\mathcal{R}_{\mathcal{E}} (\lambda) u_j^b (\cdot; \lambda_*)) (x)$
we combine (\ref{mathcalR-inequalities}) (with $k=1$) and (\ref{k1estimate1})
to obtain the estimate 
\begin{equation*}
    | \frac{\partial}{\partial \lambda} (\mathcal{R}_{\mathcal{E}} (\lambda)^2 u_j^b (\cdot; \lambda_*))(x)|
    \le 2 C(x; \lambda_*) 
    \| \mathcal{T}_{c, b}^{\gamma} (\lambda_*)^{-1} \|
    \| \mathcal{E} (\cdot; \lambda, \lambda_*) \| 
    \| \mathcal{E}_{\lambda} (\cdot; \lambda, \lambda_*) \| 
    \| u^b_j (\cdot; \lambda_*)  \|_{\mathbb{B}_{\lambda}}.
\end{equation*}
Continuing in this way, we obtain, for each $k \in \mathbb{N}$, the inequality 
\begin{equation} \label{more-derivative-estimates}
\begin{aligned}
    | \frac{\partial}{\partial \lambda} &(\mathcal{R}_{\mathcal{E}} (\lambda)^k u_j^b (\cdot; \lambda_*))(x)| \\
    &\le k C(x; \lambda_*) \| \mathcal{T}_{c, b}^{\gamma} (\lambda_*)^{-1} \|^{k-1}
    \| \mathcal{E} (\cdot; \lambda, \lambda_*) \|^{k-1} 
    \| \mathcal{E}_{\lambda} (\cdot; \lambda, \lambda_*) \| 
    \| u^b_j (\cdot; \lambda_*)  \|_{\mathbb{B}_{\lambda}},
\end{aligned}
\end{equation}
from which the absolute convergence of (\ref{pointwise-derivatives}),
uniform for $\lambda$ in compact subsets of $I_{\lambda_*, r}$,
is immediate from the ratio test, along with the inequality 
(\ref{the-inequality}). We can conclude that 
for each fixed $x \in [c, b)$, the map 
$\lambda \mapsto u_j^b (x; \lambda)$ taking $I_{\lambda_*, r}$
to $\mathbb{C}^{2n}$ is continuously differentiable in $\lambda$. 

For the final statements in the lemma, addressing differentiability 
in $\lambda$, we first observe that by using (\ref{the-real-neumann-series})
along with the estimates (\ref{undifferentiated-estimates}) we can 
conclude that that given any $\lambda_* \in [\lambda_1, \lambda_2]$ 
and any compact set $[c, d] \subset (a, b)$, there exists
a value $r_* > 0$ and a constant $K_0$, depending only on $c, d, \lambda_*$, and $r_*$ so that 
\begin{equation*}
    |u_j^b (x; \lambda)|
    \le K_0, \quad 
    \forall \, (x, \lambda) \in [c, d] \times I_{\lambda_*, r_*}. 
\end{equation*}
Likewise, if we combine (\ref{pointwise-derivatives}), established above
as a uniformly converging expression of $\partial_{\lambda} u_j^b (x; \lambda)$,
with the estimates (\ref{more-derivative-estimates}), along with the 
relation
\begin{equation*}
\| \mathcal{T}_{c, b}^{\gamma} (\lambda_*)^{-1} \| \| \mathcal{E} (\cdot; \lambda, \lambda_*) \| < 1,    
\end{equation*}
we see that given any $\lambda_* \in [\lambda_1, \lambda_2]$ 
and any compact set $[c, d] \subset (a, b)$, there exists
a value $r_* > 0$ and a constant $K_1$, depending only on $c, d, \lambda_*$, and $r_*$ so that 
\begin{equation*}
    |\partial_{\lambda} u_j^b (x; \lambda)|
    \le K_1, \quad 
    \forall \, (x, \lambda) \in [c, d] \times I_{\lambda_*, r_*}. 
\end{equation*}

At this point, we proceed by integrating $J (u_j^b)' = \mathbb{B} (x; \lambda) u_j^b$
on $(x, c)$ to obtain the integral relation
\begin{equation} \label{nonlinear-integral}
J u_j^b (x; \lambda) = J u_j^b (c; \lambda)
- \int_x^c \mathbb{B} (\xi; \lambda) u_j^b (\xi; \lambda) d\xi. 
\end{equation}
In order to establish the claimed derivative relation, we need to 
justify differentiating in $\lambda$ through the integral, 
and for this we use our assumption that there exist 
$b_0, b_1 \in L^1_{\loc} ((a,b), \mathbb{R})$,
independent of $\lambda$, so that 
\begin{equation*}
    |\mathbb{B} (x; \lambda)| \le b_0 (x), \, \forall \, \lambda \in I,\, {\rm a.e.}\, x \in (a, b),
    \quad  |\mathbb{B}_{\lambda} (x; \lambda)| \le b_1 (x), \forall \, \lambda \in I,\, {\rm a.e.}\, x \in (a, b).
\end{equation*}
With these assumptions, along with the pointwise estimates obtained above, we
can justify differentiating (\ref{nonlinear-integral}) in $\lambda$ and 
$x$ to get 
\begin{equation*}
    J (\partial_{\lambda} u_j^b)' = \mathbb{B}_{\lambda} (x; \lambda) u_j^b
    + \mathbb{B} (x; \lambda) \partial_{\lambda} u_j^b. 
\end{equation*}

Last, we need to verify the limit (\ref{diff2}). 
For this, we work again with (\ref{pointwise-derivatives}),
using the estimates (\ref{more-derivative-estimates}). 
Recalling that, by assumption, 
\begin{equation*}
    \lim_{\lambda \to \lambda_*} \|\mathcal{E} (\cdot; \lambda, \lambda_*)\| = 0,
\end{equation*}
we see that the only non-zero summand in (\ref{pointwise-derivatives}) 
is the one with $k = 1$. I.e., 
\begin{equation*}
    \partial_{\lambda} u_j^b (x; \lambda_*)
    = \frac{\partial}{\partial \lambda} (\mathcal{R}_{\mathcal{E}} (\lambda) u_j^b (\cdot; \lambda_*)) (x).
\end{equation*}
Here, 
\begin{equation*}
(\mathcal{R}_{\mathcal{E}} (\lambda) u_j^b (\cdot; \lambda_*)) (x)
= \int_c^b G_{c,b}^{\gamma} (x; \xi; \lambda_*) \mathbb{B}_{\lambda} (\xi; \lambda_*) 
\mathcal{E} (\xi; \lambda, \lambda_*) u_j^b (\xi; \lambda_*) d\xi,
\end{equation*}
and we would like to differentiate in $\lambda$ through the integral sign. For this we need to 
uniformly dominate the integrand 
\begin{equation*}
    G_{c,b}^{\gamma} (x; \xi; \lambda_*) \mathbb{B}_{\lambda} (\xi; \lambda_*) 
\mathcal{E}_{\lambda} (\xi; \lambda, \lambda_*) u_j^b (\xi; \lambda_*),
\end{equation*}
and this is precisely what we assume is possible in 
Assumption {\bf (E)}(iv). 

We can now write 
\begin{equation*}
    \partial_{\lambda} u_j^b (x; \lambda_*)
    = \mathcal{T}_{c, b}^{\gamma} (\lambda_*)^{-1} 
    (\mathcal{E}_{\lambda} (\cdot; \lambda, \lambda_*) u_j^b (\cdot; \lambda)) (x).
\end{equation*}
Finally, $\mathcal{T}_{c, b}^{\gamma} (\lambda_*)^{-1}$ maps into $\mathcal{D}_{c, b}^{\gamma}$,
and elements in $\mathcal{D}_{c, b}^{\gamma}$ satisfy the limit we need. 
\end{proof}

\begin{lemma} \label{lemma2-11prime}
Let Assumptions {\bf (A)} through {\bf (E)} hold, and  
suppose that for all $\lambda \in [\lambda_1, \lambda_2]$,
$0 \notin \sigma_{\ess} (\mathcal{T} (\lambda))$. 
In addition, for each $\lambda \in [\lambda_1, \lambda_2]$, 
let $\ell_b (x; \lambda)$ denote the Lagrangian subspace
associated with the basis $\{u^b_j (x; \lambda)\}_{j=1}^n$ constructed in 
Lemma \ref{lemma2-9prime}. 
Then $\ell_b: (c, b) \times [\lambda_1, \lambda_2] \to \Lambda (n)$
is continuous. Moreover, 
we can choose the bases $\{u^b_j (x; \lambda)\}_{j=1}^n$
so that for each $j \in \{1, 2, \dots, n\}$, the 
function $u_j^b (x; \lambda)$ is piecewise continuously 
differentiable 
in $\lambda$ on $[\lambda_1, \lambda_2]$. 
\end{lemma}

\begin{proof}
First, for each fixed $\lambda_* \in [\lambda_1, \lambda_2]$,
we can use Lemma \ref{continuation-lemma} to obtain a
differentiable-in-$\lambda$ family of bases $\{u^{b, \lambda_*}_j (x; \lambda)\}_{j=1}^n$,
for all $\lambda \in I_{\lambda_*, r_*}$, where $r_* > 0$
is a constant depending on $\lambda_*$ 
(and $\mathcal{T}_{c, b}^{\gamma} (\lambda_*)$, 
including the boundary matrix $\gamma$). By Lemma 
\ref{lemma2-7prime}, these elements must comprise a 
basis for the same Lagrangian subspaces $\ell_b (x; \lambda)$
as the bases $\{u^b_j (x; \lambda)\}_{j=1}^n$
constructed in Lemma \ref{lemma2-9prime}. 
This process creates an open cover of $[\lambda_1, \lambda_2]$, 
created by the union of all of these intervals. Next, we use 
compactness of the interval $[\lambda_1, \lambda_2]$ to extract 
a finite subcover, comprising intervals 
$\{I_{\lambda_*^j, r_*^j}\}_{j=1}^N$ for point-radius 
pairs $\{(\lambda_*^j, r_*^j)\}_{j=1}^N$,
where for notational convenience, we can select the values 
$\{\lambda_*^j\}_{j=1}^N$ so that 
\begin{equation*}
    \lambda_1 =: \lambda_*^1 < \lambda_*^2 < \dots < \lambda_*^N := \lambda_2,
\end{equation*}
and where the values $r_*^j > 0$ are constants respectively 
associated with the values $\lambda_*^j$ in our construction of the family of 
intervals. 

Starting at $\lambda_*^1$, we can take $\{u^b_j (x; \lambda_*^1)\}_{j=1}^n$
to be a basis for the Lagrangian subspace $\ell_b (x; \lambda_*^1)$. As
$\lambda$ increases from $\lambda_*^1$, the extensions
$\{u^{b, \lambda_*^1}_j (x; \lambda)\}_{j=1}^n$ in $I_{\lambda_*^1, r_*^1}$
comprise bases for the Lagrangian subspaces $\ell_b (x; \lambda)$. 
By construction, the set 
$I_{\lambda_*^1, r_*^1} \cap I_{\lambda_*^2, r_*^2}$ must 
be non-empty. We take any $\lambda_*^{1, 2}$ in this intersection,
and we note that at this value of $\lambda$ the 
extensions $\{u^{b, \lambda_*^1}_j (x; \lambda_*^{1, 2})\}_{j=1}^n$ in 
$I_{\lambda_*^1, r_*^1}$ 
serve as a basis for the same Lagrangian subspace as 
the extensions 
$\{u^{b, \lambda_*^2}_j (x; \lambda_*^{1, 2})\}_{j=1}^n$ 
in $I_{\lambda_*^2, r_*^2}$. 
This allows us to continuously switch from the
frame $\{u^{b, \lambda_*^1}_j (x; \lambda_*^{1, 2})\}_{j=1}^n$ to the 
frame $\{u^{b, \lambda_*^2}_j (x; \lambda_*^{1, 2})\}_{j=1}^n$. 

We now allow $\lambda$ to increase from $\lambda_*^{1, 2}$, 
and take the elements $\{u^{b, \lambda_*^2}_j (x; \lambda)\}_{j=1}^n$
as our choice of bases for the Lagrangian subspaces $\ell_b (x; \lambda)$. 
By construction, the set $I_{\lambda^2_*, r_*^2} \cap I_{\lambda_*^3, r_*^3}$
must be non-empty, and we take any $\lambda_*^{2, 3}$
in this intersection, noting that at this value of $\lambda$ the 
analytic extensions
$\{u^{b, \lambda_*^2}_j (x; \lambda_*^{2, 3})\}_{j=1}^n$ in $I_{\lambda_*^2, r_*^2}$ 
serve as a basis for the same Lagrangian subspace as 
the extensions 
$\{u^{b, \lambda_*^3}_j (x; \lambda_*^{2, 3})\}_{j=1}^n$ in $I_{\lambda_*^3, r_*^3}$. 
Continuing in this way, we see that 
$\ell_b: (c, b) \times [\lambda_1, \lambda_2] \to \Lambda (n)$ 
is continuous. 

Summarising our notation, the interval $[\lambda_1, \lambda_2]$ has been 
partitioned into values 
\begin{equation*}
    \lambda_1 =: \lambda_*^{0, 1} < \lambda_*^{1, 2} < \lambda_*^{2, 3} 
    < \dots < \lambda_*^{N-1, N} < \lambda_*^{N, N+1} := \lambda_2,
\end{equation*}
and we use the frame $\{u_j^{b, \lambda_*^k} (x; \lambda)\}_{j=1}^n$ on 
the interval $[\lambda_*^{k-1, k}, \lambda_*^{k, k+1}]$ for all 
$k = 1, 2, \dots, N$. It's clear from the construction 
in Lemma \ref{continuation-lemma} that 
for each $j \in \{1, 2, \dots, n\}$, $u_j^{b, \lambda_*^k} (x; \lambda)$
is continuously differentiable in $\lambda$ 
on $(\lambda_*^{k-1, k}, \lambda_*^{k, k+1})$, so the frame obtained 
by patching these bases together at the points $\{\lambda_*^{i, j}\}_{i, j = 0,1}^{N, N+1}$
is piecewise continuously differentiable. 
\end{proof}

With appropriate modifications, 
Lemmas \ref{lemma2-6prime}--\ref{lemma2-11prime}
can be stated with $\{u_j^b (x; \lambda)\}_{j=1}^n$ replaced by 
$\{u_j^a (x; \lambda)\}_{j=1}^n$.
In addition, under the assumption {\bf (A)$^\prime$}, the analysis 
of $\mathcal{T} (\cdot)$ in this section can be carried out 
for $\mathcal{T}^{\alpha} (\cdot)$, and in particular, Lemmas 
\ref{lemma2-9prime} through \ref{lemma2-11prime}  
hold with $\mathcal{T} (\lambda)$ replaced by 
$\mathcal{T}^{\alpha} (\lambda)$.

\section{The Maslov Index} \label{maslov-section}

Our framework for computing the Maslov index is adapted from 
Section 2 of \cite{HS2}, and we briefly sketch the main ideas
here. Given any pair of Lagrangian subspaces $\ell_1$ and 
$\ell_2$ with respective frames $\mathbf{X}_1 = \genfrac{(}{)}{0pt}{1}{X_1}{Y_1}$
and $\mathbf{X}_2 = \genfrac{(}{)}{0pt}{1}{X_2}{Y_2}$, we consider the matrix
\begin{equation} \label{tildeW}
\tilde{W} := - (X_1 + iY_1)(X_1-iY_1)^{-1} (X_2 - iY_2)(X_2+iY_2)^{-1}. 
\end{equation}
In \cite{HS2}, the authors establish: (1) the inverses 
appearing in (\ref{tildeW}) exist; (2) $\tilde{W}$ is independent
of the specific frames $\mathbf{X}_1$ and $\mathbf{X}_2$ (as long
as these are indeed frames for $\ell_1$ and $\ell_2$); (3) $\tilde{W}$ is 
unitary; and (4) the identity 
\begin{equation} \label{key}
\dim (\ell_1 \cap \ell_2) = \dim (\ker (\tilde{W} + I)).
\end{equation}
Given two continuous paths of Lagrangian subspaces 
$\ell_i: [0, 1] \to \Lambda (n)$, $i = 1, 2$, with 
respective frames $\mathbf{X}_i: [0,1] \to \mathbb{C}^{2n \times n}$,
relation (\ref{key}) allows us to compute the Maslov 
index $\mas (\ell_1, \ell_2; [0,1])$ as a spectral flow
through $-1$ for the path of matrices 
\begin{equation} 
\tilde{W} (t) := - (X_1 (t) + iY_1 (t))(X_1 (t)-iY_1 (t))^{-1} 
(X_2 (t) - iY_2 (t))(X_2 (t)+iY_2 (t))^{-1}. 
\end{equation}

In \cite{HS2}, the authors provide a rigorous definition 
of the Maslov index based on the spectral flow developed 
in \cite{P96}. Here, rather, we give only an intuitive 
discussion. As a starting point, 
if $-1 \in \sigma (\tilde{W} (t_*))$ for some $t_* \in [0, 1]$, 
then we refer to $t_*$ as a crossing point, and its 
multiplicity is taken to be $\dim (\ell_1 (t_*) \cap \ell_2 (t_*))$, 
which by virtue of (\ref{key}) is equivalent to the 
multiplicity of $-1$ as an eigenvalue of $\tilde{W} (t_*)$. 
We compute the Maslov index $\mas (\ell_1, \ell_2; [0, 1])$ 
by allowing $t$ to increase from $0$ to $1$ and incrementing 
the index whenever an eigenvalue crosses $-1$ in the 
counterclockwise direction, while decrementing the index
whenever an eigenvalue crosses $-1$ in the clockwise
direction. These increments/decrements are counted with 
multiplicity, so for example, if a pair of eigenvalues 
crosses $-1$ together in the counterclockwise direction, 
then a net amount of $+2$ is added to the index. Regarding
behavior at the endpoints, if an eigenvalue of $\tilde{W}$
rotates away from $-1$ in the clockwise direction as $t$ increases
from $0$, then the Maslov index decrements (according to 
multiplicity), while if an eigenvalue of $\tilde{W}$
rotates away from $-1$ in the counterclockwise direction as $t$ increases
from $0$, then the Maslov index does not change. Likewise, 
if an eigenvalue of $\tilde{W}$ rotates into $-1$ in the 
counterclockwise direction as $t$ increases
to $1$, then the Maslov index increments (according to 
multiplicity), while if an eigenvalue of $\tilde{W}$
rotates into $-1$ in the clockwise direction as $t$ increases
to $1$, then the Maslov index does not change. Finally, 
it's possible that an eigenvalue of $\tilde{W}$ will arrive 
at $-1$ for $t = t_*$ and remain at $-1$ as $t$ traverses
an interval. In these cases, the 
Maslov index only increments/decrements upon arrival or 
departure, and the increments/decrements are determined 
as for the endpoints (departures determined as with $t=0$,
arrivals determined as with $t = 1$).

One of the most important features of the Maslov index is homotopy invariance, 
for which we need to consider continuously varying families of Lagrangian 
paths. To set some notation, we denote by $\mathcal{P} (\mathcal{I})$ the collection 
of all paths $\mathcal{L} (t) = (\ell_1 (t), \ell_2 (t))$, where 
$\ell_1, \ell_2: \mathcal{I} \to \Lambda (n)$ are continuous paths in the 
Lagrangian--Grassmannian. We say that two paths 
$\mathcal{L}, \mathcal{M} \in \mathcal{P} (\mathcal{I})$ are homotopic in $\Lambda (n)$ provided 
there exists a family $\mathcal{H}_s$ so that 
$\mathcal{H}_0 = \mathcal{L}$, $\mathcal{H}_1 = \mathcal{M}$, 
and $\mathcal{H}_s (t)$ is continuous as a map from $(t,s) \in \mathcal{I} \times [0,1]$
into $\Lambda (n) \times \Lambda (n)$. 

The Maslov index has the following properties. 

\medskip
\noindent
{\bf (P1)} (Path Additivity) If $\mathcal{L} \in \mathcal{P} (\mathcal{I})$
and $a, b, c \in \mathcal{I}$, with $a < b < c$, then 
\begin{equation*}
\mas (\mathcal{L};[a, c]) = \mas (\mathcal{L};[a, b]) + \mas (\mathcal{L}; [b, c]).
\end{equation*}

\medskip
\noindent
{\bf (P2)} (Homotopy Invariance) If $\mathcal{I} = [a, b]$ and  
$\mathcal{L}, \mathcal{M} \in \mathcal{P} (\mathcal{I})$ 
are homotopic in $\Lambda (n)$ with $\mathcal{L} (a) = \mathcal{M} (a)$ and  
$\mathcal{L} (b) = \mathcal{M} (b)$ (i.e., if $\mathcal{L}, \mathcal{M}$
are homotopic with fixed endpoints) then 
\begin{equation*}
\mas (\mathcal{L};[a, b]) = \mas (\mathcal{M};[a, b]).
\end{equation*} 

Straightforward proofs of these properties appear in \cite{HLS2017}
for Lagrangian subspaces of $\mathbb{R}^{2n}$, and proofs in the current setting of 
Lagrangian subspaces of $\mathbb{C}^{2n}$ are essentially identical. 

As noted previously, the direction we associate with a 
crossing point is determined by the direction in which eigenvalues
of $\tilde{W}$ rotate through $-1$ (counterclockwise is positive, 
while clockwise is negative). In order to analyze this direction
in specific cases, we will make use of the following lemma from 
\cite{HS2}. 

\begin{lemma} \label{monotonicity1}
Suppose $\ell_1, \ell_2: \mathcal{I} \to \Lambda (n)$ denote paths of 
Lagrangian subspaces of $\mathbb{C}^{2n}$ with respective frames 
$\mathbf{X}_1 = \genfrac{(}{)}{0pt}{1}{X_1}{Y_1}$ and $\mathbf{X}_2 = \genfrac{(}{)}{0pt}{1}{X_2}{Y_2}$
that are differentiable at $t_0 \in \mathcal{I}$. If the matrices 
\begin{equation*}
- \mathbf{X}_1 (t_0)^* J \mathbf{X}_1' (t_0) = X_1 (t_0)^* Y_1' (t_0) - Y_1 (t_0)^* X_1'(t_0)
\end{equation*}
and (noting the sign change)
\begin{equation*}
\mathbf{X}_2 (t_0)^* J \mathbf{X}_2' (t_0) = - (X_2 (t_0)^* Y_2' (t_0) - Y_2 (t_0)^* X_2'(t_0))
\end{equation*}
are both non-negative, and at least one is positive definite, then the eigenvalues of 
$\tilde{W} (t)$ rotate in the counterclockwise direction as $t$ increases through $t_0$. 
Likewise, if both of these matrices are non-positive, and at least one is 
negative definite, then the eigenvalues of $\tilde{W} (t)$ rotate in the clockwise direction as 
$t$ increases through $t_0$.
\end{lemma}

\section{Proofs of the Main Theorems} \label{theorems-section}

In this section, we use our Maslov index framework to prove Theorems 
\ref{regular-singular-theorem} and \ref{singular-theorem}.

\subsection{Proof of Theorem \ref{regular-singular-theorem}} 
\label{regular-singular-section}

Fix any pair $\lambda_1, \lambda_2 \in I$, $\lambda_1 < \lambda_2$
so that for all $\lambda \in [\lambda_1, \lambda_2]$, we have
the exclusion $0 \notin \sigma_{\ess} (\mathcal{T}^{\alpha} (\lambda))$,
and let $\ell_{\alpha} (x; \lambda)$ denote the map of Lagrangian subspaces
associated with the frames $\mathbf{X}_{\alpha} (x; \lambda)$ specified 
in (\ref{frame-alpha}). Keeping in mind that $\lambda_2$ is fixed, let 
$\ell_b (x; \lambda_2)$ denote the map of Lagrangian subspaces
associated with the frames $\mathbf{X}_b (x; \lambda_2)$ specified 
in (\ref{frame-b}). We emphasize that since $\lambda_2$ is fixed we
don't yet require Lemma \ref{lemma2-11prime} to extend the 
frame $\mathbf{X}_b (x; \lambda_2)$ to additional values
$\lambda \in [\lambda_1, \lambda_2]$. 
We will establish Theorem \ref{regular-singular-theorem} by 
considering the Maslov index for $\ell_{\alpha} (x; \lambda)$ and 
$\ell_b (x; \lambda_2)$ along a path designated as the 
{\it Maslov box} in the next paragraph. As described in 
Section \ref{maslov-section}, this Maslov index is computed as
a spectral flow for the matrix 
\begin{equation} \label{tildeW-bc1}
\begin{aligned}
\tilde{W} (x; \lambda) &= - (X_{\alpha} (x; \lambda) + i Y_{\alpha} (x; \lambda))
(X_{\alpha} (x; \lambda) - i Y_{\alpha} (x; \lambda))^{-1} \\
& \times (X_b (x; \lambda_2) - i Y_b (x; \lambda_2))
(X_b (x; \lambda_2) + i Y_b (x; \lambda_2))^{-1}.
\end{aligned}
\end{equation}

By Maslov Box in this case we mean the following sequence of contours, 
specified for some value $c \in (a, b)$ to be chosen sufficiently  
close to $b$ during the analysis (sufficiently large if $b = + \infty$):
(1) fix $x = a$ and let $\lambda$ increase from $\lambda_1$ to $\lambda_2$ 
(the {\it bottom shelf}); 
(2) fix $\lambda = \lambda_2$ and let $x$ increase from $a$ to $c$ 
(the {\it right shelf}); (3) fix $x = c$ and let $\lambda$
decrease from $\lambda_2$ to $\lambda_1$ (the {\it top shelf}); and (4) fix
$\lambda = \lambda_1$ and let $x$ decrease from $c$ to $a$ (the 
{\it left shelf}). (See Figure \ref{box-figure}.)

\begin{figure}[ht]
\begin{center}
\begin{tikzpicture}
\draw[<->] (-8,-1.5) -- (1,-1.5);	
\draw[<->] (0,-2) -- (0,4.5);	
\node at (.5,4.3) {$x$};
\node at (-7.5,-2) {$\lambda$};
\node at (-6,-2) {$\lambda_1$};
\node at (-1,-2) {$\lambda_2$};
\node at (.5,3) {$c$};
\draw[-] (-.1,3) -- (.1,3);
\node at (.5,-1) {$a$};
\draw[-] (-.1,-1) -- (.1,-1);
%
\draw[thick, ->] (-6,-1) -- (-3.5,-1);
\draw[thick] (-3.5,-1) -- (-1,-1);	
\draw[thick, ->] (-1,-1) -- (-1,1);
\draw[thick] (-1,1) -- (-1, 3);
\draw[thick,->] (-1,3) -- (-3.5, 3);
\draw[thick] (-3.5,3) -- (-6, 3);
\draw[thick,->] (-6,3) -- (-6, 1);
\draw[thick] (-6,1) -- (-6, -1);
%
\node[scale = .75] at (-3.5, -.55) {$\mas (\ell_{\alpha} (a; \cdot), \ell_b (a; \lambda_2); [\lambda_1, \lambda_2])$};
\node[scale = .75, rotate=90] at (-.55, 1.1) {$\mas (\ell_{\alpha} (\cdot; \lambda_2), \ell_b (\cdot; \lambda_2); [a, c])$};
\node[scale = .75] at (-3.6, 3.45) {$- \mas (\ell_{\alpha} (c; \cdot), \ell_b (c; \lambda_2); [\lambda_1, \lambda_2])$};
\node[scale = .75, rotate=90] at (-6.45, 1.1) {$- \mas (\ell_{\alpha} (\cdot; \lambda_1), \ell_b (\cdot; \lambda_2); [a, c])$};
\end{tikzpicture}
\end{center}
\caption{The Maslov Box.} \label{box-figure}
\end{figure}
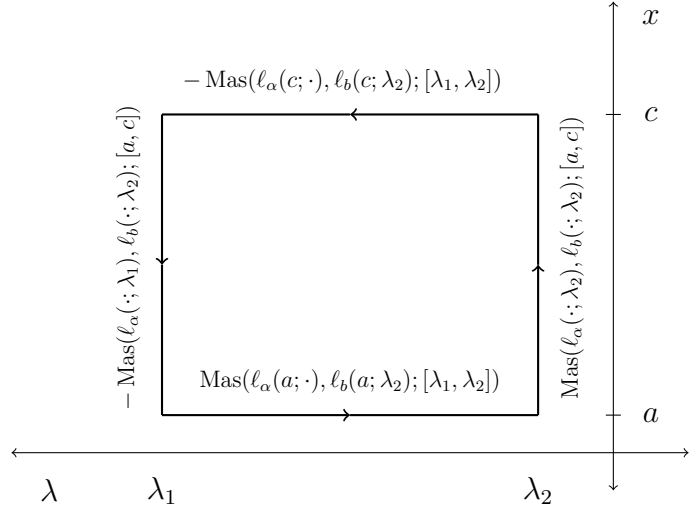

{\it Right shelf.}
We begin our analysis with the right shelf, for which $\mathbf{X}_{\alpha}$ 
and $\mathbf{X}_b$ are both evaluated at $\lambda_2$. By construction, 
$\ell_{\alpha} (x; \lambda_2)$ and $\ell_b (x; \lambda_2)$ will have non-trivial
intersection at some $x \in [a, c]$ (and so for all $x \in [a, c]$) with dimension $m$ 
if and only if $\lambda_2$ is an eigenvalue of the operator pencil 
$\mathcal{T}^{\alpha} (\cdot)$ with multiplicity $m$ (i.e., $0$ is an 
eigenvalue of $\mathcal{T}^{\alpha} (\lambda_2)$ with multiplicity $m$). 
In the event that $\lambda_2$
is not an eigenvalue of $\mathcal{T}^{\alpha} (\cdot)$, there will be no 
crossing points along the right shelf. On the other hand, if
$\lambda_2$ is an eigenvalue of $\mathcal{T}^{\alpha} (\cdot)$ with multiplicity
$m$, then $\tilde{W} (x; \lambda_2)$ will have $-1$ as an eigenvalue
with multiplicity $m$ for all $x \in [a, c]$. In either case,
\begin{equation*}
\mas (\ell_{\alpha} (\cdot; \lambda_2), \ell_b (\cdot; \lambda_2); [a, c]) = 0,
\end{equation*} 
so there is no contribution from the right shelf. 

{\it Bottom shelf.} For the bottom shelf, $\ell_{\alpha} (a; \lambda)$ is determined 
by the boundary matrix $\alpha (\lambda)$, leading to the Maslov index 
\begin{equation*}
\mas (\ell_{\alpha} (a; \cdot), \ell_b (a; \lambda_2); [\lambda_1, \lambda_2]),
\end{equation*} 
which appears in the statement of Theorem \ref{regular-singular-theorem}. As 
observed in Remark \ref{regular-singular-remark}, there is no contribution 
from this term if $\alpha (\lambda)$ is constant. 

{\it Top shelf.} For the top shelf, $\tilde{W} (c; \lambda)$ detects intersections
between $\ell_{\alpha} (c; \lambda)$ and $\ell_b (c; \lambda_2)$ as $\lambda$
decreases from $\lambda_2$ to $\lambda_1$. Such intersections
correspond precisely with eigenvalues of  
$\mathcal{T}_{a, c}^{\alpha} (\cdot)$, where 
$\mathcal{T}_{a, c}^{\alpha} (\cdot)$ denotes the 
finite-interval (or {\it truncated}) operator specified
similarly as $\mathcal{T}^{\alpha} (\cdot)$, except on the domain
\begin{equation*}
    \mathcal{D}_{a, c}^{\alpha} := 
    \{y \in \mathcal{D}_{a, c, M}: \alpha (\lambda) y(a) = 0, 
    \quad \mathbf{X}_b (c; \lambda_2)^* J y (c) = 0 \},
\end{equation*}
where $\mathcal{D}_{a, c, M}$ denotes the domain of the maximal
operator specified as in Definition \ref{maximal-operator}, 
except on $(a, c)$.
Similarly as in Section \ref{operator-section}, we can check that 
for each $\lambda \in [\lambda_1, \lambda_2]$
$\mathcal{T}_{a, c}^{\alpha} (\lambda)$ is a self-adjoint operator. (In fact,
since $\mathcal{T}_{a, c}^{\alpha} (\lambda)$ is posed on a bounded interval
$(a, c)$ with $\mathbb{B} (\cdot; \lambda), \mathbb{B}_{\lambda} (\cdot; \lambda) \in L^1 ((a, c), \mathbb{C}^{2n \times 2n})$, 
for all $\lambda \in [\lambda_1, \lambda_2]$, 
self-adjointness can be established by more routine considerations.)

We know from Lemma \ref{monotonicity1} that monotonicity
in $\lambda$ is determined by 
$- \mathbf{X}_{\alpha} (c; \lambda)^* J \partial_{\lambda} \mathbf{X}_{\alpha} (c; \lambda)$, 
and we readily compute 
\begin{equation*}
\begin{aligned}
\frac{\partial}{\partial x} \mathbf{X}_{\alpha}^* &(x; \lambda) J \partial_{\lambda} \mathbf{X}_{\alpha} (x; \lambda)
= \mathbf{X}_{\alpha}^{\prime} (x; \lambda)^* J \partial_{\lambda} \mathbf{X}_{\alpha} (x; \lambda) 
+ \mathbf{X}_{\alpha} (x; \lambda)^* J \partial_{\lambda} \mathbf{X}_{\alpha}^{\prime} (x; \lambda) \\
&= - \mathbf{X}_{\alpha}^{\prime}(x; \lambda)^* J^* \partial_{\lambda} \mathbf{X}_{\alpha} (x; \lambda) 
+ \mathbf{X}_{\alpha} (x; \lambda)^* \partial_{\lambda} J \mathbf{X}_{\alpha}^{\prime} (x; \lambda) \\
&= - \mathbf{X}_{\alpha} (x; \lambda)^* \mathbb{B} (x; \lambda) \partial_{\lambda} \mathbf{X}_{\alpha} (x; \lambda) 
+ \mathbf{X}_{\alpha} (x; \lambda)^* \mathbb{B} (x; \lambda) \partial_{\lambda} \mathbf{X}_{\alpha} (x; \lambda) \\
&+ \mathbf{X}_{\alpha}^* \partial_{\lambda} \mathbb{B} (x; \lambda) \mathbf{X}_{\alpha} (x; \lambda) 
= \mathbf{X}_{\alpha} (x; \lambda)^* \mathbb{B}_{\lambda} (x; \lambda) \mathbf{X}_{\alpha} (x; \lambda),
\end{aligned}
\end{equation*}
where the differentiation of $\mathbf{X}_{\alpha} (x; \lambda)$ in $x$ and $\lambda$, including the 
exchange of order of these derivatives, is straightforward since the columns of $\mathbf{X}_{\alpha} (x; \lambda)$
are simply solutions to standard initial value problems. 
Integrating on $[a,x]$, and noting that 
\begin{equation*}
    \mathbf{X}_{\alpha}^* (a; \lambda) J \partial_{\lambda} \mathbf{X}_{\alpha} (a; \lambda)
    = \alpha (\lambda) J \partial_\lambda \alpha^* (\lambda),
\end{equation*}
we obtain the relation   
\begin{equation*}
\mathbf{X}_{\alpha} (x; \lambda)^* J \partial_{\lambda} \mathbf{X}_{\alpha} (x; \lambda)
= \alpha (\lambda) J \partial_\lambda \alpha^* (\lambda)
+ \int_a^x \mathbf{X}_{\alpha} (y;\lambda)^* \mathbb{B}_{\lambda} (y; \lambda) \mathbf{X}_{\alpha} (y; \lambda) dy. 
\end{equation*}
By assumption, the matrix $\alpha (\lambda) J \partial_\lambda \alpha^* (\lambda)$
is non-negative for all $\lambda \in [\lambda_1, \lambda_2]$, so
monotonicity along the top shelf follows by setting $x = c$ and appealing 
to Assumption {\bf (B)}. 
In this way, we see that Assumption {\bf (B)} (along with our assumptions on 
the nature of $\alpha (\lambda)$) ensures that 
as $\lambda$ increases, the eigenvalues of $\tilde{W} (c; \lambda)$ will 
rotate monotonically in the clockwise direction. 
Since each crossing along the top shelf
corresponds with a value $\lambda$ for which $0$ is an eigenvalue of 
$\mathcal{T}_{a, c}^{\alpha} (\lambda)$, we can conclude that 
\begin{equation} \label{truncated-count-alpha}
\mathcal{N}_{a, c}^{\alpha} ([\lambda_1, \lambda_2)) = 
- \mas (\ell_{\alpha} (c; \cdot), \ell_b (c; \lambda_2); [\lambda_1, \lambda_2]),
\end{equation}
where $\mathcal{N}_{a, c}^{\alpha} ([\lambda_1, \lambda_2))$ denotes a count, including
multiplicities, of the values $\lambda \in [\lambda_1, \lambda_2)$ 
for which $0$ is an eigenvalue of $\mathcal{T}_{a, c}^{\alpha} (\lambda)$. 
We note that $\lambda_1$ is included in the count, because in the event
that $(c, \lambda_1)$ is a crossing point, eigenvalues of 
$\tilde{W} (c; \lambda)$ will rotate away from $-1$ in the clockwise 
direction as $\lambda$ increases from $\lambda_1$ (thus decrementing the 
Maslov index). Likewise, $\lambda_2$ is not included in the count, because in the event
that $(c, \lambda_2)$ is a crossing point, eigenvalues of 
$\tilde{W} (c; \lambda)$ will rotate into $-1$ in the clockwise 
direction as $\lambda$ increases to $\lambda_2$ (thus leaving the 
Maslov index unchanged). 

{\it Left shelf.} 
Our analysis so far leaves only the left shelf to consider, and 
we observe that the Maslov index on the left shelf can be expressed as 
\begin{equation*}
- \mas (\ell_{\alpha} (\cdot; \lambda_1), \ell_b (\cdot; \lambda_2); [a, c]).
\end{equation*} 
Using path additivity and homotopy invariance, we can sum the Maslov
indices on each shelf of the Maslov Box to arrive at the relation 
\begin{equation} \label{box-sum-alpha}
\mathcal{N}_{a, c}^{\alpha} ([\lambda_1, \lambda_2)) 
= \mas (\ell_{\alpha} (\cdot; \lambda_1), \ell_b (\cdot; \lambda_2); [a, c])
- \mas (\ell_{\alpha} (a; \cdot), \ell_b (a; \lambda_2); [\lambda_1, \lambda_2]).
\end{equation}

In order to obtain a statement about $\mathcal{N}^{\alpha} ([\lambda_1, \lambda_2))$, 
we observe that eigenvalues of $\mathcal{T}^{\alpha} (\cdot)$ correspond precisely 
with intersections of $\ell_{\alpha} (c; \lambda)$
and $\ell_b (c; \lambda)$. (We emphasize that in this last statement, 
$\ell_b$ is evaluated at $\lambda$, not $\lambda_2$, and so we 
are using Lemma \ref{lemma2-11prime}). Employing a monotonicity
argument similar to the one above for the top shelf, we can conclude that 
\begin{equation} \label{full-count-alpha}
\mathcal{N}^{\alpha} ([\lambda_1, \lambda_2)) = 
- \mas (\ell_{\alpha} (c; \cdot), \ell_b (c; \cdot); [\lambda_1, \lambda_2]).
\end{equation}

\begin{remark}
The monotonicity argument in the case of (\ref{full-count-alpha}) is a bit more subtle than in 
the case above for the top shelf, and in order to keep the analysis 
as complete as possible, we include the full argument in the 
appendix, Section \ref{monotonicity-lambda-section}. 
\end{remark}

Our next goal is to relate the Maslov index on the right-hand side of 
(\ref{full-count-alpha}) to Maslov indices in which $\lambda$ only varies
in one or the other of $\ell_{\alpha} (c; \lambda)$ and $\ell_b (c; \lambda)$.
For this, we have the following claim. 

\begin{claim} \label{triangle-claim-alpha}
Under the assumptions of Theorem \ref{regular-singular-theorem}, 
and for any $c \in (a, b)$,
\begin{equation*}
\begin{aligned}
\mas (\ell_{\alpha} (c; \cdot), \ell_b (c; \cdot); [\lambda_1, \lambda_2])
& =
\mas (\ell_{\alpha} (c; \lambda_1), \ell_b (c; \cdot); [\lambda_1, \lambda_2]) \\
&\quad + \mas (\ell_{\alpha} (c; \cdot), \ell_b (c; \lambda_2); [\lambda_1, \lambda_2]). 
\end{aligned}
\end{equation*}
\end{claim}

\begin{proof} With $c \in (a, b)$ fixed, we consider 
$\ell_{\alpha} (c; \cdot), \ell_b (c; \cdot): [\lambda_1, \lambda_2] \to \Lambda (n)$
and set 
\begin{equation*}
\begin{aligned}
\tilde{W}_c (\lambda, \mu) &:= 
- (X_{\alpha} (c; \lambda) + i Y_{\alpha} (c; \lambda)) (X_{\alpha} (c; \lambda) - i Y_{\alpha} (c; \lambda))^{-1} \\
& \quad \quad \quad \times (X_b (c; \mu) - i Y_b (c; \mu)) (X_b (c; \mu) + i Y_b (c; \mu))^{-1}. 
\end{aligned}
\end{equation*}
We now compute the Maslov index associated with $\tilde{W}_c (\lambda, \mu)$ along
the triangular path in $[\lambda_1, \lambda_2] \times [\lambda_1, \lambda_2]$ comprising 
the following three paths: (1) fix $\lambda = \lambda_1$ and let $\mu$ increase from 
$\lambda_1$ to $\lambda_2$; (2) fix $\mu = \lambda_2$ and let $\lambda$ increase 
from $\lambda_1$ to $\lambda_2$; and (3) let $\lambda$ and $\mu$ decrease together 
(i.e., with $\lambda = \mu$) from $\lambda_2$ to $\lambda_1$. 
(See Figure \ref{triangle-claim-figure}.)
The claim follows from path additivity and homotopy invariance.
\end{proof}

\begin{figure}[ht]
\begin{center}
\begin{tikzpicture}
\draw[<-, thick] (-5.5,0) -- (1,0);	
\draw[->, thick] (.5,.5) -- (.5,-5.5);	
%
\draw[->, thick] (-4.5,-4.5) -- (-4.5,-2.5);
\draw[thick] (-4.5,-2.5) -- (-4.5,-1);
\draw[->, thick] (-4.5,-1) -- (-2.5,-1);
\draw[thick] (-2.5, -1) -- (-.5, -1);
\draw[->, thick] (-.5,-1) -- (-2.5,-2.75);
\draw[thick] (-2.5,-2.75) -- (-4.5,-4.5);
%
\node at (-5.5, .5) {$\lambda$};
\node at (-4.5,.5) {$\lambda_1$};
\draw[thick] (-4.5,.1) -- (-4.5,-.1);
\node at (-.5,.5) {$\lambda_2$};
\draw[thick] (-.5,.1) -- (-.5,-.1);
\node at (1,-1) {$\lambda_2$};
\draw[thick] (.4,-1) -- (.6,-1);
\node at (1,-4.5) {$\lambda_1$};
\draw[thick] (.4,-4.5) -- (.6,-4.5);
\node at (1,-5.5) {$\mu$};
%
\node[scale = .75] at (-2.5, -.6) {$\mas (\ell_{\alpha} (c; \cdot), \ell_b (c; \lambda_2); [\lambda_1, \lambda_2])$};
\node[scale = .75, rotate=90] at (-4.9, -2.7) {$\mas (\ell_{\alpha} (c; \lambda_1), \ell_b (c; \cdot); [\lambda_1, \lambda_2])$};
\node[scale = .75, rotate=43] at (-2.2, -3.1) {$- \mas (\ell_{\alpha} (c; \cdot), \ell_b (c; \cdot); [\lambda_1, \lambda_2])$};
\end{tikzpicture}
\end{center}
\caption{Triangular path in the $(\lambda, \mu)$-plane for Claim \ref{triangle-claim-alpha}.} \label{triangle-claim-figure}
\end{figure}
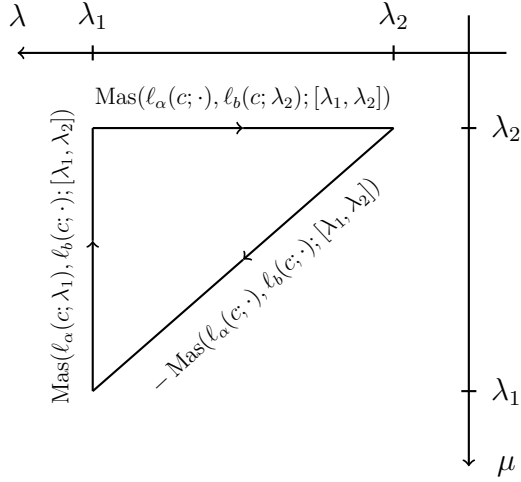

We can conclude from (\ref{truncated-count-alpha}), (\ref{full-count-alpha}), and Claim \ref{triangle-claim-alpha} 
that 
\begin{equation} \label{N-Nc-alpha}
\mathcal{N}^{\alpha} ([\lambda_1, \lambda_2)) = \mathcal{N}_{a, c}^{\alpha} ([\lambda_1, \lambda_2))
- \mas (\ell_{\alpha} (c; \lambda_1), \ell_b (c; \cdot); [\lambda_1, \lambda_2]). 
\end{equation}
By monotonicity, 
\begin{equation*}
\mas (\ell_{\alpha} (c; \lambda_1), \ell_b (c; \cdot); [\lambda_1, \lambda_2]) \le 0,
\end{equation*}
from which we can additionally conclude that 
\begin{equation*}
\mathcal{N}^{\alpha} ([\lambda_1, \lambda_2)) 
\ge \mathcal{N}_{a, c}^{\alpha} ([\lambda_1, \lambda_2)). 
\end{equation*}
In light of (\ref{box-sum-alpha}), this gives 
\begin{equation} \label{count-inequality-alpha}
  \mathcal{N}^{\alpha} ([\lambda_1, \lambda_2)) 
\ge \mas (\ell_{\alpha} (\cdot; \lambda_1), \ell_b (\cdot; \lambda_2); [a, c])
- \mas (\ell_{\alpha} (a; \cdot), \ell_b (a; \lambda_2); [\lambda_1, \lambda_2]).  
\end{equation}
Here, we emphasize that it follows from (\ref{full-count-alpha}), 
monotonicity in the calculation of the right-hand side of 
(\ref{full-count-alpha}), and continuity in $\lambda$ that
the count $\mathcal{N}^{\alpha} ([\lambda_1, \lambda_2))$ must be
finite. 

The first summand on the right-hand side of (\ref{count-inequality-alpha}) is computed over the 
compact interval $[a, c]$ on which (\ref{linear-hammy}) can be viewed as 
a regular system, as analyzed in \cite{HS2}. In \cite{HS2}, the authors
show that the direction of crossing points for such systems are all positive as 
$x$ increases from $a$ to $c$. (See the statement and proof of Theorem 1.1
in \cite{HS2}.) It follows that as $c \to b^-$ the values 
$\mas (\ell_{\alpha} (\cdot; \lambda_1), \ell_b (\cdot; \lambda_2); [a, c])$
are monotonically non-decreasing, and since $\mathcal{N}^{\alpha} ([\lambda_1, \lambda_2))$
and $\mas (\ell_{\alpha} (a; \cdot), \ell_b (a; \lambda_2); [\lambda_1, \lambda_2])$ 
are both finite and independent of $c$, we can conclude that the limit
\begin{equation*}
    \lim_{c \to b^-} \mas (\ell_{\alpha} (\cdot; \lambda_1), \ell_b (\cdot; \lambda_2); [a, c]),
\end{equation*}
must exist, and in fact that it must be the case that this limit is obtained 
for all $c$ sufficiently close to $b$ (sufficiently large if $b = + \infty$). 
As noted in Remark \ref{regular-singular-remark}, we denote this limit by
$\mas (\ell_{\alpha} (\cdot; \lambda_1), \ell_b (\cdot; \lambda_2); [a, b))$.
In this way, the first assertion of Theorem \ref{regular-singular-theorem}
is obtained by taking a limit on both sides of (\ref{count-inequality-alpha})
as $c \to b^-$.

For the second assertion in Theorem \ref{regular-singular-theorem}, we can combine 
(\ref{N-Nc-alpha}) and (\ref{box-sum-alpha}) to write 
\begin{equation*}
\begin{aligned}
\mathcal{N}^{\alpha} ([\lambda_1, \lambda_2)) 
&= \mas (\ell_{\alpha} (\cdot; \lambda_1), \ell_b (\cdot; \lambda_2); [a, c])
- \mas (\ell_{\alpha} (a; \cdot), \ell_b (a; \lambda_2); [\lambda_1, \lambda_2]) \\
&\quad - \mas (\ell_{\alpha} (c; \lambda_1), \ell_b (c; \cdot); [\lambda_1, \lambda_2]).     
\end{aligned}
\end{equation*}
If there exists a value $c_0 \in (a, b)$ so that for 
all $c \in (c_0, b)$
\begin{equation*} 
    \ell_{\alpha} (c; \lambda_1) \cap \ell_b (c; \lambda)
    = \{0\}, \quad \forall \,\, \lambda \in [\lambda_1, \lambda_2), 
\end{equation*}
then for all $c \in (c_0, b)$ we have 
\begin{equation*}
\mathcal{N}^{\alpha} ([\lambda_1, \lambda_2)) 
= \mas (\ell_{\alpha} (\cdot; \lambda_1), \ell_b (\cdot; \lambda_2); [a, c])
- \mas (\ell_{\alpha} (a; \cdot), \ell_b (a; \lambda_2); [\lambda_1, \lambda_2]).
\end{equation*}
The second claim in Theorem \ref{regular-singular-theorem} follows immediately
upon taking $c \to b^-$, noting particularly that by monotonicity the right-hand
side of this last expression remains fixed for $c \in (c_0, b)$. 
\hfill $\square$

\subsection{Proof of Theorem \ref{singular-theorem}} 
\label{singular-theorem-section}

Similarly as in the proof of Theorem \ref{regular-singular-theorem}, we fix any pair 
$\lambda_1, \lambda_2 \in I$, $\lambda_1 < \lambda_2$, 
for which we have that for each $\lambda \in [\lambda_1, \lambda_2]$,  
$0 \notin \sigma_{\ess} (\mathcal{T} (\lambda))$. 
For the proof of Theorem \ref{singular-theorem}, we let $\ell_b (x; \lambda_2)$
be as in the proof of Theorem \ref{regular-singular-theorem}, and we 
let $\ell_a (x; \lambda)$ denote the map 
of Lagrangian subspaces associated with the frames $\mathbf{X}_a (x; \lambda)$
constructed as in Lemma \ref{lemma2-11prime}, except with the analysis 
on $(a, c)$ rather than $(c, b)$.
We will establish Theorem \ref{singular-theorem} by considering the Maslov
index for $\ell_a (x; \lambda)$ and $\ell_b (x; \lambda_2)$
along the Maslov box designated just below. As described in 
Section \ref{maslov-section}, this Maslov index is computed as
a spectral flow for the matrix 
\begin{equation} \label{singular-tildeW}
\begin{aligned}
\tilde{W} (x; \lambda) &= - (X_a (x; \lambda) + i Y_a (x; \lambda))
(X_a (x; \lambda) - i Y_a (x; \lambda))^{-1} \\
& \times (X_b (x; \lambda_2) - i Y_b (x; \lambda_2))
(X_b (x; \lambda_2) + i Y_b (x; \lambda_2))^{-1}
\end{aligned}
\end{equation}
(re-defined from Section \ref{regular-singular-section}). 

In this case, the Maslov Box will consist of the following sequence of contours, 
specified for some values $c_1, c_2 \in (a, b)$, $c_1 < c_2$ to be 
chosen sufficiently close to $a$ and $b$ (respectively) during the analysis:
(1) fix $x = c_1$ and let $\lambda$ increase from $\lambda_1$ to $\lambda_2$ 
(the {\it bottom shelf}); 
(2) fix $\lambda = \lambda_2$ and let $x$ increase from $c_1$ to $c_2$ 
(the {\it right shelf}); (3) fix $x = c_2$ and let $\lambda$
decrease from $\lambda_2$ to $\lambda_1$ (the {\it top shelf}); and (4) fix
$\lambda = \lambda_1$ and let $x$ decrease from $c_2$ to $c_1$ 
(the {\it left shelf}). (The figure is similar to Figure \ref{box-figure}). 

{\it Right shelf.} 
In this case, our calculation along the right shelf detects intersections 
between $\ell_a (x; \lambda_2)$ and $\ell_b (x; \lambda_2)$ as $x$ increases
from $c_1$ to $c_2$. By construction, 
$\ell_a (x; \lambda_2)$ and $\ell_b (x; \lambda_2)$ will have non-trivial intersection
at some value $x \in [c_1, c_2]$ with dimension $m$ if and only if $\lambda_2$ is an 
eigenvalue of the operator pencil $\mathcal{T} (\cdot)$ with multiplicity $m$. 
In the event that $\lambda_2$
is not an eigenvalue of $\mathcal{T} (\cdot)$, there will be no 
crossing points along the right shelf. On the other hand, if
$\lambda_2$ is an eigenvalue of $\mathcal{T} (\cdot)$ with multiplicity
$m$, then $\tilde{W} (x; \lambda_2)$ will have $-1$ as an eigenvalue
with multiplicity $m$ for all $x \in [c_1, c_2]$. In either case,
\begin{equation} \label{right-shelf-singular}
\mas (\ell_a (\cdot; \lambda_2), \ell_b (\cdot; \lambda_2); [c_1, c_2]) = 0.
\end{equation} 

{\it Bottom shelf.} For the bottom shelf, we're looking for 
intersections between $\ell_a (c_1;\lambda)$ and 
$\ell_b (c_1; \lambda_2)$ as $\lambda$ increases from 
$\lambda_1$ to $\lambda_2$. Since $\ell_a (c_1; \lambda)$
corresponds with solutions that lie left in $(a, b)$, 
this leads to a calculation similar to the calculation of 
\begin{equation*}
\mas (\ell_{\alpha} (c; \cdot), \ell_b (c; \lambda_2); [\lambda_1, \lambda_2]),
\end{equation*}  
which arose in our analysis of the top shelf for the proof
of Theorem \ref{regular-singular-theorem}. For the moment, 
the only thing we will note about this quantity is that due 
to monotonicity in $\lambda$ (following similarly as in 
Section \ref{monotonicity-lambda-section}), we have
the inequality 
\begin{equation} \label{bottom-shelf-inequality}
\mas (\ell_a (c_1; \cdot), \ell_b (c_1; \lambda_2); [\lambda_1, \lambda_2]) \le 0.
\end{equation}  

{\it Top shelf.} For the top shelf, $\tilde{W} (c_2; \lambda)$ detects intersections
between $\ell_a (c_2; \lambda)$ and $\ell_b (c_2; \lambda_2)$ as $\lambda$
decreases from $\lambda_2$ to $\lambda_1$. In this way, intersections
correspond precisely with eigenvalues of the restriction
$\mathcal{T}_{a, c_2} (\cdot)$ of the maximal operator associated with
(\ref{linear-hammy}) on $(a, c_2)$ to the domain 
\begin{equation*}
    \mathcal{D}_{a, c_2} := 
    \{y \in \mathcal{D}_{a, c_2, M}: \lim_{x \to a^+} U^a (x; \mu_0, \lambda_0)^* J y (x) = 0, 
    \quad \mathbf{X}_b (c_2; \lambda_2)^* J y (c_2) = 0 \}.
\end{equation*}
Similarly as in Section \ref{operator-section}, we can check that 
for each $\lambda \in [\lambda_1, \lambda_2]$, $\mathcal{T}_{a, c_2} (\lambda)$ 
is a self-adjoint operator. 

We can verify monotonicity along the top shelf almost precisely as 
in Section \ref{monotonicity-lambda-section}, and we can conclude from 
this that 
\begin{equation} \label{truncated-count}
\mathcal{N}_{a, c_2} ([\lambda_1, \lambda_2)) = 
- \mas (\ell_a (c_2; \cdot), \ell_b (c_2; \lambda_2); [\lambda_1, \lambda_2]),
\end{equation}
where $\mathcal{N}_{a, c_2} ([\lambda_1, \lambda_2))$ denotes a count 
of the number of eigenvalues that $\mathcal{T}_{a, c_2} (\cdot)$ has on 
the interval $[\lambda_1, \lambda_2)$. (The inclusion of $\lambda_1$ and exclusion 
of $\lambda_2$ are precisely as discussed in the proof of 
Theorem \ref{regular-singular-theorem}.)

Similarly as with Claim \ref{triangle-claim-alpha}, we obtain the relation 
\begin{equation} \label{triangle-claim}
\begin{aligned}
\mas (\ell_a (c_2; \cdot), \ell_b (c_2; \cdot); [\lambda_1, \lambda_2])
&= 
\mas (\ell_a (c_2; \lambda_1), \ell_b (c_2; \cdot); [\lambda_1, \lambda_2]) \\
&\quad + \mas (\ell_a (c_2; \cdot), \ell_b (c_2; \lambda_2); [\lambda_1, \lambda_2]).
\end{aligned}
\end{equation}
Recalling that $\mathcal{N} ([\lambda_1, \lambda_2))$ denotes 
the number of eigenvalues that $\mathcal{T} (\cdot)$
has on the interval $[\lambda_1, \lambda_2)$, we can write
\begin{equation} \label{nlam1lam2}
\begin{aligned}
\mathcal{N}([\lambda_1, \lambda_2)) &= - \mas (\ell_a (c_2; \cdot), \ell_b (c_2; \cdot); [\lambda_1, \lambda_2]) \\
&= - \mas (\ell_a (c_2; \lambda_1), \ell_b (c_2; \cdot); [\lambda_1, \lambda_2])
- \mas (\ell_a (c_2; \cdot), \ell_b (c_2; \lambda_2); [\lambda_1, \lambda_2]) \\
&= \mathcal{N}_{a, c_2} ([\lambda_1, \lambda_2)) - \mas (\ell_a (c_2; \lambda_1), \ell_b (c_2; \cdot); [\lambda_1, \lambda_2]).   
\end{aligned}
\end{equation}

{\it Left shelf.} 
Our analysis so far leaves only the left shelf to consider, and 
we observe that it can be expressed as 
\begin{equation*}
- \mas (\ell_a (\cdot; \lambda_1), \ell_b (\cdot; \lambda_2); [c_1, c_2]).
\end{equation*} 
Using path additivity and homotopy invariance, we can sum the Maslov
indices on each shelf of the Maslov Box to arrive at the relation 
\begin{equation} \label{box-sum-equation}
\mathcal{N}_{a, c_2} ([\lambda_1, \lambda_2)) 
= \mas (\ell_a (\cdot; \lambda_1), \ell_b (\cdot; \lambda_2); [c_1, c_2])
- \mas (\ell_a (c_1;\cdot), \ell_b (c_1; \lambda_2); [\lambda_1, \lambda_2]).
\end{equation}
Using (\ref{nlam1lam2}) and (\ref{box-sum-equation}), we can now write 
\begin{equation} \label{full-count-equation}
\begin{aligned}
\mathcal{N} ([\lambda_1, \lambda_2)) &=
\mathcal{N}_{a, c_2} ([\lambda_1, \lambda_2)) - \mas (\ell_a (c_2; \lambda_1), \ell_b (c_2; \cdot); [\lambda_1, \lambda_2]) \\
&= 
\mas (\ell_a (\cdot; \lambda_1), \ell_b (\cdot; \lambda_2); [c_1, c_2])
- \mas (\ell_a (c_1; \cdot), \ell_b (c_1; \lambda_2); [\lambda_1, \lambda_2]) \\
& - \mas (\ell_a (c_2; \lambda_1), \ell_b (c_2; \cdot); [\lambda_1, \lambda_2]).
\end{aligned}
\end{equation}

Recalling the monotonicity relation (\ref{bottom-shelf-inequality}), and noting likewise
the inequality 
\begin{equation*}
\mas (\ell_a (c_2; \lambda_1), \ell_b (c_2; \cdot); [\lambda_1, \lambda_2]) \le 0,
\end{equation*}
we can conclude the inequality 
\begin{equation} \label{count-inequality} 
\mathcal{N} ([\lambda_1, \lambda_2)) 
\ge \mas (\ell_a (\cdot; \lambda_1), \ell_b (\cdot; \lambda_2); [c_1, c_2]).
\end{equation}
The right-hand side of (\ref{count-inequality}) is computed over the 
compact interval $[c_1, c_2]$ on which (\ref{linear-hammy}) can be viewed as 
a regular system, as analyzed in \cite{HS2}. In \cite{HS2}, the authors
show that the direction of crossing points for such systems are all positive as 
$x$ increases from $c_1$ to $c_2$. (See the statement and proof of Theorem 1.1
in \cite{HS2}.) It follows that as $c_1 \to a^+$ and $c_2 \to b^-$ the values 
$\mas (\ell_{\alpha} (\cdot; \lambda_1), \ell_b (\cdot; \lambda_2); [c_1, c_2])$
are monotonically non-decreasing, and since $\mathcal{N}([\lambda_1, \lambda_2))$
is finite, we can conclude that the limit
\begin{equation*}
    \lim_{\genfrac{}{}{0pt}{}{c_1 \to a^+}{c_2 \to b^-}} 
    \mas (\ell_{\alpha} (\cdot; \lambda_1), \ell_b (\cdot; \lambda_2); [c_1, c_2]),
\end{equation*}
must exist, and in fact that it must be the case that this limit is obtained 
for all $c_1$ sufficiently close to $a$ (sufficiently negative if $a = - \infty$)
and all $c_2$ sufficiently close to $b$ (sufficiently large if $b = + \infty$). 
As noted in Remark \ref{singular-remark}, we denote this limit by
$\mas (\ell_{\alpha} (\cdot; \lambda_1), \ell_b (\cdot; \lambda_2); (a, b))$.
In this way, the first assertion of Theorem \ref{singular-theorem}
is obtained by taking a limit on both sides of (\ref{count-inequality})
as $c_1 \to a^+$ and $c_2 \to b^-$.

For the second claim in Theorem \ref{singular-theorem}, we observe that 
under the assumption that there exists a value $c_a \in (a, b)$ so that for 
all $c \in (a, c_a)$
\begin{equation*}
    \ell_a (c; \lambda) \cap \ell_b (c; \lambda_2)
    = \{0\}, \quad \forall \,\, \lambda \in [\lambda_1, \lambda_2), 
\end{equation*}
we can take $c_1$ sufficiently close to $a$ (sufficiently negative
if $a = -\infty$) so that 
\begin{equation*}
    \mas (\ell_a (c_1; \cdot), \ell_b (c_1; \lambda_2); [\lambda_1, \lambda_2]) = 0,
\end{equation*}
and likewise under the assumption that there exists a value 
$c_b \in (a, b)$ so that for all $c \in (c_b, b)$
\begin{equation*}
    \ell_a (c; \lambda_1) \cap \ell_b (c; \lambda)
    = \{0\}, \quad \forall \,\, \lambda \in [\lambda_1, \lambda_2), 
\end{equation*}
we can take $c_2$ sufficiently close to $b$ (sufficiently large 
if $b = + \infty$) so that 
\begin{equation*}
    \mas (\ell_a (c_2; \lambda_1), \ell_b (c_2; \cdot); [\lambda_1, \lambda_2]) = 0.
\end{equation*}
If follows then from (\ref{full-count-equation}) that 
\begin{equation*}
    \mathcal{N} ([\lambda_1, \lambda_2))
    = \mas (\ell_a (\cdot; \lambda_1), \ell_b (\cdot; \lambda_2); [c_1, c_2])
\end{equation*}
for all such $c_1$ and $c_2$. The second claim of Theorem \ref{singular-theorem}
now follows trivially upon taking $c_1 \to a^+$ and $c_2 \to b^-$.
\hfill $\square$

\section{Applications} \label{applications-section}

In this section, we will discuss two two specific 
applications of our main results, 
though we first need to make one further observation 
associated with Niessen's approach. We recall that 
the key element in Niessen's approach is an emphasis 
on the matrix 
\begin{equation*}
    \mathcal{A} (x; \mu, \lambda)
    = \frac{1}{2 {\rm Im}\,\mu} \Phi (x; \mu, \lambda)^* (J/i) \Phi(x; \mu, \lambda),
\end{equation*}
where $\Phi (x; \mu, \lambda)$ denotes a fundamental matrix for 
(\ref{linear-hammy}), and we clearly require 
${\rm Im }\,\mu \ne 0$. 
We saw in Section \ref{operator-section} that 
if $\{\nu_j (x; \mu, \lambda)\}_{j=1}^{2n}$ denote the eigenvalues
of $\mathcal{A} (x; \mu, \lambda)$, then the number of solutions of 
(\ref{linear-hammy}) that lie left in $(a, b)$ is precisely 
the number of these eigenvalues with a finite limit as 
$x$ approaches $a$, while the number of solutions of 
(\ref{linear-hammy}) that lie right in $(a, b)$ is precisely 
the number of these eigenvalues with a finite limit as 
$x$ approaches $b$. Under Assumption {\bf (D)}, these numbers 
are constant in $\mu$ and $\lambda$ on the set 
$(\mathbb{C} \backslash \mathbb{R}) \times I$,
and so we can categorize the limit-case 
(i.e., limit-point, limit-circle, or limit-$m$) of 
(\ref{linear-hammy}) at $x = a$ (resp. $x=b$) by fixing any 
$(\mu, \lambda) \in (\mathbb{C} \backslash \mathbb{R}) \times I$ 
and computing the values 
$\{\nu_j (x; \mu, \lambda)\}_{j=1}^{2n}$ as $x$ tends to $a$
(resp. as $x$ tends to $b$). 
(This is precisely what we will do
in our examples below.) Furthermore, we have additionally seen in Section \ref{operator-section} 
that for each $\nu_j (x; \mu, \lambda)$ (with or without a finite 
limit), we can associate a (sub)sequence of eigenvectors 
$\{v_j (x_k; \mu, \lambda)\}_{k=1}^{\infty}$
that converges, as $x_k \to a^+$, to some $v_j^a (\mu, \lambda)$ that lies on the 
unit circle in $\mathbb{C}^{2n}$, and similarly for a 
sequence $x_k \to b^-$. If $\nu_j (x; \mu, \lambda)$
has a finite limit as $x \to a^+$, then $\Phi (x; \mu, \lambda) v_j^a (\mu, \lambda)$
will lie left in $(a, b)$, while if $\nu_j (x; \mu, \lambda)$
has a finite limit as $x \to b^-$, then $\Phi (x; \mu, \lambda) v_j^b (\mu, \lambda)$
will lie right in $(a, b)$.

In practice, in order to construct the frames $\mathbf{X}_a (x; \lambda_i)$, 
$\mathbf{X}_b (x; \lambda_i)$, $i = 1, 2$, 
we would like to fix either $\lambda = \lambda_1$ or $\lambda = \lambda_2$
and extend these ideas to values $\mu = 0$. Following 
\cite{HS2022}, we do this by replacing 
$\mathcal{A} (x; \mu, \lambda)$ with 
\begin{equation} \label{mathcal-B}
    \mathcal{B} (x; \lambda) := \Phi(x; 0, \lambda)^* J \partial_{\lambda} \Phi (x; 0, \lambda).
\end{equation}
At this point, $\mu$ is fixed at zero, so we are working directly with 
(\ref{hammy}) rather than (\ref{linear-hammy}), but for notational 
consistency we will continue to view $\Phi$ as a function of 
$x$, $\mu$, and $\lambda$. 

If we differentiate (\ref{mathcal-B}) with respect to $x$, we find that 
\begin{equation} \label{mathcal-B-prime}
    \mathcal{B}' (x; \lambda) 
    = \Phi(x; 0, \lambda)^* \mathbb{B}_{\lambda} (x; \lambda) \Phi (x; 0, \lambda), 
\end{equation}
and upon integrating we find that we can alternatively express 
$\mathcal{B} (x; \lambda)$ as 
\begin{equation} \label{mathcal-B-integrated}
    \mathcal{B} (x; \lambda) 
    = \int_c^x \Phi (\xi; 0, \lambda)^* \mathbb{B}_{\lambda} (\xi; \lambda) \Phi (\xi; 0, \lambda) d\xi,
\end{equation}
where we've observed that since $\Phi (c; 0, \lambda) = I_{2n}$,
we have $\mathcal{B} (c; \lambda) = 0$.
Recalling that $\mathbb{B}_{\lambda} (x; \lambda)$ 
is self-adjoint for a.e. $x \in (a, b)$,
we see from this relation that $\mathcal{B} (x; \lambda)$ is 
self-adjoint for all $x \in (a, b)$. 
Consequently, the eigenvalues of $\mathcal{B} (x; \lambda)$
must be real-valued, and we denote these values $\{\eta_j (x; \lambda)\}_{j=1}^{2n}$.
Since $\mathcal{B} (c; \lambda) = 0$, we can conclude that 
$\eta_j (c; \lambda) = 0$ for all $j \in \{1, 2, \dots, 2n\}$,
and all $\lambda \in I$. In addition, according to 
(\ref{mathcal-B-prime}), along with Condition {\bf (B)}, 
for each fixed $\lambda \in \mathbb{R}$,
the eigenvalues $\{\eta_j (x; \lambda)\}_{j=1}^{2n}$ will 
be non-decreasing as $x$ increases. It follows that as $x \to b^-$, each 
eigenvalue $\eta_j (x; \lambda)$ will either approach $+ \infty$
or a finite limit. In the latter case, we set 
\begin{equation*}
    \eta_j^b (\lambda) := \lim_{x \to b^-} \eta_j (x; \lambda).
\end{equation*}
Likewise, as $x \to a^+$, each eigenvalue $\eta_j (x; \lambda)$
will either approach $- \infty$ or a finite limit. In the latter
case, we set 
\begin{equation*}
    \eta_j^a (\lambda) := \lim_{x \to a^+} \eta_j (x; \lambda).
\end{equation*}

The following lemma can be adapted directly from Lemma 5.1 
in \cite{HS2022}. 

\begin{lemma} \label{subspace-dimensions-lemma-real}
Let Assumptions {\bf (A)} and {\bf (B)} hold,
and let $\lambda \in [\lambda_1, \lambda_2]$ be fixed.
Then the dimension $m_a (0, \lambda)$ of the subspace of solutions to 
(\ref{hammy}) that lie left in $(a, b)$ is precisely 
the number of eigenvalues $\eta_j (x; \lambda) \in \sigma (\mathcal{B} (x; \lambda))$
that approach a finite limit as $x \to a^+$. Likewise, 
the dimension $m_b (0, \lambda)$ of the subspace of solutions to 
(\ref{hammy}) that lie right in $(a, b)$ is precisely 
the number of eigenvalues $\eta_j (x; \lambda) \in \sigma (\mathcal{B} (x; \lambda))$
that approach a finite limit as $x \to b^-$.
\end{lemma}

\begin{remark}
We emphasize that as opposed to the case $\mu \in \mathbb{C} \backslash \mathbb{R}$,
we cannot conclude from these considerations that $m_a (0, \lambda), m_b (0, \lambda) \ge n$.
Rather, in this case we conclude these inequalities for all $\lambda \in [\lambda_1, \lambda_2]$ 
from Lemma \ref{lemma2-9prime} (under the assumptions stated for that lemma). 
\end{remark}

If, for each $x \in (a, b)$, we let $\{w_j (x; \lambda)\}_{j=1}^{2n}$ denote 
an orthonormal collection of eigenvectors associated with 
the eigenvalues $\{\eta_j (x; \lambda)\}_{j=1}^{2n}$, then as in the 
proof of Lemma 2.1 of \cite{HS2022}, we can find 
(for each $j \in \{1, 2, \dots, 2n\}$) a sequence $\{w_j (x_k; \lambda)\}_{k=1}^{\infty}$
that converges, as $x_k \to a^+$, to some $w_j^a (\lambda)$ on the unit circle in 
$\mathbb{C}^{2n}$, and likewise we can find 
a sequence $\{w_j (x_k; \lambda)\}_{k=1}^{\infty}$
that converges, as $x_k \to b^-$, to some $w_j^b (\lambda)$ on the unit circle in 
$\mathbb{C}^{2n}$. Moreover, if $\eta_j (x; \lambda)$
has a finite limit as $x \to a^+$, then $\Phi (x; 0, \lambda) w_j^a (\lambda)$
will lie left in $(a, b)$, while if $\eta_j (x; \lambda)$
has a finite limit as $x \to b^-$, then $\Phi (x; 0, \lambda) w_j^b (\lambda)$
will lie right in $(a, b)$.

These considerations provide a practical method for constructing the frames
$\mathbf{X}_a (x; \lambda)$ and $\mathbf{X}_b (x; \lambda)$ that 
we will need in order to implement Theorems \ref{regular-singular-theorem}
and \ref{singular-theorem}. Most directly, if (\ref{linear-hammy})
is limit-point at $x = a$ (resp., $x=b$), then the procedure 
described in the previous paragraph will provide precisely $n$ 
linearly independent solutions to (\ref{hammy}) that 
lie left in $(a, b)$ (resp., right in $(a, b)$), and 
these can be taken to comprise the columns of $\mathbf{X}_a (x; \lambda)$
(resp., $\mathbf{X}_b (x; \lambda)$). 

More generally, Lemma \ref{subspace-dimensions-lemma} can be used 
to construct left and right lying solutions of (\ref{linear-hammy})
for some $(\lambda_0, \mu_0) \in I \times (\mathbb{C} \backslash \mathbb{R})$,
and these can then be used to specify the Niessen elements described 
in the lead-in to Lemma \ref{krall-niessen-lemma}. I.e., the matrices
$U^a (x; \mu_0, \lambda_0)$ and $U^b (x; \mu_0, \lambda_0)$ discussed in 
Section \ref{operator-section} can be constructed in this way. 
Working, for example, with the solutions constructed above 
for $\mu = 0$ that lie left in $(a, b)$, we 
can identify $n$ linearly independent solutions to (\ref{hammy})
$\{u_j^a (x; \lambda)\}_{j=1}^n$
that satisfy 
\begin{equation*}
    \lim_{x \to a^+} U^a (x; \mu_0, \lambda_0)^* J u_j^a (x; \lambda) = 0. 
\end{equation*}
This collection $\{u_j^a (x; \lambda)\}_{j=1}^n$ can be taken to comprise 
the columns of $\mathbf{X}_a (x; \lambda)$, and we can proceed 
similarly for $\mathbf{X}_b (x; \lambda)$. 

Before turning to two specific applications, we verify that our full assumptions
{\bf (A)} through {\bf (F)} hold for three important general cases.

\subsection{Verification of Assumptions {\bf (A)} through {\bf (F)}}
\label{verifications-section}

In this section, we verify that our full assumptions {\bf (A)} through 
{\bf (F)} hold for systems (\ref{hammy}) that are linear in $\lambda$
(i.e., with $\mathbb{B} (x; \lambda)$ as in (\ref{linear-in-lambda})),
and also for systems with two common types of nonlinear dependence
on $\lambda$. While these constitute important cases in their own 
right, our primary goal is to establish a straightforward framework
for checking these assumptions in additional cases that come up 
in applications. 

\subsubsection{The Linear Case}
\label{linear-section}

In order to be clear that this analysis is a genuine extension of the case 
in which $\mathbb{B} (x; \lambda)$ is linear in $\lambda$, we verify that our 
assumptions hold in the case (\ref{linear-in-lambda})
under the following assumptions from \cite{HS2022}, which include the cases
considered in \cite{GZ2017}:

\medskip
$\mathbf{ (\tilde{A})}$ We take $B_0, B_1 \in L^1_{\loc} ((a, b), \mathbb{C}^{2n \times 2n})$,
and additionally assume that $B_0 (x)$ and $B_1 (x)$ are both self-adjoint
for a.e. $x \in (a, b)$, with also $B_1 (x)$ non-negative for 
a.e. $x \in (a, b)$.

\medskip
$\mathbf{ (\tilde{B})}$ If $y(\cdot; \lambda) \in \AC_{\loc} ((a, b), \mathbb{C}^{2n})$ 
is any non-trivial solution of (\ref{hammy}) with $\mathbb{B} (x; \lambda)$
as in (\ref{linear-in-lambda}), then 
\begin{equation*}
\int_c^d (B_1 (x) y(x; \lambda), y(x; \lambda)) dx > 0,
\end{equation*}
for all $[c, d] \subset (a, b)$, $c < d$. 

\medskip
$\mathbf{ (\tilde{C})}$ If $\tilde{m}_a (\lambda)$ denotes the dimension 
of the subspace of solutions to (\ref{hammy})-(\ref{linear-in-lambda})
(i.e., the system (\ref{hammy}) with $\mathbb{B} (x; \lambda)$ as 
in (\ref{linear-in-lambda}))
that lie left in $(a, b)$ and $\tilde{m}_b (\lambda)$ denotes the dimension
of the subspace of solutions to (\ref{hammy})-(\ref{linear-in-lambda})
that lie right in $(a, b)$, then the values $\tilde{m}_a (\lambda)$ and $\tilde{m}_b (\lambda)$ are 
both constant for all $\lambda \in \mathbb{C} \backslash \mathbb{R}$.

\begin{lemma} \label{linear-in-lambda-lemma}
Under Assumptions $\mathbf{ (\tilde{A})}$, $\mathbf{ (\tilde{B})}$, and $\mathbf{ (\tilde{C})}$,
the system (\ref{hammy})-(\ref{linear-in-lambda}) satisfies Assumptions 
{\bf (A)} through {\bf (E)} for $I = \mathbb{R}$, and Assumption {\bf (F)} holds for all 
$\lambda_1, \lambda_2 \in \mathbb{R}$, $\lambda_1 < \lambda_2$.
\end{lemma}

\begin{proof}
First, noting that 
$\mathbb{B}_{\lambda} (x; \lambda) = B_1 (x)$, we immediately 
see that Assumptions {\bf (A)} hold with $b_0 (x) = |B_0 (x)|$ 
and $b_1 (x) = |B_1 (x)|$, and moreover that Assumption {\bf (B)}
is just a restatement of Assumption $\mathbf{ (\tilde{B})}$ in 
this case. 

For Assumption {\bf (C)}, in the case of (\ref{linear-in-lambda}),
the norm on $L^2_{\mathbb{B}_{\lambda}} ((a, b), \mathbb{C}^{2n})$ becomes 
\begin{equation*}
    \|f\|_{\mathbb{B}_{\lambda}} 
    = \Big(\int_a^b (B_1 (x) f(x), f(x)) dx\Big)^{1/2},  
\end{equation*}
from which it's clear that the space $L^2_{\mathbb{B}_{\lambda}} ((a, b), \mathbb{C}^{2n})$
is independent of $\lambda$. Likewise, the maximal domain constructed in 
Definition \ref{maximal-operator} is based on the equation 
\begin{equation*}
    Jy' - B_0 (x) y = B_1 (x) f,
\end{equation*}
and so is clearly identical for all $\lambda \in \mathbb{R}$ (since 
$\lambda$ doesn't appear).

For Assumption {\bf (D)}, we note that in this case (\ref{linear-hammy}) 
can be expressed as 
\begin{equation*} 
    Jy' = B_0 (x) y + (\lambda + \mu) B_1 (x) y. 
\end{equation*}
By Assumption $\mathbf{ (\tilde{C})}$, 
$m_a (\mu, \lambda) = \tilde{m}_a (\lambda + \mu)$ is constant for all 
$\lambda + \mu \in \mathbb{C} \backslash \mathbb{R}$, and this condition
implies that $m_a (\mu, \lambda)$ is constant for all 
$(\lambda, \mu) \in \mathbb{R} \times (\mathbb{C} \backslash \mathbb{R})$, 
which is the first half of Assumption {\bf (D)}. The condition on 
$m_b (\mu; \lambda)$ follows similarly. 

Turning now to Assumption {\bf (E)}, we observe that in the case of 
(\ref{linear-in-lambda}), we have simply 
\begin{equation*}
    \mathcal{E} (x; \lambda, \lambda_*)
    = (\lambda - \lambda_*) I,
\end{equation*}
for which Items (i) through (iv) of Assumption {\bf (E)} are immediately 
seen to hold. 

Finally, for Assumption {\bf (F)}, in the case of (\ref{linear-in-lambda}) 
we can write 
\begin{equation*}
    \mathbb{B} (x; \lambda_2) - \mathbb{B} (x; \lambda_1)
    = (\lambda_2 - \lambda_1) B_1 (x),
\end{equation*}
making (\ref{difference-definite}) from Remark \ref{assumption-F-remark} an 
immediate consequence of Assumption $\mathbf{ (\tilde{B})}$. 
\end{proof}

\subsubsection{Quadratic Schr\"odinger Systems}
\label{quadratic-schrodinger-systems-section}

Next, we identify conditions under which Assumptions {\bf (A)} through 
{\bf (F)} hold for Hamiltonian systems associated with Quadratic  
Schr\"odinger Systems
\begin{equation*}
    -\phi'' + V(x) \phi = (\lambda Q_1 (x) + \lambda^2 Q_2 (x)) \phi,
    \quad \phi(x) \in \mathbb{C}^n,
\end{equation*}
where $V(x)$, $Q_1 (x)$, and $Q_2 (x)$ are taken to be 
$n \times n$ matrices that 
will be characterized below in Assumption {\bf (Q)}. 
Writing $y = \genfrac{(}{)}{0pt}{1}{y_1}{y_2}$ with $y_1 = \phi$ and $y_2 = \phi'$,
we obtain the system $J y' = \mathbb{B} (x; \lambda) y$ with 
\begin{equation} \label{quadratic-sturm-liouville-B}
    \mathbb{B} (x; \lambda)
    = \begin{pmatrix}
    \lambda Q_1 (x) + \lambda^2 Q_2 (x) - V(x) & 0 \\
    0 & I
    \end{pmatrix},
\end{equation}
and 
\begin{equation} \label{quadratic-sturm-liouville-B-derivative}
    \mathbb{B}_{\lambda} (x; \lambda)
    = \begin{pmatrix}
    Q_1 (x) + 2 \lambda Q_2 (x) & 0 \\
    0 & 0
    \end{pmatrix}.
\end{equation}
We make the following assumptions on $V$, $Q_1$, and $Q_2$:

\medskip
{\bf (Q)} We assume $Q_1, Q_2, V \in L^1_{\loc} ((a, b), \mathbb{R}^{n \times n})$, and 
also that  $Q_1, Q_2 \in L^{\infty} ((a, b), \mathbb{R}^{n \times n})$, with 
$Q_1 (x)$, $Q_2 (x)$ and $V(x)$ symmetric for a.e. $x \in (a, b)$.
In addition, we assume there exists a closed interval $I \subset \mathbb{R}$ 
and a constant $\theta > 0$ so that for each $\lambda \in I$ the following 
holds: for a.e. $x \in (a, b)$
\begin{equation*}
    ((Q_1 (x) + 2 \lambda Q_2 (x)) v, v) \ge \theta |v|^2,
    \quad \forall \, v \in \mathbb{R}^n.
\end{equation*}

\medskip

\begin{lemma} \label{quadratic-sturm-liouville-lemma}
Let Assumptions {\bf (Q)} hold for some closed interval $I \subset \mathbb{R}$. 
Then the system (\ref{hammy})-(\ref{quadratic-sturm-liouville-B}) satisfies Assumptions 
{\bf (A)} through {\bf (E)} on $I$, and Assumption {\bf (F)} holds for all 
$\lambda_1, \lambda_2 \in I$, $\lambda_1 < \lambda_2$.
\end{lemma}

\begin{proof} 
Under Assumptions {\bf (Q)}, and for $\mathbb{B} (x; \lambda)$ 
and $\mathbb{B}_{\lambda} (x; \lambda)$ as in (\ref{quadratic-sturm-liouville-B})
and (\ref{quadratic-sturm-liouville-B-derivative}), we can set 
\begin{equation*}
    b_0 (x) := \max_{\lambda \in I} |\mathbb{B} (x; \lambda)|,
    \quad b_1 (x) := \max_{\lambda \in I} |\mathbb{B}_{\lambda} (x; \lambda)|,
\end{equation*}
and it's clear that Assumptions {\bf (A)} hold with these choices of 
$b_0 (x)$ and $b_1 (x)$. 

For {\bf (B)}, we take any non-trivial solution 
$y(\cdot; \lambda) \in \AC_{\loc} ((a, b), \mathbb{C}^{2n})$
of (\ref{hammy})-(\ref{quadratic-sturm-liouville-B}),
and for any fixed $\lambda_* \in I$, compute
(recalling $y = \genfrac{(}{)}{0pt}{1}{y_1}{y_2}$)
\begin{equation*}
\begin{aligned}
\|y\|_{\mathbb{B}_{\lambda_*}}^2
&= \int_c^d (\mathbb{B}_{\lambda} (x; \lambda_*)  y(x; \lambda), y(x; \lambda)) dx \\
&= \int_c^d ((Q_1 (x) + 2 \lambda_* Q_2 (x)) y_1, y_1)_{\mathbb{C}^{n}} dx
\ge \theta \int_c^d |y_1 (x; \lambda)|^2 dx
\end{aligned}
\end{equation*}
for all $[c, d] \subset (a, b)$, $c < d$. If this final integral is 0, we 
must have $y_1 (x; \lambda) = 0$ for all $x \in (c, d)$, and consequently 
(since $y_2 = y_1'$) we must have $y (x; \lambda) = 0$ for all $x \in (c, d)$. 
But this contradicts
the assumption that $y (x; \lambda)$ is non-trivial, so we can conclude 
that this integral is positive, from which {\bf (B)} is immediate. 

Turning to {\bf (C)}, for any measurable $f: (a, b) \to \mathbb{C}^{2n}$, 
expressed as $f = \genfrac{(}{)}{0pt}{1}{f_1}{f_2}$, and any $\lambda \in I$, we can write  
\begin{equation*}
    \|f\|_{\mathbb{B}_{\lambda}}^2 
    = \int_a^b ((Q_1 (x) + 2 \lambda Q_2 (x))f_1, f_1) dx.
\end{equation*}
Since $Q_1, Q_2 \in L^{\infty} ((a, b), \mathbb{C}^{n \times n})$, we 
have the inequality
\begin{equation*}
    \int_a^b ((Q_1 (x) + 2 \lambda Q_2 (x))f_1, f_1) dx 
    \le \max_{\lambda \in I} \| Q_1 (x) + 2 \lambda Q_2 (x) \|_{L^{\infty} ((a, b), \mathbb{C}^{2n})}
    \|f_1\|^2_{L^2 ((a, b), \mathbb{C}^n)},
\end{equation*}
so in particular there exists a constant $C_1$ so that 
\begin{equation} \label{quadratic-C1} 
    \|f\|_{\mathbb{B}_{\lambda}} \le C_1 \|f_1\|_{L^2 ((a, b), \mathbb{C}^n)},
    \quad \forall \, \lambda \in I.
\end{equation}
Likewise we have 
\begin{equation} \label{quadratic-lower-bound}
    \int_a^b ((Q_1 (x) + 2 \lambda Q_2 (x))f_1, f_1) dx 
    \ge \theta \|f_1\|^2_{L^2 ((a, b), \mathbb{C}^n)},
    \quad \forall \, \lambda \in I,
\end{equation}
so that 
\begin{equation} \label{quadratic-theta}
     \|f\|_{\mathbb{B}_{\lambda}}
     \ge \sqrt{\theta} \|f_1\|_{L^2 ((a, b), \mathbb{C}^n)}.
\end{equation}
In total, we can conclude that for all $\lambda \in I$ 
\begin{equation*}
    L^2_{\mathbb{B}_{\lambda}} ((a, b), \mathbb{C}^{2n})
    = \{\textrm{Measurable functions} \, f: (a, b) \to \mathbb{C}^{2n}:
    f_1 \in L^2 ((a, b), \mathbb{C}^n)\},
\end{equation*}
and this space is clearly independent of $\lambda$. For the second 
part of {\bf (C)}, for any fixed $\lambda \in I$, we can characterize
$\mathcal{D}_M (\lambda)$ in this case as the collection of 
functions 
\begin{equation*}
    y \in \AC_{\loc} ((a, b), \mathbb{C}^{2n}) \cap L^2_{\mathbb{B}_{\lambda}} ((a, b), \mathbb{C}^{2n})
\end{equation*}
for which there exists $f \in L^2_{\mathbb{B}_{\lambda}} ((a, b), \mathbb{C}^{2n})$, 
$f = \genfrac{(}{)}{0pt}{1}{f_1}{f_2}$, so 
that (using (\ref{domain-relation2}))
\begin{equation} \label{quadratic-s-item6}
    \begin{aligned}
    -y_2' + (\lambda^2 Q_2 (x) + V(x)) y_1 &= (Q_1 (x) + 2 \lambda Q_2 (x)) f_1 \\
    y_1' &= y_2,
    \end{aligned}
\end{equation}
for a.e. $x \in (a, b)$. Given that $y \in \mathcal{D}_M (\lambda)$, for some 
$\lambda \in I$, we need to verify that for any other 
$\tilde{\lambda} \in I$ we also have 
$y \in \mathcal{D}_M (\tilde{\lambda})$. I.e., we need to show that for any 
$\tilde{\lambda} \in I$ there exists 
$\tilde{f} \in L^2_{\mathbb{B}_{\lambda}} ((a, b), \mathbb{C}^{2n})$, 
$\tilde{f} = \genfrac{(}{)}{0pt}{1}{\tilde{f}_1}{\tilde{f}_2}$, so that 
\begin{equation*} 
    \begin{aligned}
    -y_2' + (\tilde{\lambda}^2 Q_2 (x) + V(x)) y_1 &= (Q_1 (x) + 2 \tilde{\lambda} Q_2 (x)) \tilde{f}_1 \\
    y_1' &= y_2,
    \end{aligned}
\end{equation*}
for a.e. $x \in (a, b)$. To see this, we begin by re-writing the first equation in (\ref{quadratic-s-item6}) as 
\begin{equation*}
    -y_2' + (\tilde{\lambda}^2 Q_2 (x) + V(x)) y_1 = (Q_1 (x) + 2 \lambda Q_2 (x)) f_1 
    + (\tilde{\lambda}^2 - \lambda^2) Q_2 (x) y_1. 
\end{equation*}
In this way, we see that we need to find 
$\tilde{f} \in L^2_{\mathbb{B}_{\lambda}} ((a, b), \mathbb{C}^{2n})$
so that 
\begin{equation*}
   (Q_1 (x) + 2 \tilde{\lambda} Q_2 (x)) \tilde{f}_1
   = (Q_1 (x) + 2 \lambda Q_2 (x)) f_1 
    + (\tilde{\lambda}^2 - \lambda^2) Q_2 (x) y_1. 
\end{equation*}
Under Assumption {\bf (Q)}, we can write 
\begin{equation*}
    \tilde{f}_1
   = (Q_1 (x) + 2 \tilde{\lambda} Q_2 (x))^{-1} \Big((Q_1 (x) + 2 \lambda Q_2 (x)) f_1 
    + (\tilde{\lambda}^2 - \lambda^2) Q_2 (x) y_1\Big),
\end{equation*}
and it follows from the assumed $L^{\infty} ((a, b), \mathbb{C}^{2n})$ bounds on 
$Q_1$, $Q_2$, and $(Q_1 + 2 \tilde{\lambda} Q_2)^{-1}$ that 
$\tilde{f}_1 \in L^2 ((a, b), \mathbb{C}^n)$, so that 
$\tilde{f} \in L^2_{\mathbb{B}_{\lambda}} ((a, b), \mathbb{C}^{2n})$.

Our verification of Assumption {\bf (D)} will make use of ideas 
developed during our verification of Assumption {\bf (E)}, so we turn 
next to this latter case, for which we begin by 
identifying the function $\mathcal{E} (x; \lambda, \lambda_*)$
so that 
\begin{equation*}
    \mathbb{B} (x; \lambda) - \mathbb{B} (x; \lambda_*)
    = \mathbb{B}_{\lambda} (x; \lambda_*) \mathcal{E} (x; \lambda, \lambda_*).
\end{equation*}
In this case, the relation can be expressed as 
\begin{equation*}
    (\lambda - \lambda_*)
    \begin{pmatrix}
    Q_1 (x) + (\lambda + \lambda_*) Q_2 (x) & 0 \\
    0 & 0
    \end{pmatrix}
    = 
    \begin{pmatrix}
    Q_1 (x) + 2 \lambda_* Q_2 (x) & 0 \\
    0 & 0
    \end{pmatrix}
    \begin{pmatrix}
    \mathcal{E}_{11} & \mathcal{E}_{12} \\
    \mathcal{E}_{21} & \mathcal{E}_{22}
    \end{pmatrix},
\end{equation*}
where we're using the block notation
\begin{equation*}
\mathcal{E} (x; \lambda, \lambda_*)
=
 \begin{pmatrix}
    \mathcal{E}_{11} (x; \lambda, \lambda_*) & \mathcal{E}_{12} (x; \lambda, \lambda_*) \\
    \mathcal{E}_{21} (x; \lambda, \lambda_*) & \mathcal{E}_{22} (x; \lambda, \lambda_*)
    \end{pmatrix}.
\end{equation*}
We may as well take $ \mathcal{E}_{12} =  \mathcal{E}_{21} = \mathcal{E}_{22} = 0$,
requiring only 
\begin{equation*}
    (\lambda - \lambda_*)  (Q_1 (x) + (\lambda + \lambda_*) Q_2 (x))
    = (Q_1 (x) + 2 \lambda_* Q_2 (x)) \mathcal{E}_{11}.
\end{equation*}
Since $Q_1 (x) + 2 \lambda_* Q_2 (x)$ is positive definite for 
a.e. $x \in (a, b)$ we can write 
\begin{equation} \label{quadratic-E11}
    \mathcal{E}_{11} (x; \lambda, \lambda_*)
    = (\lambda - \lambda_*) (Q_1 (x) + 2 \lambda_* Q_2 (x))^{-1} 
    (Q_1 (x) + (\lambda + \lambda_*) Q_2 (x)),
\end{equation}
and likewise 
\begin{equation} \label{quadratic-lambda-derivative}
    \partial_{\lambda} \mathcal{E}_{11} (x; \lambda, \lambda_*)
    = (Q_1 (x) + 2 \lambda_* Q_2 (x))^{-1} 
    (Q_1 (x) + 2\lambda Q_2 (x)).
\end{equation}

For Assumption {\bf (E)}(i), we need to determine the behavior of 
$\mathcal{E} (\cdot; \lambda, \lambda_*)$ as a map from 
$L^2_{\mathbb{B}_{\lambda}} ((a, b), \mathbb{C}^n)$ to 
itself. Here, for any $f \in L^2_{\mathbb{B}_{\lambda}} ((a, b), \mathbb{C}^n)$,
expressed as $f = \genfrac{(}{)}{0pt}{1}{f_1}{f_2}$, we can write 
\begin{equation*}
    \begin{aligned}
    \|\mathcal{E} (\cdot; \lambda, \lambda_*) f\|_{\mathbb{B}_{\lambda_*}}^2
    &= \int_a^b (\mathcal{E} (x; \lambda, \lambda_*) f(x))^* 
    \mathbb{B}_{\lambda} (x; \lambda_*) \mathcal{E} (x; \lambda, \lambda_*) f(x) dx \\
    &= \int_a^b (\mathcal{E}_{11} (x; \lambda, \lambda_*) f_1(x))^* 
    (Q_1 (x) + 2 \lambda_* Q_2 (x)) \mathcal{E}_{11} (x; \lambda, \lambda_*) f_1 (x) dx \\
    &\le \|\mathcal{E}_{11} (\cdot; \lambda, \lambda_*)\|_{L^{\infty}}^2 
    \| Q_1 (\cdot) + 2 \lambda_* Q_2 (\cdot) \|_{L^{\infty}} \|f_1\|_{L^2}^2,
    \end{aligned}
\end{equation*}
where each norm in this last expression is unweighted and computed on $(a, b)$. 
Under our assumptions on $Q_1 (x)$ and $Q_2 (x)$, the multiplier of 
$\|f_1\|_{L^2}^2$ is bounded, and in light of (\ref{quadratic-theta}) we 
can conclude that there exists a constant $C_2$ so that 
\begin{equation} \label{quadratic-E-bound}
\|\mathcal{E} (\cdot; \lambda, \lambda_*)\| 
\le C_2 \|\mathcal{E}_{11} (\cdot; \lambda, \lambda_*)\|_{L^{\infty}},     
\end{equation}
from which it's clear that 
$\mathcal{E} (\cdot; \lambda, \lambda_*) 
\in \mathbb{B} (L^2_{\mathbb{B}_{\lambda}} ((a, b), \mathbb{C}^{2n}))$. 

For Assumption {\bf (E)}(ii), we see from (\ref{quadratic-E11}) that 
\begin{equation*}
    \|\mathcal{E}_{11} (\cdot; \lambda, \lambda_*)\|_{L^{\infty}}
    \le |\lambda - \lambda_*| \| (Q_1 (\cdot) + 2 \lambda_* Q_2 (\cdot))^{-1} \|_{L^{\infty}}
    \| Q_1 (\cdot) + (\lambda + \lambda_*) Q_2 (\cdot) \|_{L^{\infty}},
\end{equation*}
from which it's clear from (\ref{quadratic-E-bound})
that $\|\mathcal{E} (\cdot; \lambda, \lambda_*)\| = \mathbf{O} (|\lambda - \lambda_*|)$
for $\lambda$ near $\lambda_*$, which is stronger than the claim.  

For Assumption {\bf (E)}(iii), we see from 
(\ref{quadratic-lambda-derivative}) that for a.e. $x \in (a, b)$, 
$\mathcal{E}_{\lambda} (x; \lambda, \lambda_*)$ is continuous 
in $\lambda$. In order to verify that 
$\mathcal{E} (\cdot; \lambda, \lambda_*)$ is differentiable
in $\lambda$ as a map from $I$ to 
$\mathcal{B} (L^2_{\mathbb{B}_{\lambda}} ((a, b), \mathbb{C}^{2n}))$, 
we let $\epsilon (h; \lambda) 
\in \mathcal{B} (L^2_{\mathbb{B}_{\lambda}} ((a, b), \mathbb{C}^{2n}))$
be the operator so that 
\begin{equation*}
\mathcal{E} (\cdot; \lambda + h, \lambda_*)
= \mathcal{E} (\cdot; \lambda, \lambda_*)
+ \mathcal{E}_{\lambda} (\cdot; \lambda, \lambda_*) h
+ \epsilon (h; \lambda) h,
\end{equation*}
and our goal is to show that 
\begin{equation} \label{quadratic-limit}
    \lim_{h \to 0} \| \epsilon (h; \lambda) \| = 0.
\end{equation}
Using (\ref{quadratic-E11}) and (\ref{quadratic-lambda-derivative})
we find that 
\begin{equation*}
\epsilon (h; \lambda)
= h \begin{pmatrix}
(Q_1 (\cdot) + 2 \lambda_* Q_2 (\cdot))^{-1} Q_2 (\cdot) & 0 \\
0 & 0
\end{pmatrix},
\end{equation*}
from which (\ref{quadratic-limit}) follows from our previous
calculations. As expected, the derivative of 
$\mathcal{E} (\cdot; \lambda, \lambda_*)$ is 
precisely the usual partial derivative $\frac{\partial \mathcal{E}}{\partial {\lambda}} (\cdot; \lambda, \lambda_*)$,
and we need to verify that this latter operator is 
continuous on $I_{\lambda_*, r}$. For this, we fix any
$\lambda_0 \in I$, and observe that
for any other $\lambda \in I$, we can write 
\begin{equation*}
    \partial_{\lambda} \mathcal{E}_{11} (x; \lambda, \lambda_*)
    - \partial_{\lambda} \mathcal{E}_{11} (x; \lambda_0, \lambda_*)
    = 2 (\lambda - \lambda_0) (Q_1 (x) + 2 \lambda_* Q_2 (x))^{-1} Q_2 (x). 
\end{equation*}
Proceeding similarly as in the calculations leading to 
(\ref{quadratic-E-bound}), we find that there exists a constant 
$C_3$ so that 
\begin{equation*}
    \|  \partial_{\lambda} \mathcal{E}_{11} (x; \lambda, \lambda_*)
    - \partial_{\lambda} \mathcal{E}_{11} (x; \lambda_0, \lambda_*) \|
    \le C_3 |\lambda - \lambda_0|,
\end{equation*}
which is more than we require.  

Turning to Assumption {\bf (E)}(iv), given any 
$f, g \in L^2_{\mathbb{B}_{\lambda}} ((a, b), \mathbb{C}^{2n})$,
expressed as $f = \genfrac{(}{)}{0pt}{1}{f_1}{f_2}$ 
and $g = \genfrac{(}{)}{0pt}{1}{g_1}{g_2}$,
we can compute directly to see that 
\begin{equation*}
    \begin{aligned}
    f(x)^* \mathbb{B}_{\lambda} (x; \lambda) g(x)
    &= f_1 (x)^* (Q_1 (x) + 2 \lambda Q_2 (x)) g_1 (x) \\
    & =  f_1 (x)^* Q_1 (x) g_1 (x)
    + 2 \lambda f_1 (x)^* Q_2 (x) g_1 (x).
    \end{aligned}
\end{equation*}
It follows that we can take the function $h$ specified in Assumption {\bf (E)}(iv)
to be 
\begin{equation*}
    h(x) := | f_1 (x)^* Q_1 (x) g_1 (x)|
    + 2 K |f_1 (x)^* Q_2 (x) g_1 (x)|,
\end{equation*}
where 
\begin{equation*}
    K = \max \{|\lambda|: \lambda \in I\}.
\end{equation*}

This completes the verification of Assumption {\bf (E)}, and 
at this point, we return to establishing Assumption {\bf (D)}, 
which we recall addresses the dimension $m_a (\mu; \lambda)$ of the space of 
solutions to (\ref{linear-hammy}) that lie left in $(a, b)$, and 
the dimension $m_b (\mu; \lambda)$ of the space of 
solutions to (\ref{linear-hammy}) that lie right in $(a, b)$. For 
a fixed value $\lambda_0 \in I$, suppose $\mathbb{B} (x; \lambda_0)$
and $\mathbb{B}_{\lambda} (x; \lambda_0)$ have real-valued entries for 
a.e. $x \in (a, b)$. Then we can conclude from Theorem V.2.2 in \cite{Krall2002}
that $m_a (\mu; \lambda_0)$ is constant for all $\mu \in \mathbb{C} \backslash \mathbb{R}$,
and similarly for $m_b (\mu; \lambda_0)$. For Assumption {\bf (D)}, we need to 
additionally check that $m_a (\mu; \lambda)$ and  $m_b (\mu; \lambda)$
remain constant as $\lambda$ varies as well. For this, we can establish the 
following lemma, which we prove in the appendix.

\begin{lemma} \label{assumption-d-lemma}
Let Assumptions {\bf (A)} through {\bf (C)} hold, along with Assumptions
{\bf (E)}, and suppose that for each $\lambda_0 \in I$ the values
$m_a (\mu, \lambda_0)$ (from Definition \ref{left-right-definition}) 
are constant for all $\mu \in \mathbb{C} \backslash \mathbb{R}$,
and likewise for the values $m_b (\mu, \lambda_0)$. In addition, 
assume that for each $\lambda_0 \in I$
\begin{equation*}
    \|\mathcal{E}_{\lambda} (\cdot; \lambda, \lambda_0) - I\| = \mathbf{o} (1),
    \quad \lambda \to \lambda_0.
\end{equation*}
Then the values $m_a (\mu; \lambda)$ are constant for all 
$(\mu, \lambda) \in (\mathbb{C} \backslash \mathbb{R}) \times I$, and likewise for 
the values $m_b (\mu; \lambda)$. 
\end{lemma}

The key observation in this lemma is that Assumption {\bf (D)}
now follows from the condition 
\begin{equation} \label{quadratic-lemma-claim}
    \|\mathcal{E}_{\lambda} (\cdot; \lambda, \lambda_0) - I\| = \mathbf{o} (1),
    \quad \lambda \to \lambda_0.
\end{equation}
For this, we first use (\ref{quadratic-lambda-derivative}) to write 
\begin{equation*}
\begin{aligned}
\partial_{\lambda} \mathcal{E}_{11} (x; \lambda, \lambda_0) - I
&= (Q_1 (x) + 2 \lambda_0 Q_2 (x))^{-1} 
    \Big{\{} (Q_1 (x) + 2\lambda Q_2 (x)) - (Q_1 (x) + 2 \lambda_0 Q_2 (x)) \Big{\}} \\
&= 2 (\lambda - \lambda_0) (Q_1 (x) + 2 \lambda_0 Q_2 (x))^{-1} Q_2 (x),
\end{aligned}
\end{equation*}
from which it follows similarly as in the calculations leading to 
(\ref{quadratic-E-bound}) that there exists a constant $C_4$ 
so that 
\begin{equation*}
 \| \mathcal{E}_{\lambda} (\cdot; \lambda, \lambda_0) - I \|
 \le C_4 |\lambda - \lambda_0|.
\end{equation*}
The claim (\ref{quadratic-lemma-claim}) is now immediate. 

Last, we verify Assumption {\bf (F)}. For this, we fix any 
$\lambda_1, \lambda_2 \in I$, $\lambda_1 < \lambda_2$, and compute 
\begin{equation*}
    \mathbb{B} (x; \lambda_2) - \mathbb{B} (x; \lambda_1)
    = \begin{pmatrix}
       (\lambda_2 - \lambda_1) (Q_1 (x) + (\lambda_1 + \lambda_2) Q_2 (x)) & 0 \\
       0 & 0
    \end{pmatrix}.
\end{equation*}
Since $(\lambda_1 + \lambda_2)/2 \in [\lambda_1, \lambda_2] \subset I$, we can conclude 
from Assumptions {\bf (Q)} that the matrix 
$Q_1 (x) + (\lambda_1 + \lambda_2) Q_2 (x)$ is positive definite, and 
consequently that $\mathbb{B} (x; \lambda_2) - \mathbb{B} (x; \lambda_1)$
is non-negative. In addition, the second part of Assumption {\bf (F)} 
follows similarly as Assumption {\bf (B)}, using Remark \ref{assumption-F-remark}.
\end{proof}

\subsubsection{Degenerate Sturm-Liouville Systems}
\label{degenerate-sturm-liouville-section}

In this section, we consider systems 
\begin{equation} \label{da-equation}
\mathcal{L} \phi 
= - (P(x) \phi')' + V(x) \phi = \lambda \phi, 
\end{equation} 
with degenerate matrices 
\begin{equation*}
P (x) = 
\begin{pmatrix}
P_{11} (x) & 0 \\
0 & 0
\end{pmatrix}.
\end{equation*}
Here, for some $0 < m < n$, $P_{11} \in AC_{\loc} ((a, b), \mathbb{C}^{m \times m})$ 
is a map into the space of self-adjoint matrices. We assume $P_{11} (x)$ is 
invertible for all $x \in (a, b)$, and additionally that
$V \in C ((a, b), \mathbb{C}^{n \times n})$.
For notational convenience, we will write 
\begin{equation} \label{degenerate-sls-v}
V (x) = 
\begin{pmatrix}
V_{11} (x) & V_{12} (x) \\
V_{12} (x)^* & V_{22} (x)
\end{pmatrix},
\end{equation}
where for each $x \in (a, b)$, $V_{11} (x)$ is a 
self-adjoint $m \times m$ matrix, 
$V_{12} (x)$ is an $m \times (n-m)$
matrix, and $V_{22} (x)$ is a self-adjoint $(n-m) \times (n-m)$
matrix. We will write 
\begin{equation*}
\phi = 
\begin{pmatrix}
\phi_1 \\
\phi_2
\end{pmatrix};
\quad \phi_1 (x; \lambda) \in \mathbb{C}^m; 
\quad \phi_2 (x; \lambda) \in \mathbb{C}^{n-m}, 
\end{equation*}
allowing us to express the system (\ref{da-equation}) as 
\begin{equation} \label{da-system-form}
\begin{aligned}
- (P_{11} (x) \phi_1')' + V_{11} (x) \phi_1 + V_{12} (x) \phi_2 &= \lambda \phi_1 \\
V_{12} (x)^* \phi_1 + V_{22} (x) \phi_2 &= \lambda \phi_2.  
\end{aligned}
\end{equation}

First, it follows from Theorem 2.2 of \cite{ALMS1994} 
that the essential spectrum associated with (\ref{da-system-form}) 
contains (though might not be limited to) the union of the ranges of the 
eigenvalues of $V_{22} (x)$ as $x$ ranges over 
$(a, b)$. More precisely, let $\{\nu_k (x)\}_{k=1}^{n-m}$ denote the 
eigenvalues of $V_{22} (x)$, and let $\mathcal{R}_k$ denote the 
range of $\nu_k : (a, b) \to \mathbb{R}$. Then the 
essential spectrum associated with (\ref{da-system-form})
contains the set 
\begin{equation*}
\mathcal{R} := \bigcup_{k=1}^{n-m} \mathcal{R}_k.
\end{equation*}
For values $\lambda \notin \mathcal{R}$, we can eliminate 
$\phi_2$ from (\ref{da-system-form}) to obtain the system 
\begin{equation*}
- (P_{11} (x) \phi_1')' + V_{11} (x) \phi_1 + V_{12} (x) (\lambda I - V_{22} (x))^{-1} V_{12} (x)^* \phi_1  
= \lambda \phi_1. 
\end{equation*} 
We can express this system as a first-order system in the usual way, writing
$y_1 = \phi_1$ and $y_2 = P_{11} \phi_1'$. In this way, we obtain (\ref{hammy})
with 
\begin{equation} \label{sls-b}
\mathbb{B} (x; \lambda) =
\begin{pmatrix}
\lambda I - \mathbf{V} (x; \lambda) & 0 \\
0 & P_{11} (x)^{-1} 
\end{pmatrix},
\end{equation}
where we've set 
\begin{equation*}
\mathbf{V} (x; \lambda) = V_{11} (x) + V_{12} (x) (\lambda I - V_{22} (x))^{-1} V_{12} (x)^*. 
\end{equation*}
In order to analyze the system (\ref{hammy})-(\ref{sls-b}) using our general 
framework, we need to verify Assumptions {\bf (A)} through {\bf (F)}.  In order to 
do this, we make the following assumptions. 

\medskip
\noindent
{\bf (DSLS)} For (\ref{hammy})-(\ref{sls-b}), assume 
$P_{11} \in AC_{\loc} ((a,b), \mathbb{R}^{m \times m})$, with 
$P_{11} (x)$ self-adjoint for all $x \in (a, b)$, and also 
that there exists a constant $\theta > 0$ so that 
\begin{equation*}
    (P_{11} (x) v, v) \ge \theta |v|^2,
    \quad \forall \, x \in (a, b), \, v \in \mathbb{C}^n. 
\end{equation*}
In addition, for $V$ as in (\ref{degenerate-sls-v}), suppose 
$V \in C((a, b), \mathbb{R}^{n \times n})$, and also that 
there exists a constant $C$ so that 
\begin{equation*}
    |V_{12} (x)|, |V_{22} (x)| \le C,
    \quad \forall \, x \in (a, b).
\end{equation*}
\medskip

\begin{lemma} \label{degenerate-sturm-liouville-lemma}
Let Assumptions {\bf (DSLS)} hold, and fix any interval 
$I = [\lambda_1, \lambda_2]$, $\lambda_1 < \lambda_2$,
for which there exists some $\delta > 0$ so that
for any $x \in (a, b)$ if $\nu (x)$ is any eigenvalue
of $V_{22} (x)$ then either $\nu (x) \le \lambda_1 - \delta$ or 
$\nu (x) \ge \lambda_2 + \delta$. Then the system (\ref{hammy})-(\ref{sls-b})
satisfies Assumptions {\bf (A)} through {\bf (F)} on $I$. 
\end{lemma}

\begin{proof}
For Assumption {\bf (A)}, we observe that it follows from our 
assumptions on $V(x)$ that for each $\lambda \in [\lambda_1, \lambda_2]$
the matrix norm $|\mathbf{V} (x; \lambda)|$ is uniformly bounded for all 
$x \in (a, b)$. It follows from compactness of $[\lambda_1, \lambda_2]$ 
that we can set 
\begin{equation*}
    b_0 (x) := \max_{\lambda \in [\lambda_1, \lambda_2]}
    |\mathbb{B} (x; \lambda)|,
\end{equation*}
and conclude that for any compact interval $[c, d] \subset (a, b)$,
$c < d$, there exists a constant $C_0^{c, d}$ so that $b_0 (x) \le C_0^{c, d}$
for all $x \in (a, b)$. It follows immediately that 
$b_0 \in L^1_{\loc} ((a, b), \mathbb{R})$, and that 
$|\mathbb{B} (x; \lambda)| \le b_0 (x)$ for all 
$x \in (a, b)$, $\lambda \in [\lambda_1, \lambda_2]$. 
Next, computing directly, we can write 
\begin{equation} \label{sls-b-derivative}
\mathbb{B}_{\lambda} (x; \lambda) =
\begin{pmatrix}
I - \mathbf{V}_{\lambda} (x; \lambda) & 0 \\
0 & 0 
\end{pmatrix},
\end{equation}
where 
\begin{equation} \label{v-derivative}
\mathbf{V}_{\lambda} (x; \lambda) = - V_{12} (x) (\lambda I - V_{22} (x))^{-2} V_{12} (x)^*.
\end{equation} 
Similarly as with $\mathbb{B} (x; \lambda)$, we can set 
\begin{equation*}
    b_1 (x) := \max_{\lambda \in [\lambda_1, \lambda_2]}
    |\mathbb{B}_{\lambda} (x; \lambda)|,
\end{equation*}
and conclude that for any compact interval $[c, d] \subset (a, b)$,
$c < d$, there exists a constant $C_1^{c, d}$ so that $b_1 (x) \le C_1^{c, d}$
for all $x \in (a, b)$. It follows immediately that 
$b_1 \in L^1_{\loc} ((a, b), \mathbb{R})$, and that 
$|\mathbb{B}_{\lambda} (x; \lambda)| \le b_1 (x)$ for all 
$x \in (a, b)$, $\lambda \in [\lambda_1, \lambda_2]$.

For Assumption {\bf (B)}, we let 
$y (\cdot; \lambda) \in AC_{\loc} ((a, b), \mathbb{C}^{2m})$
be any non-trivial solution of (\ref{hammy})-(\ref{sls-b}), 
expressed as $y = \genfrac{(}{)}{0pt}{1}{y_1}{y_2}$, and for any 
$[c, d] \subset (a, b)$, $c < d$, compute 
\begin{equation} \label{sls-integral}
    \int_c^d (\mathbb{B}_{\lambda} (x; \lambda) y (x; \lambda), y(x; \lambda)) dx
    = \int_c^d \Big((I - \mathbf{V}_{\lambda} (x; \lambda)) y_1 (x; \lambda), y_1 (x; \lambda)\Big) dx.
\end{equation}
It's clear from (\ref{v-derivative}) that $\mathbf{V}_{\lambda} (x; \lambda)$ is non-positive,
and so $I - \mathbf{V}_{\lambda} (x; \lambda)$ is positive definite for all $x \in [c, d]$ and all 
$\lambda \in [\lambda_1, \lambda_2]$. We can conclude that (\ref{sls-integral})
can only be 0 if $y_1 (x; \lambda) = 0$ for all $x \in [c, d]$, and in this case 
we must also have $y_2 (x; \lambda) = 0$ for all $x \in [c, d]$. But this contradicts
our assumption that $y(x; \lambda)$ is a non-trivial solution, allowing us to conclude
that the left-hand side of (\ref{sls-integral}) must be positive; i.e., that 
Assumption {\bf (B)} must hold. 

For Assumption {\bf (C)}, we first observe that given any measurable function $f$, expressed 
as $f = \genfrac{(}{)}{0pt}{1}{f_1}{f_2}$, we can write 
\begin{equation*}
    \int_a^b (\mathbb{B}_{\lambda} (x; \lambda) f(x), f(x)) dx
    = \int_a^b ((I - \mathbf{V}_{\lambda} (x; \lambda)) f_1(x), f_1(x)) dx.
\end{equation*}
Since $\mathbf{V}_{\lambda} (x; \lambda)$ is non-positive we see that 
\begin{equation*}
    |((I - \mathbf{V}_{\lambda} (x; \lambda)) f_1(x), f_1(x))| \ge |f_1 (x)|^2
\end{equation*}
for a.e. $x \in (a, b)$, and in addition, under Assumptions {\bf (DSLS)}, there exists
a constant $C_1$ so that 
\begin{equation*}
    |((I - \mathbf{V}_{\lambda} (x; \lambda)) f_1(x), f_1(x))| \le C_1 |f_1 (x)|^2.
\end{equation*}
Similarly as in Section \ref{quadratic-schrodinger-systems-section}, we can conclude that  
\begin{equation*}
    L^2_{\mathbb{B}_{\lambda}} ((a, b), \mathbb{C}^{2n})
    = \{\textrm{Measurable functions} \, f: (a, b) \to \mathbb{C}^{2n}:
    f_1 \in L^2 ((a, b), \mathbb{C}^n)\},
\end{equation*}
and this space is independent of $\lambda$.

Turning to the second part of Assumption {\bf (C)}, a function 
\begin{equation*}
    y \in \AC_{\loc} ((a, b), \mathbb{C}^{2m}) \cap L^2_{\mathbb{B}_{\lambda}} ((a, b), \mathbb{C}^{2m})
\end{equation*}
is in $\mathcal{D}_M (\lambda)$ if and only if there exists some 
$f \in L^2_{\mathbb{B}_{\lambda}} ((a, b), \mathbb{C}^{2m})$, expressed as
$f = \genfrac{(}{)}{0pt}{1}{f_1}{f_2}$, so that 
\begin{equation} \label{degenerate-sls-C}
    \begin{pmatrix}
    -y_2' \\ y_1'
    \end{pmatrix}
    -
    \begin{pmatrix}
    (\lambda I - \mathbf{V} (x; \lambda)) y_1 \\ P_{11} (x)^{-1} y_2
    \end{pmatrix}
    = \begin{pmatrix}
    (I - \mathbf{V}_{\lambda} (x; \lambda)) f_1 \\ 0
    \end{pmatrix},
\end{equation}
for a.e. $x \in (a, b)$. 
For $\lambda, \tilde{\lambda} \in [\lambda_1, \lambda_2]$, 
in order to show that $\mathcal{D}_M (\lambda)$ and 
$\mathcal{D}_M (\tilde{\lambda})$ are identical sets, it
is sufficient to show the following: given that we can 
find $f_1 \in L^2 ((a, b), \mathbb{C}^m)$ for which 
(\ref{degenerate-sls-C}) holds, we can find 
$\tilde{f}_1 \in L^2 ((a, b), \mathbb{C}^m)$ for which the 
same relations hold with $\lambda$ replaced by $\tilde{\lambda}$. 
As a starting point, we can rearrange the first equation in 
(\ref{degenerate-sls-C}) as 
\begin{equation*}
    - y_2' - (\tilde{\lambda} I - \mathbf{V} (x; \tilde{\lambda})) y_1
    = - (\tilde{\lambda} I - \mathbf{V} (x; \tilde{\lambda})) y_1 
    + (\lambda I - \mathbf{V} (x; \lambda)) y_1
    + (I - \mathbf{V}_{\lambda} (x; \lambda)) f_1.
\end{equation*}
We see from this that we need to find $\tilde{f}_1 \in L^2 ((a, b), \mathbb{C}^m)$
so that 
\begin{equation*}
    \Big{\{} (\lambda I - \mathbf{V} (x; \lambda)) - (\tilde{\lambda} I - \mathbf{V} (x; \tilde{\lambda})) \Big{\}} y_1
    + (I - \mathbf{V}_{\lambda} (x; \lambda)) f_1 
    = (I - \mathbf{V}_{\lambda} (x; \tilde{\lambda})) \tilde{f}_1.
\end{equation*}
We've seen that under Assumptions {\bf (DSLS)} the matrix 
$I - \mathbf{V}_{\lambda} (x; \tilde{\lambda})$ is invertible 
for all $x \in (a, b)$, allowing us to solve for $\tilde{f}_1$
as 
\begin{equation*}
    \tilde{f}_1 (x)
    = (I - \mathbf{V}_{\lambda} (x; \tilde{\lambda}))^{-1}
    \Big{\{} (\lambda - \tilde{\lambda})y_1 - (\mathbf{V} (x; \lambda) - \mathbf{V} (x; \tilde{\lambda})) y_1  
    + (I - \mathbf{V}_{\lambda} (x; \lambda)) f_1 \Big{\}}. 
\end{equation*}
Here, 
\begin{equation} \label{delta-V}
\mathbf{V} (x; \lambda) - \mathbf{V} (x; \tilde{\lambda})
= - (\lambda - \tilde{\lambda}) V_{12} (x) (\lambda I - V_{22} (x))^{-1}
(\tilde{\lambda} I - V_{22} (x))^{-1} V_{12} (x)^*,
\end{equation}
from which it follows from Assumptions {\bf (DSLS)} that there 
exists a constant $C_2$ so that
\begin{equation*}
\| \mathbf{V} (\cdot; \lambda) - \mathbf{V} (\cdot; \tilde{\lambda}) \|_{L^{\infty} ((a, b), \mathbb{C}^{m \times m})}
\le C_2 |\lambda - \tilde{\lambda}|,
\end{equation*}
for all $\lambda, \tilde{\lambda} \in [\lambda_1, \lambda_2]$.
Using the $L^{\infty} ((a, b), \mathbb{C}^{m \times m})$ control we have on 
$(\tilde{\lambda} I - \mathbf{V}_{\lambda} (x; \tilde{\lambda}))^{-1}$
and $\mathbf{V} (x; \lambda) - \mathbf{V} (x; \tilde{\lambda})$, along with 
$L^2 ((a, b), \mathbb{C}^m)$ control on $y_1$ and $f_1$, we see that 
$\tilde{f}_1 \in L^2 ((a, b), \mathbb{C}^m)$ as required. 

Similarly as in Section \ref{quadratic-schrodinger-systems-section}, we 
will next verify Assumption {\bf (E)}, then return to the verification  
of Assumption {\bf (D)}. For Assumption {\bf (E)}, we begin by 
identifying the function $\mathcal{E} (x; \lambda, \lambda_*)$. In this 
case, 
\begin{equation} \label{sls-B-difference}
    \mathbb{B} (x; \lambda) - \mathbb{B} (x; \lambda_*)
    = \begin{pmatrix}
    (\lambda - \lambda_*) I - (\mathbf{V} (x; \lambda) - \mathbf{V} (x; \lambda_*)) & 0 \\
    0 & 0 
    \end{pmatrix},
\end{equation}
and by comparing this with $\mathbb{B}_{\lambda} (x; \lambda_*)$ 
(from {\ref{sls-b-derivative}}), we see that we can express 
$\mathcal{E} (x; \lambda, \lambda_*)$ as 
\begin{equation*}
    \mathcal{E} (x; \lambda, \lambda_*)
    = \begin{pmatrix}
    \mathcal{E}_{11} (x; \lambda, \lambda_*) & 0 \\
    0 & 0
    \end{pmatrix},
\end{equation*}
with 
\begin{equation} \label{E11}
\mathcal{E}_{11} (x; \lambda, \lambda_*)
= (I - \mathbf{V}_{\lambda} (x; \lambda_*))^{-1}
 \Big{\{} (\lambda - \lambda_*) I - (\mathbf{V} (x; \lambda) - \mathbf{V} (x; \lambda_*)) \Big{\}}.
\end{equation}
(Here, the entries in the bottom row of $\mathcal{E} (x; \lambda, \lambda_*)$ don't have a role in the analysis, and 
are chosen to be zero for simplicity.) Computing directly, we also see that 
\begin{equation} \label{sls-E11-derivative}
\partial_{\lambda} \mathcal{E}_{11} (x; \lambda, \lambda_*)
= (I - \mathbf{V}_{\lambda} (x; \lambda_*))^{-1}
(I - \mathbf{V}_{\lambda} (x; \lambda)).
\end{equation}

For Assumptions {\bf (E)}(i) and {\bf (E)}(ii), we observe that for 
any $f \in L^2_{\mathbb{B}_{\lambda}} ((a, b), \mathbb{C}^{2m})$, 
expressed as $f = \genfrac{(}{)}{0pt}{1}{f_1}{f_2}$, we have 
\begin{equation*}
    \| \mathcal{E} (\cdot; \lambda, \lambda_*) f  \|_{\mathbb{B}_{\lambda}}^2
    = \int_a^b (\mathcal{E}_{11} (x; \lambda, \lambda_*) f_1 (x))^* 
    (I - \mathbf{V}_{\lambda} (x; \lambda)) \mathcal{E}_{11} (x; \lambda, \lambda_*) f_1 (x) dx, 
\end{equation*}
for which we can combine (\ref{E11}) with (\ref{delta-V}) (with $\lambda_*$ 
in place of $\tilde{\lambda}$ in the latter) to see that there exists a 
constant $C_3$ so that 
\begin{equation} \label{sls-E-bound}
    \|\mathcal{E} (\cdot; \lambda, \lambda_*) \| \le C_3 |\lambda - \lambda_*|.
\end{equation}
Assumptions {\bf (E)}(i) and {\bf (E)}(ii) both follow immediately. 

For Assumption {\bf (E)}(iii), it's clear from 
(\ref{sls-E11-derivative}) that for a.e. $x \in (a, b)$, 
$\mathcal{E}_{\lambda} (x; \lambda, \lambda_*)$ is continuous 
in $\lambda$. In order to verify that 
$\mathcal{E} (\cdot; \lambda, \lambda_*)$ is differentiable
as a map from $I_{\lambda_*, r} \subset I$ to 
$\mathcal{B} (L^2_{\mathbb{B}_{\lambda}} ((a, b), \mathbb{C}^{2n}))$, 
we let $\epsilon (h; \lambda) 
\in \mathcal{B} (L^2_{\mathbb{B}_{\lambda}} ((a, b), \mathbb{C}^{2n}))$
be the operator so that 
\begin{equation*}
\mathcal{E} (\cdot; \lambda + h, \lambda_*)
= \mathcal{E} (\cdot; \lambda, \lambda_*)
+ \mathcal{E}_{\lambda} (\cdot; \lambda, \lambda_*) h
+ \epsilon (h; \lambda) h,
\end{equation*}
and our goal is to show that 
\begin{equation} \label{degenerate-limit}
    \lim_{h \to 0} \| \epsilon (h; \lambda) \| = 0.
\end{equation}
Using (\ref{E11}) and (\ref{sls-E11-derivative})
we find that 
\begin{equation*}
\epsilon (h; \lambda)
= \begin{pmatrix} 
\epsilon_{11} (h; \lambda)  & 0 \\
0 & 0
\end{pmatrix},
\end{equation*}
where 
\begin{equation*}
\epsilon_{11} (h; \lambda)
= - h (I - \mathbf{V}_{\lambda} (\cdot; \lambda_*))^{-1} V_{12} (\cdot)
((\lambda + h)I - V_{22} (\cdot))^{-1} (\lambda I - V_{22} (\cdot))^{-2}
V_{12} (x)^*.
\end{equation*}
Proceeding now similarly as in the calculations leading to 
(\ref{sls-E-bound}), we find that (\ref{degenerate-limit}) 
holds. As expected, the derivative of 
$\mathcal{E} (\cdot; \lambda, \lambda_*)$ is 
precisely the usual partial derivative 
$\frac{\partial \mathcal{E}}{\partial \lambda} (\cdot; \lambda, \lambda_*)$,
and we need to verify that this latter operator is 
continuous on $I_{\lambda_*, r} \subset I$. For this, we fix any
$\lambda_0 \in I_{\lambda_*, r}$, and observe that
for any other $\lambda \in I_{\lambda_*, r}$, we can write 
\begin{equation} \label{degenerate-derivative-difference}
    \partial_{\lambda} \mathcal{E}_{11} (x; \lambda, \lambda_*)
    - \partial_{\lambda} \mathcal{E}_{11} (x; \lambda_0, \lambda_*)
    = - (I - \mathbf{V}_{\lambda} (x; \lambda_*))^{-1} 
    (\mathbf{V}_{\lambda} (x; \lambda) - \mathbf{V}_{\lambda} (x; \lambda_0)). 
\end{equation}
Focusing on the difference 
$\mathbf{V}_{\lambda} (x; \lambda) - \mathbf{V}_{\lambda} (x; \lambda_*)$, we 
can write 
\begin{equation} \label{sls-V-derivative-difference}
    \begin{aligned}
    \mathbf{V}_{\lambda} &(x; \lambda) - \mathbf{V}_{\lambda} (x; \lambda_0)
    = - V_{12} (x) \Big{\{} ((\lambda I - V_{22} (x))^{-1})^* (\lambda I - V_{22} (x))^{-1} \\
    & \quad \quad - ((\lambda_0 I - V_{22} (x))^{-1})^* (\lambda_0 I - V_{22} (x))^{-1} \Big{\}} V_{12} (x)^* \\
    &= - V_{12} (x) ((\lambda I - V_{22} (x))^{-1})^* 
    \Big{\{} (\lambda I - V_{22} (x))^{-1} - (\lambda_0 I - V_{22} (x))^{-1} \Big{\}} V_{12} (x)^* \\
    &\quad \quad + V_{12} (x)\Big{\{} (\lambda I - V_{22} (x))^{-1} - (\lambda_0 I - V_{22} (x))^{-1} \Big{\}}
    (\lambda_0 I - V_{22} (x))^{-1} V_{12} (x)^* \\
    &= - (\lambda - \lambda_0) V_{12} (x) \Big{\{} ((\lambda I - V_{22} (x))^{-1})^2 (\lambda_0 I - V_{22} (x))^{-1} \\
    &\quad \quad + (\lambda I - V_{22} (x))^{-1} ((\lambda_0 I - V_{22} (x))^{-1})^2 \Big{\}} V_{12} (x)^*.
    \end{aligned}
\end{equation}
Combining this relation with (\ref{degenerate-derivative-difference}), we find that 
there exists a constant $C_4$ so that 
\begin{equation*}
    \| \partial_{\lambda} \mathcal{E} (\cdot; \lambda, \lambda_*)
    -  \partial_{\lambda} \mathcal{E} (\cdot; \lambda_0, \lambda_*) \|
    \le C_4 |\lambda - \lambda_0|,
\end{equation*}
which implies continuity. 

For Assumption {\bf (E)}(iv), we let $f, g \in L^2_{\mathbb{B}_{\lambda}} ((a, b), \mathbb{C}^{2m})$
and observe that 
\begin{equation*}
    |f (x)^* \mathbb{B}_{\lambda} (x; \lambda) g(x)|
     = |f_1 (x)^* (I - \mathbf{V}_{\lambda} (x; \lambda)) g_1 (x)|.
\end{equation*}
If we set 
\begin{equation*}
    \tilde{h} (x) := \max_{\lambda \in [\lambda_1, \lambda_2]}
    \| I - \mathbf{V}_{\lambda} (x; \lambda) \|_{L^{\infty} ((a, b), \mathbb{C}^{m \times m})},
\end{equation*}
then according to Assumptions {\bf (DSLS)} (and using compactness of $[\lambda_1, \lambda_2]$),
there exists a constant $C_5$ so that $|\tilde{h} (x)| \le C_5$ for all $x \in (a, b)$. It follows
that Assumption {\bf (E)}(iv) is satisfied with 
\begin{equation*}
    h(x) := \tilde{h} (x)^2  |(I - \mathbf{V}_{\lambda} (x; \lambda_*))^{-1}|
    |f_1 (x)| |g_1 (x)|.
\end{equation*}

At this point, we return to establishing Assumption {\bf (D)}, which 
now follows via Lemma \ref{assumption-d-lemma} from the 
condition 
\begin{equation*}
    \|\mathcal{E}_{\lambda} (\cdot; \lambda, \lambda_0) - I\| = \mathbf{o} (1),
    \quad \lambda \to \lambda_0.
\end{equation*}
For this, we first use (\ref{sls-E11-derivative}) to write
\begin{equation*}
    \begin{aligned}
    \partial_{\lambda} \mathcal{E}_{11} (x; \lambda, \lambda_*) - I
    &= (I - \mathbf{V}_{\lambda} (x; \lambda_*))^{-1} 
    (I - \mathbf{V}_{\lambda} (x; \lambda)) - I \\
    &= (I - \mathbf{V}_{\lambda} (x; \lambda_*))^{-1} 
    \Big{\{}  (I - \mathbf{V}_{\lambda} (x; \lambda)) -  (I - \mathbf{V}_{\lambda} (x; \lambda_*)) \Big{\}} \\
    &= (I - \mathbf{V}_{\lambda} (x; \lambda_*))^{-1} 
    (\mathbf{V}_{\lambda} (x; \lambda) - \mathbf{V}_{\lambda} (x; \lambda_*)). 
    \end{aligned}
\end{equation*}
In view of (\ref{sls-V-derivative-difference}), we can conclude that there exists
a constant $C_6$ so that 
\begin{equation*}
    \| \partial_{\lambda} \mathcal{E}_{11} (x; \lambda, \lambda_*) - I \|_{L^{\infty} ((a, b), \mathbb{C}^{m \times m})}
    \le C_6 |\lambda - \lambda_*|,
\end{equation*}
which is stronger than we require. 

Finally, for Assumption {\bf (F)} we see from 
(\ref{sls-B-difference}) (with $\lambda_1$ and $\lambda_2$ 
respectively in place of $\lambda_*$ and $\lambda$) that 
we need to understand the matrix 
$\mathbf{V} (x; \lambda_2) - \mathbf{V} (x; \lambda_1)$. 
Using (\ref{delta-V}), we can write 
\begin{equation*} 
\mathbf{V} (x; \lambda_2) - \mathbf{V} (x; \lambda_1)
= - (\lambda_2 - \lambda_1) V_{12} (x) (\lambda_2 I - V_{22} (x))^{-1}
(\lambda_1 I - V_{22} (x))^{-1} V_{12} (x)^*.
\end{equation*}
In order to understand the sign of this resulting matrix, we observe
by spectral mapping that for each $x \in (a, b)$ the eigenvalues of
the matrix
$(\lambda_2 I - V_{22} (x))^{-1} (\lambda_1 I - V_{22} (x))^{-1}$
are precisely the values 
$(\lambda_2 - \nu (x))^{-1} (\lambda_1 - \nu (x))^{-1}$, where 
$\nu (x)$ is an eigenvalue of $V_{22} (x)$. 

If $\nu (x)$ is an eigenvalue of $V_{22} (x)$, then by assumption we have
either $\nu (x) + \delta < \lambda_1$ or $\nu (x) - \delta > \lambda_2$.
In the former case, we have $\lambda_1 - \nu (x) > \delta$, which 
also implies $\lambda_2 - \nu (x) > \delta$, so that 
\begin{equation} \label{sls-ratio}
    \frac{1}{(\lambda_1 - \nu (x)) (\lambda_2 - \nu (x))} > 0.
\end{equation}
Likewise, if $\nu (x) - \delta > \lambda_2$ then we have 
$\lambda_2 - \nu (x) < -\delta$, which 
also implies $\lambda_1 - \nu (x) < - \delta$, so that
the left-hand side of (\ref{sls-ratio}) is again 
positive. We conclude that for every $x \in (a, b)$, 
the eigenvalues of the matrix 
$(\lambda_2 I - V_{22} (x))^{-1} (\lambda_1 I - V_{22} (x))^{-1}$
are all positive, and consequently the matrix 
$\mathbf{V} (x; \lambda_2) - \mathbf{V} (x; \lambda_1)$
is necessarily non-positive (and negative definite if 
$\rank V_{12} (x) = m$). Since $\lambda_2 > \lambda_1$, it follows
that the matrix 
\begin{equation*}
    (\lambda_2 - \lambda_1) I - (\mathbf{V} (x; \lambda_2) - \mathbf{V} (x; \lambda_1))
\end{equation*}
is positive-definite. In this way, we see that 
$\mathbb{B} (x; \lambda_2) - \mathbb{B} (x; \lambda_1)$ is non-negative. In addition, 
if $y (\cdot; \lambda_1) \in \AC_{\loc} ((a, b), \mathbb{C}^{2m})$ is any solution
of $Jy' = \mathbb{B} (x; \lambda_1) y$, then for any interval 
$[c, d] \subset (a, b)$, $c < d$,
\begin{equation*}
\begin{aligned}
    \int_c^d &((\mathbb{B} (x; \lambda_2) - \mathbb{B} (x; \lambda_1)) y_1 (x; \lambda_1), y(x; \lambda_1)) dx \\
    & = \int_c^d (((\lambda_2 - \lambda_1) I - (\mathbf{V} (x; \lambda_2) - \mathbf{V} (x; \lambda_1)) y_1 (x; \lambda_1), y_1 (x; \lambda_1)) dx,
\end{aligned}    
\end{equation*}
and this must be strictly positive unless $y_1 (x; \lambda_1) = 0$ for all $x \in (c, d)$. 
But since $y_2 = P_{11} (x)^{-1} y_1$, this would contradict the assumption that 
$y(x; \lambda)$ is a non-trivial solution. This completes the verification of
Assumption {\bf (F)}, and so concludes the verifications of Assumptions 
{\bf (A)} through {\bf (F)}. 
\end{proof}

\subsection{Specific Applications}
\label{specific-applications-section}

In this section, we apply our framework to a specific case of the 
quadratic Schr\"odinger equation, and to an equation of Hain-L\"ust 
type arising in the study of magnetohydrodynamics.

\subsubsection{Quadratic Schr\"odinger Equation}
\label{quadratic-schrodinger-example}

When Schr\"odinger's equation for the hydrogen atom is expressed in 
spherical coordinates and analyzed by separation of variables, the 
resulting radial equation can be expressed as 
\begin{equation} \label{schrodinger-hydrogen}
H\phi := - \frac{1}{x^2} (x^2 \phi')' - \frac{\gamma}{x} \phi
+ \frac{\ell (\ell+1)}{x^2} \phi = \lambda \phi,
\quad x \in (0, \infty),
\end{equation}
where $\gamma>0$ is a physical constant and $\ell$ is a non-negative
integer associated with angular momentum (see, e.g., Chapter 12 
in \cite{Ga1974}). In the case $\ell = 0$, the change of 
dependent variable $\psi = x \phi$ reduces (\ref{schrodinger-hydrogen})
to
\begin{equation} \label{schrodinger-hydrogen-transformed}
\mathcal{H}\psi := - \psi'' - \frac{\gamma}{x} \psi = \lambda \psi.
\end{equation}
Equations (\ref{schrodinger-hydrogen}) (with $\ell=0$) and 
(\ref{schrodinger-hydrogen-transformed}) were analyzed in 
\cite{{HS2022}} with the renormalized oscillation approach, 
and in this section we extend that analysis 
to the case in which (\ref{schrodinger-hydrogen-transformed})
is augmented with a quadratic nonlinearity in $\lambda$. As 
background for this, we observe that in the setting
of (\ref{schrodinger-hydrogen})
eigenvalues $\lambda$ correspond with energies, and additionally 
that these energies can affect the potential function so that
Schr\"odinger's equation becomes nonlinear in the spectral 
parameter, 
\begin{equation} \label{schrodinger-nonlinear}
-\psi'' + V(x; \lambda) \psi = \lambda \psi
\end{equation}
(see, e.g., \cite{JJ1972}). In particular, in \cite{JJ1972}, 
the authors consider potentials $V(x;\lambda)$ of the 
forms 
\begin{equation*}
    V^{\pm} (x; \lambda) = U(x) \pm 2 \sqrt{\lambda} Q(x),
\end{equation*}
with $U(x)$ and $Q(x)$ appropriately specified,
leading to Schr\"odinger's equations with quadratic dependence on 
a spectral parameter $\kappa := \sqrt{\lambda}$. For the current
analysis, we take precisely this form with 
$U(x) = - \frac{\gamma}{x}$ and the choice $V^+ (x; \lambda)$
to obtain 
\begin{equation} \label{schrodinger-quadratic-eqn}
\mathcal{H} \psi = - \psi'' - \frac{\gamma}{x} \psi
= - 2 \kappa Q(x) \psi + \kappa^2 \psi.  
\end{equation}
As discussed in \cite{JJ1972}, the eigenvalues of interest 
in (\ref{schrodinger-quadratic-eqn}) are negative-real, so 
$\kappa$ is purely imaginary. In order to work with a real 
spectral parameter, we write $\kappa = i\tau$, and for a 
specific example, we adopt 
from \cite{JJ1972} a form of $Q(x)$, which we
index with a parameter $\delta \ge 0$ to study its effect on 
the location of the eigenvalues. In particular, we use 
\begin{equation*}
    Q(x; \delta) := \frac{i \delta}{2} \frac{e^{-x/2}}{1-.5e^{-x/2}}.
\end{equation*}
Combining these observations, we arrive at the family of equations
\begin{equation} \label{schrodinger-family}
    H\psi = - \psi'' + V(x) \psi
    = (Q_1 (x; \delta) \lambda + Q_2 (x) \lambda^2) \psi,
    \quad x \in (0, \infty),
\end{equation}
where $V (x) = -\frac{\gamma}{x}$, 
\begin{equation*}
    Q_1 (x; \delta) = \frac{\delta}{2} \frac{e^{-x/2}}{1-.5e^{-x/2}},
\end{equation*}
$Q_2 (x) = -1$, and we are using $\lambda$ rather than $\tau$ 
to denote the spectral parameter. For the calculations
carried out below we will complete the specification of 
$H$ by taking $\gamma = 4$. 

The first step of our analysis of (\ref{schrodinger-family}) is to 
verify that Assumptions {\bf (Q)} from 
Section \ref{quadratic-schrodinger-systems-section} are satisfied. 
For this, we first observe by inspection that 
$V, Q_1, Q_2 \in L^1_{\loc} ((0, \infty), \mathbb{R})$,
and also that $Q_1, Q_2 \in L^{\infty} ((0, \infty), \mathbb{R})$. 
In addition, since $V$, $Q_1$, and $Q_2$ are all functions 
mapping into the real numbers they are immediately symmetric
when viewed as $1 \times 1$ matrices. For the final assumption in {\bf (Q)}, we can write 
\begin{equation*}
    Q_1 (x) + 2 \lambda Q_2 (x)
    = \frac{\delta}{2} \frac{e^{-x/2}}{1-.5e^{-x/2}} - 2\lambda
    \ge - 2\lambda,
\end{equation*}
so the condition is satisfied on any interval $I = [\lambda_1, \lambda_2]$,
$\lambda_1 < \lambda_2$, for which $\lambda_2 < 0$. 

Next, we fix a particular interval $[\lambda_1, \lambda_2]$ that will remain 
unchanged throughout the analysis. In order to anchor our work to the 
case in which $\lambda$ appears linearly, we will start by thinking about 
the case $\delta = 0$, for which (\ref{schrodinger-family}) becomes
\begin{equation} \label{schrodinger-family2}
    - \psi'' + V(x) \psi
    = - \lambda^2 \psi.
\end{equation}
This is precisely the case analyzed in \cite{HS2022}, except that in 
\cite{HS2022} the appearance of $- \lambda^2$ here is 
replaced by $\lambda \le 0$. As discussed in \cite{HS2022}, 
the eigenvalues 
(with appropriate boundary conditions, discussed in detail below) are known
to be
\begin{equation*} \label{schrodinger-eigenvalues}
    \lambda = \pm \frac{\gamma}{2n},
    \quad n = 1, 2, 3, \dots.
\end{equation*}
The lowest four eigenvalues are $-2$, $-1$, $-\frac{2}{3}$,
and $-\frac{1}{2}$, and in order to focus on the first three of these
we will take $\lambda_1 = -3$ and $\lambda_2 = -7/12$. 

In the usual way, we will express (\ref{schrodinger-family})
as a first-order system by setting $y_1 = \psi$ and $y_2 = \psi'$. We 
obtain (\ref{hammy}) with  
\begin{equation} \label{quadratic-b}
    \mathbb{B} (x; \lambda)
    = \begin{pmatrix}
    \frac{\lambda \delta}{2} \frac{e^{-x/2}}{1-.5e^{-x/2}} - \lambda^2 + \frac{\gamma}{x} & 0 \\
    0 & 1
    \end{pmatrix}.
\end{equation}
For each $\lambda \in [-3,-7/12]$, the maximal operator $\mathcal{T}_M (\lambda)$
is now defined precisely as in Definition \ref{maximal-operator}, Part (ii).

Our starting point for the analysis is to determine the 
limit-case of (\ref{hammy}) with (\ref{quadratic-b}). For this,
we fix any $\lambda_0 \in [-3, -7/12]$ and any 
$\mu_0 \in \mathbb{C} \backslash \mathbb{R}$, and we consider
the linear Hamiltonian system 
\begin{equation} \label{for-niessen1}
    Jy' = \mathbb{B} (x; \lambda_0) y + \mu_0 \mathbb{B}_{\lambda} (x; \lambda_0)y.
\end{equation}
Specifically, we will take $\lambda_0 = -1$ and $\mu_0 = i$. We let 
$\Phi (x; i, -1)$ denote a fundamental matrix for this 
system specified with $\Phi (1; i, -1) = I_2$ (i.e., we 
take $c=1$ in our general development). With 
$\mathcal{A} (x; i, -1)$ correspondingly defined 
as in (\ref{mathcal-A-defined}), the value $m_0 (i,-1)$
can be determined by considering the eigenvalues of 
$\mathcal{A} (x; i, -1)$ as $x \to 0^+$, and 
the value $m_{\infty} (i,-1)$
can be determined by considering the eigenvalues of 
$\mathcal{A} (x; i, -1)$ as $x \to + \infty$. We 
denote the eigenvalues of $\mathcal{A} (x; i, -1)$
by $\{\nu_j (x; i, -1)\}_{j=1}^2$, and proceeding numerically,
we find that for the case $\delta = 0$ we approximately have 
\begin{equation} \label{eg1evals}
    \begin{aligned}
    \nu_1 (\epsilon;i,-1) &= -1.3305 \\
    \nu_2 (\epsilon;i,-1) &= .1879,
    \end{aligned}
\end{equation}
and 
\begin{equation*}
    \begin{aligned}
    \nu_1 (10;i,-1) &= -1.0302 \times 10^{-8} \\
    \nu_2 (10;i,-1) &= 2.7857 \times 10^7,
    \end{aligned}
\end{equation*}
where here and below we are taking $\epsilon = 10^{-10}$ as 
a suitably small value for the approximations. 
From the first pair we expect that $m_0 (i;-1) = 2$, from 
which we can conclude from Assumption {\bf (D)} 
(verified for this case in Section 
\ref{quadratic-schrodinger-systems-section}) 
that $m_0 (\mu, \lambda) = 2$ for all 
$\lambda \in [-3, -7/12]$ and 
$\mu \in \mathbb{C} \backslash \mathbb{R}$. By definition, 
then, $\mathcal{T}_M (\lambda)$ is in the limit-circle
case at $x=0$ for all $\lambda \in [-3, -7/12]$. Likewise, 
$m_{\infty} (i;-1) = 1$, from which we can conclude
that $\mathcal{T}_M (\lambda)$ is in the limit-point
case at $x=\infty$ for all $\lambda \in [-3, -7/12]$.
Finally, we note that similar values and conclusions
hold for the cases $\delta = 1$ and $\delta = 5$,
though we omit details in those cases.

\begin{remark} \label{quadratic-limit-case-remark}
Throughout this section, our numerical calculations are 
intended only to illustrate the theory, and we make no 
effort to rigorously justify either the values we obtain
or the conclusions we draw from them. For example, 
in this last calculation, we have not attempted to 
find a rigorous error interval for the values of
$\nu_1 (10^{-10};i,-1)$, $\nu_2 (10^{-10};i,-1)$, 
and $\nu_1 (10;i,-1)$,
and we offer no additional 
direct justification that $\nu_2 (x;i,-1)$ is indeed
tending to $+ \infty$ as $x$ tends to $+ \infty$. 
Nonetheless, we observe that 
in the current relatively simple setting the
associated observations about limiting behavior
can be verified rigorously using Frobenius regular singular-point
theory as $x \to 0^+$ (as described, e.g., in Section 5.6
of \cite{BDM2017}) and a standard asymptotic analysis as
$x \to + \infty$ (as in \cite{HS2020}). Finally, 
we note that in all of our numerical calculations, we work with more 
precise numbers than those given in the text. 
\end{remark}

Since $\mathcal{T}_M (\lambda)$ is in the limit-point case
at $x = + \infty$, no boundary condition is necessary on 
that side, but on the left we must specify a boundary 
condition based on a choice of Niessen elements. As
discussed in Section 5.2 of \cite{HS2022}, it's natural 
in physical arguments to specify boundary conditions 
so that solutions of (\ref{schrodinger-hydrogen}) remain
bounded as $x \to 0^+$. Through the map $\psi = x\phi$, such 
solutions correspond with solutions of (\ref{schrodinger-hydrogen-transformed})
for which $\psi (x; \lambda) \to 0$ as $x \to 0^+$. 
Motivated by these observations, we will choose the Niessen 
elements for our boundary condition at $x = 0$ so that 
solutions vanish as $x \to 0^+$. 

As discussed in Section \ref{operator-section}, the Niessen 
space associated with $x=0$ comprises two solutions to 
(\ref{for-niessen1}), constructed as 
\begin{equation*}
    y_1^0 (x; i, -1) = \Phi (x; i, -1) v_1^0 (i, -1),
    \quad 
    y_2^0 (x; i, -1) = \Phi (x; i, -1) v_2^0 (i, -1),
\end{equation*}
where $v_1^0 (i, -1)$ and $v_2^0 (i, -1)$ are respectively 
vectors obtained as limits ($x \to 0^+$) of eigenvectors
of $\mathcal{A} (x; i,-1)$ associated respectively with
the eigenvalues $\nu_1 (x; i, -1)$ and $\nu_2 (x; i, -1)$
(obtained as described in Lemma \ref{subspace-dimensions-lemma}). 
For computational purposes, we approximate these vectors 
in the case $\delta = 0$ with 
\begin{equation*}
\begin{aligned}
    v_1^0 (i, -1) 
    &\cong 
    v_1 (\epsilon; i, -1)
    =
    \begin{pmatrix}
    .8631 \\ -.1762 + 4733i
    \end{pmatrix}; \\
     v_2^0 (i, -1) 
    &\cong 
    v_2 (\epsilon; i, -1)
    =
    \begin{pmatrix}
    .1762 + 4733i \\ .8631
    \end{pmatrix}.
\end{aligned}
\end{equation*}

\begin{remark} \label{hydrogen-eigs-relation-remark}
    We see in these calculations that 
    \begin{equation*}
        v_2^0 (i; -1) = J \overline{v_1^0 (i; -1)},
    \end{equation*}
    where the overbar denotes complex conjugate. In 
    order to verify that this is generally the case, we 
    first recall the identities 
    \begin{equation} \label{phiJphi}
        \begin{aligned}
            \Phi (x; \bar{\mu}; \lambda)^* J \Phi (x; \mu, \lambda) &= J \\
            \Phi (x; \bar{\mu}, \lambda) &= \overline{\Phi (x; \mu, \lambda)}.
        \end{aligned}
    \end{equation}
The first of these is equation (2.7) from \cite{HS2022}, and 
the second follows from (\ref{for-niessen1}) since 
$\mathbb{B} (x; \lambda)$ has real-valued entries. Now, suppose
$\nu_1$ is an eigenvalue of $\mathcal{A} (x; \mu, \lambda)$
with eigenvector $v_1$ so that 
\begin{equation*}
    \frac{1}{2\im \mu} \Phi (x; \mu, \lambda)^* (J/i) \Phi (x; \mu, \lambda) v_1 = \nu_1 v_1.
\end{equation*}
If we take the complex conjugate of this relation, we obtain 
\begin{equation} \label{new-remark1}
    \frac{1}{2\im \mu} \Phi (x; \bar{\mu}, \lambda)^* (J/i) \Phi (x; \bar{\mu}, \lambda) \bar{v}_1 = \nu_1 \bar{v}_1,
\end{equation}
where we have recalled that $\nu_1 \in \mathbb{R}$. According to 
(\ref{phiJphi}), the matrix $-J \Phi (x; \mu, \lambda) J$ is the 
inverse of $\Phi (x; \bar{\mu}, \lambda)^*$, so if we multiply
(\ref{new-remark1}) on the left by  $-\Phi (x; \mu, \lambda)^* J \Phi (x; \mu, \lambda) J$
we obtain the relation
\begin{equation*}
    \frac{1}{2\im \mu} (J/i) \bar{v}_1
    = \nu_1 \Phi (x; \mu, \lambda)^* J \Phi (x; \mu, \lambda) J \bar{v}_1.
\end{equation*}
Rearranging, we see that 
\begin{equation*}
     \frac{1}{2\im \mu} \Phi (x; \mu, \lambda)^* (J/i) \Phi (x; \mu, \lambda) J \bar{v}_1
     = - \frac{1}{(2 \im \mu)^2 \nu_1} J \bar{v}_1,
\end{equation*}
showing that if $\nu_1$ is an eigenvalue of $\mathcal{A} (x; \mu, \lambda)$
with eigenvector $v_1$ then $\nu_2 = -1/((2\im \mu)^2 \nu_1)$ is an 
eigenvalue of $\mathcal{A} (x; \mu, \lambda)$ with eigenvector
$v_2 = J \bar{v}_1$. (Cf. Lemma 2.2 in \cite{HS2022}.)
\end{remark}

Our boundary conditions at $x=0$ will be specified via a 
vector function
\begin{equation} \label{U0defined}
    \begin{aligned}
    U^0 (x; i,-1) 
    &= y_1^0 (x; i, -1) + \beta z_1^0 (x; i, -1) \\
    &= \Phi (x; i, -1) (v_1^0 (i, -1) + \beta v_2^0 (i, -1)),
    \end{aligned}
\end{equation}
where $\beta$ is taken to be a complex number so that 
\begin{equation*}
    |\beta| = \sqrt{-\nu_1^0 (i, -1)/\nu_2^0 (i, -1)}.
\end{equation*}
Using (\ref{eg1evals}), we can approximate the right-hand side of this relation 
with 
\begin{equation*}
\sqrt{-\nu_1 (\epsilon; i, -1)/\nu_2 (\epsilon; i, -1)}
= 2.6609.
\end{equation*}
We've seen that the physical boundary condition we expect
is 
\begin{equation*}
    \lim_{x \to 0^+} y(x; \lambda)
    = \begin{pmatrix}
    0 \\ 1
    \end{pmatrix}
\end{equation*}
(the value $1$ taken as a choice of scaling),
and with $y(x; \lambda) = \Phi (x; 0, \lambda) w$ we can
identify the vector $w$ by computing 
\begin{equation*}
    w = \lim_{x \to 0^+} \Phi (x; 0, \lambda)^{-1} 
    \begin{pmatrix}
    0 \\ 1
    \end{pmatrix}.
\end{equation*}
Precisely, we fix $\lambda = \lambda_1 = -3$, and approximate 
$w$ with 
\begin{equation*}
    w \cong \Phi (\epsilon; 0, -3)^{-1} 
    \begin{pmatrix}
    0 \\ 1
    \end{pmatrix}
    = 
    \begin{pmatrix}
    .4720 \\ .8816
    \end{pmatrix}.
\end{equation*}
The boundary condition that we would like to identify will have 
the form 
\begin{equation*}
    \lim_{x \to 0^+} U^0 (x; i, -1)^* J \Phi (x; 0, -3) w 
    = 0, 
\end{equation*}
so we can approximate the required choice of $\beta$ by 
taking $\beta \in \mathbb{C}$ so that 
\begin{equation*}
  U^0 (\epsilon; i, -1)^* J \Phi (\epsilon; 0, -3) w = 0.
\end{equation*}
We can express this relation as 
\begin{equation*}
    \Big(\Phi (\epsilon; i, -1) (v_1^0 (i, -1) + \beta v_2^0 (i, -1)) \Big)^*
    J \Phi (\epsilon; 0, -3) w = 0,
\end{equation*}
allowing us to solve for 
\begin{equation*}
    \bar{\beta}
    \cong - \frac{v_1^0 (i, -1)^* \Phi (\epsilon; i, -1)^* J \Phi (\epsilon; 0, -3)w}
    {v_2^0 (i, -1)^* \Phi (\epsilon; i, -1)^* J \Phi (\epsilon; 0, -3)w}.
\end{equation*}
Proceeding with the above numerical approximations of $v_1^0 (i, -1)$ and $v_2^0 (i, -1)$,
we find $\bar{\beta} \cong -.9494 - 2.4858i$, so that 
$\beta \cong -.9494 + 2.4858i$.

At this point, we can fully specify the self-adjoint restriction of 
$\mathcal{T}_M (\lambda)$ that we will work with; namely, for each $\lambda \in [\lambda_1, \lambda_2]$
we will take $\mathcal{T} (\lambda)$ to be the restriction of $\mathcal{T}_M (\lambda)$
to the domain 
\begin{equation*}
    \mathcal{D} := 
    \{y \in \mathcal{D}_M:  \lim_{x \to 0^+} U^0 (x; i, -1)^* J  y (x) = 0 \},
\end{equation*}
where $U^0 (x; i, -1)$ is as in (\ref{U0defined}) with 
$\beta = -.9494 + 2.4858i$.
Having specified the operator under consideration, we next 
check the condition on essential spectrum, namely that 
for each $\lambda \in [\lambda_1, \lambda_2]$, 
$0 \notin \sigma_{\ess} (\mathcal{T} (\lambda))$. To this end, 
we define the corresponding second-order operator 
\begin{equation} \label{H-lambda}
    \mathbf{H} (\lambda) \psi := 
    - \psi'' + (V (x) - \lambda Q_1 (x) + \lambda^2) \psi
\end{equation}
with domain 
\begin{equation*}
    \begin{aligned}
    &\dom(\mathbf{H} (\lambda)) = \Big{\{}\psi \in L^2 ((0,\infty),\mathbb{C}): 
    \psi, \psi' \in \AC_{\loc} ((0,\infty),\mathbb{C}), \\
    & \quad \mathbf{H} (\lambda) \psi \in L^2 ((0,\infty),\mathbb{C}), \,
    \lim_{x \to 0^+} U^0 (x; i, -1)^*
    J \genfrac{(}{)}{0pt}{0}{\psi (x)}{\psi' (x)} = 0 
    \Big{\}}.
    \end{aligned}
\end{equation*}
As discussed in \cite{HS2022}, it's straightforward to check that 
for each $\lambda \in [-3, -7/12]$, the operators $\mathcal{T}_M (\lambda)$
and $\mathbf{H} (\lambda)$ have precisely the same sets of essential spectrum, and 
also the same sets of discrete eigenvalues. In addition, we know from 
\cite{Rejto1966} that any self-adjoint restriction of the operator
$-\partial_x^2 + V(x)$ has essential spectrum $[0, \infty)$, and it follows
immediately that any self-adjoint restriction of the operator 
$-\partial_x^2 + (V(x) + \lambda^2)$ has essential spectrum 
$[\lambda^2, \infty)$. The remaining term $-\lambda Q_1 (x)$ is 
a compact perturbation of $-\partial_x^2 + (V(x) + \lambda^2)$, and so 
we can conclude that for each $\lambda \in [-3, -7/12]$
the essential spectrum of $\mathbf{H} (\lambda)$
(and so of $\mathcal{T} (\lambda)$) is precisely $[\lambda^2, + \infty)$.
From these observations it's clear that for each 
$\lambda \in [\lambda_1, \lambda_2]$, we have 
$0 \notin \sigma_{\ess} (\mathcal{T} (\lambda))$, which is 
exactly the condition we require on essential spectrum.

Next, in order to fully apply Theorem \ref{singular-theorem}, we 
must check conditions (\ref{condition2}) and (\ref{condition3}). 
Beginning with (\ref{condition2}), we first observe that 
$x = 0$ is a regular singular point of (\ref{schrodinger-family}),
so we have from standard Frobenius theory that for each 
$\lambda \in [-3, -7/12]$ there exist constants 
$\{a_k (\lambda)\}_{k=1}^{\infty}$ and 
$\{b_k (\lambda)\}_{k=0}^{\infty}$ so that for all $x > 0$
sufficiently small, the functions 
\begin{equation*}
    \begin{aligned}
    \psi_1 (x; \lambda) 
    &= x + \sum_{k=1}^{\infty} a_k (\lambda) x^{k+1} \\
    \psi_2 (x; \lambda) 
    &= 1 + b_0 (\lambda) \psi_1 (x; \lambda) \ln x
    + \sum_{k=1}^{\infty} b_k (\lambda) x^k
    \end{aligned}
\end{equation*}
comprise a linearly independent pair of solutions. 
(See, e.g., \cite{BDM2017}). In addition, our boundary 
condition at $x = 0$ selects precisely the solution of 
(\ref{schrodinger-family}) that approaches $0$ as 
$x \to 0^+$, so we can take 
\begin{equation*}
    \mathbf{X}_0 (x; \lambda)
    = \begin{pmatrix}
    \psi_1 (x; \lambda) \\ 
    \psi_1' (x; \lambda)
    \end{pmatrix}.
\end{equation*}
Moreover, if $\lambda_2$ is not an eigenvalue of 
(\ref{schrodinger-family}), then we can take 
\begin{equation*}
    \mathbf{X}_{\infty} (x; \lambda_2) 
    = \begin{pmatrix}
    \psi_2 (x; \lambda_2) + \kappa_0 (\lambda_2) \psi_1 (x; \lambda_2) \\
    \psi_2' (x; \lambda_2) + \kappa_0 (\lambda_2) \psi_1' (x; \lambda_2)
    \end{pmatrix},
\end{equation*}
for some constant $\kappa_0 (\lambda_2)$, where the key point in the 
specification of $\mathbf{X}_{\infty} (x; \lambda_2)$ is that if 
$\lambda_2$ is not an eigenvalue of (\ref{schrodinger-family}), then
$\psi_2 (x; \lambda_2)$ must appear non-trivially. 

At this point, we can detect intersections between $\ell_0 (x; \lambda)$
and $\ell_{\infty} (x; \lambda_2)$ (respectively the Lagrangian subspaces
with frames $\mathbf{X}_0 (x; \lambda)$ and $\mathbf{X}_{\infty} (x; \lambda_2)$) 
by computing 
\begin{equation*}
    \det \begin{pmatrix}
     \mathbf{X}_0 (x; \lambda) & \mathbf{X}_{\infty} (x; \lambda_2) \\
    \end{pmatrix}
    = \det 
    \begin{pmatrix}
     \psi_1 (x; \lambda) & \psi_2 (x; \lambda_2) + \kappa_0 (\lambda_2) \psi_1 (x; \lambda_2) \\
     \psi_1' (x; \lambda) & \psi_2' (x; \lambda_2) + \kappa_0 (\lambda_2) \psi_1' (x; \lambda_2)
    \end{pmatrix}.
\end{equation*}
The only term that doesn't vanish to first order in $x$ is $\psi_2 (x; \lambda_2) \psi_1' (x; \lambda)$,
and this product approaches $1$ as $x \to 0^+$. We can conclude that for $x>0$ sufficiently 
small the spaces $\ell_0 (x; \lambda)$ and $\ell_{\infty} (x; \lambda_2)$ do not 
intersect for any $\lambda \in [-3, -7/12]$, giving (\ref{condition2}). 

Turning to condition (\ref{condition3}), it's convenient to express (\ref{schrodinger-family})
as 
\begin{equation} \label{wkb-form}
    \psi'' = \mathcal{V} (x; \lambda) \psi,
\end{equation}
where 
\begin{equation} \label{wkb-potential}
    \mathcal{V} (x; \lambda)
    := - \frac{\gamma}{x} 
    - \frac{\delta}{2} \frac{e^{-x/2}}{1 - .5 e^{-x/2}} \lambda 
    + \lambda^2.
\end{equation}
According to Theorem 2.1 in Chapter 6 of \cite{Olver1974}, 
for each $\lambda \in [-3, -7/12]$, we can express a linearly 
independent pair of solutions to (\ref{wkb-form}) as 
\begin{equation*}
    \begin{aligned}
    \psi_3 (x; \lambda) &= 
    \mathcal{V} (x; \lambda)^{-1/4} 
    e^{- \int_M^x \mathcal{V} (y; \lambda) dy} (1 + \epsilon_3 (x; \lambda)) \\
    \psi_4 (x; \lambda) &= 
    \mathcal{V} (x; \lambda)^{-1/4} 
    e^{\int_M^x \mathcal{V} (y; \lambda) dy} (1 + \epsilon_4 (x; \lambda)), 
    \end{aligned}
\end{equation*}
where by taking $0 \ll M$ we can ensure that for $i = 1,2$, $\epsilon_i (x; \lambda)$ and 
$\epsilon_i' (x; \lambda)$ are suitably small for all $x > M$, $i = 3, 4$.

With the above construction, we can take a frame for $\ell_{\infty} (x; \lambda)$
to be 
\begin{equation*}
    \mathbf{X}_{\infty} (x; \lambda) 
    =  \begin{pmatrix}
    \psi_3 (x; \lambda) \\ \psi_3'(x; \lambda)
    \end{pmatrix}.
\end{equation*}
In addition, since $\lambda_1 = -3$ is not an eigenvalue of $\mathcal{T} (\cdot)$,
we can take a frame for $\ell_0 (x; \lambda_1)$
to be 
\begin{equation*}
    \mathbf{X}_0 (x; \lambda_1)
    = \begin{pmatrix}
    \psi_4 (x; \lambda_1) + \kappa_{\infty} (\lambda_1) \psi_3 (x; \lambda_1) \\ 
    \psi_4' (x; \lambda_1) + \kappa_{\infty} (\lambda_1) \psi_3' (x; \lambda_1)
    \end{pmatrix},
\end{equation*}
where the key point is that $\psi_4 (x; \lambda_1)$ must appear non-trivally
in the linear combination. In order to detect intersections between 
$\ell_0 (x; \lambda_1)$ and $\ell_{\infty} (x; \lambda)$ (respectively the 
Lagrangian subspaces with frames $\mathbf{X}_0 (x; \lambda_1)$ and 
$\mathbf{X}_{\infty} (x; \lambda)$), we can compute the determinant 
\begin{equation*}
    \begin{aligned}
     \det \begin{pmatrix}
     \mathbf{X}_0 (x; \lambda_1) & \mathbf{X}_{\infty} (x; \lambda) \\
    \end{pmatrix}
    = \det 
    \begin{pmatrix}
    \psi_4 (x; \lambda_1) + \kappa_{\infty} (\lambda_1) \psi_3 (x; \lambda_1) & \psi_3 (x; \lambda) \\
    \psi_4' (x; \lambda_1) + \kappa_{\infty} (\lambda_1) \psi_3' (x; \lambda_1) & \psi_3' (x; \lambda) 
    \end{pmatrix}.
    \end{aligned}
\end{equation*}
For $M$ sufficiently large, the leading order term from this determinant is 
\begin{equation*}
    (\frac{\mathcal{V} (x; \lambda)}{\mathcal{V} (x; \lambda_1)})^{1/4}
    + (\frac{\mathcal{V} (x; \lambda_1)}{\mathcal{V} (x; \lambda)})^{1/4}
    \cong
    (\frac{\lambda}{\lambda_1})^{1/4} + (\frac{\lambda_1}{\lambda})^{1/4},
\end{equation*}
where the approximation becomes better as $x$ increases. Since this last relation
is non-zero for all $\lambda \in [-3, -7/12]$, we can conclude that for $x$ 
sufficiently large the spaces $\ell_0 (x; \lambda_1)$ and $\ell_{\infty} (x; \lambda)$
do not intersect, ensuring that Condition (\ref{condition3}) holds. 

At this point, we have (approximately) specified our frame that lies left in $(0, \infty)$,
\begin{equation*}
    \mathbf{X}_0 (x; -3) = \Phi (x; 0, -3) w,
\end{equation*}
and we next turn to identifying the solution of (\ref{hammy}) with 
(\ref{quadratic-b}) and $\lambda = \lambda_2 = - 7/12$ that lies 
right in $(0, \infty)$. For this, we begin with the eigenvalues
of 
\begin{equation*}
    \mathcal{B} (x; \lambda_2)
    = \Phi (x; 0, \lambda_2)^* J (\partial_{\lambda} \Phi) (x; 0, \lambda_2).
\end{equation*}
In order to approximate the behavior of these eigenvalues as $x \to + \infty$, we 
choose a large value $M$ and evaluate $\nu_1 (M; \lambda_2)$ and 
$\nu_2 (M; \lambda_2)$. Precisely, in this case, we take $M = 50$,
for which we find 
\begin{equation*}
    \begin{aligned}
    \nu_1 (50; \lambda_2) &= 32.7812 \\
    \nu_2 (50; \lambda_2) &= 1.6480 \times 10^{14}.
    \end{aligned}
\end{equation*}
The eigenvector associated with $\nu_1 (50; \lambda_2)$ is 
\begin{equation*}
    v_1 (50; \lambda_2)
    = \begin{pmatrix}
    -.9332 \\ .3593
    \end{pmatrix}.
\end{equation*}
This approximately specifies the solution $\mathbf{X}_{\infty} (x; \lambda_2)$
that lies right in $(0,\infty)$ as 
\begin{equation*}
   \mathbf{X}_{\infty} (x; \lambda_2)
   = \Phi (x; 0, -\frac{7}{12}) v_1 (50; \lambda_2).
\end{equation*}
With $\mathbf{X}_0 (x; \lambda_1)$ and $\mathbf{X}_{\infty} (x; \lambda_2)$ specified, we can 
now compute the Maslov index 
\begin{equation*}
    \mas (\ell_0 (\cdot; \lambda_1), \ell_{\infty} (\cdot; \lambda_2); (0,\infty))
\end{equation*}
as a rotation number through $-1$ for the complex number
\begin{equation*}
    \begin{aligned}
    \tilde{W} (x; \lambda_1, \lambda_2) 
    &= - (X_0 (x; \lambda_1) + i Y_0 (x; \lambda_1)) (X_0 (x; \lambda_1) - i Y_0 (x; \lambda_1))^{-1}  \\
    & \quad \times (X_{\infty} (x; \lambda_2) + i Y_{\infty} (x; \lambda_2)) 
    (X_{\infty} (x; \lambda_2) - i Y_{\infty} (x; \lambda_2))^{-1}.  
    \end{aligned}
\end{equation*}
By generating the frames $\mathbf{X}_0 (x; \lambda_1)$ and 
$\mathbf{X}_{\infty} (x; \lambda_2)$ numerically, we can track 
the complex value $\tilde{W} (x; \lambda_1, \lambda_2)$, and in this 
way, we observe crossings at values $x = .653$, $x = 2.137$, 
and $x = 5.489$ (with numerical increment $.001$). From Theorem 
\ref{singular-theorem}, we can conclude that there are three eigenvalues
on the interval $[-3, -7/12]$, as expected. 

In order to see more fully the dynamics associated with this count, 
we identify the {\it spectral curves} passing through  
the Maslov box $[-3, -7/12] \times [0, 10]$, where the endpoint $0$ can be 
included in a limiting sense and, due to exponential decay, $10$
is sufficient for approximating the right endstate at $+ \infty$. 
In the left-hand side of Figure \ref{eps0-fig},
the top two spectral curves (with spectral curves counted top to 
bottom as they cross the left shelf) appear to cross at about 
$(\lambda, x) = (-2.5, .68)$, and in order to better understand
the nature of this point, we depict an amplified view of it 
on the right-hand side of Figure \ref{eps0-fig}. We see that, 
in fact, the spectral curves do not cross, but rather the top
curve sharply rises and the second curve turns sharply to the 
right. 

\begin{figure}[ht]  
\begin{center}\includegraphics[%
  width=8cm,
  height=6cm]{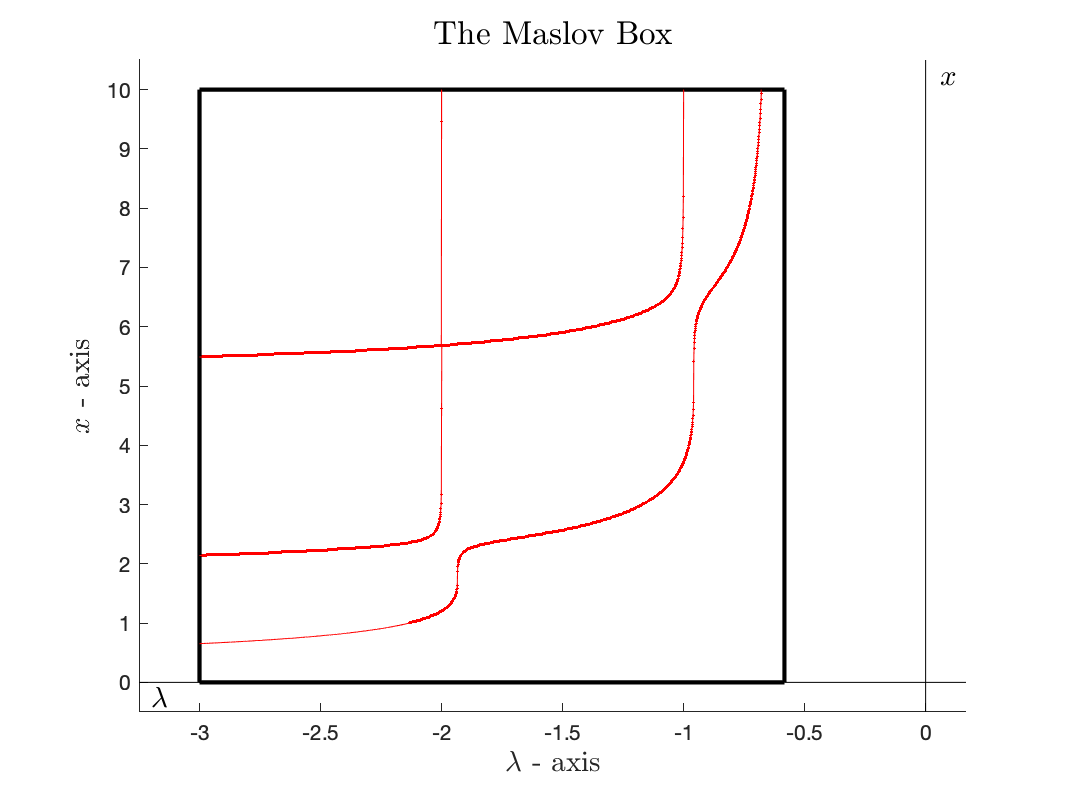}
  \includegraphics[%
  width=8cm,
  height=6cm]{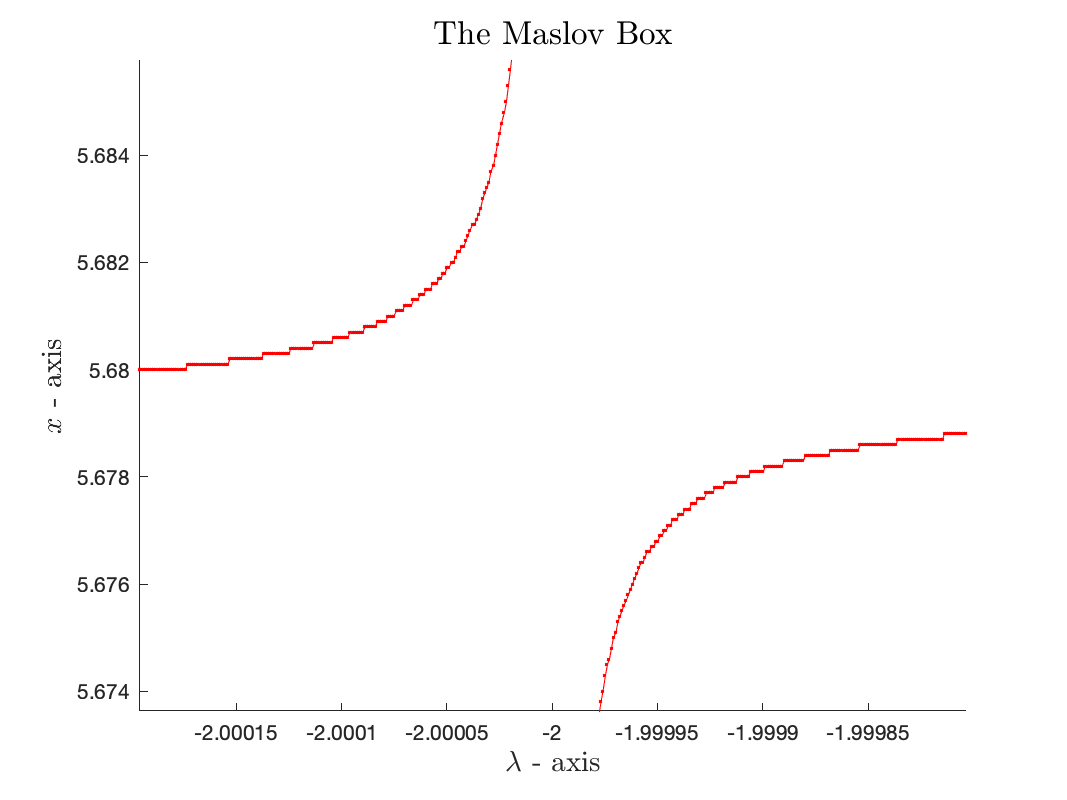}
\end{center}
\caption{The Full Maslov Box for $H$ on $[-3, -7/12] \times [0, 10]$ with $\delta = 0$. \label{eps0-fig} }
\end{figure}

\begin{figure}[ht]  
\begin{center}\includegraphics[%
  width=11cm,
  height=8cm]{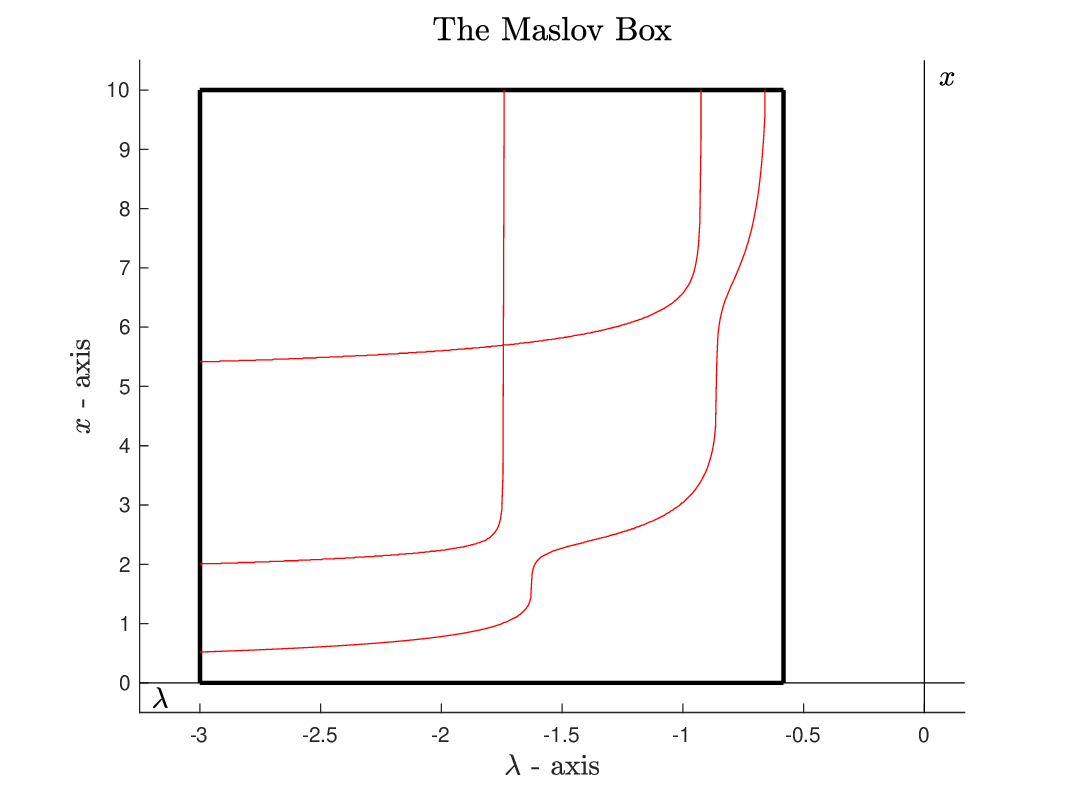}
\end{center}
\caption{The Full Maslov Box for $H$ on $[-3, -7/12] \times [0, 10]$ with $\delta = 1$. \label{eps1-fig}}
\end{figure}

\begin{figure}[ht]  
\begin{center}\includegraphics[%
  width=11cm,
  height=8cm]{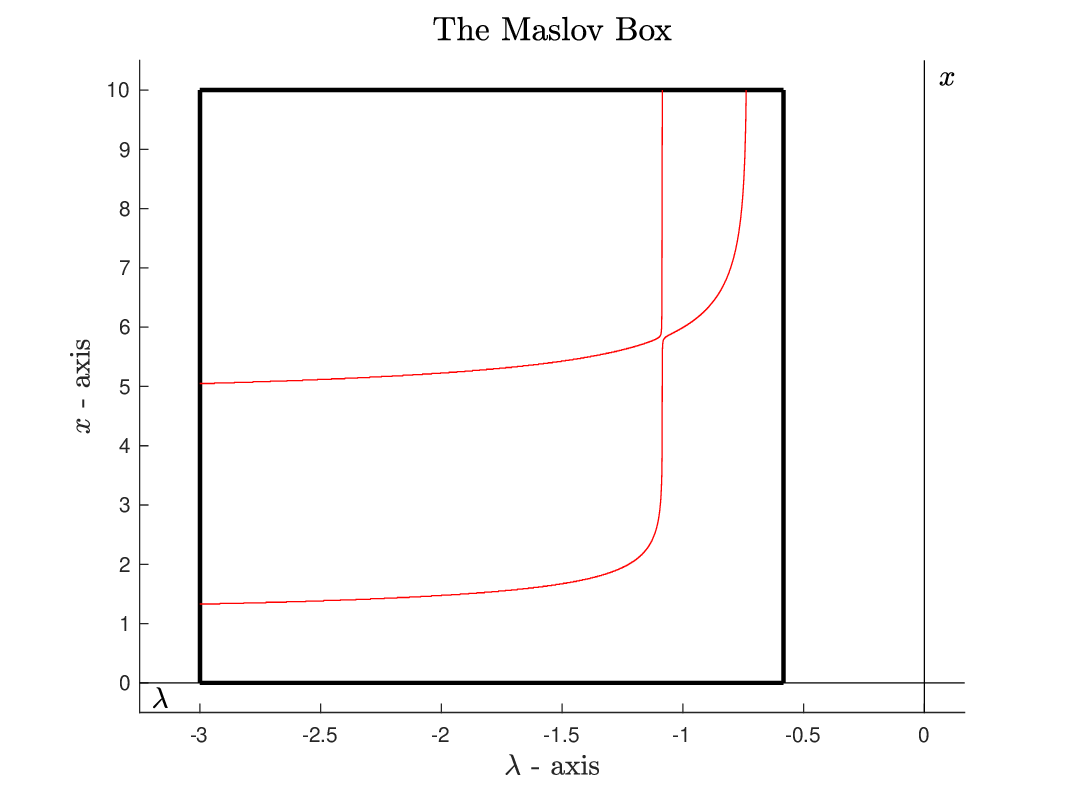}
\end{center}
\caption{The Full Maslov Box for $H$ on $[-3, -7/12] \times [0, 10]$ with $\delta = 5$. \label{eps5-fig} }
\end{figure}

Next, we begin to increase $\delta$ so that we obtain 
interaction between the terms $Q_1 (x; \delta) \lambda$ 
and $Q_2 (x) \lambda^2$. The details of these calculations 
are nearly identical to those for the case $\delta = 0$,
so we only give a summary of results. First, for $\delta = 1$,
we again count the number of eigenvalues in the interval 
$[-3, -7/12]$. In this case, we obtain crossing points along
the left shelf at approximately $x = .520$, $x = 1.999$, 
and $x = 5.410$. These values are lower than their counterparts
for the case $\delta = 0$, so it's natural to expect that
the eigenvalues have moved to the right. Again, in order to 
understand the full dynamics, we identify the spectral curves 
passing through the Maslov box $[-3, -7/12] \times [0, 10]$.
(See Figure \ref{eps1-fig}).
As expected, we find that the eigenvalues are now located 
approximately at values $\lambda = -1.740$, $\lambda = -.925$,
and $\lambda = -.660$ (working with an increment of $.005$
in the numerical calculations). As with the case $\delta = 0$,
the top two spectral curves don't cross. 

Last, we increase $\delta$ to $\delta = 5$, and again we start by 
computing the Maslov index on the left shelf. In this case, we find
only two intersections, at approximately $x = 1.323$ and $x = 5.043$.
These values are lower than $1.999$ and $5.410$ respectively, and 
the lowest crossing no longer occurs. We conclude that for $\delta = 5$
there are only two eigenvalues on the interval $[-3, - 7/12]$. Once more, 
in order to understand the full dynamics associated with this count, 
we identify the spectral curves 
passing through the Maslov box $[-3, -7/12] \times [0, 10]$.
(See Figure \ref{eps5-fig}). We see that the two eigenvalues 
reside at approximately $\lambda = -1.085$ and $\lambda = -.738$.

\subsubsection{Application to Incompressible Ideal MHD}
\label{mhd-section}

As a second application, we consider a model of 
incompressible magnetohydrodynamics (MHD), adapted 
from \cite{Freidberg1970}, with the precise form 
considered taken from \cite{GP2004}. As discussed 
in detail in \cite{Freidberg2014, GP2004}, the behavior
of a plasma can be modeled by the MHD equations 
\begin{equation} \label{mhd-system}
    \begin{aligned}
    \rho (v_t + (v \cdot \nabla)v) &= J \times B - \nabla p \\
    \rho_t + \nabla \cdot (\rho v) &= 0 \\
    B_t &= - \nabla \times E \\
    p_t + v \cdot \nabla p + \gamma p \nabla \cdot v &= 0,
    \end{aligned}
\end{equation}
along with the Maxwell relations $\nabla \times B = \mu_0 J$
and $\nabla \cdot B = 0$. (In Amp\`ere's Law, $\mu_0 \epsilon_0 E_t$
is taken to be negligible; see p. 38 in \cite{GP2004} for discussion.) 
For the current application, we will also assume 
perfect conductivity $E + v \times B = 0$ (see p. 70 
in \cite{GP2004}; also equation (5) in \cite{Freidberg1970}), 
along with incompressibility, $\nabla \cdot v = 0$.
Here, $\rho$ denotes fluid density, $v$ denotes fluid velocity,
$J$ denotes current density, $B$ denotes magnetic field 
strength, $E$ denotes electric field strength, 
and $p$ denotes fluid pressure. The constant $\mu_0$ is 
the usual permeability of free space, and the constant 
$\gamma$ is a ratio of specific heats taken in \cite{GP2004}
to be $\gamma = 5/3$. (See p. 53 of \cite{GP2004} for 
details.) 

For this application, we will focus on a plasma confined
to a cylinder with radius $b > 0$, in which case cylindrical 
coordinates $(r, \theta, z)$ become convenient. In what follows, we
will let $\hat{r}$, $\hat{\theta}$, and $\hat{z}$ denote
the usual basis elements for cylindrical coordinates, 
and we will write 
\begin{equation*}
    B = B_r \hat{r} + B_{\theta} \hat{\theta} + B_z \hat{z},
\end{equation*}
and similarly for other vectors. 
Our focus will be on the spectrum of the operator obtained 
when the system (\ref{mhd-system}) is linearized about 
a stationary solution depending only on the radial variable 
$r$, $(v_0 (r), \rho_0 (r), B_0 (r), p_0 (r))$, with 
also $J_0 (r) = \mu_0^{-1} \nabla \times B_0 (r)$. 
As in Section 6.1.2 of \cite{GP2004}, we focus on the 
{\it static equilibrium} case with $v_0$ taken to be 
identically zero. Linearization leads us to 
equation (9.47) from \cite{GP2004}, though since we're 
not aware of a reference in which 
this equation is derived in detail, we briefly include
such a derivation here. For this, we begin by observing 
that since $v_0 (r)$ is identically zero, the functions
$B_0 (r)$ and $p_0 (r)$ are seen to satisfy 
\begin{equation} \label{pressure-magnetism1}
    \nabla p_0 
    = J_0 \times B_0
    = \frac{1}{\mu_0} (\nabla \times B_0) \times B_0.
\end{equation}
The condition $\nabla \cdot B = 0$ implies that $B_r (r) = 0$,
and subsequently (\ref{pressure-magnetism1}) reduces to 
\begin{equation*} \label{pressure-magnetism2}
    (p_0 + \frac{1}{2} |B|^2)' 
    = - \frac{1}{r} B_{\theta}^2,
\end{equation*}
where $|B|$ denotes the length of $B$.
This still leaves considerable freedom in the choice of 
$p_0 (r)$, $B_{\theta} (r)$, and $B_z (r)$.  For our 
numerical calculations, we will take a particular 
triple of functions from Section 9.1.1 in \cite{GP2004},
but for our general derivation we will leave these functions
unspecified. 

We linearize (\ref{mhd-system}) about the stationary 
solution, writing our perturbations as $\rho_1 (r, \theta, z, t)$,
$v_1 (r, \theta, z, t)$, $p_1 (r, \theta, z, t)$,  
and $B_1 (r, \theta, z, t)$,
with also $J_1 (r, \theta, z, t) = (1/\mu_0) \nabla \times B_1 (r, \theta, z, t)$.  
In this way, we obtain the system of perturbations,  
\begin{equation*}
    \begin{aligned}
    \rho_0 {v_1}_t &= J_0 \times B_1 + J_1 \times B_0 - \nabla p_1 \\
    {\rho_1}_t &= 0 \\
    {B_1}_t &= \nabla \times (v_1 \times B_0) \\
    {p_1}_t &= 0.
    \end{aligned}
\end{equation*}
Following \cite{Freidberg2014, GP2004}, we introduce $\xi$ so that 
\begin{equation*}
    v_1 (r, \theta, z, t) = \xi_t (r, \theta, z, t),
\end{equation*}
allowing us to express the momentum equation as 
\begin{equation} \label{momentum-expressed-in-xi}
 \rho_0 {\xi}_{tt} = J_0 \times B_1 + J_1 \times B_0 - \nabla p_1.    
\end{equation}
Similarly, our equation for $B_1$ becomes 
\begin{equation} \label{B1-equation}
    {B_1}_t = \nabla \times (\xi_t \times B_0).
\end{equation}
We now look for solutions of the form 
\begin{equation*}
    \xi (r, \theta, z, t)
    = \tilde{\xi} (r) e^{i (m \theta + kz - \omega t)},
\end{equation*}
and similarly for $B_1 (r, \theta, z, t)$, where 
$m, k \in \mathbb{Z}$, and identifying appropriate values of 
$\omega$ will be the primary focus of our attention. With this notation,
(\ref{B1-equation}) becomes 
\begin{equation*}
    - i \omega B_1 = \nabla \times (-i \omega \xi \times B_0)
    \implies  B_1 = \nabla \times (\xi \times B_0),
\end{equation*}
and likewise $-i\omega \tilde{\rho}_1 = 0$ and 
$-i\omega \tilde{p}_1 = 0$. 
Upon substitution into the momentum equation we see that 
\begin{equation} \label{xi-F-equation}
    - \rho_0 \omega^2 \xi = \mathcal{F} (\xi) 
    := \frac{1}{\mu_0} (\nabla \times B_0) \times (\nabla \times (\xi \times B_0))
    + \frac{1}{\mu_0} (\nabla \times (\nabla \times (\xi \times B_0))).
\end{equation}

Following \cite{GP2004}, we introduce 
\begin{equation} \label{Q-defined}
    Q := \nabla \times (\xi \times B_0),
\end{equation}
so that (\ref{xi-F-equation}) can be expressed in the more 
compact form 
\begin{equation} \label{main-xi-equation}
- \rho_0 \omega^2 \xi
= \frac{1}{\mu_0} \Big{\{} (\nabla \times B_0) \times Q
+ (\nabla \times Q) \times B_0 \Big{\}}.
\end{equation}
By direct calculation we find that if 
\begin{equation*}
\tilde{\xi} (r) = \xi_r (r) \hat{r} + \xi_{\theta} (r) \hat{\theta} + \xi_z (r) \hat{r},     
\end{equation*}
then 
\begin{equation*}
    Q e^{- i (m\theta + kz - \omega t)}
    = i F \xi_r \hat{r} 
    + \Big(k |B| \eta - (B_{\theta} \xi_r)'\Big) \hat{\theta}
    - \Big(\frac{1}{r} (r B_z \xi_r)' + \frac{m |B| \eta}{r}\Big) \hat{z}.
\end{equation*}
where we've introduced (from \cite{GP2004}) the 
convenient notation
\begin{equation} \label{F-defined}
    F := \frac{m B_{\theta}}{r} + k B_z,
\end{equation}
and 
\begin{equation} \label{eta-defined}
    \eta = - \frac{B_z \xi_{\theta} - B_{\theta} \xi_x}{|B|}.
\end{equation}

The system (\ref{main-xi-equation}) comprises a matrix operator 
problem in the variables $\xi_r$, $\xi_{\theta}$, and $\xi_z$. 
In \cite{GP2004}, the authors show that this system can alternatively 
be expressed in terms $\xi_r$, $\eta$, and 
\begin{equation} \label{zeta-defined}
    \zeta := i \frac{B_{\theta} \xi_{\theta} + B_z \xi_z}{|B|},
\end{equation}
with the result taking the form 
\begin{equation} \label{m-system}
    M \begin{pmatrix} \xi_r \\ \eta \\ \zeta \end{pmatrix}
    = - \rho_0 \mu_0 \omega^2 
    \begin{pmatrix} \xi_r \\ \eta \\ \zeta \end{pmatrix},
\end{equation}
with 
\begin{equation*}
    M = \begin{pmatrix}
    \frac{d}{dr} \frac{|B|^2}{r} \frac{d}{dr} r - F^2 - r (\frac{B_{\theta}^2}{r^2})' &
    \frac{d}{dr} GB - \frac{2k B_{\theta} |B|}{r} & 0 \\
    - \frac{GB}{r} \frac{d}{dr} r - \frac{2k B_{\theta} |B|}{r} & - F^2 - G^2 & 0 \\
    0 & 0 & 0
    \end{pmatrix},
\end{equation*}
where we have introduced 
\begin{equation} \label{G-defined}
    G = \frac{m B_z}{r} - k B_{\theta}.
\end{equation}
For comparison, this is the analogue of equation (9.28) in \cite{GP2004}, appearing 
here in the incompressible case. (In \cite{GP2004}, the authors scale $\mu_0$ to 1, so it 
doesn't appear.) 

In order to reduce (\ref{m-system}) to a single second-order equation, we 
will restrict our attention to values $\omega^2$ larger than 0 
and on closed intervals disjoint from 
\begin{equation*}
    \ran \Big{(}\frac{F^2}{\mu_0 \rho_0}\Big{)} \Big|_{[0, b]},
\end{equation*}
where we recall that $b$ is the radius of the cylinder, so 
$r \in [0, b]$. We see from the third equation in (\ref{m-system}) 
that we must have $\zeta = 0$, requiring 
\begin{equation*}
    \xi_z = - \frac{B_{\theta}}{B_z} \xi_{\theta},
\end{equation*}
from which it follows immediately that 
\begin{equation} \label{eta-xi-theta-relation}
    \eta = i \frac{|B|}{B_z} \xi_{\theta}.
\end{equation}
In the incompressible case, we additionally have 
\begin{equation*}
    \nabla \cdot \xi 
    = \Big(\frac{1}{r} (r \xi_r (r))' + \frac{im}{r} \xi_{\theta (r)} + ik \xi_z (r)\Big)
    e^{i (m \theta + kz - \omega t)}
    = 0,
\end{equation*}
from which we see that 
\begin{equation} \label{xi-r-xi-theta-relation}
(r \xi_r (r))' = - i (m - k r \frac{B_{\theta}}{B_z}) \xi_{\theta} (r)    
= - i \frac{r}{B_z} G \xi_{\theta} (r). 
\end{equation}
If we combine (\ref{xi-r-xi-theta-relation}) with (\ref{eta-xi-theta-relation}),
we find that 
\begin{equation} \label{eta-xi-r-relation}
    \eta = - \frac{|B|}{r G} (r \xi_r (r))'
\end{equation}

Following \cite{GP2004}, we introduce the radial variable $\chi = r \xi_r$, for which 
the second equation in (\ref{m-system}) can be expressed as 
\begin{equation} \label{second-chi-equation}
    - \frac{G |B|}{r} \chi' - \frac{2kB_{\theta} |B|}{r^2} \chi 
    - (F^2 + G^2) \eta = - \rho_0 \mu_0 \omega^2 \eta.
\end{equation}
In order to obtain a single second-order equation for $\chi$, our 
strategy will be to use this equation to solve for $\eta$, and then 
to substitute the expression we obtain for $\eta$ into 
the first equation in (\ref{m-system}). More precisely, 
our goal is to obtain equation (9.47) from \cite{GP2004}, which 
requires the use of relations we have from incompressibility. 
To this end, we will use (\ref{eta-xi-r-relation}) for the right-hand 
side of (\ref{second-chi-equation}), and solve the resulting 
equation for $\eta$. Noting additionally the relation 
\begin{equation} \label{F-squared-plus-F-squared}
F^2 + G^2
= (\frac{m^2}{r^2} + k^2) |B|^2,
\end{equation}
we find that 
\begin{equation} \label{main-eta}
    \eta = - \frac{G^2 + \mu_0 \rho_0 \omega^2}{(\frac{m^2}{r^2} + k^2) |B| G r} \chi'
    - \frac{2 k B_{\theta}}{(\frac{m^2}{r^2} + k^2) |B| r^2} \chi.
\end{equation}
Alternatively, we can use the relation (\ref{eta-xi-r-relation}) 
in the term $- G^2 \eta$ in (\ref{second-chi-equation}) to see that 
\begin{equation} \label{second-eta-relation}
    (F^2 - \rho_0 \mu_0 \omega^2) \eta
     = - \frac{2kB_{\theta} |B|}{r^2} \chi
     \implies 
     \eta = - \frac{2kB_{\theta} |B|}{(F^2 - \rho_0 \mu_0 \omega^2)r^2} \chi.
\end{equation}
If we now combine (\ref{second-eta-relation}) with (\ref{eta-xi-r-relation})
we obtain a first-order equation for $\chi$, 
\begin{equation} \label{chi-prime-chi}
    \chi' = \frac{2kB_{\theta} G}{(F^2 - \rho_0 \mu_0 \omega^2) r} \chi.
\end{equation}

Upon substitution of (\ref{main-eta}) into the first equation in 
(\ref{m-system}), we obtain the relation 
\begin{equation} \label{penultimate-chi-equation}
\begin{aligned}
    \frac{d}{dr} &\Big[ \frac{F^2 - \rho_0 \mu_0 \omega^2}{m^2 + k^2 r^2} r \frac{d \chi}{dr} \Big]
    + \frac{2 k B_{\theta} \rho_0 \mu_0 \omega^2}{(m^2 + k^2 r^2) G} \chi' \\
    &+ \Big[\frac{\rho_0 \mu_0 \omega^2 - F^2}{r} - (\frac{B_{\theta}^2}{r^2})'
    - (\frac{2 k B_{\theta} G}{m^2 + k^2 r^2})' 
    + \frac{4 k^2 B_{\theta}^2}{(m^2 + k^2 r^2) r} \Big] \chi = 0. 
\end{aligned}
\end{equation}
In addition, if we replace $\chi'$ in (\ref{penultimate-chi-equation}) with 
(\ref{chi-prime-chi}), we obtain equation (9.47) from \cite{GP2004}, 
namely 
\begin{equation} \label{ultimate-chi-equation}
\begin{aligned}
    \frac{d}{dr} &\Big[ \frac{F^2 - \rho_0 \mu_0 \omega^2}{m^2 + k^2 r^2} r \frac{d \chi}{dr} \Big]
    + \Big[\frac{\rho_0 \mu_0 \omega^2 - F^2}{r} - (\frac{B_{\theta}^2}{r^2})' \\
    &- (\frac{2 k B_{\theta} G}{m^2 + k^2 r^2})' 
    + \frac{4 k^2 B_{\theta}^2 F^2}{(F^2 - \rho_0 \mu_0 \omega^2)(m^2 + k^2 r^2) r} \Big] \chi = 0. 
\end{aligned}    
\end{equation}
In order to ensure that 
the denominator $m^2 + k^2 r^2$ remains bounded away from 0, we will focus
on cases for which $m$ is non-zero. We will see that (\ref{ultimate-chi-equation})
is limit-point at $x = 0$, and so no boundary condition is required, while 
for the boundary condition at the regular endpoint $x = b$ we will take 
the Dirichlet condition $\chi (b) = 0$. (See Section 9.2.1 in \cite{GP2004}.)

For the calculations that follow, we assume  
$B_{\theta} (r)$, $B_z (r)$, $\rho_0 (r)$, and $p_0 (r)$
are all analytic in $r$ in an open ball containing the interval 
$[0, b]$, and also that $B_{\theta} (0) = 0$. 
For notational consistency, we begin by expressing (\ref{ultimate-chi-equation}) 
in the form used throughout the current analysis, namely by replacing $r$ with $x$ 
and $\omega^2$ with $- \lambda$ (for which we will then be interested in 
$\lambda < 0$), so that (\ref{ultimate-chi-equation}) can be 
expressed as 
\begin{equation} \label{mhd-second-order}
    - (P(x;\lambda) \phi')' + V(x; \lambda) \phi = 0,
\end{equation}
where 
\begin{equation} \label{mhd-P}
P(x; \lambda) = \frac{\mu_0 \rho_0 \lambda + F(x)^2}{m^2 + k^2 x^2} x,
\end{equation}
and 
\begin{equation} \label{mhd-V}
    \begin{aligned}
    V(x; \lambda) &= \frac{1}{x} (\mu_0 \rho_0 \lambda + F(x)^2)
    + \Big(\frac{B_{\theta}^2}{x^2}\Big)' 
    - \frac{4 k^2 B_{\theta}^2F^2}{x (m^2 + k^2x^2)(\mu_0 \rho_0 \lambda + F(x)^2)}
    + \Big( \frac{2 k B_{\theta} G}{m^2 + k^2 x^2} \Big)'.
    \end{aligned}
\end{equation}
We now express (\ref{mhd-second-order}) in our standard form (\ref{hammy})
by setting $y_1 = \phi$ and $y_2 = - P(x; \lambda) \phi'$, noting the 
sign choice on $y_2$. We obtain  
\begin{equation} \label{mhd-B}
Jy' = \mathbb{B} (x; \lambda) y,
\quad
\mathbb{B} (x; \lambda) 
= \begin{pmatrix}
V(x; \lambda) & 0 \\
0 & -P(x; \lambda)^{-1}
\end{pmatrix}.
\end{equation}
For notational convenience, we will set $Q (x; \lambda) = - P (x; \lambda)^{-1}$. 

In order to employ Theorem \ref{regular-singular-theorem} in this case, 
we need to check 
Assumptions {\bf (A)} through {\bf (F)}, along with the additional
assumptions of Theorem \ref{regular-singular-theorem}. 
Starting with Assumption {\bf (A)},
we first observe that since $F$ is continuous, the range of 
$F^2/(\mu_0 \rho_0)|_{[0,b]}$ is a 
closed set, say $[c, d]$. We will apply Theorem \ref{regular-singular-theorem} 
on a fixed interval $[\lambda_1, \lambda_2]$, $\lambda_1 < \lambda_2 < 0$ 
so that $[\lambda_1, \lambda_2] \cap [c, d] = \emptyset$. On such 
intervals, the quantity $\mu_0 \rho_0 \lambda + F(x)^2$ is 
bounded below. It follows that there exists a constant $C_0$
sufficiently large so that for all $\lambda \in [\lambda_1, \lambda_2]$
we have 
\begin{equation*}
    |B (x; \lambda)| \le b_0 (x) := \frac{C_0}{x}
    \quad \forall \, x \in (0, b).
\end{equation*}
 Next, we compute 
\begin{equation} \label{mhd-B-derivative}
    \mathbb{B}_{\lambda} (x; \lambda)
    = \begin{pmatrix}
    V_{\lambda} (x; \lambda) & 0 \\
    0 & Q_{\lambda} (x; \lambda) 
    \end{pmatrix},
\end{equation}
where 
\begin{equation} \label{V-lambda-Q-lambda}
    \begin{aligned}
    V_{\lambda} (x; \lambda)
    &= \frac{\rho_0 \mu_0}{x} 
    + \frac{4 \rho_0 \mu_0 B_{\theta} (x)^2 F(x)^2}{x (m^2 + k^2 x^2) (\rho_0 \mu_0 \lambda + F(x)^2)} \\
    Q_{\lambda} (x; \lambda)
    &= \frac{\rho_0 \mu_0 (m^2 + k^2 x^2)}{x (\rho_0 \mu_0 \lambda + F(x)^2)^2}.
    \end{aligned}
\end{equation}
We readily see that there exists a constant $C_1$
sufficiently large so that for all $\lambda \in [\lambda_1, \lambda_2]$
we have 
\begin{equation*}
    |B (x; \lambda)| \le b_1 (x) := \frac{C_1}{x}
    \quad \forall \, x \in (0, b),
\end{equation*}
completing the verification of Assumption {\bf (A)}. 

For Assumption {\bf (B)}, we observe that 
\begin{equation} \label{mhd-B-lambda-relation}
    (\mathbb{B}_{\lambda} (x; \lambda) y, y) 
    = V_{\lambda} (x; \lambda) |y_1|^2 + Q_{\lambda} (x; \lambda) |y_2|^2,
\end{equation}
and it's clear from (\ref{V-lambda-Q-lambda}) that this sum can only 
vanish on an interval $[c, d]$, $c < d$, if $y_1 (x)$ and $y_2 (x)$ both 
vanish on this interval. This establishes Assumption {\bf (B)}. 

Turning to Assumption {\bf (C)},  our starting point is to be clear
about the nature of $L^2_{\mathbb{B}_{\lambda}} ((0,b),\mathbb{C}^2)$
for this application.
By definition, this is the collection of Lebesgue measurable 
functions $f = \genfrac(){0pt}{1}{f_1}{f_2}$ for which 
\begin{equation*}
    \int_0^b (\mathbb{B}_{\lambda} (x; \lambda) f(x), f(x)) dx 
    < \infty.
\end{equation*}
Using (\ref{mhd-B-lambda-relation}), we immediately see that
$L^2_{\mathbb{B}_{\lambda}} ((0,b),\mathbb{C}^2)$ is equivalent 
to the collection of Lebesgue measurable functions $f = \genfrac(){0pt}{1}{f_1}{f_2}$
for which 
\begin{equation*}
    \int_0^b \frac{1}{x} |f_i (x)|^2 dx < \infty,
    \quad i = 1, 2.
\end{equation*}
This space, which we will denote $L^2_{1/x} ((0, b), \mathbb{C}^2)$
is clearly independent of $\lambda$, establishing the first part of 
Assumption {\bf (C)}.
For the second part of Assumption {\bf (C)}, we need to check 
that the maximal domain $\mathcal{D}_M (\lambda)$ is independent 
of $\lambda$. For (\ref{mhd-B}), we can 
characterize $\mathcal{D}_M (\lambda)$ as the collection of 
functions
\begin{equation*}
    y \in \AC_{\loc} ((0, b], \mathbb{C}^2) 
    \cap L^2_{\mathbb{B}_{\lambda}} ((0,b), \mathbb{C}^2)
\end{equation*}
for which there exists 
$f = \genfrac(){0pt}{1}{f_1}{f_2} \in L^2_{\mathbb{B}_{\lambda}} ((0,b), \mathbb{C}^2)$ 
so that 
\begin{equation} \label{mhd-y-system}
    \begin{aligned}
    - y_2' - V(x; \lambda) y_1 &= V_{\lambda} (x; \lambda) f_1 (x) \\
    y_1' - Q (x; \lambda) y_2 &= Q_{\lambda} (x; \lambda) f_2 (x),
    \end{aligned}
\end{equation}
for a.e. $x \in (0, b)$. Given that $y \in \mathcal{D}_M (\lambda)$
for some $\lambda \in [\lambda_1, \lambda_2]$, we need to verify 
that for any other $\tilde{\lambda} \in [\lambda_1, \lambda_2]$
we have $y \in \mathcal{D}_M (\tilde{\lambda})$. Specifically, we 
need to show that for any such $\tilde{\lambda} \in [\lambda_1, \lambda_2]$
there exists 
$\tilde{f} = \genfrac(){0pt}{1}{\tilde{f}_1}{\tilde{f}_2} \in L^2_{\mathbb{B}_{\lambda}} ((0,b), \mathbb{C}^2)$ 
so that 
\begin{equation} \label{mhd-y-system-tilde}
    \begin{aligned}
    - y_2' - V(x; \tilde{\lambda}) y_1 &= V_{\lambda} (x; \tilde{\lambda}) \tilde{f}_1 (x) \\
    y_1' - Q (x; \tilde{\lambda}) y_2 &= Q_{\lambda} (x; \tilde{\lambda}) \tilde{f}_2 (x).
    \end{aligned}
\end{equation}
To this end, we can write (\ref{mhd-y-system}) as 
\begin{equation*}
    \begin{aligned}
    - y_2' - V(x; \tilde{\lambda}) y_1 
    &= V_{\lambda} (x; \lambda) f_1 + (V(x; \lambda) - V (x; \tilde{\lambda})) y_1 \\
    y_1' - Q (x; \tilde{\lambda}) y_2 
    &= Q_{\lambda} (x; \lambda) f_2 + (Q (x; \lambda) - Q (x; \tilde{\lambda})) y_2,
    \end{aligned}
\end{equation*}
from which we see by inspection that in order to obtain the relations (\ref{mhd-y-system-tilde}), we need to 
have 
\begin{equation*}
    \begin{aligned}
    V_{\lambda} (x; \tilde{\lambda}) \tilde{f}_1 (x)
    &= V_{\lambda} (x; \lambda) f_1 (x) + (V(x; \lambda) - V (x; \tilde{\lambda})) y_1 \\
    Q_{\lambda} (x; \tilde{\lambda}) \tilde{f}_2 (x)
    &= Q_{\lambda} (x; \lambda) f_2 (x) + (Q (x; \lambda) - Q (x; \tilde{\lambda})) y_2,
    \end{aligned}
\end{equation*}
for a.e. $x \in (0, b)$. For the first, we can write 
\begin{equation*}
    \tilde{f}_1 (x) = V_{\lambda} (x; \tilde{\lambda})^{-1} V_{\lambda} (x; \lambda) f_1 (x)
    + V_{\lambda} (x; \tilde{\lambda})^{-1} (V(x; \lambda) - V (x; \tilde{\lambda})) y_1 (x),
\end{equation*}
where by direct calculation we find the relations 
\begin{equation} \label{V-lambda-inverse-V-lambda}
\begin{aligned}
V_{\lambda} (x; \tilde{\lambda})^{-1} V_{\lambda} (x; \lambda)
& = \Big{\{} 1 + \frac{4 k^2 B_{\theta} (x)^2 F(x)^2}{(m^2+k^2 x^2) (\mu_0 \rho_0 \tilde{\lambda} + F(x)^2)^2} \Big{\}}^{-1} \\
&\quad \times \Big{\{} 1 + \frac{4 k^2 B_{\theta} (x)^2 F(x)^2}{(m^2+k^2 x^2) (\mu_0 \rho_0 \lambda + F(x)^2)^2} \Big{\}},
\end{aligned}
\end{equation}
and 
\begin{equation*}
\begin{aligned}
 V_{\lambda} (x; \tilde{\lambda})^{-1} (V(x; \lambda) - V (x; \tilde{\lambda}))
 &= (\lambda - \tilde{\lambda}) \Big{\{} 1 + \frac{4 k^2 B_{\theta} (x)^2 F(x)^2}{(m^2+k^2 x^2) (\mu_0 \rho_0 \tilde{\lambda} + F(x)^2)^2} \Big{\}}^{-1} \\
 &\quad \times \Big{\{} 1 + \frac{4 k^2 B_{\theta} (x)^2 F(x)^2}{(m^2+k^2 x^2) (\mu_0 \rho_0 \lambda + F(x)^2) (\mu_0 \rho_0 \tilde{\lambda} + F(x)^2)} \Big{\}}.
 \end{aligned}
\end{equation*}
In each case, we see that the expression on the right-hand side is uniformly bounded 
both above and below, so $\tilde{f}_1$ inherits the integrability properties of 
$f_1$ and $y_1$, which is precisely what we need. 

For $\tilde{f}_2$, we need 
\begin{equation*}
    \tilde{f}_2 
    = Q_{\lambda} (x; \tilde{\lambda})^{-1} Q_{\lambda} (x; \lambda) f_2
    + Q_{\lambda} (x; \tilde{\lambda})^{-1} (Q(x; \lambda) - Q (x; \tilde{\lambda})) y_2,
\end{equation*}
where 
\begin{equation*}
 Q_{\lambda} (x; \tilde{\lambda})^{-1} Q_{\lambda} (x; \lambda)
 = \frac{(\rho_0 \mu_0 \tilde{\lambda} + F(x)^2)^2}{(\rho_0 \mu_0 \lambda + F(x)^2)^2},
\end{equation*}
and 
\begin{equation} \label{Q-lambda-inverse-Q-lambda}
    Q_{\lambda} (x; \tilde{\lambda})^{-1} (Q(x; \lambda) - Q (x; \tilde{\lambda}))
    = (\lambda - \tilde{\lambda}) \frac{\rho_0 \mu_0 \tilde{\lambda} + F(x)^2}{\rho_0 \mu_0 \lambda + F(x)^2}.
\end{equation}
The expressions on the right-hand sides are uniformly bounded above and 
below, so $\tilde{f}_2$ inherits the integrability properties of 
$f_2$ and $y_2$. This completes the verification of Assumption {\bf (C)}.

As with our previous applications, we set Assumption {\bf (D)} aside until
we've verified Assumption {\bf (E)}. For this latter verification, we begin 
by identifying $\mathcal{E} (x; \lambda, \lambda_*)$ in this case. Taking 
advantage of the diagonal structure of $\mathbb{B} (x; \lambda)$, we can 
choose $\mathcal{E}_{12} (x; \lambda, \lambda_*) \equiv 0$ and 
$\mathcal{E}_{21} (x; \lambda, \lambda_*) \equiv 0$, and take 
\begin{equation} \label{mhd-E11}
    \begin{aligned}
    \mathcal{E}_{11} (x; \lambda, \lambda_*)
    &= V_{\lambda} (x; \lambda_*)^{-1} (V(x; \lambda) - V(x; \lambda_*)) \\
    &= (\lambda - \lambda_*) 
    \Big{\{} 1 + \frac{4k^2B_{\theta} (x)^2 F(x)^2}{(m^2 + k^2 x^2)(\mu_0 \rho_0 \lambda + F(x)^2)^2} \Big{\}}^{-1} \\
    &\quad \times \Big{\{} 1 + \frac{4k^2B_{\theta} (x)^2 F(x)^2}
    {(m^2 + k^2 x^2)(\mu_0 \rho_0 \lambda + F(x)^2) (\mu_0 \rho_0 \lambda_* + F(x)^2)} \Big{\}},
    \end{aligned}
\end{equation}
and 
\begin{equation} \label{mhd-E22}
    \begin{aligned}
    \mathcal{E}_{22} (x; \lambda, \lambda_*)
    &= Q_{\lambda} (x; \lambda_*)^{-1} (Q (x; \lambda) - Q (x; \lambda_*)) \\
    &= (\lambda - \lambda_*) \frac{\rho_0 \mu_0 \lambda_* + F(x)^2}{\rho_0 \mu_0 \lambda + F(x)^2}.
    \end{aligned}
\end{equation}
We observe here that $\mathcal{E}_{11} (x; \lambda, \lambda_*)$ and 
$\mathcal{E}_{22} (x; \lambda, \lambda_*)$ can be expressed in the form 
\begin{equation*}
    \mathcal{E}_{ii} (x; \lambda, \lambda_*) = (\lambda - \lambda_*) \Theta_{ii} (x; \lambda, \lambda_*),
    \quad i = 1, 2,
\end{equation*}
where we see from (\ref{mhd-E11}) and (\ref{mhd-E22}) that the functions  
$\Theta_{ii} (x; \lambda, \lambda_*)$, $i = 1, 2$, are uniformly bounded 
above and below on $[\lambda_1, \lambda_2] \times [0, b]$. In addition, 
\begin{equation*}
\Theta_{ii} (x; \lambda, \lambda_*) \to 1, \quad \lambda \to \lambda_*    
\end{equation*}
uniformly for $x \in [0, b]$. 

Now we're prepared to check the four parts of Assumption {\bf (E)}. For (i), we need to 
check that $\mathcal{E} (x; \lambda, \lambda_*)$ is bounded as a multiplication 
operator mapping $L^2_{\mathbb{B}_{\lambda}} ((0, b), \mathbb{C}^2)$ to itself, and 
for (ii) we need to verify that in fact the norm of this operator is $\mathbf{o} (1)$
as $\lambda \to \lambda_*$. We check these conditions together by computing 
\begin{equation*}
\begin{aligned}
    \| \mathcal{E} (\cdot; &\lambda, \lambda_*) \|
    = \sup_{\|f\|_{\mathbb{B}_{\lambda}} = 1} \| \mathcal{E} (\cdot; \lambda, \lambda_*) f (\cdot) \|_{\mathbb{B}_{\lambda}} \\
    &= \sup_{\|f\|_{\mathbb{B}_{\lambda}} = 1} \Big( \int_0^b \mathbb{B}_{\lambda} (x; \lambda_*) \mathcal{E} (x; \lambda, \lambda_*) f(x),
    \mathcal{E} (x; \lambda, \lambda_*) f(x)) dx \Big)^{1/2} \\
    &= |\lambda - \lambda_*| \sup_{\|f\|_{\mathbb{B}_{\lambda}} = 1} \Big(\int_0^b V_{\lambda} (x; \lambda) 
    |\Theta_{11} (x; \lambda, \lambda_*) f_1 (x)|^2
    + Q_{\lambda} (x; \lambda) |\Theta_{22} (x; \lambda, \lambda_*) f_2 (x)|^2\Big)^{1/2} \\
    &\le C |\lambda - \lambda_*| \sup_{\|f\|_{\mathbb{B}_{\lambda}} = 1} \Big(\int_0^b V_{\lambda} (x; \lambda) |f_1 (x)|^2
    + Q_{\lambda} (x; \lambda) |f_2 (x)|^2\Big)^{1/2} \\
    &= C |\lambda - \lambda_*| \sup_{\|f\|_{\mathbb{B}_{\lambda}} = 1} \| f \|_{\mathbb{B}_{\lambda}}
    = C |\lambda - \lambda_*|,
\end{aligned}
\end{equation*}
where 
\begin{equation*}
    C = \max_{[\lambda_1, \lambda_2]\times [0, b]}
    \Big{\{}|\Theta_{11} (x; \lambda)| + |\Theta_{22} (x; \lambda)|\Big{\}}. 
\end{equation*}
We see that
\begin{equation*}
    \| \mathcal{E} (\cdot; \lambda, \lambda_*) \| \le C |\lambda - \lambda_*|,
    \quad \lambda \to \lambda_*,
\end{equation*}
which is more than we require for (i) and (ii). 

For (iii), we need to show that the matrix function $\mathcal{E} (x; \lambda, \lambda_*)$
is continuously differentiable in $\lambda$ for $\lambda \in I_{\lambda_*, r}$, and that 
the map $\lambda \mapsto \mathcal{E} (x; \lambda, \lambda_*)$ is continuously 
differentiable as a map from $I_{\lambda_*, r_*}$ to 
$\mathcal{B} (L^2_{\mathbb{B}_{\lambda}} ((0, b), \mathbb{C}^2))$. 

First, continuous differentiability of $\mathcal{E} (x; \lambda, \lambda_*)$
in $\lambda$ is clear from the explicit relations (\ref{mhd-E11}) and 
(\ref{mhd-E22}). Specifically, we see that 
\begin{equation} \label{mhd-mathcal-E-derivatives}
    \begin{aligned}
    \partial_{\lambda} \mathcal{E}_{11} (x; \lambda, \lambda_*)
    &= V_{\lambda} (x; \lambda_*)^{-1} V_{\lambda} (x; \lambda) \\
    \partial_{\lambda} \mathcal{E}_{22} (x; \lambda, \lambda_*)
    &= Q_{\lambda} (x; \lambda_*)^{-1} Q_{\lambda} (x; \lambda),
    \end{aligned}
\end{equation}
for which we respectively have explicit expressions (after replacing $\tilde{\lambda}$
with $\lambda_*$) from (\ref{V-lambda-inverse-V-lambda}) 
and (\ref{Q-lambda-inverse-Q-lambda}). 

In order to verify that $\mathcal{E} (\cdot; \lambda, \lambda_*)$
is differentiable as a map from $I_{\lambda_*, r} \subset [\lambda_1, \lambda_2]$
to $\mathcal{B} (L^2_{\mathbb{B}_{\lambda}} ((0, b), \mathbb{C}^2))$, we let
$\epsilon (h; \lambda)$ denote the operator so that 
\begin{equation*}
\mathcal{E} (\cdot; \lambda+h, \lambda_*)
= \mathcal{E} (\cdot; \lambda, \lambda_*) + \mathcal{E}_{\lambda} (\cdot; \lambda, \lambda_*) h
+ \epsilon (h; \lambda, \lambda_*) h,
\end{equation*}
and our goal is to check that 
\begin{equation} \label{mhd-limit-for-epsilon}
    \lim_{h \to 0} \| \epsilon (h; \lambda, \lambda_*) \| = 0.
\end{equation}
Computing directly, we find 
\begin{equation*}
    \begin{aligned}
     \epsilon_{11} (h; \lambda, \lambda_*)
     &= \frac{1}{h} \Big{\{} \mathcal{E}_{11} (x; \lambda + h, \lambda_*)
     - \mathcal{E}_{11} (x; \lambda, \lambda_*) 
     - \partial_{\lambda} \mathcal{E}_{11} (x; \lambda, \lambda_*) h \Big{\}} \\
     &= \Big{\{} 1 + \frac{4 k^2 B_{\theta} (x)^2 F(x)^2}
     {(m^2+k^2 x^2) (\mu_0 \rho_0 \lambda_* + F(x)^2)^2} \Big{\}}^{-1} 
     \frac{4 k^2 B_{\theta} (x)^2 F(x)^2}{(m^2+k^2 x^2) (\mu_0 \rho_0 \lambda + F(x)^2)} \\
 &\quad \times \Big{\{} \frac{1}{(\mu_0 \rho_0 (\lambda + h) + F(x)^2)}
 - \frac{1}{(\mu_0 \rho_0 \lambda + F(x)^2)} \Big{\}},
    \end{aligned}
\end{equation*}
where we mean here that $\epsilon_{11} (h; \lambda, \lambda_*)$ acts as a multiplication
operator mapping $L^2_{\mathbb{B}_{\lambda}} ((0, b), \mathbb{C}^2)$ to itself. We see
directly from this relation that (\ref{mhd-limit-for-epsilon}) must hold. The verification
for $\epsilon_{22} (h; \lambda, \lambda_*)$ is similar, allowing us to conclude that 
the map $\lambda \mapsto \mathcal{E} (x; \lambda, \lambda_*)$ is differentiable as a map 
from $I_{\lambda_*, r}$ to $\mathcal{B} (L^2_{\mathbb{B}_{\lambda}} ((0, b), \mathbb{C}^2))$,
with derivative the multiplication operator $\mathcal{E}_{\lambda} (\cdot; \lambda, \lambda_*)$. 

For the final part of Assumption {\bf (E)}(iii), we need to show that 
$\mathcal{E}_{\lambda} (\cdot; \lambda, \lambda_*)$ is continuous as a map
from $I_{\lambda_*, r}$ to $\mathcal{B} (L^2_{\mathbb{B}_{\lambda}} ((0, b), \mathbb{C}^2))$. 
For this, we fix $\lambda_0 \in [\lambda_1, \lambda_2]$ and for any 
$\lambda \in [\lambda_1, \lambda_2]$ consider the difference 
\begin{equation*}
\mathcal{E}_{\lambda} (\cdot; \lambda, \lambda_*)
- \mathcal{E}_{\lambda} (\cdot; \lambda_0, \lambda_*).
\end{equation*}
Focusing again on $\mathcal{E}_{11}$, we can use (\ref{mhd-mathcal-E-derivatives})
to compute 
\begin{equation*}
\begin{aligned}
\partial_{\lambda} \mathcal{E}_{11} (\cdot; \lambda, \lambda_*)
&- \partial_{\lambda} \mathcal{E}_{11} (\cdot; \lambda_0, \lambda_*)
= - (\lambda - \lambda_0)  \Big{\{} 1 + \frac{4k^2B_{\theta} (x)^2 F(x)^2}{(m^2 + k^2 x^2)(\mu_0 \rho_0 \lambda_* + F(x)^2)^2} \Big{\}}^{-1} \\
&\times \Big{\{} \frac{4k^2 B_{\theta} (x)^2 F(x)^2}{m^2 + k^2 x^2} \Big{\}}
\Big{\{} \frac{\mu_0^2 \rho_0^2 (\lambda + \lambda_0) + 2 \mu_0 \rho_0 F(x)^2}
{(\mu_0 \rho_0 \lambda + F(x)^2)^2 (\mu_0 \rho_0 \lambda_0 + F(x)^2)^2} \Big{\}}.
\end{aligned}
\end{equation*}
The key observation here is that this difference goes to 0 uniformly in $x$
as $\lambda$ approaches $\lambda_0$, and this is enough to allow us to 
conclude that $\mathcal{E}_{\lambda} (\cdot; \lambda, \lambda_*)$ is continuous as a map
from $I_{\lambda_*, r}$ to $\mathcal{B} (L^2_{\mathbb{B}_{\lambda}} ((0, b), \mathbb{C}^2))$. 
The analysis of $\partial_{\lambda} \mathcal{E}_{11} (\cdot; \lambda, \lambda_*)$ is 
similar, finalizing the verification of Assumption {\bf (E)}(iii).

This brings us to Assumption {\bf (E)}(iv). For this, we can compute 
\begin{equation*}
    |f(x)^* \mathbb{B}_{\lambda} (x; \lambda) g(x)|
    \le |f_1 (x)^* V_{\lambda} (x; \lambda) g_1 (x)|
    + |f_2 (x)^* Q_{\lambda} (x; \lambda) g_2 (x)|,
\end{equation*}
and Assumption {\bf (E)}(iv) follows from uniformity of 
$V_{\lambda} (x; \lambda)$ and $Q_{\lambda} (x; \lambda)$
in $\lambda$. This completes the verification of Assumption 
{\bf (E)}. 

Returning to Assumption {\bf (D)}, according to Lemma
\ref{assumption-d-lemma} we only need to check that 
\begin{equation*}
    \| \mathcal{E}_{\lambda} (\cdot; \lambda, \lambda_0) - I \|
    = \mathbf{o} (1), \quad \lambda \to \lambda_0.
\end{equation*}
Using (\ref{mhd-mathcal-E-derivatives}), we find 
\begin{equation*}
\begin{aligned}
\partial_{\lambda} \mathcal{E}_{11} (\cdot; \lambda, \lambda_0)
&- 1
= - (\lambda - \lambda_0)  \Big{\{} 1 + \frac{4k^2B_{\theta} (x)^2 F(x)^2}{(m^2 + k^2 x^2)(\mu_0 \rho_0 \lambda_0 + F(x)^2)^2} \Big{\}}^{-1} \\
&\times \Big{\{} \frac{4k^2 B_{\theta} (x)^2 F(x)^2}{m^2 + k^2 x^2} \Big{\}}
\Big{\{} \frac{\mu_0^2 \rho_0^2 (\lambda + \lambda_0) + 2 \mu_0 \rho_0 F(x)^2}
{(\mu_0 \rho_0 \lambda + F(x)^2)^2 (\mu_0 \rho_0 \lambda_0 + F(x)^2)^2} \Big{\}}.
\end{aligned}
\end{equation*}
We see that in fact
\begin{equation*}
    \|\mathcal{E}_{\lambda} (\cdot; \lambda, \lambda_0) - I\|
    = \mathbf{O} (|\lambda - \lambda_0|),
\end{equation*}
which is substantially stronger than the required condition.

This concludes the verification of Assumptions {\bf (A)} through 
{\bf (E)}. For Assumption {\bf (F)}, we have 
\begin{equation*}
    \mathbb{B} (x; \lambda_2) - \mathbb{B} (x; \lambda_1)
    = \begin{pmatrix}
    V(x; \lambda_2) - V(x; \lambda_1) & 0 \\
    0 & Q(x; \lambda_2) - Q(x; \lambda_1)
    \end{pmatrix},
\end{equation*}
where 
\begin{equation*}
    \begin{aligned}
     V(x; \lambda_2) - V(x; \lambda_1)
     &= \frac{\mu_0 \rho_0 (\lambda_2 - \lambda_1)}{x}
     \Big{\{} 1 + \frac{4 k^2 B_{\theta} (x)^2 F(x)^2}
     {(m^2+k^2x^2)(\mu_0 \rho_0 \lambda_2 + F(x)^2)(\mu_0 \rho_0 \lambda_1 + F(x)^2)} \Big{\}} \\
     Q(x; \lambda_2) - Q(x; \lambda_1)
     &= \frac{\mu_0 \rho_0 (\lambda_2 - \lambda_1)}{x}
     \Big{\{} \frac{m^2 + k^2 x^2}{(\mu_0 \rho_0 \lambda_2 + F(x)^2)(\mu_0 \rho_0 \lambda_1 + F(x)^2)} \Big{\}}.
    \end{aligned}
\end{equation*}
Since the quantity $\mu_0 \rho_0 \lambda + F(x)^2$ is taken to have the same 
sign for all $\lambda \in [\lambda_1, \lambda_2]$, we have that 
\begin{equation*}
(\mu_0 \rho_0 \lambda_2 + F(x)^2)(\mu_0 \rho_0 \lambda_1 + F(x)^2) > 0.    
\end{equation*}
It follows immediately that the condition from Remark \ref{assumption-F-remark} 
follows, so Assumption {\bf (F)} is seen to hold. 

We turn next to the assumption in Theorem \ref{regular-singular-theorem} 
on asymptotic intersections. 
For this application, the singular point is on the left, so the condition 
becomes existence of a value $c_0 > 0$ sufficiently small so that 
for all $0 < c < c_0$
\begin{equation*}
    \ell_0 (c; \lambda) \cap \ell_b (c; \lambda_2)
    = \{0\},\quad \forall\, \lambda \in (\lambda_1, \lambda_2],
\end{equation*}
where $\ell_0 (c; \lambda)$ denotes the one-dimensional space 
of solutions to (\ref{mhd-B}) that lie left in $(0, b)$, and 
$\ell_b (c; \lambda_2)$ denotes the one-dimensional space of 
solutions to (\ref{mhd-B}) that satisfy the Dirichlet boundary
condition at $x = b$. 

In order to verify this, we observe that for all $\lambda \in [\lambda_1, \lambda_2]$,
our equation (\ref{mhd-second-order}) has a regular singularity at 
$x = 0$. This allows us to construct linearly independent series solutions
\begin{equation} \label{mhd-regular-singular-solutions}
    \begin{aligned}
    \phi (x; \lambda) 
    &= x^{|m|} \Big(1+\sum_{j=1}^{\infty} a_j (\lambda) x^j \Big) \\
    \varphi (x; \lambda)
    &= b_0 (\lambda) \phi_1 (x; \lambda) \ln x
    + x^{-|m|} \Big(1+ \sum_{j=1}^{\infty} b_j (\lambda) x^j \Big),
    \end{aligned}
\end{equation}
for some constants $\{a_j (\lambda)\}_{j=1}^{\infty}$ 
and $\{b_j (\lambda)\}_{j=1}^{\infty}$, with each series
converging on an open interval containing $[0, b]$.    

Using (\ref{mhd-regular-singular-solutions}), we can take as our 
frame for solutions of (\ref{mhd-B}) lying left in $(0, b)$ the vector 
\begin{equation*}
    \mathbf{X}_0 (x; \lambda)
    = \begin{pmatrix}
    \phi (x; \lambda) \\
    - P(x; \lambda) \phi' (x; \lambda)
    \end{pmatrix},
\end{equation*}
noting that for $x$ near $0$ this is approximately 
\begin{equation*}
    \mathbf{X}_0 (x; \lambda)
    \cong \begin{pmatrix}
    x^{|m|} \\
    - \frac{\mu_0 \rho_0 \lambda + F(0)^2}{m} x^{|m|}
    \end{pmatrix}.
\end{equation*} 
For $\mathbf{X}_b (x; \lambda_2)$, since $\lambda_2$ is not an 
eigenvalue, we can take 
\begin{equation*}
   \mathbf{X}_b (x; \lambda_2)
    = \begin{pmatrix}
    \varphi (x; \lambda_2) \\
    -P(x; \lambda_2) \varphi' (x; \lambda_2)
    \end{pmatrix}
    +  C \begin{pmatrix}
    \phi (x; \lambda_2) \\
    - P(x; \lambda_2) \phi' (x; \lambda_2)
    \end{pmatrix},
\end{equation*}
for some constant $C$, where the key point is that if $\lambda_2$ is not
an eigenvalue then the first term on the right-hand side of this 
expression must appear non-trivially. We can detect intersections between 
$\ell_0 (x; \lambda)$ and $\ell_b (x; \lambda_2)$ by computing 
\begin{equation*}
\begin{aligned}
    \det &\begin{pmatrix}
    \phi (x; \lambda) & \varphi (x; \lambda_2) + C \phi (x; \lambda_2) \\
    - P(x; \lambda) \phi' (x; \lambda) & - P(x; \lambda_2) (\varphi' (x; \lambda_2) + C \phi' (x; \lambda_2)) 
    \end{pmatrix} \\
    & \cong 
    \begin{pmatrix}
    x^{|m|} & x^{- |m|} \\
    - \frac{\mu_0 \rho_0 \lambda + F(0)^2}{m} x^{|m|} & \frac{\mu_0 \rho_0 \lambda_2 + F(0)^2}{m} x^{-|m|} 
    \end{pmatrix} \\
    &= \frac{\mu_0 \rho_0 (\lambda + \lambda_2) + 2 F(0)^2}{m} \ne 0.
\end{aligned}
\end{equation*}

The final item we need to check is the assumption that 
\begin{equation*}
    \sigma_{\ess} (\mathcal{T}) \cap [\lambda_1, \lambda_2]
    = \emptyset.
\end{equation*}
For this, we recall that a value $\lambda \in [\lambda_1, \lambda_2]$
is in the essential spectrum of the operator pencil 
$\mathcal{T} (\cdot)$ if and only if $\mu = 0$ is in the 
essential spectrum of the operator $\mathcal{T} (\lambda)$. 
In order to understand the essential spectrum of the 
operator $\mathcal{T} (\lambda)$, we consider the eigenvalue 
problem $\mathcal{T} (\lambda) y = \mu y$, which can be expressed as 
\begin{equation} \label{mhd-mu-problem}
    J y' = \mathbb{B} (x; \lambda) y + \mu \mathbb{B}_{\lambda} (x; \lambda) y
\end{equation}
(see (\ref{linear-ev-problem}) and (\ref{linear-hammy})). 
If we set 
\begin{equation*}
    \begin{aligned}
    \mathcal{P} (x; \lambda, \mu)
    &= - (Q (x; \lambda) + \mu Q_{\lambda} (x; \lambda))^{-1} \\
    \mathcal{V} (x; \lambda, \mu) 
    &= V(x; \lambda) + \mu V_{\lambda} (x; \lambda),
    \end{aligned}
\end{equation*}
then we can express (\ref{mhd-mu-problem}) as a second-order equation 
for the component $y_1$, namely 
\begin{equation} \label{mhd-mu-second-order}
    - (\mathcal{P} (x; \lambda, \mu) y_1')' 
    + \mathcal{V} (x; \lambda, \mu) y_1 = 0.
\end{equation}
Our approach to the essential spectrum will now be as follows: 
(1) we will show that for each $\lambda \in [\lambda_1, \lambda_2]$
we can construct a solution $\phi (x; \lambda, \mu)$ to 
(\ref{mhd-mu-second-order}), analytic in $\mu$ for $\mu$ in 
a neighborhood of $0$, so that 
\begin{equation*}
    \mathbf{X}_0 (x; \lambda, \mu)
    = \begin{pmatrix}
    \phi (x; \lambda, \mu) \\ - \mathcal{P} (x; \lambda, \mu) \phi (x; \lambda, \mu) 
    \end{pmatrix}
\end{equation*}
is a solution of (\ref{mhd-mu-problem}) that lies left in $(0, b)$, and likewise
that we can construct a solution $\psi (x; \lambda, \mu)$ to 
(\ref{mhd-mu-second-order}), analytic in $\mu$ for $\mu$ in 
a neighborhood of $0$, so that 
\begin{equation*}
    \mathbf{X}_b (x; \lambda, \mu)
    = \begin{pmatrix}
    \psi (x; \lambda, \mu) \\ - \mathcal{P} (x; \lambda, \mu) \psi (x; \lambda, \mu) 
    \end{pmatrix}
\end{equation*}
is a solution of (\ref{mhd-mu-problem}) such that $\psi (b; \lambda, \mu) = 0$; 
and (2) we will show that if the Evans function 
\begin{equation} \label{mhd-evans}
    \mathcal{D} (\mu; \lambda)
    = \det \Big( \mathbf{X}_0 (x; \lambda, \mu) \,\,\, \mathbf{X}_b (x; \lambda, \mu)\Big) 
\end{equation}
(which is independent of $x$)
satisfies $D(0; \lambda) \ne 0$, then $0 \in \rho (\mathcal{T} (\lambda))$. Since 
$\mathcal{D} (\cdot; \lambda)$ is analytic in $\mu$ (due to analyticity of $ \mathbf{X}_0 (x; \lambda, \mu)$
and $\mathbf{X}_b (x; \lambda, \mu)$ in $\mu$) and not identically 0, we can conclude that if 
$\mu = 0$ is a root of $\mathcal{D} (\cdot; \lambda)$, then it must be isolated, and so must 
be an element of the point spectrum of $\mathcal{T} (\lambda)$. In summary, for 
each $\lambda \in [\lambda_1, \lambda_2]$, we will be able to conclude that 
$\mu = 0$ is either in the resolvent set of $\mathcal{T} (\lambda)$ or in the 
point spectrum of $\mathcal{T} (\lambda)$. It will follow immediately that $\mu = 0$ is not 
in the essential spectrum of $\mathcal{T} (\lambda)$ for any 
$\lambda \in [\lambda_1, \lambda_2]$.

For the first step in this approach, we observe that (\ref{mhd-mu-second-order}) has a regular
singular point at $x = 0$, allowing us to construct a series solutions of the form 
\begin{equation} \label{mhd-mu-series1}
    \phi (x; \lambda, \mu)
    = x^r \Big(1 + \sum_{j=0}^{\infty} a_j (\lambda, \mu) x^{j}\Big),
\end{equation}
where $a_0 (\lambda, \mu) = 1$ and 
\begin{equation} \label{mhd-r}
    r (\mu; \lambda) = \sqrt{\frac{(\mu_0 \lambda_0 (\lambda + \mu) + F(0)^2) (\mu_0 \lambda_0 (\lambda - \mu) + F(0)^2)}
    {(\mu_0 \rho_0 \lambda + F(0)^2)^2}} |m|.
\end{equation}
Here, we observe that $r(\cdot; \lambda)$ is analytic in $\mu$ for $\mu$ near 0, with also 
$r (0; \lambda) = |m|$. (We can also construct a second solution of 
(\ref{mhd-mu-second-order}), linearly independent of $\phi (x; \lambda, \mu)$, 
but that solution won't be needed for the current analysis.) 
In (\ref{mhd-mu-series1}), we see that $x^r$ is analytic in 
$\mu$, so we only need to verify that the sum $\sum_{j=1}^{\infty} a_j (\lambda, \mu) x^{r+j}$
is also analytic in $\mu$. To this end, we will check that each coefficient 
$a_j (\lambda, \mu)$ is analytic in $\mu$, and also that the sum converges
uniformly in $\mu$. In addition to (\ref{mhd-mu-series1}), 
this calculation will require the following series expansions: 
\begin{equation} \label{mhd-mu-series2}
\begin{aligned}
    \phi' (x; \lambda, \mu)
    &= \sum_{j=0}^{\infty} a_j (\lambda, \mu) (r+j) x^{r+j-1} \\
    \phi'' (x; \lambda, \mu)
    &= \sum_{j=0}^{\infty} a_j (\lambda, \mu) (r+j) (r+j-1) x^{r+j-2},
\end{aligned}
\end{equation}
and 
\begin{equation} \label{mhd-mu-series3}
    \begin{aligned}
    \mathcal{P} (x; \lambda, \mu) 
    &= \sum_{j=1}^{\infty} p_j (\lambda, \mu) x^j \\
    \mathcal{P}' (x; \lambda, \mu) 
    &= \sum_{j=1}^{\infty} j p_j (\lambda, \mu) x^{j-1} \\
     \mathcal{V} (x; \lambda, \mu) 
    &= \sum_{j=0}^{\infty} v_j (\lambda, \mu) x^{j-1},
    \end{aligned}
\end{equation}
where 
\begin{equation*}
    \begin{aligned}
    p_1 (\lambda, \mu)
    &= \frac{(\mu_0 \rho_0 \lambda + F(0)^2)^2}{m^2 (\mu_0 \rho_0 (\lambda - \mu) + F(0)^2)} \\
    v_0 (\mu, \lambda) &= \mu_0 \rho_0 (\lambda + \mu) + F(0)^2.    
    \end{aligned}
\end{equation*}

Upon substitution of these relations into (\ref{mhd-mu-second-order}),
we obtain the relation
\begin{equation*}
\begin{aligned}
    - \Big( &\sum_{j=1}^{\infty} p_j (\lambda, \mu) x^j \Big)
    \Big( \sum_{j=0}^{\infty} a_j (\lambda, \mu) (r+j) (r+j-1) x^{r+j-2} \Big) \\
    & - \Big( \sum_{j=1}^{\infty} j p_j (\lambda, \mu) x^{j-1} \Big)
    \Big( \sum_{j=0}^{\infty} a_j (\lambda, \mu) (r+j) x^{r+j-1} \Big) \\
    &+ x^{-1} \Big(\sum_{j=0}^{\infty} v_j (\lambda, \mu) x^{j-1} \Big)
    \Big( \sum_{j=0}^{\infty} a_j (\lambda, \mu) x^{r+j} \Big)
    = 0.
\end{aligned}
\end{equation*}
For the subsequent calculations, it will be convenient to write 
\begin{equation*}
    \begin{aligned}
    \Big( &\sum_{j=1}^{\infty} p_j (\lambda, \mu) x^j \Big)
    \Big( \sum_{j=0}^{\infty} a_j (\lambda, \mu) (r+j) (r+j-1) x^{r+j-2} \Big) \\
    &= x^{r-2} \Big( \sum_{j=1}^{\infty} p_j (\lambda, \mu) x^j \Big)
    \Big( \sum_{j=0}^{\infty} a_j (\lambda, \mu) (r+j) (r+j-1) x^{j} \Big) \\
    &= x^{r-2} \sum_{j=1}^{\infty} c_j (\lambda, \mu) x^j,
    \end{aligned}
\end{equation*}
where 
\begin{equation} \label{mhd-c-expansion}
    c_j
    = p_1 a_{j-1} (r+j-1) (r+j-2) 
    + p_2 a_{j-2} (r+j-2)(r+j-3) 
    + \dots + p_j a_0 r (r-1),
\end{equation}
for $j = 1, 2, \dots$. (Here, dependence on $\lambda$ and $\mu$ has been suppressed
to emphasize the index relations.)
Likewise, 
\begin{equation*}
    \begin{aligned}
    \Big( &\sum_{k=1}^{\infty} j p_j (\lambda, \mu) x^{j-1} \Big)
    \Big( \sum_{j=0}^{\infty} a_j (\lambda, \mu) (r+j) x^{r+j-1} \Big) \\
    &=  x^{r-2} \Big( \sum_{j=1}^{\infty} j p_j (\lambda, \mu) x^j \Big)
    \Big( \sum_{j=0}^{\infty} a_j (\lambda, \mu) (r+j) x^{r+j-1} \Big) \\
    &= x^{r-2} \sum_{j=1}^{\infty} d_j (\lambda, \mu) x^j,
    \end{aligned}
\end{equation*}
where 
\begin{equation} \label{mhd-d-expansion}
    d_j = p_1 a_{j-1} (r+j-1) + 2 p_2 a_{j-2} (r+j-2)
    + \dots + j p_{j} a_0 r,
\end{equation}
for $j = 1, 2, 3, \dots$,
and 
\begin{equation*}
    \begin{aligned}
    x^{-1} &\Big(\sum_{j=0}^{\infty} v_j (\lambda, \mu) x^j \Big)
    \Big( \sum_{j=0}^{\infty} a_j (\lambda, \mu) x^{r+j} \Big) \\
    &= x^{r-1} \Big(\sum_{j=0}^{\infty} v_j (\lambda, \mu) x^j \Big)
    \Big( \sum_{j=0}^{\infty} a_j (\lambda, \mu) x^{j} \Big) \\
    &= x^{r-1} \sum_{j=0}^{\infty} e_j (\lambda, \mu) x^j
    = x^{r-2} \sum_{j=1}^{\infty} e_{j-1} (\lambda, \mu) x^j,
    \end{aligned}
\end{equation*}
where 
\begin{equation} \label{mhd-e-expansion}
    e_j = v_0 a_j  
    + v_1 a_{j-1} 
    + \dots + v_j a_0, 
\end{equation}
for $j = 0, 1, 2, \dots$. 
Combining these expressions, we can write (\ref{mhd-mu-second-order})
as 
\begin{equation*}
- x^{r-2} \sum_{j=1}^{\infty} c_j (\lambda, \mu) x^j 
- x^{r-2} \sum_{j=1}^{\infty} d_j (\lambda, \mu) x^j
+ x^{r-2} \sum_{j=1}^{\infty} e_{j-1} (\lambda, \mu) x^{j} = 0.
\end{equation*}
Upon matching coefficients of powers of $x$,
we obtain the coefficient relations
\begin{equation} \label{mhd-relations}
    - c_j (\lambda, \mu) - d_j (\lambda, \mu) + e_{j-1} (\lambda, \mu) = 0,
    \quad j = 1, 2, 3, \dots. 
\end{equation}
For $j = 1$, we have 
\begin{equation*}
    \begin{aligned}
    c_1 &= p_1 (\lambda, \mu) a_0 r (r-1) \\
    d_1 &= p_1 (\lambda, \mu) a_0 r \\
    e_0 &= v_0 (\lambda, \mu) a_0,
    \end{aligned}
\end{equation*}
from which the relation $-c_1 - d_1 + e_0 = 0$ is seen to hold from the above 
expressions for $p_1$, $a_0$, and $v_0$. 

Next, we use the equation (\ref{mhd-relations}) to obtain an iterative
relation for $a_{j+1} (\lambda, \mu)$ in terms of the previous 
coefficients. Using (\ref{mhd-c-expansion}), (\ref{mhd-d-expansion}),
and (\ref{mhd-e-expansion}), we obtain the relation
\begin{equation*}
    A_{\mu} a_{j-1} = B_{\mu},
\end{equation*}
where 
\begin{equation*}
\begin{aligned}
    A_{\mu} &= - p_1 (r+j-1)(r+j-2) - p_1 (r+j-1) + v_0
    = - p_1 (r+j-1)^2 + v_0 \\
    &= - \frac{(\mu_0 \rho_0 \lambda + F(0)^2)^2}{m^2 (\mu_0 \rho_0 (\lambda - \mu) + F(0)^2)} (r+j-1)^2
    + \mu_0 \rho_0 (\lambda + \mu) + F(0)^2,
\end{aligned}
\end{equation*}
and 
\begin{equation*}
    \begin{aligned}
    B_{\mu} &= p_2 (r+j-2) (r+j-3) a_{j-2} + \dots + p_j (r-1) (r-2) a_0 \\
    &+ 2 p_2 (r+j-2) a_{j-2} + \dots + j p_j r a_0 
    - v_1 a_{j-2} - \dots - v_{j-1} a_0. 
    \end{aligned}
\end{equation*}
Here, since $r(0,\lambda) = |m|$,
\begin{equation*}
    A_0 = (\mu_0 \rho_0 (\lambda + \mu) + F(0)^2) (1 - \frac{(|m|+j-1)^2}{m^2}).
\end{equation*}
The quantity $(1 - (|m|+j-1)^2/(m^2))$ is strictly negative for 
$j = 2, 3, \dots$, so that we can divide by $A_{\mu}$ for all $\mu$
sufficiently close to $0$. Since $A_{\mu}$ and $B_{\mu}$ both grow  
at quadratic rate in $j$, we see that for the ratio $B_{\mu}/A_{\mu}$,
the coefficients of the values $\{a_i\}_{i = 0}^{j-2}$ can be bounded
by some constant $M$, uniform in $j$ and $\mu$ for $\mu$ in a fixed 
neighborhood of $0$. Specifically, we obtain an inequality of the form 
\begin{equation*}
    |a_{j-1}| \le M (|a_0| + |a_1| + \dots + |a_{j-2}|),
\end{equation*}
or upon shifting the index
\begin{equation} \label{mhd-a-bound}
    |a_j| \le M (|a_0| + |a_1| + \dots + |a_{j-1}|).
\end{equation}

\begin{claim} \label{mhd-a-claim} For any sequence of values $\{a_j\}_{j=0}^{\infty}$,
suppose $|a_0| \le 1$ and that inequality (\ref{mhd-a-bound}) holds for all $j \in \mathbb{N}$. 
Then for any constant $C > 1$ chosen large enough so that 
\begin{equation} \label{mhd-a-claim-inequality}
    \frac{M/C}{1 - 1/C} \le 1,
\end{equation}
we have $|a_j| \le C^j$ for all $j \in \{0, 1, 2, \dots\}$. 
\end{claim}

\begin{proof}
We proceed inductively, first observing that by assumption 
$|a_0| \le 1 = C^0$. For the induction step, we assume that for some $n \in \mathbb{N}$,
we have $|a_j| \le C^j$ for all $j \in \{0, 1, 2, \dots, n\}$, and our goal is to 
show that this implies $|a_{n+1}| \le C^{n+1}$. Upon combining (\ref{mhd-a-bound})
with the induction hypothesis, we obtain the inequality 
\begin{equation*}
    \begin{aligned}
    |a_{n+1}| &\le M (|a_0| + |a_1| + \dots + |a_n|)
    \le M (1 + C + \dots + C^n) \\
    &= M C^{n+1} (\frac{1}{C} + \frac{1}{C^2} + \dots + \frac{1}{C^{n+1}})
    \le M C^{n+1} \sum_{j=1}^{\infty} \frac{1}{C^j} \\
    &= M C^{n+1} \frac{1/C}{1 - 1/C} \le C^{n+1},
    \end{aligned}
\end{equation*}
where in the last step we've used (\ref{mhd-a-claim-inequality}).
\end{proof}

Now, we return to (\ref{mhd-mu-series1}), and fix any value $x \le 1/(2C)$,
so that 
\begin{equation*}
    |a_j (\lambda, \mu) x^j| \le C^j (\frac{1}{2})^j \frac{1}{C^j}
    = (\frac{1}{2})^j.
\end{equation*}
We see that for fixed values $\lambda \in [\lambda_1, \lambda_2]$
and $x \le 1/(2C)$, the sum $\sum_{j=1}^{\infty} a_j (\lambda, \mu) x^j$ 
converges uniformly in $\mu$ for $\mu$ confined to a sufficiently small
neighborhood of $0$. Since the functions $\{a_j (\lambda, \mu)\}_{j=0}^{\infty}$
are analytic in $\mu$ (a property inherited from the analyticity
of $\mathcal{P} (x; \lambda, \mu)$ and $\mathcal{V} (x; \lambda, \mu)$ in 
$\mu$), the function $\phi (x; \lambda, \mu)$ in (\ref{mhd-mu-series1}) 
is analytic in $\mu$ for $0 < x < 1/(2C)$. Subsequently, we obtain 
analyticity in $\mu$ for for all $x \in (0,b)$ by analytic continuation.
Using $\phi (x; \lambda, \mu)$, we can now construct a solution 
\begin{equation} \label{mhd-left}
\mathbf{X}_0 (x; \lambda, \mu) 
= \begin{pmatrix}
\phi (x; \lambda, \mu) \\
- \mathcal{P} (x; \lambda, \mu) \phi' (x; \lambda, \mu) 
\end{pmatrix},
\end{equation}
analytic in $\mu$, that solves $\mathcal{T} (\lambda) \mathbf{X}_0 = \mu \mathbf{X}_0$ and lies 
left in $(0, b)$, and since (\ref{mhd-mu-problem}) is regular at 
$x = b$ we can likewise construct a solution $\mathbf{X}_b (x; \lambda, \mu)$,
also analytic in $\mu$, that solves $\mathcal{T} (\lambda) \mathbf{X}_b = \mu \mathbf{X}_b$,
along with the Dirichlet condition 
$\mathbf{X}_b (b; \lambda, \mu) = \genfrac(){0pt}{1}{0}{1}$. This allows us 
to analytically construct the Evans function $\mathcal{D} (\mu)$
specified in (\ref{mhd-evans}). 

\begin{claim} \label{mhd-evans-claim} For any fixed $\lambda \in [\lambda_1, \lambda_2]$,
there exists $\epsilon > 0$ sufficiently small so that a value 
$\mu \in B(0; \epsilon)$ is in the resolvent set of $\mathcal{T} (\lambda)$ if and only if
$\mathcal{D} (\mu; \lambda) \ne 0$.
\end{claim}

\begin{proof}
First, if $\mathcal{D} (\mu; \lambda) = 0$, then $\mu$ is an eigenvalue of 
$\mathcal{T} (\lambda)$, and so excluded from the resolvent set. For the other
direction, we assume $\mathcal{D} (\mu; \lambda) \ne 0$, and as in Appendix 
Section \ref{green-function-section}, we construct a Green's function 
\begin{equation*}
    G (x, \xi; \lambda, \mu)
    = \begin{cases}
    - \begin{pmatrix} 0 & \mathbf{X}_b (x; \lambda, \mu) \end{pmatrix}
    \mathbb{M} (\lambda, \mu) \begin{pmatrix} \mathbf{X}_0 (x; \lambda, \bar{\mu}) & 0\end{pmatrix}^* & 0 < \xi < x < b \\
    \begin{pmatrix} \mathbf{X}_0 (x; \lambda, \mu) & 0 \end{pmatrix}
    \mathbb{M} (\lambda, \mu) \begin{pmatrix} 0 & \mathbf{X}_b (x; \lambda, \bar{\mu})\end{pmatrix}^* & 0 < x < \xi < b, 
    \end{cases}
\end{equation*}
where 
\begin{equation*}
    \mathbb{M} (\lambda, \mu)
    = \mathbb{E} (x; \lambda, \mu)^{-1} J (\mathbb{E} (x; \lambda, \bar{\mu})^{-1})^*,
    \quad \mathbb{E} (x; \lambda, \mu) 
    = \begin{pmatrix}
    \mathbf{X}_0 (x; \lambda, \mu) & \mathbf{X}_b (x; \lambda, \mu)
    \end{pmatrix},
\end{equation*}
and as indicated by the notation $\mathbb{M}$ does not depend on $x$.
Precisely, if $\mathcal{D} (\mu; \lambda) \ne 0$, then for any 
$f \in L^2_{\mathbb{B}_{\lambda}} ((0, b), \mathbb{C}^2)$ we can solve 
the inhomogeneous problem $(\mathcal{T} (\lambda) - \mu I) y = f$ with 
the integral 
\begin{equation} \label{mhd-y-integral}
    y (x; \lambda, \mu)
    = (\mathcal{T} (\lambda) - \mu I)^{-1} f
    = \int_0^b G(x, \xi; \lambda, \mu) \mathbb{B}_{\lambda} (\xi; \lambda) f(\xi) d\xi.
\end{equation}
All that remains is to verify that $(\mathcal{T} (\lambda) - \mu)^{-1}$ 
maps $L^2_{\mathbb{B}_{\lambda}} ((0, b), \mathbb{C}^2)$ to the 
domain of $\mathcal{T} (\lambda)$, namely 
\begin{equation*}
    \mathcal{D} = \{y \in \mathcal{D}_M: y_1 (b) = 0 \}.
\end{equation*}

If we write 
\begin{equation*}
    \mathbb{M} (\lambda, \mu)
    = \begin{pmatrix}
    m_{11} (\lambda, \mu) & m_{12} (\lambda,\mu) \\
    m_{21} (\lambda,\mu) & m_{22} (\lambda,\mu) 
    \end{pmatrix},
\end{equation*}
then the first component of $y$ from (\ref{mhd-y-integral}) satisfies 
\begin{equation} \label{mhd-y1-integral}
    \begin{aligned}
    y_1 (x) 
    &= m_{21} \int_0^x \Big{\{}V_{\lambda} (\xi, \lambda) \phi (\xi) \psi (x) f_1 (\xi)
    + \phi' (\xi) Q_{\lambda} (\xi; \lambda) P (\xi; \lambda) \psi (x) f_2 (\xi) \Big{\}} d\xi \\
    &+ m_{12} \int_x^b \Big{\{}V_{\lambda} (\xi; \lambda) \phi (x) \psi (\xi) f_1 (\xi)
    + \phi (x) Q_{\lambda} (\xi; \lambda) P (\xi) \psi' (\xi) f_2 (\xi) \Big{\}} d\xi,
    \end{aligned}
\end{equation}
and $y_2 (x)$ satisfies a similar relation. Here, dependence of $m_{12}$,
$m_{21}$, $y_1$, $\phi$ 
and $\psi$ on the values $\mu$ and $\lambda$ has been suppressed for 
notational brevity. In order to verify 
that $y \in L^2_{\mathbb{B}_{\lambda}} ((0, b), \mathbb{C}^2)$, we need to 
check that 
\begin{equation} \label{mhd-weighted-integrability}
    \int_0^b \frac{1}{x} y_i (x)^2 dx < \infty,
    \quad i = 1, 2.
\end{equation}
We will show this for $i = 1$, noting that verification for the case 
$i = 2$ is almost identical. 

Recalling (\ref{V-lambda-Q-lambda}) 
and (\ref{mhd-mu-series1}), we see that
\begin{equation*}
    V_{\lambda} (\xi; \lambda) \phi (\xi; \lambda, \mu)
    = \rho_0 \mu_0 \xi^{r-1} + \mathbf{O} (\xi^{r}),
\end{equation*}
and likewise from (\ref{mhd-P}), (\ref{V-lambda-Q-lambda}) 
and (\ref{mhd-mu-series1}), 
\begin{equation*}
    \phi'(\xi; \lambda, \mu) Q_{\lambda} (\xi; \lambda) P (\xi; \lambda)
    = \frac{r \mu_0 \rho_0}{\mu_0 \rho_0 \lambda + F(0)^2} \xi^{r-1}
    + \mathbf{O} (\xi^r).
\end{equation*}
In this way, we see that to leading order the first integral in 
(\ref{mhd-y1-integral}) is 
\begin{equation*}
    I_1 (x) = x^{-r} \int_0^x \xi^{r-1} f_1 (\xi) d \xi,
\end{equation*}
where the constant multiplier is left off for notational convenience.
In order to check (\ref{mhd-weighted-integrability}), our starting point will
be 
\begin{equation} \label{mhd-integral1}
    \int_0^b \frac{1}{x} I_1 (x)^2 dx
    = \int_0^b x^{-2r-1} \Big(\int_0^x \xi^{r-1} f_1 (\xi) d \xi \Big)^2 dx.
\end{equation}
Here, using the Cauchy-Schwarz inequality, we see that 
\begin{equation*}
\begin{aligned}
\Big|\int_0^x \xi^{r-1} f_1 (\xi) d \xi \Big|
&\le \Big(\int_0^x \xi^{2(r-1)} d\xi \Big)^{1/2}
\Big(\int_0^x f_1 (\xi)^2 d\xi \Big)^{1/2} \\
&= \frac{1}{\sqrt{2r-1}} x^{r - 1/2} \Big(\int_0^x f_1 (\xi)^2 d\xi \Big)^{1/2},
\end{aligned}
\end{equation*}
where $2r-1 > 0$ for $\mu$ small since $r(0;\lambda) = |m|$, and 
we're working in the case $m \ne 0$.
The right-hand side of (\ref{mhd-integral1}) now becomes 
\begin{equation*}
    \frac{1}{2r-1} \int_0^b x^{-2} \int_0^x f_1 (\xi)^2 d\xi dx 
    = \frac{1}{2r-1} \Big[ -x^{-1} \int_0^x f_1 (\xi)^2 d\xi \Big|_{0}^b
    + \int_0^b \frac{1}{x} f_1 (x)^2 dx \Big]. 
\end{equation*}
The second summand in this last expression is finite because 
$f \in L^2_{\mathbb{B}_{\lambda}} ((0, b), \mathbb{C}^2)$. For the 
first, 
\begin{equation*}
\int_0^x f_1 (\xi)^2 d\xi 
= \int_0^x \xi \frac{1}{\xi} f_1 (\xi)^2 d\xi
\le x \int_0^x \frac{1}{\xi} f_1 (\xi)^2 d\xi,
\end{equation*}
from which we see that the evaluations at both $x = 0$
and $x = b$ are finite. 

We turn next to the integral in (\ref{mhd-y1-integral}) multiplying $m_{12}$. In 
this case, the integral to leading order is 
\begin{equation}
I_2 (x) = x^{r} \int_x^b \xi^{-r-1} f_1 (\xi) d\xi,    
\end{equation}
for which we can compute 
\begin{equation*}
\begin{aligned}
    \int_0^b \frac{1}{x} I_2 (x)^2 dx
    &= \int_0^b x^{2r-1} \Big( \int_x^b \xi^{-r-1} f_1 (\xi) d\xi \Big)^2 dx \\
    &\le \int_0^b x^{2r-1} \Big( \int_x^b \xi^{-2r} d\xi \Big) 
    \Big( \int_x^b \xi^{-2} f_1 (\xi)^2 d\xi  \Big) dx \\
    &= \int_0^b x^{2r-1} \Big( \frac{1}{-2r + 1} \xi^{-2r +1} \Big) \Big|_{\xi = x}^{\xi = b}
    \Big( \int_x^b \xi^{-2} f_1 (\xi)^2 d\xi  \Big) dx \\
    &= \frac{1}{-2r + 1} \int_0^b x^{2r-1} \Big( \frac{1}{x^{2r-1}} - \frac{1}{b^{2r-1}} \Big) 
    \Big( \int_x^b \xi^{-2} f_1 (\xi)^2 d\xi  \Big) dx.
\end{aligned}
\end{equation*}
The key integral here is 
\begin{equation*}
    \int_0^b \int_x^b \xi^{-2} f_1 (\xi)^2 d\xi dx,
\end{equation*}
and we can integrate by parts to see that this integral can be expressed as 
\begin{equation} \label{mhd-last-expression}
    \Big(x \int_x^b \xi^{-2} f_1 (\xi)^2 d\xi\Big) \Big|_{0}^b
    + \int_0^b x^{-1} f_1 (x)^2 dx.
\end{equation}
For the first summand in (\ref{mhd-last-expression}), we can write 
\begin{equation*}
 \Big(x \int_x^b \xi^{-2} f_1 (\xi)^2 d\xi\Big) \Big|_{0}^b
 \le x \int_x^b \frac{1}{x} \xi^{-1} f_1 (\xi)^2 d\xi
 = \int_x^b \xi^{-1} f_1 (\xi)^2 d\xi,
\end{equation*}
from which we see that the evaluation is bounded at each
endpoint. The second expression in (\ref{mhd-last-expression}) is 
bounded since 
$f \in L^2_{\mathbb{B}_{\lambda}} ((0, b), \mathbb{C}^2)$. 
This completes the analysis of $y_1$, and a similar analysis can
be carried out for $y_2$, allowing us to conclude 
$y \in L^2_{\mathbb{B}_{\lambda}} ((0, b), \mathbb{C}^2)$. 

Next, we see directly from (\ref{mhd-y1-integral}) that 
$y_1 (\cdot; \lambda, \mu) \in \AC_{\loc} ((0, b), \mathbb{C}^2)$, 
and also that the evaluation $y_1 (b; \lambda, \mu) = 0$ holds 
(because $\psi (b; \lambda, \mu) = 0$). 
Likewise, we can show that 
$y_2 (\cdot; \lambda, \mu) \in \AC_{\loc} ((0, b), \mathbb{C}^2)$, 
and also that $y_2 (b; \lambda, \mu)$ is bounded, allowing us to 
conclude that $y(\cdot; \lambda, \mu) \in \mathcal{D}$. 
\end{proof}

We see now that for each fixed $\lambda \in [\lambda_1, \lambda_2]$
there is a ball $B (0; \epsilon) \subset \mathbb{C}$, 
$\epsilon > 0$, so that 
the function $\mathcal{D} (\mu, \lambda)$ is analytic for 
$\mu \in B (0; \epsilon)$, and so that $\mu \in B(0;\epsilon)$ is  
in the spectrum of $\mathcal{T} (\lambda)$ if and only 
if $\mathcal{D} (\mu, \lambda) = 0$. Moreover, by monotonicity 
in the spectral parameter $\mu$, it cannot be the case
that $\mathcal{D} (\mu, \lambda) \equiv 0$ in a neighborhood
of $\mu = 0$. Since the zeros of 
an analytic function are isolated, if $\mathcal{D} (0, \lambda) = 0$
then $\mu = 0$ must be isolated as an eigenvalue of 
$\mathcal{T} (\lambda)$, and so cannot be in the essential spectrum. 
In summary, we can conclude that for each $\lambda \in [\lambda_1, \lambda_2]$,
$\mu = 0$ is either in the resolvent set of $\mathcal{T} (\lambda)$
or is an eigenvalue of $\mathcal{T} (\lambda)$, ensuring that 
$0 \notin \sigma_{\ess} (\mathcal{T} (\lambda))$. 

We have now verified that for our MHD application 
(\ref{ultimate-chi-equation}), our Assumptions 
{\bf (A)} through {\bf (F)} all hold, along with 
the full assumptions of Theorem \ref{regular-singular-theorem}, 
adjusted for the change of singular side from the right 
to the left. We turn now to numerical calculations
associated with this example. For our example
calculations, we will make the specific choices
\begin{equation} \label{mhd-B-choices}
\begin{aligned}
    B_{\theta} (x) 
    &= \frac{B_0 \kappa x}{1 + \kappa^2 x^2} \\
    B_{z} (x) 
    &= \frac{B_0}{1 + \kappa^2 x^2},
\end{aligned}
\end{equation}
along with constant values $b = .01$, $m = -1$, $k=1$,  
$\rho_0 = 1$, $\mu_0 = 1$, $B_0 = 1$, and 
\begin{equation*}
    \kappa = \frac{1/B_0 - m}{k} = .9.
\end{equation*}

\begin{remark} \label{mhd-choices-remark}
The choices (\ref{mhd-B-choices}) are taken from 
Section 9.1.1 of \cite{GP2004}. The specific 
parameter values were taken as a simple case for 
which oscillation counts reveal an accumulation 
point of eigenvalues at $\lambda = -1$.
\end{remark}

With the choices made above, we can compute  
\begin{equation*}
    F (x) = \frac{m B_{\theta} (x)}{x} + k B_z (x) \\
    = \frac{B_0 (m \kappa + k)}{1+\kappa^2 x^2}
    = \frac{1}{1 + .81 x^2},
\end{equation*}
from which we see that 
\begin{equation*}
    \ran F|_{[0, .01]} 
    =  \Big{[}\frac{1}{1+.81 (.01)^2}, 1\Big{]}
    \cong  \Big{[}.999919, 1\Big{]}.
\end{equation*}
It follows from our general theory that we can work 
on any interval $[\lambda_1, \lambda_2]$ contained 
in the union 
\begin{equation*}
    (-\infty, -1) \cup (-.999919, 0),
\end{equation*}
i.e., on intervals either to the left of $-1$ 
or contained in the interval $(-.999919, 0)$. 
Numerical computations in \cite{Kerner1987} 
suggest that eigenvalues accumulate at $-1$
from the left, so we will proceed by working 
with small intervals near $-1$, starting 
with $[\lambda_1, \lambda_2] = [-1.1, -1.03]$.

First, since our system is regular at the right 
boundary point $x = b = .01$, we can readily 
generate the frame $\mathbf{X}_b (x; \lambda)$
for all $\lambda \in [\lambda_1, \lambda_2]$
by solving the first-order ODE 
\begin{equation*}
    J {\mathbf{X}_b}' = \mathbb{B} (x; \lambda) \mathbf{X}_b,
    \quad \mathbf{X}_b (.01; \lambda) 
    = \begin{pmatrix} 0 \\ 1 \end{pmatrix}.
\end{equation*}
We know from our Frobenius theory that the system 
is limit-point at $x = 0$, so we only need to 
generate the unique solution that lies left in 
$(0, .01)$. For this, we can use the development 
outlined at the beginning of Section \ref{applications-section}, 
beginning with the eigenvalues
of 
\begin{equation*}
    \mathcal{B} (x; \lambda_1)
    = \Phi (x; 0, \lambda_1)^* J (\partial_{\lambda} \Phi) (x; 0, \lambda_1).
\end{equation*}
In order to approximate the behavior of these eigenvalues as $x \to 0^+$, we 
choose a small value $\delta$ and evaluate $\nu_1 (\delta; \lambda_1)$ and 
$\nu_2 (\delta; \lambda_1)$. Precisely, in this case, we take $\delta = 10^{-5}$,
for which we find 
\begin{equation*}
    \begin{aligned}
    \nu_1 (10^{-5}; \lambda_1) &= -9.3997 \times 10^6 \\
    \nu_2 (10^{-5}; \lambda_1) &= -2.3953.
    \end{aligned}
\end{equation*}
The eigenvector associated with $\nu_2 (10^{-5}; \lambda_1)$ is 
\begin{equation*}
    v_2 (10^{-5}; \lambda_1)
    = \begin{pmatrix}
    .99997 \\ -.00717
    \end{pmatrix}.
\end{equation*}
This approximately specifies the solution $\mathbf{X}_0 (x; \lambda_1)$
that lies left in $(0,.01)$ as 
\begin{equation*}
   \mathbf{X}_{0} (x; \lambda_1)
   = \Phi (x; 0, \lambda_1) v_2 (10^{-5}; \lambda_1).
\end{equation*}
With $\mathbf{X}_0 (x; \lambda_1)$ and $\mathbf{X}_{b} (x; \lambda_2)$ specified, we can 
now compute the Maslov index 
\begin{equation*}
    \mas (\ell_0 (\cdot; \lambda_1), \ell_{b} (\cdot; \lambda_2); (0,.01))
\end{equation*}
as a rotation number through $-1$ for the complex number
\begin{equation*}
    \begin{aligned}
    \tilde{W} (x; \lambda_1, \lambda_2) 
    &= - (X_0 (x; \lambda_1) + i Y_0 (x; \lambda_1)) (X_0 (x; \lambda_1) - i Y_0 (x; \lambda_1))^{-1}  \\
    & \quad \quad  \times (X_{b} (x; \lambda_2) + i Y_{b} (x; \lambda_2)) 
    (X_{b} (x; \lambda_2) - i Y_{b} (x; \lambda_2))^{-1}.  
    \end{aligned}
\end{equation*}
By generating the frames $\mathbf{X}_0 (x; \lambda_1)$ and 
$\mathbf{X}_{b} (x; \lambda_2)$ numerically, we can track 
the complex value $\tilde{W} (x; \lambda_1, \lambda_2)$, and in this 
way, we observe a single crossing at the value $x = .006388$
(with numerical increment $10^{-6}$). From Theorem 
\ref{regular-singular-theorem}, we can conclude that there is one 
eigenvalue on the interval $[-1.1, -1.03]$ in this case. 

\begin{figure}[h]  
\begin{center}\includegraphics[%
  width=11cm,
  height=8cm]{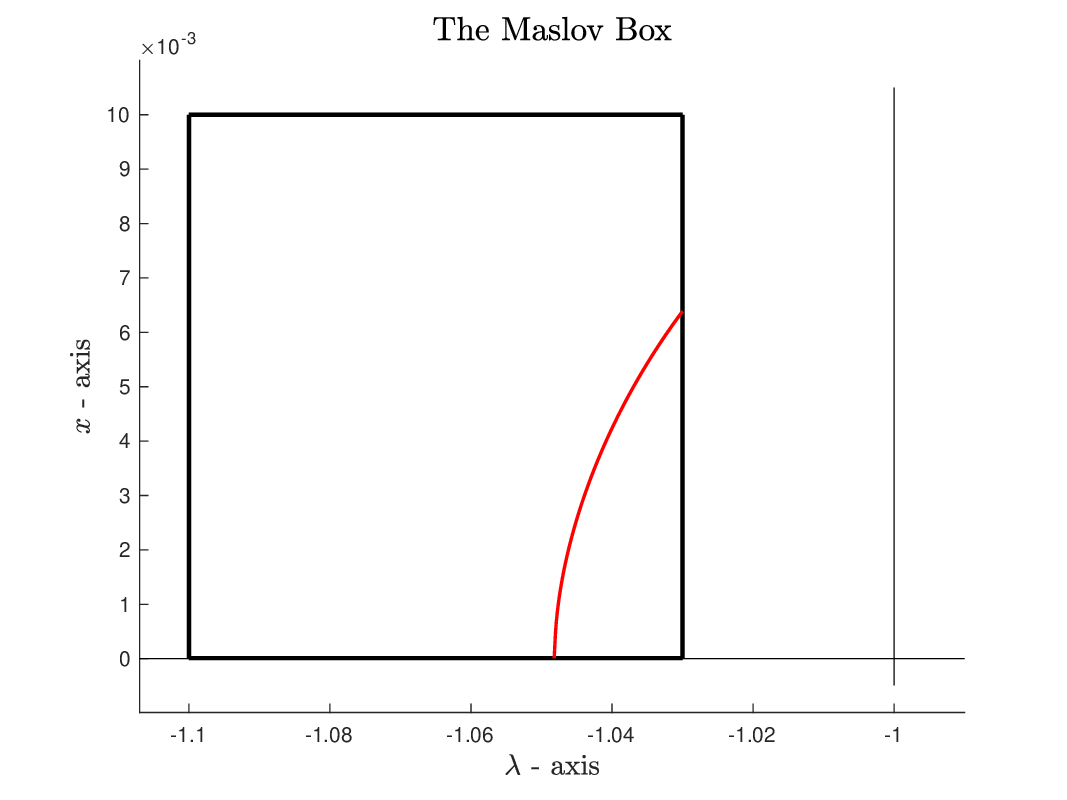}
\end{center}
\caption{The Maslov Box for (\ref{ultimate-chi-equation}) on $[-1.1, -1.03] \times [0, .01]$. \label{mhdbox1-fig}}
\end{figure}

\begin{figure}[h]  
\begin{center}\includegraphics[%
  width=11cm,
  height=8cm]{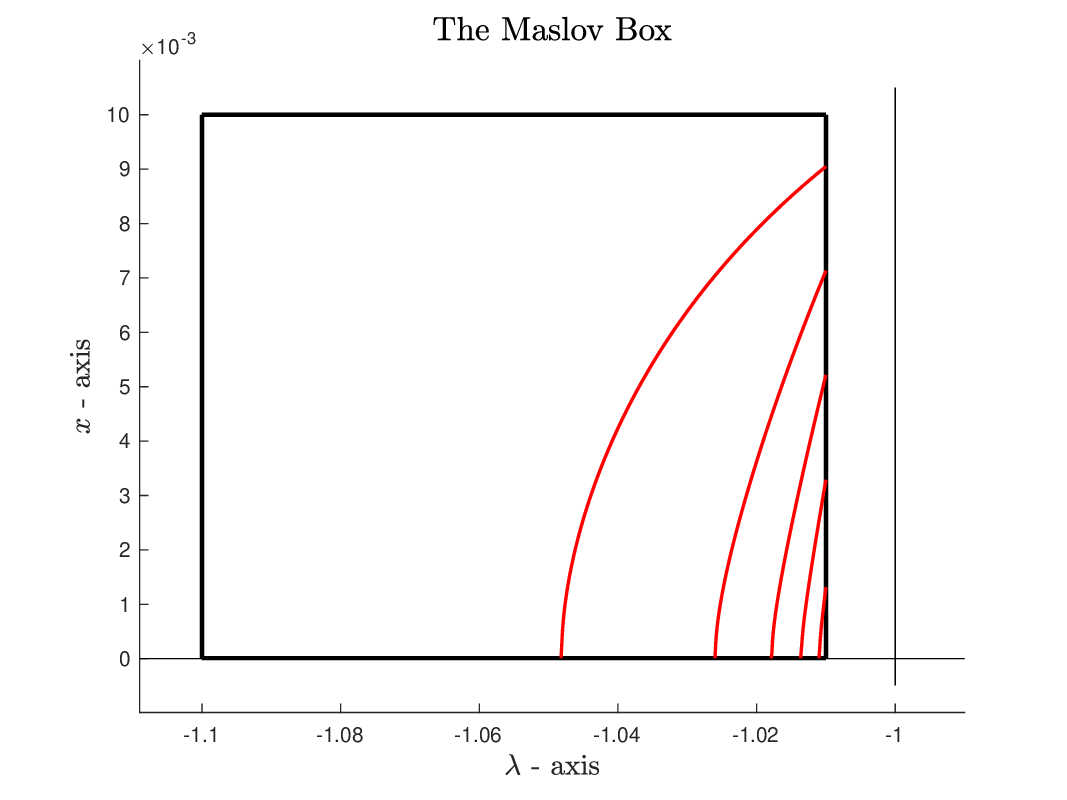}
\end{center}
\caption{The Maslov Box for (\ref{ultimate-chi-equation}) on $[-1.1, -1.01] \times [0, .01]$. \label{mhdbox3-fig}}
\end{figure}

In order to see more fully the dynamics associated with this count, 
we generate the spectral curves passing through  
the Maslov box $[-1.1, -1.03] \times [0, .01]$, where the endpoint $0$ can be 
included in a limiting sense. (See Figure \ref{mhdbox1-fig}.)
As expected, we see that there is a single monotonic spectral curve, 
which passes through the right shelf at about $x = .006$
and approaches the bottom shelf at about $\lambda = -1.0482$
(with a stepsize of $.0001$). This value of $\lambda$ is the 
eigenvalue that we expected to find on $[-1.1, -1.03]$. 

Based on numerical calculations from \cite{Kerner1987}, discussed in 
Section 9.3.2 of \cite{GP2004}, we expect that $-1$ will be 
an accumulation point, and in order to investigate this, we 
move the right endpoint of our interval closer to $-1$. Leaving 
$\lambda_1 = -1.1$, we move $\lambda_2$ to $\lambda_2 = -1.01$. 
In this case, we find five crossings along the right shelf, 
approximately located at $x = .000315$, $.003288$, $.005216$,
$.007133$, and $.009049$ (with a stepsize again of $10^{-6}$). 
From Theorem \ref{regular-singular-theorem}, we can conclude that there are five 
eigenvalue on the interval $[-1.1, -1.01]$. A Maslov box associated
with this calculation is included as Figure \ref{mhdbox3-fig}, in which we 
see that the eigenvalues are located at $\lambda = -1.4082$, 
$\lambda = -1.0260$, $\lambda = -1.0179$, $\lambda = - 1.0136$,
and $\lambda = -1.0110$ (again with a stepsize of $.0001$).

\subsubsection{Application to Hydraulic Shocks}
\label{hydraulic-shocks-section}

In \cite{SYZ2020, YZ2020}, the authors consider traveling waves (specifically 
{\it hydraulic shock profiles}) arising as solutions of the inviscid Saint-Venant
equations for inclined shallow-water flow, namely 
\begin{equation} \label{isl}
    \begin{aligned}
        h_t + q_x &= 0 \\
        q_t + \Big(\frac{q^2}{h} + \frac{h^2}{2F^2} \Big)_x &= h - \frac{|q|q}{h^2},
    \end{aligned}
\end{equation}
where $h$ denotes fluid height, $q = hu$ denotes fluid flow ($u$ is fluid velocity),
$F > 0$ denotes Froude number, and the system is posed on $\mathbb{R} \times \mathbb{R}_+$.
System \eqref{isl} is of relaxation form with associated formal equilibrium equation 
\begin{equation} \label{isl-equilibrium}
    h_t + q_*(h)_x = 0,
\end{equation}
where $q_* (h) := h^{3/2}$ is the value of $q$ for which gravity and bottom forces 
cancel. 

In \cite{YZ2020}, the authors establish existence of traveling-wave solutions 
\begin{equation*}
    \begin{aligned}
        h(x,t) &= H (x-st) \\
        q(x, t) &= Q (x - st),
    \end{aligned}
\end{equation*}
with well-defined asymptotic endstates 
\begin{equation*}
    \begin{aligned}
        \lim_{z \to - \infty} (H(z),Q(z)) &=: (H_L, Q_L) \\ 
        \lim_{z \to + \infty} (H(z),Q(z)) &=: (H_R, Q_R).
    \end{aligned}
\end{equation*}
Specifically, they prove the following proposition. 

\begin{proposition}[Proposition 1.1 from \cite{YZ2020}] \label{YZ-prop}
    Let $(H_L, H_R, c)$ be a triple for which there exists an entropy-admissible shock 
    solution of \eqref{isl-equilibrium} in the sense of Lax \cite{Lax1957} with speed 
    $c$ connecting left state $H_L$ to right state $H_R$; i.e., $H_L > H_R > 0$ and 
    $c[H] = [H^{3/2}]$, where $[\cdot]$ denotes a jump difference, e.g., 
    $[H] = H_R - H_L$. Then there exists a corresponding hydraulic shock profile
    with $Q_L = H_L^{3/2}$ and $Q_R = H_R^{3/2}$ if and only if $0 < F < 2$. The 
    profile is smooth (i.e., $C^{\infty} (\mathbb{R})$) for 
    \begin{equation*}
        H_L > H_R > \frac{2F^2}{1+2F+\sqrt{1+4F}} H_L,
    \end{equation*}
    and nondegenerate in the sense that $c$ is not a characteristic speed of 
    \eqref{isl-equilibrium} at any point along the profile. For 
    \begin{equation} \label{OurCase}
        0 < H_R < \frac{2F^2}{1+2F+\sqrt{1+4F}} H_L,
    \end{equation}
    the profile is nondegenerate and piecewise smooth with a single discontinuuity 
    consisting of an entropy-admissible shock of \eqref{isl-equilibrium}. At the 
    critical value 
    \begin{equation*}
        H_R = \frac{2F^2}{1+2F+\sqrt{1+4F}} H_L,
    \end{equation*}
    $H_R$ is characteristic, and there exists a degenerate profile that is continuous
    but not smooth, with discontinuous derivative at $H_R$. 
\end{proposition}

Our focus will be on the case characterized by \eqref{OurCase}. 
Upon linearization of \eqref{isl} about $(H (x-st), Q (x-st))$, we arrive (proceeding
as in \cite{MZ2002}) at the eigenvalue problem 
\begin{equation} \label{syz2020evp1}
    Av' = (E - \lambda I - A') v,
\end{equation}
where 
\begin{equation*}
    A(z) = 
    \begin{pmatrix}
        -c & 1 \\
        \frac{H(z)}{F^2} - \frac{Q(z)^2}{H(z)^2} & \frac{2Q(z)}{H(z)} - c 
    \end{pmatrix};
    \quad 
    E(z) = 
    \begin{pmatrix}
        0 & 0 \\
        \frac{2 Q(z)^2}{H(z)^3} + 1 & - \frac{2Q(z)}{H(z)^2} 
    \end{pmatrix},
\end{equation*}
and prime denotes differentiation with respect to the translated variable 
$z = x - ct$. Here, the matrix $A (z)$ is invertible, so we can express 
\eqref{syz2020evp1} as $v' = Bv$, $B = A^{-1} (E - \lambda I - A')$. For 
notational convenience, we'll denote the elements of $A$ and $E$ respectively
as $\{a_{ij}\}_{i,j = 1}^2$ and $\{e_{ij}\}_{i,j = 1}^2$, in which case the 
elements $\{b_{ij}\}_{i,j = 1}^2$ of $B$ can be expressed 
\begin{equation*}
\begin{aligned}
    b_{11} (z) &= \frac{\lambda a_{22} (z) + (e_{21} (z) - a_{21}' (z))}{a_{21} (z)+ca_{22} (z)};
    \quad 
    b_{12} (z) = \frac{e_{22} (z) - \lambda - a_{22}' (z)}{a_{21} (z)+ca_{22} (z)} \\
    b_{21} (z) &= \frac{c (e_{21} (z) - a_{21}' (z)) - \lambda a_{21} (z)}{a_{21} (z)+ca_{22} (z)}
    \quad
    b_{22} (z) = \frac{c (e_{22} (z) - \lambda - a_{22}' (z))}{a_{21} (z)+ca_{22} (z)}.
\end{aligned}
\end{equation*}

So far, we've been viewing \eqref{isl} and the subsequent equations as posed on 
$\mathbb{R}$, but due to the discontinuity of the wave $(H (z), Q(z))$ at $z = 0$
it's convenient to re-formulate the problem on $(-\infty, 0]$. Following an approach 
of Erpenbeck and Majda (see \cite{Erpenbeck1962a, Erpenbeck1962b, Majda1983}), the 
authors of \cite{SYZ2020} carry this out, restricting \eqref{syz2020evp1} to 
$(-\infty,0)$ and adding the following boundary condition:
\begin{equation*}
    \Big(\lambda \begin{pmatrix} H(0) \\ Q (0) \end{pmatrix}
    - R \begin{pmatrix} H(0) \\ Q (0) \end{pmatrix} \Big)^* J A (0) v(0) = 0,
    \quad J = \begin{pmatrix}
        0 & -1 \\
        1 & 0
    \end{pmatrix}
\end{equation*}
where 
\begin{equation*}
    R \begin{pmatrix} H(0) \\ Q (0) \end{pmatrix}
    := \begin{pmatrix} 0 \\ H (0) - \frac{|Q(0)| Q(0)}{H(0)^2} \end{pmatrix}.
\end{equation*}

In order to formulate \eqref{syz2020evp1} in the current setting, we make
a transformation $v = TY$, with 
\begin{equation*}
    T = e^{\frac{1}{2} \int_0^z b_{11} (\zeta) + c b_{12} (\zeta) d\zeta}
    \begin{pmatrix}
        \frac{1}{2} (b_{11} + cb_{12}) & 1 \\
        c - \lambda & c
    \end{pmatrix},
\end{equation*}
obtaining $JY' = \mathcal{B} Y$ with (replacing now $z$ with $x$)
\begin{equation*}
    \mathcal{B} (x; \lambda) = 
    \begin{pmatrix}
        v(x) + q_1 (x) \lambda + q_2 (x) \lambda^2 & 0 \\
        0 & 1
    \end{pmatrix},
\end{equation*}
where the three functions $v(x)$, $q_1 (x)$, and $q_2 (x)$ are all carefully 
analyzed in \cite{SYZ2020}, namely in terms of functions $f_1$, $f_2$, $f_3$,
and $f_4$ that we won't specify precisely here (see Appendix A of \cite{SYZ2020} for 
precise specifications). Here, we only briefly summarize the salient properties
of these functions. First, we have the following relations:
\begin{equation*}
    \begin{aligned}
    v(x) &= - \frac{1}{4} f_2 (x)^2 - \frac{1}{2} f_2'(x) \\
    q_1 (x) &= f_4 (x) - \frac{1}{2} f_1 (x) f_2 (x) - \frac{1}{2} f_1'(x) \\
    q_2 (x) &= f_3 (x) - \frac{1}{4} f_1 (x)^2.
    \end{aligned}
\end{equation*}
From \cite{SYZ2020}, we have the following properties. The functions $f_1$,
$f_2$, $f_3$, and $f_4$ are all uniformly bounded on $(-\infty, 0]$ with 
bounded derivatives, so $v, q_1, q_2 \in L^{\infty} ((-\infty,0], \mathbb{R})$.
From the proof of Lemma 3.1 in \cite{SYZ2020}, there exists a constant $\delta_1 > 0$
so that $q_1 (x) \le -\delta_1$ for all $x \in (- \infty, 0])$, and likewise from 
equation (4.13) in \cite{SYZ2020}, there exists a constant $\delta_2 > 0$ so that 
$q_2 (x) \le -\delta_2$ for all $x \in (- \infty, 0])$. 

For our purposes, it will be convenient to express our system on $[0, \infty)$, and for 
this we introduce a new variable, $y(x) := Y (-x)$, $x \ge 0$. Our system then becomes 
\eqref{hammy} with 
\begin{equation} \label{isl-B}
    \mathbb{B} (x; \lambda) = 
    \begin{pmatrix}
        - v(-x) - q_1 (-x) \lambda - q_2 (-x) \lambda^2 & 0 \\
        0 & -1
    \end{pmatrix},
\end{equation}
and the boundary condition 
\begin{equation} \label{isl-bc}
    \alpha y(0) = 0; 
    \quad \alpha = \begin{pmatrix} (c_1 + c_2 \lambda) & -1 \end{pmatrix}.
\end{equation}
where the constants $c_1$ and $c_2$ are specified respectively in (4.8) and 
(4.9) of \cite{SYZ2020}. There is a typo in (4.8) from \cite{SYZ2020}, and 
the correct specification should be 
\begin{equation} \label{c1specified}
    c_1 = \frac{1}{2} f_2 (0) 
    - F^2 \frac{H_* (H_R + \sqrt{H_R} + 1)^2 - H_R (2 F^2 + 1)}{(\sqrt{H_R} +1)^2 (H_*^3 - H_S^3)},
\end{equation}
where $H_R$ is as in Proposition \ref{YZ-prop} and 
\begin{equation} \label{isl-specifications}
    H_* = \frac{-\nu-1+\sqrt{8F^2\nu^4+\nu^2+2\nu+1}}{2(\nu+1)} H_R;
    \quad H_s = \Big(\frac{F \nu^2}{\nu+1} \Big)^{2/3} H_R; 
    \quad \nu = \sqrt{\frac{H_L}{H_R}} > 1.
\end{equation}
According to (4.9) in \cite{SYZ2020}, $c_2 < 0$.

In order to apply Lemma \ref{quadratic-sturm-liouville-lemma}, we need to establish 
that Assumptions {\bf (Q)} hold. For this, a key point is the boundedness of 
$v(x)$, $q_1 (x)$, and $q_2 (x)$, along with other properties established in 
\cite{YZ2020, SYZ2020}. For convenient reference, we summarize these in the 
following proposition.

\begin{proposition} \label{hydraulic-shocks-coefficients-proposition}
Let $0 < F < 2$, $\nu > 1$, and suppose \eqref{OurCase}
holds. If $H (x)$ denotes the profile from Proposition 
\ref{YZ-prop} connecting $H_L$ to $H_R$, shifted so that  
the unique point of discontinuity of $H(x)$ occurs at $x = 0$, then 
the following hold:

\vspace{.1in}
(i) $H' (x) < 0$ for all $x \in (-\infty, 0]$, with $H'(0)$ specified
as $x$ approaches 0 from the left;

\vspace{.1in}
(ii) The left-sided limit of $H(x)$ at the point of discontinuity exists and 
is given by 
\begin{equation*}
    \lim_{x \to 0^-} H(x) = H_*;
\end{equation*}

\vspace{.1in}
(iii) The right-sided limit of $H(x)$ at the point of discontinuity exists and 
is given by 
\begin{equation*}
    \lim_{x \to 0^+} H(x) = H_R;
\end{equation*} 

\vspace{.1in}
(iv) $H_R < H_S < H_* < H_L$;

\vspace{.1in}
(v) $v, q_1, q_2 \in L^{\infty} ((-\infty,0),\mathbb{R})$;

\vspace{.1in}
(vi) There exists some $\delta > 0$ so that for all $x \in (-\infty, 0])$,
we have $q_1 (x) \le - \delta$ and $q_2 (x) \le - \delta$. 
    
\end{proposition}

\begin{proof}
    Most of these properties have been established in our references \cite{SYZ2020, YZ2020}, so 
    we primarily just indicate where the results can be found in those references. First, 
    Items \textit{(i)}, \textit{(ii)}, and \textit{(iii)} simply combine observations from \cite{SYZ2020}, p. 3 
    (discussion surrounding equation (2.4) of that reference). 

    For Item \textit{(iv)}, we start with the inequality $H_R < H_s$. Noting that 
    $\nu^2 = H_L/H_R$, we can express \eqref{OurCase} as 
    \begin{equation*}
        \nu^2 > \frac{1 + 2F + \sqrt{1+4F}}{2F^2}.
    \end{equation*}
    Rewriting this inequality as 
    \begin{equation*}
        2F^2 \nu^2 - 1 - 2F > \sqrt{1+4F},
    \end{equation*}
    we observe that if this inequality holds then both sides must be positive, and we 
    can square both sides. If we do this and cancel the quantity $1+4F$ from either 
    side of the resulting equation, we arrive at the inequality 
    \begin{equation*}
        \Big((F\nu^2 - 1) - \nu \Big) \Big((F\nu^2 - 1) + \nu \Big) > 0.
    \end{equation*}
    The second factor on the left-hand side is clearly positive (since $\nu > 1$),
    and so the first factor must be positive as well, ensuring the inequality 
    \begin{equation} \label{isl-key-inequality}
     \frac{F\nu^2}{1 + \nu}  > 1,  
    \end{equation}
    whence 
    \begin{equation*}
    H_s = \Big(\frac{F \nu^2}{\nu+1} \Big)^{2/3} H_R > H_R.    
    \end{equation*}
    
    Turning next to the inequality $H_s < H_L$, we see that our goal can be expressed 
    as 
    \begin{equation*}
    \Big(\frac{F \nu^2}{\nu+1} \Big)^{2/3} H_R < H_L,    
    \end{equation*}
    or equivalently 
    \begin{equation*}
    \Big(\frac{F \nu^2}{\nu+1} \Big)^{2/3} < \nu^2.    
    \end{equation*}
    Rearranging this last inequality, we arrive at the relation 
    \begin{equation*}
        F < \nu (\nu+1),
    \end{equation*}
    and this holds for all $0 < F < 2$ and $\nu > 1$.

    To see that $H_* < H_L$, we first observe that $H_*$ can be expressed as 
    \begin{equation*}
        \begin{aligned}
            H_* &= \frac{-\nu-1+\sqrt{8F^2\nu^4+\nu^2+2\nu+1}}{2(\nu+1)} \frac{H_L}{\nu^2} \\
            &= \Big(-\frac{1}{2} + \frac{\sqrt{8F^2 + (\frac{1}{\nu} + \frac{1}{\nu^2})^2}}{2 (\nu+1)} \Big) H_L 
        \end{aligned}
    \end{equation*}
    Since $F \nu^2 > 1 + \nu$ (as shown during the verification that $H_R < H_s$), we see 
    that $F > (1/\nu + 1/\nu^2)$, so that 
    \begin{equation*}
        H_* < \Big(-\frac{1}{2} + \frac{\sqrt{8F^2 + F^2}}{2 (\nu+1)} \Big) H_L
        = \Big(-\frac{1}{2} + \frac{3F}{2 (\nu+1)} \Big) H_L < H_L,
    \end{equation*}
    where the final inequality follows for all $0 < F < 2$ and $\nu > 1$. 

    For the inequality $H_* > H_s$, we first express it as 
    \begin{equation*}
    \frac{-\nu-1+\sqrt{8F^2\nu^4+\nu^2+2\nu+1}}{2(\nu+1)} H_R
    < \Big(\frac{F \nu^2}{\nu+1} \Big)^{2/3} H_R,
    \end{equation*}
    which is equivalent to 
    \begin{equation*}
    \frac{-\nu-1+\sqrt{8F^2\nu^4+(\nu+1)^2}}{2(\nu+1)^{1/3}}
    < (F \nu^2)^{2/3}.
    \end{equation*}
    From \eqref{isl-key-inequality}, we can conclude the inequality 
    \begin{equation*}
        \frac{1}{(\nu+1)^{1/3}} > \frac{1}{(F \nu^2)^{1/3}}, 
    \end{equation*}
    which allows us to write 
    \begin{equation*}
        \frac{-\nu-1+\sqrt{8F^2\nu^4+(\nu+1)^2}}{2(\nu+1)^{1/3}}
        > \frac{-\nu-1+\sqrt{8F^2\nu^4+(\nu+1)^2}}{2(F \nu^2)^{1/3}}.
    \end{equation*}
    In this way, we see that the sought inequality will follow if we can show that 
    \begin{equation*}
     \frac{-\nu-1+\sqrt{8F^2\nu^4+(\nu+1)^2}}{2(F \nu^2)^{1/3}}
    > (F \nu^2)^{2/3}. 
    \end{equation*}
    This is equivalent to the inequality 
    \begin{equation*}
        -\nu-1+\sqrt{8F^2\nu^4+(\nu+1)^2 > 2F \nu^2},
    \end{equation*}
    which we can rearrange as 
    \begin{equation*}
    \sqrt{8F^2\nu^4+(\nu+1)^2} > 2F \nu^2 + \nu + 1.    
    \end{equation*}
    Upon squaring both sides, we find that this reduces to (and so follows from)
    \eqref{isl-key-inequality}. This concludes the proof of Item \textit{(iv)}.
    
    For Item \textit{(v)}, since $H > H_* > H_S$,
    we have from (3.4) of \cite{SYZ2020} that $q_1 \in L^{\infty} ((-\infty,0),\mathbb{R})$,
    and likewise from (4.13) of the same reference that $q_2 \in L^{\infty} ((-\infty,0),\mathbb{R})$.
    For $v(x)$, we see from the specification of $f_2$ in the appendix of \cite{SYZ2020} that 
    $v(x)$ is bounded provided 
    \begin{equation} \label{to-verify}
        H^3 (\sqrt{H_R} + 1)^2 - F^2 H_R^2 \ne 0 \quad 
        \forall H \in [H_*, H_L]. 
    \end{equation}
    From (3.8) of \cite{SYZ2020}, we additionally have the inequality 
    $F < H_R + \sqrt{H_R}$, and this allows us to compute 
    \begin{equation*}
        \begin{aligned}
            H^3 (\sqrt{H_R}+1)^2 - F^2 H_R^2 
            &> H_R^3 (H_R + 2 \sqrt{H_R} + 1) - (H_R + \sqrt{H_R})^2 H_R^2 \\
            &= H_R^4 - 2H_R^{7/2} + H_R^3 - (H_R^2 + 2H_R^{3/2} + H_R) H_R^2 = 0,
        \end{aligned}
    \end{equation*}
    verifying \eqref{to-verify}. 

    Finally Item \textit{(vi)} follows from the proofs of Lemma 3.1 in \cite{SYZ2020}
    (the claim for $q_1$) and Lemma 4.1 in the same reference (the claim for $q_2$).
    
\end{proof}

We now verify that Assumptions {\bf (Q)} hold for \eqref{hammy} with \eqref{isl-B}
and boundary condition \eqref{isl-bc}. For this, we first observe the correspondences
$Q_1 (x) = - q_1 (-x)$, $Q_2 (x) = - q_2 (-x)$, and $V(x) = v (-x)$, and note that
the appearance of $-1$ in the lower right entry of $\mathbb{B} (x; \lambda)$ (rather
than $+1$ as in Section \ref{quadratic-schrodinger-systems-section}) has no bearing 
on the proof of Lemma \ref{quadratic-sturm-liouville-lemma}. 

The conditions $Q_1, Q_2, V \in L^1_{\loc} ((0, \infty), \mathbb{R}^{n \times n})$, and 
also $Q_1, Q_2 \in L^{\infty} ((0, \infty), \mathbb{R}^{n \times n})$, all follow from 
Item \textit{(iii)} in Proposition \ref{hydraulic-shocks-coefficients-proposition}. 
For the final statement in Assumption {\bf (Q)}, we need to verify that there exists
some $\theta > 0$ so that for all $x \in [0, \infty)$, and for all 
$\lambda \in [0, \lambda_{\infty}]$, we have the relation 
\begin{equation*}
    (- q_1 (-x) - 2 \lambda q_2 (-x)) \ge \theta.
\end{equation*}
Since $q_2 (-x) < 0$ and $\lambda \ge 0$, it is sufficiently to verify that 
$q_1 (-x) \le -\theta$ for all $x \in [0, \infty)$. But this is an immediate
consequence of Item \textit{(iv)} in Proposition \ref{hydraulic-shocks-coefficients-proposition}. 

In order to apply Theorem \ref{regular-singular-theorem}, we still need to verify 
the positivity condition associated with the boundary, that $\lambda = 0$ isn't an eigenvalue,
and that we can choose $\lambda_{\infty}$ sufficiently large so that there are no eigenvalues 
greater than or equal to $\lambda_{\infty}$. 
For this latter condition, we observe that if we set $w (x) = y_1 (-x)$, the $w$ solves 
the second-order ODE 
\begin{equation} \label{isl-second-order}
\begin{aligned}
    w'' &= - (v(x) + q_1 (x) \lambda + q_2 (x) \lambda^2) w; \quad x \in (-\infty, 0] \\
    w' (0) &= (c_1 + c_2 \lambda) w(0),
\end{aligned}
\end{equation}
which is precisely (2.14)--(4.6) in \cite{SYZ2020}. In Lemma 3.3 of \cite{SYZ2020} the 
authors show that $\lambda = 0$ isn't an eigenvalue of \eqref{isl-second-order}, so we only
need to verify the existence of $\lambda_{\infty}$. In \cite{SYZ2020}, the authors additionally 
show that the eigenvalues of \eqref{isl-second-order} are all real-valued, so we can focus on 
real values of $\lambda$. We fix $\lambda \in (0, \infty)$ and let $w(x; \lambda)$ denote 
an associated eigenfunction. Taking an $L^2 ((-\infty,0),\mathbb{C})$ inner product of $w^*$ with 
\eqref{isl-second-order}, we obtain the relation 
\begin{equation} \label{isl-energy-argument}
    \int_{-\infty}^0 w^* w '' dx
    = - \int_{-\infty}^0 (v(x) + q_1 (x) \lambda + q_2 (x) \lambda^2) |w|^2 dx,
\end{equation}
where for the left-hand side we have 
\begin{equation} \label{isl-boundary-calculation}
    \int_{-\infty}^0 w^* w '' dx
    = - \int_{-\infty}^0 |w'|^2 dx + w^* (0) w' (0)
    = - \int_{-\infty}^0 |w'|^2 dx + (c_1 + c_2 \lambda) |w (0)|^2. 
\end{equation}
If $w(0) = 0$, then $w' (0) = 0$, and consequently we must have $w \equiv 0$ on $(-\infty, 0]$,
so that $w$ isn't an eigenfunction. In this way, we see that we must have $w (0) \ne 0$,
and since $c_2 < 0$, this implies that $\lambda$ can be chosen sufficiently large so 
that the right-hand side of \eqref{isl-boundary-calculation} is negative. We can conclude
from \eqref{isl-energy-argument} that for $\lambda$ sufficiently large 
\begin{equation*}
    \int (v(x) + q_1 (x) \lambda + q_2 (x) \lambda^2) |w|^2 dx > 0,
\end{equation*}
but since $q_2 (x) < 0$ this is a contradiction (for $\lambda$ sufficiently large). 

Having verified the assumptions of Theorem \ref{regular-singular-theorem}, we are now
in a position to consider the appropriate Maslov indices. According to Theorem 
\ref{regular-singular-theorem}, 
\begin{equation} \label{isl-indices}
    \mathcal{N}^{\alpha} ((0, \lambda_{\infty}))
    = \mas (\ell_{\alpha} (\cdot; \lambda_{\infty}), \ell_{\infty} (\cdot; 0); [0, \infty)])
    - \mas (\ell_{\alpha} (0; \cdot), \ell_{\infty} (0; 0); [0, \lambda_{\infty}]),
\end{equation}
where $\lambda_{\infty}$ has been chosen sufficiently large (as discussed above) so 
that our system \eqref{hammy} with \eqref{isl-B} and boundary conditions \eqref{isl-bc}
has no eigenvalues greater than or equal to $\lambda_{\infty}$. 

Our goal now is to show that each of the Maslov indices in \eqref{isl-indices} is 0. For
both of these calculations, a key point will be that we can identify the space $\ell_{\infty} (x; 0)$
explicitly, so we begin by clarifying that observation. First, for $\lambda = 0$, 
we have the system 
\begin{equation*}
    J \mathbf{X}_{\infty}' (x; 0) = \mathbb{B} (x; 0) \mathbf{X}_{\infty} (x; 0);
    \quad \lim_{x \to \infty} \mathbf{X}_{\infty} (x;0) = 0.
 \end{equation*}
If we write $\mathbf{X}_{\infty} = \genfrac{(}{)}{0pt}{1}{\zeta_1}{\zeta_2}$, this system becomes 
\begin{equation*}
    \begin{aligned}
    \zeta_1' &= - \zeta_2 \\
    \zeta_2' &= v (-x) \zeta_1,
    \end{aligned}
\end{equation*}
which is easily solved with (for $x > 0$)
\begin{equation*}
    \begin{aligned}
        \zeta_1 (x) &= e^{\frac{1}{2}\int_0^{-x} f_2 (y) dy} H (-x) \\
        \zeta_2 (x) &= e^{\frac{1}{2}\int_0^{-x} f_2 (y) dy} \Big(H' (-x) + \frac{1}{2} f_2 (-x) H (-x) \Big).
    \end{aligned}
\end{equation*}
In this way, we identify a frame for $\ell_{\infty} (0;0)$ as 
\begin{equation*}
    \mathbf{X}_{\infty} (0;0)
    = \begin{pmatrix} \zeta_1 (0) \\ \zeta_2 (0) \end{pmatrix} 
    = \begin{pmatrix} H (0) \\ H' (0) + \frac{1}{2}f_2 (0)H(0) \end{pmatrix}. 
\end{equation*}
A frame for $\ell_{\alpha} (0;\lambda)$ is readily seen to be 
\begin{equation*}
    \mathbf{X}_{\alpha} (0;\lambda) = J \alpha (\lambda)^*
    = \begin{pmatrix} 1 \\ c_1 + c_2 \lambda \end{pmatrix},
\end{equation*}
so intersections in the calculation of the second Maslov index on the 
right-hand side of \eqref{isl-indices} correspond precisely with zeros
of the determinant 
\begin{equation} \label{isl-determinant}
    \det \begin{pmatrix}
    1 & H(0) \\
    c_1 + c_2 \lambda & H'(0) + \frac{1}{2} f_2 (0) H(0) 
    \end{pmatrix}
    = H'(0) + \Big(\frac{1}{2} f_2 (0) - (c_1 + c_2 \lambda) \Big) H(0).
\end{equation} 
Here, $H(0) = H_*$, and from \cite{YZ2020} we have 
\begin{equation*}
    H' (x) = \frac{H(x)^3 - Q(x)^2}{\frac{H(x)^3}{F^2} - q_0^2};
    \quad q_0 = Q (x) - cH (x). 
\end{equation*}
(In \cite{YZ2020}, this is the first unnumbered equation after (2.7); we emphasize
that $Q(x) - cH(x)$ is independent of $x$.)  Upon
setting $x = 0$, we obtain the relation 
\begin{equation*}
    H' (0) = \frac{H_*^3 - Q_*^2}{\frac{H_*^3}{F^2} - q_0^2}, 
    \quad Q_* = q_0 + cH_*.
\end{equation*}
From (2.9) in \cite{YZ2020}, we have $H_S = q_0^2 F^2$, allowing us to 
write 
\begin{equation} \label{Hprime}
    H' (0) = F^2 \frac{H_*^3 - Q_*^2}{H_*^3 - H_S^3}.
\end{equation}
Focusing now on the numerator $H_*^3 - Q_*^2$, we can write 
\begin{equation*}
    H_*^3 - Q_*^2
    = H_*^3 - (cH_* - q_0)^2
    = H_*^3 - c^2 H_*^2 + 2c H_* q_0 - q_0^2,
\end{equation*}
which can be expressed as 
\begin{equation} \label{HstarminusQstar}
    H_*^3 - Q_*^2
    = H_*^3 - (cH_* - q_0)^2
    = H_*^3 - c^2 H_*^2 + 2c H_* \frac{H_S^{3/2}}{F} - \frac{H_S^3}{F^2},
\end{equation}
(using again $H_S = q_0^2 F^2$). For $c$, we can combine the relations 
$Q_R = cH_R - q_0$ and $Q_R = H_R^{3/2}$ to see that 
\begin{equation*}
    c = \frac{H_S^{3/2}}{F H_R} + \sqrt{H_R}. 
\end{equation*}
If we substitute this last expression for $c$ into \eqref{HstarminusQstar}, we find 
\begin{equation*}
    \begin{aligned}
     H_*^3 - Q_*^2 &= H_*^3 - (\sqrt{H_R} + \frac{H_S^{3/2}}{F H_R})^2 H_*^2 
     + 2 (\sqrt{H_R} + \frac{H_S^{3/2}}{F H_R}) H_* \frac{H_S^{3/2}}{F} - \frac{H_S^3}{F^2} \\
     &= \frac{H_* - H_R}{F^2 H_R^2} \Big{\{} F^2 H_R^2 H_*^2
     - 2F H_R^{3/2} H_S^{3/2} H_* - H_S^3 (H_* - H_R) \Big{\}}.
    \end{aligned}
\end{equation*}
Rearranging, we find the relation 
\begin{equation} \label{HstarQstarRatio}
    \frac{H_*^3 - Q_*^2}{H_R - H_*}
    = - H_*^2 + \frac{2 H_S^{3/2} H_*}{F \sqrt{H_R}} 
    + \frac{H_S^3 (H_* - H_R)}{F^2 H_R^2}. 
\end{equation}
From (4.7) in \cite{SYZ2020}, we know 
\begin{equation*}
    H_*^2 (\sqrt{H_R} +1)^2 + H_* H_R (\sqrt{H_R} +1)^2 - 2F^2 H_R = 0,
\end{equation*}
which we can rearrange as 
\begin{equation*}
    H_*^2 = - H_* H_R + \frac{2F^2 H_R}{(\sqrt{H_R}+1)^2}.
\end{equation*}
Upon substitution of this last expression into \eqref{HstarQstarRatio},
we obtain the relation 
\begin{equation*} 
    \frac{H_*^3 - Q_*^2}{H_R - H_*}
    = H_* H_R - \frac{2F^2 H_R}{(\sqrt{H_R}+1)^2}
    + \frac{2 H_S^{3/2}H_*}{F \sqrt{H_R}} + \frac{H_S^3 (H_* - H_R)}{F^2 H_R^2}.
\end{equation*}
In order to make a final simplification, we observe the relation 
\begin{equation*}
    H_S^3 = \frac{F^2 H_R^2}{(1+\sqrt{H_R})^2}, 
\end{equation*}
which is easily seen from \eqref{isl-specifications}. Namely, we 
can write 
\begin{equation*}
    H_S^3 = \frac{F^2 \nu^4}{(\nu+1)^2} H_R^3
    = \frac{F^2 \frac{H_L^2}{H_R^2}}{(\sqrt{\frac{H_L}{H_R}} + 1)^2} H_R^3
    = \frac{F^2 H_L H_R^2}{(1 + \sqrt{\frac{H_R}{H_L}})^2},
\end{equation*}
and by a choice of scaling we take $H_L = 1$. This allows us to write 
\begin{equation*} 
    \frac{H_*^3 - Q_*^2}{H_R - H_*}
    = \Big(\frac{H_R (1+\sqrt{H_R})^2 + 2 \sqrt{H_R}{(1+\sqrt{H_R})+ 1}}{(1+\sqrt{H_R})^2} \Big) H_*
    - \frac{(1+2F^2) H_R}{(1+\sqrt{H_R})^2}.
\end{equation*}
Noting the identity 
\begin{equation*}
     H_R (1+\sqrt{H_R})^2 + 2\sqrt{H_R} (1+\sqrt{H_R}) +1
    = (H_R + \sqrt{H_R}+1)^2,
\end{equation*}
we arrive at the relation 
\begin{equation*}
    \frac{H_*^3 - Q_*^2}{H_R - H_*}
    = \frac{(H_R + \sqrt{H_R}+1)^2 H_* - (1+2F^2)H_R}{(\sqrt{H_R} + 1)^2}.
\end{equation*}
By incorporating this last relation into \eqref{c1specified}, we arrive at the 
relations
\begin{equation} \label{TheRest}
    \frac{1}{2} f_2 (0) - c_1
    = F^2 \frac{H_*^3 - Q_*^2}{H_R - H_*} \cdot \frac{1}{H_*^3 - H_S^3}. 
\end{equation}

We are now in a position to better understand the right-hand side of 
\eqref{isl-determinant}. Using \eqref{Hprime} and \eqref{TheRest}, we can write 
\begin{equation} \label{AllTogether}
\begin{aligned}
    H' (0) &+ (\frac{1}{2} f_2 (0) - c_1) H (0) 
    = F^2 \frac{H_*^3 - Q_*^2}{H_*^3 - H_S^3}
    + F^2 \frac{H_*^3 - Q_*^2}{(H_R - H_*) (H_*^3 - H_S^3)} H_* \\
    & = F^2 \frac{H_*^3 - Q_*^2}{H_*^3 - H_S^3} \cdot \frac{H_R}{H_R - H_*}
\end{aligned}
\end{equation}
From Proposition \ref{hydraulic-shocks-coefficients-proposition}, we have the inequalities 
$H_*^3 - H_S^3 > 0$ and $H_R - H_* < 0$, and since $H'(0) < 0$ we also see (from \eqref{Hprime}) that 
$H_*^3 - Q_*^2 < 0$. Combining these observations with \eqref{AllTogether}, 
we see at last that 
\begin{equation*}
    H' (0) + (\frac{1}{2} f_2 (0) - c_1) H (0) > 0. 
\end{equation*}
Since $c_2 < 0$, we can conclude that the right-hand side of \eqref{isl-determinant}
is positive for all $\lambda \in [0, \lambda_{\infty}]$. In this way, we conclude 
that 
\begin{equation*}
    \mas (\ell_{\alpha} (0; \cdot), \ell_{\infty} (0; 0); [0, \lambda_{\infty}])
    = 0.
\end{equation*}

For the first Maslov index in \eqref{isl-indices}, we note from our construction above
that 
\begin{equation*}
    \mathbf{X}_{\infty} (x;0)
    = \begin{pmatrix} \zeta_1 (x) \\ \zeta_2 (x) \end{pmatrix} 
    = \begin{pmatrix} e^{\frac{1}{2} \int_0^{-x} f_2 (y) dy} H(-x) \\ 
    e^{\frac{1}{2} \int_0^{-x} f_2 (y) dy} (H' (-x) + \frac{1}{2}f_2 (-x)H(-x)) \end{pmatrix}. 
\end{equation*}
This allows us to identify intersections arising in the Maslov index under 
consideration with solutions to the eigenvalue problem \eqref{isl-second-order}. In 
particular, the question becomes, for $\lambda = \lambda_{\infty}$, are there values
of $x \in (-\infty, 0]$ for which $\genfrac{(}{)}{0pt}{1}{w (x; \lambda_{\infty})}{ w' (x; \lambda_{\infty})}$ 
intersects 
\begin{equation*}
    \begin{pmatrix} W (x) \\ W'(x) \end{pmatrix}
    := \begin{pmatrix} e^{\frac{1}{2} \int_0^{x} f_2 (y) dy} H(x) \\ 
    e^{\frac{1}{2} \int_0^{x} f_2 (y) dy} (H' (x) + \frac{1}{2}f_2 (x)H(x)) \end{pmatrix}.
\end{equation*}
An intersection will occur if there exists $s \le 0$ so that 
\begin{equation*}
    \det \begin{pmatrix}
        w (s; \lambda_{\infty}) & W (s) \\
        w' (s; \lambda_{\infty}) & W' (s)
    \end{pmatrix}
    = 0.
\end{equation*}
I.e.,  the condition for an intersection is 
\begin{equation*}
    W(s) w' (s; \lambda_{\infty}) - W'(s) w(s; \lambda_{\infty}) = 0,
\end{equation*}
and since the exponential factor is irrelevant, this can be expressed as 
\begin{equation*}
    H(s) w' (s; \lambda_{\infty}) 
    - (H' (s) + \frac{1}{2} f_2 (s) H(s)) w(s; \lambda_{\infty}) = 0.
\end{equation*}
We will have an intersection in the calculation of the first Maslov index 
in \eqref{isl-indices} if and only if there exists a function 
$w (\cdot; \lambda_{\infty}) \in H^2 (s, 0)$ satisfying the boundary 
value problem 
\begin{equation} \label{isl-bvp}
\begin{aligned}
    w'' &= - (v(x) + q_1 (x) \lambda + q_2 (x) \lambda^2) w; \quad x \in (s, 0] \\
    w' (0) &= (c_1 + c_2 \lambda) w(0), \\
    w' (s; \lambda) &= \Big(\frac{H'(s)}{H(s)} + \frac{1}{2} f_2 (s) \Big) w (s; \lambda).
\end{aligned}
\end{equation}
Suppose that some $\lambda > 0$ is an eigenvalue of this problem, and let $w (x; \lambda)$
denote its associated eigenfunction. As noted in \eqref{isl-energy-argument}, if we 
multiply this ODE by $w^* (x; \lambda)$ and integrate on $[s, 0]$, the right-hand side becomes 
\begin{equation*}
   -  \int_s^0 (v(x) + q_1 (x) \lambda + q_2 (x) \lambda^2) |w (x; \lambda)|^2 dx. 
\end{equation*}
From Proposition \ref{hydraulic-shocks-coefficients-proposition}, we have that 
$v \in L^{\infty} (s,0)$ and that for some $\delta > 0$ we have $q_1 (x) \le - \delta$
and $q_2 (x) \le - \delta$ for all $x \in [s,0]$. It follows that we can choose 
$\lambda_{\infty}$ sufficiently large so that 
\begin{equation} \label{delta-inequality}
   -  \int_s^0 (v(x) + q_1 (x) \lambda + q_2 (x) \lambda^2) |w (x; \lambda_{\infty})|^2 dx
   \ge \frac{\lambda_{\infty}^2}{2} \delta \|w\|_{L^2 (s,0)}.
\end{equation}
On the left-hand side, we obtain the relation 
\begin{equation} \label{parts-inequality}
\begin{aligned}
    \int_s^0 w'' w^* dx &= - \int_s^0 |w_x (x; \lambda_{\infty})|^2 dx
    + w_x (0; \lambda_{\infty}) w^* (0; \lambda_{\infty})
    - w_x (s; \lambda_{\infty}) w^* (s; \lambda_{\infty}) \\
    &= - \int_s^0 |w_x (x; \lambda_{\infty})|^2 dx
    + (c_1 + c_2 \lambda_{\infty}) |w(0; \lambda_{\infty})|^2
    - (\frac{H'(s)}{H(s)} + \frac{1}{2} f_2 (s)) |w(s; \lambda_{\infty})|^2.
\end{aligned}
\end{equation}
Here, $H'(s)$ is uniformly bounded on $(-\infty, 0)$, and $H(s)$ is 
bounded away from 0 on this same interval, so the ratio $H' (s)/ H(s)$ is uniformly 
bounded on $(-\infty, 0)$. Likewise, $f_2 (s)$ is uniformly bounded on $(-\infty, 0)$,
so the quantity $\frac{H'(s)}{H(s)} + \frac{1}{2} f_2 (s)$ is uniformly bounded on 
this interval. 

Precisely in the spirit of Lemma 1.3.8 from \cite{BK2013}, we find that for any 
$\epsilon > 0$, we have the inequality 
\begin{equation} \label{BK2013inequality}
    |w (s; \lambda_{\infty})|^2 
    \le |w (0; \lambda_{\infty})|^2 + \frac{1}{\epsilon} \|w (\cdot; \lambda_{\infty})\|_{L^2 (s, 0)}^2
    + \epsilon \|w_x (\cdot; \lambda_{\infty})\|_{L^2 (s, 0)}^2. 
\end{equation}
This inequality, combined with \eqref{delta-inequality} and \eqref{parts-inequality},   
allows us to write 
\begin{equation*}
\begin{aligned}
\frac{\lambda_{\infty}^2}{2} \delta \|w\|_{L^2 (s,0)} &+ \|w_x (\cdot; \lambda_{\infty})\|_{L^2 (s, 0)}^2
- (c_1 + c_2 \lambda_{\infty}) w(0; \lambda_{\infty})^2  
\le - (\frac{H'(s)}{H(s)} + \frac{1}{2} f_2 (s)) w(s; \lambda_{\infty})^2 \\
& \le \Big| \frac{H'(s)}{H(s)} + \frac{1}{2} f_2 (s) \Big| 
\Big{\{} |w(0; \lambda_{\infty})|^2 + \frac{1}{\epsilon} \|w (\cdot; \lambda_{\infty})\|_{L^2 (s, 0)}^2
    + \epsilon \|w_x (\cdot; \lambda_{\infty})\|_{L^2 (s, 0)}^2 \Big{\}}.
\end{aligned}
\end{equation*}
Since $c_2 < 0$ and $\frac{H'(s)}{H(s)} + \frac{1}{2} f_2 (s)$ is bounded, we can choose $\lambda_{\infty}$
large enough to ensure a contradiction. 

\appendix

\section{Appendix}

In this appendix, we include the following: (1) a full derivation of our Green's function 
$G^{\alpha} (x, \xi; \lambda)$ associated with the operator $\mathcal{T}^{\alpha} (\lambda)$
(i.e., a full proof of Lemma \ref{green-function-lemma}); (2) 
details on the monotonicity (in $\lambda$) arguments 
from the proofs of Theorems \ref{regular-singular-theorem} and 
\ref{singular-theorem}; and (3) a proof of Lemma \ref{assumption-d-lemma}.

\subsection{Derivation of the Green's Function $G^{\gamma}_{c, b} (x, \xi; \lambda)$}
\label{green-function-section}

In this section, we prove Lemma \ref{green-function-lemma}, establishing a representation
for the Green's function associated with (\ref{inhomogeneous-problem}). 
We begin by letting $\Phi (x; \lambda)$ denote a 
fundamental matrix for (\ref{extension-hammy}) satisfying
\begin{equation*}
    J \Phi' = \mathbb{B} (x; \lambda) \Phi, \quad \Phi (c; \lambda) = I_{2n}.
\end{equation*}
Proceeding by variation of parameters, we look for solutions of 
(\ref{inhomogeneous-problem})
of the form $y (x; \lambda) = \Phi (x; \lambda) v(x; \lambda)$, where
the vector function $v(x;\lambda)$ is to be determined. This leads
immediately to the relation $J \Phi (x; \lambda) v'(x; \lambda) = \mathbb{B}_{\lambda} (x; \lambda) f$.
Using the relation 
\begin{equation*}
    \Phi (x; \lambda)^* J \Phi (x; \lambda) = J,
\end{equation*}
we can solve for 
\begin{equation*}
    v'(x; \lambda) = - J \Phi (x; \lambda)^* \mathbb{B}_{\lambda} (x; \lambda) f.
\end{equation*}
Integrating, we obtain the relation
\begin{equation*}
    v(x; \lambda) = - \int_c^x J \Phi (\xi; \lambda)^* \mathbb{B}_{\lambda} (\xi; \lambda) f (\xi) d\xi
    + K(\lambda),
\end{equation*}
where $K (\lambda)$ is a (vector) constant of integration to be determined, and consequently 
we can write 
\begin{equation} \label{green-y1}
    y(x; \lambda)
    = - \Phi (x; \lambda) \int_c^x J \Phi (\xi; \lambda)^* \mathbb{B}_{\lambda} (\xi; \lambda) f (\xi) d\xi
    + \Phi (x; \lambda) K (\lambda).
\end{equation}
In order to identify $K(\lambda)$, we employ the boundary conditions associated with 
$\mathcal{T}_{c, b}^{\gamma} (\lambda)$, namely (\ref{boundary-c}) and (\ref{boundary-b}).
For the former, since $y (x; \lambda) = \Phi (x; \lambda) v (x; \lambda)$ and 
$\Phi (c; \lambda) = I_{2n}$, we see
that $y (c; \lambda) = v (c; \lambda) = K (\lambda)$, so that the condition is
$\gamma K (\lambda)  = 0$, which is equivalent to 
\begin{equation*}
    (J \gamma^*)^* J K (\lambda) = 0.
\end{equation*}
For the latter, in Lemma \ref{lemma2-9prime}, we constructed a basis $\{u_j^b (x; \lambda)\}_{j=1}^n$
for the space of solutions to (\ref{extension-hammy}) that lie
right in $(c, b)$ and satisfy (\ref{boundary-b}). We let $U^b (x; \lambda)$
denote the matrix comprising the vector functions $\{u_j^b (x; \lambda)\}_{j=1}^n$
as its columns, noting that we then have  
\begin{equation} \label{green-bc}
    \lim_{x \to b^-} U^b (x; \mu_0, \lambda_0)^* J U^b (x; \lambda) = 0. 
\end{equation}
If we alternatively impose the boundary condition
\begin{equation} \label{boundary-b-equation-lambda}
    \lim_{x \to b^-} U^b (x; \lambda)^* J y(x) = 0, 
\end{equation}
then by the Lagrangian property we are effectively 
looking for a Green's function that can be expressed 
in terms of $U^b (x; \lambda)$ for $a < \xi < x < b$.
It follows from (\ref{green-bc}) that $G^{\gamma}_{c,b} (x, \xi; \lambda)$
will then satisfy the required boundary condition
(\ref{boundary-b}) (which can be checked 
directly with our final form of the Green's function). 

We now act with $U^b (x; \lambda)^* J$ on (\ref{green-y1})
to obtain the relation 
\begin{equation} \label{green-y2}
\begin{aligned}
    U^b (x; \lambda)^* J y(x; \lambda)
    &= - U^b (x; \lambda)^* J \Phi (x; \lambda) \int_c^x J \Phi (\xi; \lambda)^* \mathbb{B}_{\lambda} (\xi; \lambda) f (\xi) d\xi \\
    & + U^b (x; \lambda)^* J \Phi (x; \lambda) K (\lambda).
\end{aligned}
\end{equation}
Since the columns of $U^b (x; \lambda)$ solve (\ref{extension-hammy}) 
and are linearly independent, there must exist a rank-$n$ $2n \times n$ matrix $\mathbf{R}^b (\lambda)$
so that $U^b (x; \lambda) = \Phi (x; \lambda) \mathbf{R}^b (\lambda)$ (i.e., $\mathbf{R}^b (\lambda) = U^b (c; \lambda)$).
This allows us to express (\ref{green-y2}) as 
\begin{equation*}
    U^b (x; \lambda)^* J y(x; \lambda)
    = \int_c^x \mathbf{R}^b (\lambda)^* \Phi (\xi; \lambda)^* \mathbb{B}_{\lambda} (\xi; \lambda) f (\xi) d\xi
    + \mathbf{R}^b (\lambda)^* J K (\lambda).
\end{equation*}
By assumption, $\Phi (\cdot; \lambda) \mathbf{R}^b (\lambda) \in L^2_{\mathbb{B}_{\lambda}} ((c, b), \mathbb{C}^{2n})$,
and we are taking $f \in L^2_{\mathbb{B}_{\lambda}} ((c, b), \mathbb{C}^{2n})$, so we are justified in taking
$x \to b^-$ on the right-hand side. Since this limit is 0 on the left, we obtain 
the relation 
\begin{equation*}
    0
    = \int_c^b \mathbf{R}^b (\lambda)^* \Phi (\xi; \lambda)^* \mathbb{B}_{\lambda} (\xi; \lambda) f (\xi) d\xi
    + \mathbf{R}^b (\lambda)^* J K (\lambda).
\end{equation*}
Combining this last relation with $(J \gamma^*)^* J K (\lambda) = 0$, we can write 
\begin{equation*}
    \begin{pmatrix}
    (J \gamma^*)^* \\ \mathbf{R}^b (\lambda)^*
    \end{pmatrix}
    J K (\lambda)
    = \begin{pmatrix}
    0 \\ - \int_c^b \mathbf{R}^b (\lambda)^* \Phi (\xi; \lambda)^* \mathbb{B}_{\lambda} (\xi; \lambda) f (\xi) d\xi
    \end{pmatrix}.
\end{equation*}

At this point, we set 
\begin{equation*}
    \mathbb{E} (\lambda) 
    :=  \begin{pmatrix}
    J \gamma^* & \mathbf{R}^b (\lambda)
    \end{pmatrix}, 
\end{equation*}
and claim that since $0 \notin \sigma (\mathcal{T}_{c, b}^{\gamma} (\lambda))$, 
it must be the case that $\mathbb{E} (\lambda)$ is 
non-singular. To see this, we observe that solutions of 
(\ref{extension-hammy}) satisfying the 
boundary condition at $x = c$ are linear combinations of the columns 
of $\Phi (x; \lambda) J \gamma^*$, while solutions satisfying 
(\ref{boundary-b-equation-lambda}) are linear combinations of the 
columns of $\Phi (x; \lambda) \mathbf{R}^b (\lambda)$. It follows
that the matrix $(\Phi (x; \lambda) J \gamma^*)^* J \Phi (x; \lambda) \mathbf{R}^b (\lambda)$
is singular if and only if $0 \in \sigma_{\pt} (\mathcal{T}_{c, b}^{\gamma} (\lambda))$. 
But 
\begin{equation*}
    (\Phi (x; \lambda) J \gamma^*)^* J \Phi (x; \lambda) \mathbf{R}^b (\lambda)
    = (J \gamma^*)^* J \mathbf{R}_b (\lambda),
\end{equation*}
and we know from Lemma 2.2 in \cite{HS2} that
$(J \gamma^*)^* J \mathbf{R}_b (\lambda)$ is non-singular if and only if 
$\mathbb{E} (\lambda)$ is non-singular. In particular, 
if $0 \notin \sigma (\mathcal{T}_{c, b}^{\gamma} (\lambda))$ then
$\mathbb{E} (\lambda)$ must be non-singular. This observation allows 
us to solve for 
\begin{equation*}
\begin{aligned}
    K(\lambda) 
    &= -J (\mathbb{E} (\lambda)^*)^{-1}
    \begin{pmatrix}
    0 \\ - \int_c^b \mathbf{R}^b (\lambda)^* \Phi (\xi; \lambda)^* \mathbb{B}_{\lambda} (\xi; \lambda) f (\xi) d\xi
    \end{pmatrix} \\
    &= J (\mathbb{E} (\lambda)^*)^{-1} 
    \int_c^b \begin{pmatrix} 0 & \mathbf{R}^b (\lambda)^* \end{pmatrix}^* 
    \Phi (\xi; \lambda)^* \mathbb{B}_{\lambda} (\xi; \lambda) f (\xi) d\xi.
\end{aligned}
\end{equation*}

Upon substitution of $K (\lambda)$ into (\ref{green-y1}), we obtain the relation
\begin{equation*}
\begin{aligned}
y(x; \lambda) &= - \Phi (x; \lambda) \int_c^x J \Phi (\xi; \lambda)^* \mathbb{B}_{\lambda} (\xi; \lambda) f(\xi) d\xi \\
&+ \Phi (x; \lambda) J (\mathbb{E} (\lambda)^*)^{-1}
    \int_c^b \begin{pmatrix}
    0 & \mathbf{R}^b (\lambda)
    \end{pmatrix}^* \Phi (\xi; \lambda)^* \mathbb{B}_{\lambda} (\xi; \lambda) f(\xi) d\xi \\
&= - \Phi (x; \lambda) J (\mathbb{E} (\lambda)^*)^{-1} \mathbb{E} (\lambda)^* 
\int_c^x \Phi (\xi; \lambda)^* \mathbb{B}_{\lambda} (\xi; \lambda) f(\xi) d\xi \\
&+ \Phi (x; \lambda) J (\mathbb{E} (\lambda)^*)^{-1}
    \int_c^b \begin{pmatrix}
    0 & \mathbf{R}^b (\lambda)
    \end{pmatrix}^* \Phi (\xi; \lambda)^* \mathbb{B}_{\lambda} (\xi; \lambda) f(\xi) d\xi.
\end{aligned}
\end{equation*} 
Continuing with this calculation, we next see that 
\begin{equation*}
    \begin{aligned}
    y(x; \lambda) &= - \Phi (x; \lambda) J (\mathbb{E} (\lambda)^*)^{-1} 
    \begin{pmatrix}
    J \gamma^* & 0
    \end{pmatrix}^*
    \int_c^x \Phi (\xi; \lambda)^* \mathbb{B}_{\lambda} (\xi; \lambda) f(\xi) d\xi \\
    &- \Phi (x; \lambda) J (\mathbb{E} (\lambda)^*)^{-1} 
    \begin{pmatrix}
    0 & \mathbf{R}^b (\lambda)
    \end{pmatrix}^*
    \int_c^x \Phi (\xi; \lambda)^* \mathbb{B}_{\lambda} (\xi; \lambda) f(\xi) d\xi \\
    &+ \Phi (x; \lambda) J (\mathbb{E} (\lambda)^*)^{-1}
    \begin{pmatrix}
    0 & \mathbf{R}^b (\lambda)
    \end{pmatrix}^*
    \int_c^b \Phi (\xi; \lambda)^* \mathbb{B}_{\lambda} (\xi; \lambda) f(\xi) d\xi \\
    &= - \Phi (x; \lambda) J (\mathbb{E} (\lambda)^*)^{-1} 
    \begin{pmatrix}
    J \gamma^* & 0
    \end{pmatrix}^*
    \int_c^x \Phi (\xi; \lambda)^* \mathbb{B}_{\lambda} (\xi; \lambda) f(\xi) d\xi \\
    &+ \Phi (x; \lambda) J (\mathbb{E} (\lambda)^*)^{-1}
    \begin{pmatrix}
    0 & \mathbf{R}^b (\lambda)
    \end{pmatrix}^*
    \int_x^b \Phi (\xi; \lambda)^* \mathbb{B}_{\lambda} (\xi; \lambda) f(\xi) d\xi. 
    \end{aligned}
\end{equation*}
It follows, by inspection, that 
\begin{equation*}
    G^{\gamma}_{c, b} (x, \xi; \lambda)
    = \begin{cases}
    - \Phi (x; \lambda) J (\mathbb{E} (\lambda)^*)^{-1} 
    \begin{pmatrix}
    J \gamma^* & 0
    \end{pmatrix}^*
    \Phi (\xi; \lambda)^* & c < \xi < x < b \\
    \Phi (x; \lambda) J (\mathbb{E} (\lambda)^*)^{-1}
    \begin{pmatrix}
    0 & \mathbf{R}^b (\lambda)
    \end{pmatrix}^*
    \Phi (\xi; \lambda)^* & c < x < \xi < b.
    \end{cases}
\end{equation*}

Last, it's convenient to note that we can express 
$G^{\gamma}_{c, b} (x, \xi; \lambda)$ in a more symmetric form. 
To see this, we first observe that 
\begin{equation*}
    \begin{aligned}
    \mathbb{E} (\lambda)^* J \mathbb{E} (\lambda)
    &= \begin{pmatrix}
    - \gamma J \\ \mathbf{R}^b (\lambda)^*
    \end{pmatrix} J
    \begin{pmatrix}
    J \gamma^* & \mathbf{R}^b (\lambda)
    \end{pmatrix} \\
    &= \begin{pmatrix}
    \gamma J \gamma^* & \gamma \mathbf{R}^b (\lambda) \\
    - \mathbf{R}^b (\lambda)^* \gamma^* & \mathbf{R}^b (\lambda)^* J \mathbf{R}^b (\lambda)
    \end{pmatrix}
    = \begin{pmatrix}
    0 & \gamma \mathbf{R}^b (\lambda) \\
    - (\gamma \mathbf{R}^b (\lambda))^*  & 0
    \end{pmatrix},
    \end{aligned}
\end{equation*}
where we've used the observations that $J \gamma^*$ and 
$\mathbf{R}^b (\lambda)$ are frames for Lagrangian subspaces of 
$\mathbb{C}^{2n}$. Here, $\gamma \mathbf{R}^b (\lambda) = (J \gamma^*)^* J \mathbf{R}^b (\lambda)$,
and we've already seen that this matrix is non-singular so long
as $0 \notin \sigma (\mathcal{T}^{\gamma}_{c, b} (\lambda))$. This allows 
us to write 
\begin{equation} \label{inverse-matrix}
   (\mathbb{E} (\lambda)^* J \mathbb{E} (\lambda))^{-1}
   = \begin{pmatrix}
    0 & - ((\gamma \mathbf{R}^b (\lambda))^*)^{-1} \\
    (\gamma \mathbf{R}^b (\lambda))^{-1}  & 0
    \end{pmatrix}.
\end{equation}
It follows that 
\begin{equation} \label{cf1}
    \begin{aligned}
    - &\begin{pmatrix}
    J \gamma^* & 0
    \end{pmatrix}
    \mathbb{E} (\lambda)^{-1} J (\mathbb{E} (\lambda)^*)^{-1} 
    \begin{pmatrix}
    0 & \mathbf{R}^b (\lambda)
    \end{pmatrix}^* \\
    & = \begin{pmatrix}
    J \gamma^* & 0
    \end{pmatrix}
    \begin{pmatrix}
    0 & - ((\gamma \mathbf{R}^b (\lambda))^*)^{-1} \\
    (\gamma \mathbf{R}^b (\lambda))^{-1}  & 0
    \end{pmatrix}
    \begin{pmatrix}
    0 \\ \mathbf{R}^b (\lambda)^*
    \end{pmatrix} \\
    &= - \begin{pmatrix}
    J \gamma^* & 0
    \end{pmatrix}
    \begin{pmatrix}
    ((\gamma \mathbf{R}^b (\lambda))^*)^{-1} \mathbf{R}^b (\lambda)^* \\ 0
    \end{pmatrix}
    = - (J \gamma^*) (\gamma \mathbf{R}^b (\lambda)^*)^{-1} \mathbf{R}^b (\lambda)^*.
    \end{aligned}
\end{equation}

On the other hand, (\ref{inverse-matrix}) also 
allows us to write 
\begin{equation*}
    (\mathbb{E} (\lambda)^*)^{-1}
    = J \mathbb{E} (\lambda)
    \begin{pmatrix}
    0 & - ((\gamma \mathbf{R}^b (\lambda))^*)^{-1} \\
    (\gamma \mathbf{R}^b (\lambda))^{-1}  & 0
    \end{pmatrix},
\end{equation*}
from which we see that 
\begin{equation} \label{cf2}
\begin{aligned}
(\mathbb{E} (\lambda)^*)^{-1}  
    &\begin{pmatrix}
    0 & \mathbf{R}^b (\lambda)
    \end{pmatrix}^* 
= J \mathbb{E} (\lambda)
    \begin{pmatrix}
    0 & - ((\gamma \mathbf{R}^b (\lambda))^*)^{-1} \\
    (\gamma \mathbf{R}^b (\lambda))^{-1}  & 0
    \end{pmatrix}
    \begin{pmatrix}
    0 \\ \mathbf{R}^b (\lambda)^*
    \end{pmatrix} \\
    &= J 
    \begin{pmatrix}
    J \gamma^* & \mathbf{R}^b (\lambda)
    \end{pmatrix}
    \begin{pmatrix}
    - ((\gamma \mathbf{R}^b (\lambda))^*)^{-1} \mathbf{R}^b (\lambda)^* \\ 0
    \end{pmatrix}
    = \gamma^* ((\gamma \mathbf{R}^b (\lambda))^*)^{-1} \mathbf{R}^b (\lambda)^*.
\end{aligned}
\end{equation}
Comparing (\ref{cf1}) and (\ref{cf2}), we see that 
\begin{equation*}
    J (\mathbb{E} (\lambda)^*)^{-1}  
    \begin{pmatrix}
    0 & \mathbf{R}^b (\lambda)
    \end{pmatrix}^* 
    =  \begin{pmatrix}
    J \gamma^* & 0
    \end{pmatrix} 
    \mathbb{E} (\lambda)^{-1} 
    J (\mathbb{E} (\lambda)^*)^{-1}  
    \begin{pmatrix}
    0 & \mathbf{R}^b (\lambda)
    \end{pmatrix}^*.
\end{equation*}
We will set 
\begin{equation*}
    \mathbb{M} (\lambda) := \mathbb{E} (\lambda)^{-1} J (\mathbb{E} (\lambda)^*)^{-1}, 
\end{equation*}
from which we observe that 
\begin{equation*}
    \mathbb{M} (\lambda)^* = - \mathbb{M} (\lambda).
\end{equation*}
For $c < x < \xi < b$, we will re-write 
$G^{\gamma}_{c, b} (x, \xi; \lambda)$ by using the relation 
\begin{equation*}
    J (\mathbb{E} (\lambda)^*)^{-1}  
    \begin{pmatrix}
    0 & \mathbf{R}^b (\lambda)
    \end{pmatrix}^* 
    = \begin{pmatrix}
    J \gamma^* & 0
    \end{pmatrix} 
    \mathbb{M} (\lambda)
    \begin{pmatrix}
    0 & \mathbf{R}^b (\lambda)
    \end{pmatrix}^*.
\end{equation*}
and similarly we will re-write 
$G^{\gamma}_{c, b} (x, \xi; \lambda)$
for $c < \xi < x < b$, by using 
the relation
\begin{equation*}
    J (\mathbb{E} (\lambda)^*)^{-1}  
    \begin{pmatrix}
    J \gamma^* & 0
    \end{pmatrix}^* 
    = \begin{pmatrix}
    0 & \mathbf{R}^b (\lambda)
    \end{pmatrix} 
    \mathbb{M} (\lambda)
    \begin{pmatrix}
    J \gamma^* & 0
    \end{pmatrix}^*.
\end{equation*}
These relations allow us to express $G^{\gamma}_{c, b} (x, \xi; \lambda)$
in precisely the claimed form.
\hfill $\square$

\subsection{Monotonicity as $\lambda$ Varies}
\label{monotonicity-lambda-section}

In this section, we verify that the Maslov index specified 
on the right-hand side of (\ref{full-count-alpha}) is a monotonic
count of crossing points, each negatively directed. From 
Lemma \ref{monotonicity1}, we know that the signs of the associated crossing points are determined 
by the matrices 
\begin{equation} \label{former}
    -\mathbf{X}_{\alpha} (c; \lambda)^* J \partial_{\lambda} \mathbf{X}_{\alpha} (c; \lambda)
\end{equation}
and 
\begin{equation} \label{latter}
    \mathbf{X}_{b} (c; \lambda)^* J \partial_{\lambda} \mathbf{X}_{b} (c; \lambda).
\end{equation}
We've already seen from our analysis of the top shelf that (\ref{former})
is negative definite for all $c \in (a, b)$, so we focus here on making 
a similar conclusion about (\ref{latter}). For this, we recall that 
the columns of $\mathbf{X}_b (x; \lambda)$ comprise the basis elements
for $\ell_b (x; \lambda)$ described in Lemma \ref{continuation-lemma}
and extended in Lemma \ref{lemma2-11prime}. 
By construction, these basis elements are differentiable in $\lambda$
on the intervals $(\lambda_1, \lambda_*^{1, 2})$, $(\lambda_*^{1, 2}, \lambda_*^{2, 3})$,
..., $(\lambda_*^{N-2, N-1}, \lambda_*^{N-1, N})$, $(\lambda_*^{N-1, N}, \lambda_2)$;
more precisely, on $(\lambda_1, \lambda_*^{1, 2})$ the columns of 
$\mathbf{X}_b (x; \lambda)$ are differentiable extensions of the basis 
elements $\{u_j^b (x; \lambda_*^1)\}$, on $(\lambda_*^{1, 2}, \lambda_*^{2, 3})$
the columns of $\mathbf{X}_b (x; \lambda)$ are differentiable extensions of the basis 
elements $\{u_j^b (x; \lambda_*^2)\}$, and so on, with the values 
$\{\lambda_*^j\}_{j=1}^N$ as specified in the proof of Lemma \ref{lemma2-11prime}.
Here, we recall that $\lambda_*^1 = \lambda_1$, $\lambda_*^N = \lambda_2$,
and $\lambda_*^j \in (\lambda_*^{j-1, j}, \lambda_*^{j, j+1})$ for all 
$j \in \{2, \dots, N-1\}$. (The precise location of $\lambda_*^j$ isn't 
required for the argument.)  In addition, we know from Lemma 
\ref{continuation-lemma}, that with this construction we have the relation
\begin{equation} \label{diff-limit}
    \lim_{x \to b^-} \mathbf{X}_b (x; \lambda_*^j)^* J (\partial_{\lambda} \mathbf{X}_b) (x; \lambda_*^j) = 0 
\end{equation}
for all $j \in \{1, 2, \dots, n\}$. 

In order to understand rotation as $\lambda$ varies near $\lambda_*^j$, we first use 
(\ref{diff1}) (from Lemma \ref{continuation-lemma}) to compute (precisely as with the corresponding calculation for 
$\mathbf{X}_{\alpha} (x; \lambda)$ in our analysis of the top shelf in the proof
of Theorem \ref{regular-singular-theorem})
\begin{equation}
    \frac{\partial}{\partial x} \mathbf{X}_b (x; \lambda_*^j)^* J (\partial_{\lambda} \mathbf{X}_b) (x; \lambda_*^j)
    = \mathbf{X}_b (x; \lambda_*^j)^* \mathbb{B}_{\lambda} (x; \lambda_*^j) \mathbf{X}_b (x; \lambda_*^j).
\end{equation}
Integrating on $(c, x)$, we can write 
\begin{equation*}
    \mathbf{X}_b (x; \lambda_*^j)^* J (\partial_{\lambda} \mathbf{X}_b) (x; \lambda_*^j)
    = \mathbf{X}_b (c; \lambda_*^j)^* J (\partial_{\lambda} \mathbf{X}_b) (c; \lambda_*^j)
    + \int_c^x \mathbf{X}_b (\xi; \lambda_*^j)^* \mathbb{B}_{\lambda} (\xi; \lambda_*^j) \mathbf{X}_b (\xi; \lambda_*^j) d \xi.
\end{equation*}
Upon taking the limit as $x \to b^-$ and using (\ref{diff-limit}), we see that 
\begin{equation} \label{appendix-matrix}
    \mathbf{X}_b (c; \lambda_*^j)^* J (\partial_{\lambda} \mathbf{X}_b) (c; \lambda_*^j)
    = - \int_c^b \mathbf{X}_b (\xi; \lambda_*^j)^* \mathbb{B}_{\lambda} (\xi; \lambda_*^j) \mathbf{X}_b (\xi; \lambda_*^j) d \xi,
\end{equation}
allowing us to conclude, similarly as we did with 
$\mathbf{X}_\alpha (c; \lambda)^* J \partial_{\lambda} \mathbf{X}_{\alpha} (c; \lambda)$
in the proof of Theorem \ref{regular-singular-theorem}, that the matrix on the 
left-hand side of (\ref{appendix-matrix}) is negative definite for all 
$c \in (a, b)$, and by continuity in $\lambda$ that 
$\mathbf{X}_b (c; \lambda)^* J \partial_{\lambda} \mathbf{X}_b (c; \lambda)$ is 
negative definite for all $\lambda$ sufficiently close to $\lambda_*^j$. Possibly by taking a finer 
partition of $[\lambda_1, \lambda_2]$ in the proof of Lemma \ref{lemma2-11prime}
(i.e., by taking $N$ larger and the associated radii smaller), we can ensure 
in this way that $\mathbf{X}_b (c; \lambda)^* J \partial_{\lambda} \mathbf{X}_b (c; \lambda)$
is negative definite on each interval in our partition, 
$(\lambda_1, \lambda_*^{1, 2})$, $(\lambda_*^{1, 2}, \lambda_*^{2, 3})$,
..., $(\lambda_*^{N-2, N-1}, \lambda_*^{N-1, N})$, $(\lambda_*^{N-1, N}, \lambda_2)$. 
In this way, we find that the direction of crossings on each of these intervals 
is negative, and since these intervals partition $[\lambda_1, \lambda_2]$,
that the direction of all crossings on $[\lambda_1, \lambda_2]$ is negative
(as $\lambda$ increases).

\subsection{Proof of Lemma \ref{assumption-d-lemma}}
\label{assumption-D-section}

We follow the proof of Theorem V.2.2 from 
\cite{Krall2002}. We will proceed for the case $m_b (\mu, \lambda)$, noting
that the case $m_a (\mu, \lambda)$ is similar. 

First, we fix any $\lambda \in I$, and for the ODE 
\begin{equation*}
    J y' - \mathbb{B} (x; \lambda) y - \mu \mathbb{B}_{\lambda} (x; \lambda) y
    = \mathbb{B}_{\lambda} (x; \lambda) f,
\end{equation*}
with $f \in L^2_{\mathbb{B}_{\lambda}} ((a, b), \mathbb{C}^{2n})$, we 
construct a Green's function $G (x, \xi; \mu, \lambda)$ so that 
\begin{equation*}
    y(x; \mu, \lambda) = 
    \int_c^b G (x, \xi; \mu, \lambda) \mathbb{B}_{\lambda} (\xi; \lambda) f(\xi) d\xi,
\end{equation*}
where $y(x; \mu, \lambda)$ satisfies (\ref{boundary-b}) on the right (for a 
selection of Niessen elements) and a (regular) self-adjoint boundary condition at 
$x = c$. Here, $G (x, \xi; \mu, \lambda)$ can be constructed similarly
as in our development in Section \ref{green-function-section}, or as in Section VI.5
in \cite{Krall2002}. For any fixed $\mu \in \mathbb{C} \backslash \mathbb{R}$,
we set 
\begin{equation*}
    s := \min_{\lambda \in I} \{m_b (\mu; \lambda)\}. 
\end{equation*}
Since $s \ge n$, this minimum is well defined, and since $m_b (\mu; \lambda)$ takes
discrete values, it must occur at some $\lambda_0 \in I$ (possibly at multiple values 
in $I$). 

Let $\{y_i (x; \mu, \lambda_0)\}_{i=1}^s$ denote a linearly independent collection of solutions 
of (\ref{linear-hammy}) (with $\lambda = \lambda_0$) that lie right in $(a, b)$.  
For $\lambda$ near $\lambda_0$, let $z (x; \mu, \lambda)$ solve (\ref{linear-hammy}),
and notice that we can write 
\begin{equation} \label{first-z-equation}
    J z' - \mathbb{B} (x; \lambda_0) z - \mu \mathbb{B}_{\lambda} (x; \lambda_0) z
    = (\mathbb{B} (x; \lambda) - \mathbb{B} (x; \lambda_0)) z
    + \mu (\mathbb{B}_{\lambda} (x; \lambda) - \mathbb{B}_{\lambda} (x; \lambda_0)) z.
\end{equation}
According to Assumption {\bf (E)}, we can write
\begin{equation*}
\mathbb{B} (x; \lambda) - \mathbb{B} (x; \lambda_0)
= \mathbb{B}_{\lambda} (x; \lambda_0) \mathcal{E} (x; \lambda, \lambda_0),
\end{equation*}
and upon differentiating in $\lambda$ 
\begin{equation*}
\mathbb{B}_{\lambda} (x; \lambda)
= \mathbb{B}_{\lambda} (x; \lambda_0) \mathcal{E}_{\lambda} (x; \lambda, \lambda_0),  
\end{equation*}
so that 
\begin{equation*}
    \mathbb{B}_{\lambda} (x; \lambda) - \mathbb{B}_{\lambda} (x; \lambda_0)
    =  \mathbb{B}_{\lambda} (x; \lambda_0) (\mathcal{E}_{\lambda} (x; \lambda, \lambda_0) - I).
\end{equation*}
Combining these observations, we can express (\ref{first-z-equation}) as 
\begin{equation} \label{second-z-equation}
    J z' - \mathbb{B} (x; \lambda_0) z - \mu \mathbb{B}_{\lambda} (x; \lambda_0) z
    = \mathbb{B}_{\lambda} (x; \lambda_0) (\mathcal{E} (x; \lambda, \lambda_0) 
    + \mu (\mathcal{E}_{\lambda} (x; \lambda, \lambda_0) - I)) z.
\end{equation}

Under our assumptions on $\mathcal{E} (x; \lambda, \lambda_0)$
and $\mathcal{E}_{\lambda} (x; \lambda, \lambda_0)$, if we set
\begin{equation*}
    \mathbb{E} (x; \mu, \lambda, \lambda_0)
    := \mathcal{E} (x; \lambda, \lambda_0) 
    + \mu (\mathcal{E}_{\lambda} (x; \lambda, \lambda_0) - I),
\end{equation*}
then the following hold immediately: for each $\lambda$ sufficiently close enough to $\lambda_0$ 
$\mathbb{E} (\cdot; \mu, \lambda; \lambda_0) \in \mathcal{B} (L^2_{\mathbb{B}_{\lambda}} ((a, b), \mathbb{C}^{2n}))$
(i.e., when viewed as a multiplication operator, the matrix function $\mathbb{E} (\cdot; \mu, \lambda; \lambda_0)$
is a bounded linear operator taking $L^2_{\mathbb{B}_{\lambda}} ((a, b), \mathbb{C}^{2n})$
to itself), and the map $\lambda \mapsto \mathbb{E} (\cdot; \mu, \lambda; \lambda_0)$ is continuous 
as a map from a neighborhood of $\lambda_0$ to $\mathcal{B} (L^2_{\mathbb{B}_{\lambda}} ((a, b), \mathbb{C}^{2n}))$,
with additionally 
$\|\mathbb{E} (\cdot; \mu, \lambda, \lambda_0)\| = \mathbf{o} (1)$, $\lambda \to \lambda_0$. 
These observations allow us to conclude that for any $\phi \in L^2_{\mathbb{B}_{\lambda}} ((a, b), \mathbb{C}^{2n})$ 
we have the inequality 
\begin{equation*}
    \|\mathbb{E} (\cdot; \mu, \lambda, \lambda_0) \phi (\cdot)\|_{\mathbb{B}_{\lambda}}
    \le \| \mathbb{E} (\cdot; \mu, \lambda, \lambda_0) \| \| \phi (\cdot) \|_{\mathbb{B}_{\lambda}},
\end{equation*}
and consequently we can take $\lambda$ sufficiently close to $\lambda_0$ so that 
\begin{equation} \label{E-inequality}
    \|\mathbb{E} (\cdot; \mu, \lambda, \lambda_0) \phi (\cdot)\|_{\mathbb{B}_{\lambda}}^2
    \le \frac{|\textrm{Im}\,\mu|}{2}  \| \phi (\cdot) \|_{\mathbb{B}_{\lambda}}^2,
\end{equation}
for all $\phi \in L^2_{\mathbb{B}_{\lambda}} ((a, b), \mathbb{C}^{2n})$. 

At this point, we take $\lambda$ close enough to $\lambda_0$ so that (\ref{E-inequality})
holds, set $S := m_b (\mu; \lambda)$, and let 
$\{z_j (x; \mu, \lambda, \lambda_0)\}_{j = 1}^S$ denote a collection of 
linearly independent solutions to (\ref{first-z-equation}) that lie right
in $(a, b)$. For each 
$j \in \{1, 2, \dots, S\}$, we {\it define} 
\begin{equation*}
\tilde{z}_j (x; \mu, \lambda, \lambda_0)    
:= \int_c^b G (x, \xi; \mu, \lambda_0) \mathbb{B}_{\lambda} (\xi; \lambda_0)
\mathbb{E} (\xi; \mu, \lambda, \lambda_0) z_j (\xi; \mu, \lambda, \lambda_0) d\xi,
\end{equation*}
noting particularly that $\tilde{z}_j (x; \mu, \lambda, \lambda_0)$ solves the 
ODE 
\begin{equation} \label{inhomogeneous-z}
    J \tilde{z}_j' - \mathbb{B} (x; \lambda_0) \tilde{z}_j - \mu \mathbb{B}_{\lambda} (x; \lambda_0) \tilde{z}_j
    = \mathbb{B}_{\lambda} (x; \lambda_0) \mathbb{E} (x; \mu, \lambda, \lambda_0) z_j.
\end{equation}
I.e., $\tilde{z}_j (x; \mu, \lambda, \lambda_0)$ is a particular solution for this 
inhomogeneous equation. The space of homogeneous solutions to (\ref{inhomogeneous-z})
that lie right in $(a, b)$ is spanned by the collection 
$\{y_i (x; \mu, \lambda_0)\}_{i=1}^s$. Since each element in the collection
$\{z_j (x; \mu, \lambda, \lambda_0)\}_{j = 1}^S$ is a solution to 
(\ref{inhomogeneous-z}) that lies right in $(a, b)$, there must exist constants
$\{c_{j i} (\mu, \lambda, \lambda_0)\}_{j, i = 1}^{S, s}$ so that 
\begin{equation*}
z_j (x; \mu, \lambda, \lambda_0)
= \sum_{i=1}^s c_{j i} (\mu, \lambda, \lambda_0) y_i (x; \mu, \lambda_0)
+ \tilde{z}_j (x; \mu, \lambda, \lambda_0), 
\quad j = 1, 2, \dots, S.
\end{equation*}

Suppose now that $S > s$, in which case we can associate the coefficients 
$(c_{j i})$ with an $S \times s$ matrix $C$. It follows that there must 
exist a non-trivial row vector $\beta = (\beta_1, \beta_2, \dots, \beta_S)$,
depending on $\mu$, $\lambda$ and $\lambda_0$, so that 
$\beta C = 0$. We now set 
\begin{equation} \label{phi-defined}
    \phi (x; \mu, \lambda, \lambda_0) 
    := \sum_{j=1}^S \beta_j (\mu, \lambda, \lambda_0) z_j (x; \mu, \lambda, \lambda_0).
\end{equation}
Then by linear independence of the collection 
$\{z_j (x; \mu, \lambda, \lambda_0)\}_{j = 1}^S$, we see that 
$\phi (x; \mu, \lambda, \lambda_0)$ cannot be identically 0. In addition, 
we can now write 
\begin{equation*}
    \begin{aligned}
    \phi (x; \mu, \lambda, \lambda_0) 
    &= \sum_{j=1}^S \beta_j (\mu, \lambda, \lambda_0) z_j (x; \mu, \lambda, \lambda_0) \\
    &= \sum_{j=1}^S \beta_j (\mu, \lambda, \lambda_0) 
    \Big( \sum_{i=1}^s c_{j i} (\mu, \lambda, \lambda_0) y_i (x; \mu, \lambda_0)
+ \tilde{z}_j (x; \mu, \lambda, \lambda_0) \Big) \\
&=  \sum_{i=1}^s \Big( \sum_{j=1}^S \beta_j (\mu, \lambda, \lambda_0) c_{j i} (\mu, \lambda, \lambda_0) \Big) 
y_i (x; \mu, \lambda_0) + \sum_{j=1}^S \beta_j (\mu, \lambda, \lambda_0) \tilde{z}_j (x; \mu, \lambda, \lambda_0).
    \end{aligned}
\end{equation*}
By the selection of $\beta$, we have $\sum_{j=1}^S \beta_j (\mu, \lambda, \lambda_0) c_{j i} (\mu, \lambda, \lambda_0) = 0$,
and this allows us to write 
\begin{equation*}
    \begin{aligned}
    \phi (x; \mu, \lambda, \lambda_0) 
    &= \sum_{j=1}^S \beta_j (\mu, \lambda, \lambda_0) \tilde{z}_j (x; \mu, \lambda, \lambda_0) \\
    &= \sum_{j=1}^S \beta_j (\mu, \lambda, \lambda_0) 
    \int_c^b G (x, \xi; \mu, \lambda_0) \mathbb{B}_{\lambda} (\xi; \lambda_0)
    \mathbb{E} (\xi; \mu, \lambda, \lambda_0) z_j (\xi; \mu, \lambda, \lambda_0) d\xi \\
    &= \int_c^b G (x, \xi; \mu, \lambda_0) \mathbb{B}_{\lambda} (\xi; \lambda_0)
    \mathbb{E} (\xi; \mu, \lambda, \lambda_0) 
    \sum_{j=1}^S \beta_j (\mu, \lambda, \lambda_0) z_j (\xi; \mu, \lambda, \lambda_0) d\xi \\
    &= \int_c^b G (x, \xi; \mu, \lambda_0) \mathbb{B}_{\lambda} (\xi; \lambda_0)
    \mathbb{E} (\xi; \mu, \lambda, \lambda_0) 
    \phi (\xi; \mu, \lambda, \lambda_0) d\xi.
    \end{aligned}
\end{equation*}
In particular, $\phi$ satisfies the equation
\begin{equation*}
     J \phi' - \mathbb{B} (x; \lambda_0) \phi - \mu \mathbb{B}_{\lambda} (x; \lambda_0) \phi
    = \mathbb{B}_{\lambda} (x; \lambda_0) \mathbb{E} (x; \mu, \lambda, \lambda_0) \phi.
\end{equation*}

Precisely as in the proof of Theorem V.2.2 in \cite{Krall2002}, we now make use of the 
following relation (see Theorem VI.6.2 in \cite{Krall2002}): if 
$\phi, f \in L^2_{\mathbb{B}_{\lambda}} ((c, b), \mathbb{C}^{2n})$ and 
$\phi$ solves the equation 
\begin{equation*}
     J \phi' - \mathbb{B} (x; \lambda_0) \phi - \mu \mathbb{B}_{\lambda} (x; \lambda_0) \phi
    = \mathbb{B}_{\lambda} (x; \lambda_0) f,
\end{equation*}
then 
\begin{equation*}
    \int_c^b \phi (x)^* \mathbb{B}_{\lambda} (x; \lambda_0) \phi (x) dx
    \le \frac{1}{|\textrm{Im}\,\mu|} \int_c^b f (x)^* \mathbb{B}_{\lambda} (x; \lambda_0) f (x) dx.
\end{equation*}
For $\phi$ as in (\ref{phi-defined}), this becomes
\begin{equation*}
    \int_c^b \phi (x)^* \mathbb{B}_{\lambda} (x; \lambda_0) \phi (x) dx
    \le \frac{1}{|\textrm{Im}\,\mu|} \int_c^b (\mathbb{E} (x) 
    \phi (x))^* \mathbb{B}_{\lambda} (x; \lambda_0) (\mathbb{E} (x) 
    \phi (x)) dx,
\end{equation*}
where most dependence on the spectral parameters has been suppressed for notational 
brevity. Expressed in terms of the norm $\|\cdot\|_{\mathbb{B}_{\lambda}}$ (on 
$(c, b)$), this becomes
\begin{equation*}
    \|\phi\|_{\mathbb{B}_{\lambda}}^2 
    \le \frac{1}{|\textrm{Im}\,\mu|} 
    \| \mathbb{E} \phi\|_{\mathbb{B}_{\lambda}}^2
    \le \frac{1}{|\textrm{Im}\,\mu|} \|\mathbb{E}\|^2  \|\phi\|_{\mathbb{B}_{\lambda}}^2
    < \frac{1}{2} \|\phi\|_{\mathbb{B}_{\lambda}}^2,
\end{equation*}
where in obtaining this final inequality we've used 
(\ref{E-inequality}). In this way, we arrive at a contradiction to the assumption that 
$S > s$, so it must be the case that $m_b (\mu; \lambda) = s$ for all 
$\lambda$ sufficiently close to $\lambda_0$. 

Finally, using compactness of the interval $[\lambda_1, \lambda_2] \subset I$, we 
can conclude that we must have $m_b (\mu; \lambda) = s$ for all 
$\lambda \in [\lambda_1, \lambda_2]$. Since $m_b (\mu; \lambda)$ is constant
in $\mu$ for each fixed $\lambda$, and constant in $\lambda$ for each 
fixed $\mu$, it must be constant for all 
$(\mu, \lambda) \in (\mathbb{C} \backslash \mathbb{R}) \times [\lambda_1, \lambda_2]$.

\bigskip
{\it Acknowledgments.} A.S. acknowledges support from the 
National Science Foundation under grant DMS-1910820.

\end{document}